\documentclass[10pt]{amsart}

\usepackage{mystyle}

\title{A proof of the Erd\H{o}s--Faber--Lov\'asz conjecture}
\date{January 25, 2023}

\author[Kang]{Dong Yeap Kang}

\author[Kelly]{Tom Kelly}

\author[K\"uhn]{Daniela K\"uhn}

\author[Methuku]{Abhishek Methuku}

\author[Osthus]{Deryk Osthus}

\thanks{This project has received partial funding from the European Research
Council (ERC) under the European Union's Horizon 2020 research and innovation programme (grant agreement no. 786198, D.~K\"uhn and D.~Osthus).
The research leading to these results was also partially supported by the EPSRC, grant nos. EP/N019504/1 (D.~Kang, T. Kelly and D.~K\"uhn) and EP/S00100X/1 (A.~Methuku and D.~Osthus).}

\makeatletter
\def\@tocline#1#2#3#4#5#6#7{\relax
  \ifnum #1>\c@tocdepth \else
    \par \addpenalty\@secpenalty\addvspace{#2}\begingroup \hyphenpenalty\@M
    \@ifempty{#4}{\@tempdima\csname r@tocindent\number#1\endcsname\relax
    }{\@tempdima#4\relax
    }\parindent\z@ \leftskip#3\relax \advance\leftskip\@tempdima\relax
    \rightskip\@pnumwidth plus4em \parfillskip-\@pnumwidth
    #5\leavevmode\hskip-\@tempdima
      \ifcase #1
       \or\or \hskip 1em \or \hskip 2em \else \hskip 3em \fi #6\nobreak\relax
    \hfill\hbox to\@pnumwidth{\@tocpagenum{#7}}\par \nobreak
    \endgroup
  \fi}
\makeatother

\begin{document}

\begin{abstract} \noindent
The Erd\H{o}s--Faber--Lov\'{a}sz conjecture (posed in 1972) states that the chromatic index of any linear hypergraph on $n$ vertices is at most $n$. In this paper, we prove this conjecture for every large $n$. We also provide stability versions of this result, which confirm a prediction of Kahn.
\end{abstract}
\maketitle

\section{Introduction}
\label{intro-section}
Graph and hypergraph colouring problems are central to combinatorics, with applications and connections to many other areas, such as geometry, algorithm design, and information theory.  As one illustrative example, the fundamental Ajtai--Koml{\'{o}}s--Pintz--Spencer--Szemer\'edi (AKPSS) theorem~\cite{AKPSS82} shows that locally sparse uniform hypergraphs have large independent sets.  This was initially designed to disprove the famous Heilbronn conjecture in combinatorial geometry but has found numerous further applications e.g., in coding theory.  The AKPSS theorem was later strengthened by Frieze and Mubayi~\cite{FM13} to show that linear $k$-uniform hypergraphs with $k \ge 3$ have small chromatic number.  Here a hypergraph $\cH$ is \textit{linear} if every two distinct edges of $\cH$ intersect in at most one vertex.

\subsection{The Erd\H{o}s--Faber--Lov\'{a}sz conjecture}
In 1972, Erd\H{o}s, Faber, and Lov\'{a}sz conjectured (see~\cite{erdos1981}) the following equivalent statements. Let $n \in \mathbb{N}$.

\begin{itemize}
    \item[(i)] If $A_1 , \dots , A_n$ are sets of size $n$ such that every pair of them shares at most one element, then the elements of $\bigcup_{i=1}^{n} A_i$ can be coloured by $n$ colours so that all colours appear in each $A_i$.
    
    \item[(ii)] If $G$ is a graph that is the union of $n$ cliques, each having at most $n$ vertices, such that every pair of cliques shares at most one vertex, then the chromatic number of $G$ is at most $n$.
    
    \item[(iii)] If $\cH$ is a linear hypergraph with $n$ vertices, then the chromatic index of $\cH$ is at most $n$.
\end{itemize}

 Here the \textit{chromatic index} $\chi'(\cH)$ of a hypergraph $\cH$ is the smallest number of colours needed to colour the edges of $\cH$ so that any two edges that share a vertex have different colours. The formulation (iii) is the one that we will consider throughout the paper. This can be viewed as the dual version of (i) and (ii). 
 For simplicity, we will refer to this conjecture as the EFL conjecture.

Erd\H{o}s considered this to be `one of his three most favorite combinatorial problems' (see e.g.,~\cite{kahn1997}). The simplicity and elegance of its formulation initially led the authors to believe it to be easily solved (see e.g., the discussion in~\cite{cg1998} and~\cite{erdos1981}). It was initially designed as a simple test case for a more general theory of hypergraph colourings. However, as the difficulty became apparent Erd\H{o}s offered successively increasing rewards for a proof of the conjecture, which eventually reached \$500. 

Previous progress towards the conjecture includes the following results. 
Seymour~\cite{seymour1982} proved that every $n$-vertex linear hypergraph $\cH$ has a matching of size at least $e(\cH)/n$, where $e(\cH)$ is the number of edges in $\cH$. (Note that this immediately follows from the validity of the EFL conjecture, but it is already difficult to prove.) Kahn and Seymour~\cite{KS1992} proved that every $n$-vertex linear hypergraph has fractional chromatic index at most $n$. Chang and Lawler~\cite{chang1988} showed that every $n$-vertex linear hypergraph has chromatic index at most $\lceil 3n/2 - 2 \rceil$. A breakthrough of Kahn~\cite{kahn1992coloring} yielded an approximate version of the conjecture, by showing that every $n$-vertex linear hypergraph has chromatic index at most $n + o(n)$. Recently Faber and Harris~\cite{FH2019} proved the conjecture for linear hypergraphs whose edge sizes range between 3 and $cn^{1/2}$ for a small absolute constant $c>0$. More background and earlier developments related to the EFL conjecture are detailed in the surveys of Kahn~\cite{kahn1995, kahn1997} and of Kayll~\cite{kayll2016}.  See also the recent survey by the authors~\cite{KKKMO2021survey}. 

\subsection{Main results}\label{subsec:intro_main}

In this paper we prove the EFL conjecture for every large $n$.

\begin{theorem}\label{main-thm}
    For every sufficiently large $n$, every linear hypergraph $\cH$ on $n$ vertices has chromatic index at most $n$.
\end{theorem}

There are three constructions for which Theorem~\ref{main-thm} is known to be tight: a complete graph $K_n$ for any odd integer $n$ (and minor modifications thereof\COMMENT{e.g., odd $K_n$ minus an edge.}), a finite projective plane of order $k$ on $n = k^2 + k + 1$ points, and a degenerate plane $\{ \{1,2\},\dots,\{1,n\},\{2,\dots,n\} \}$. Note that the first example has bounded edge size (two), while the other two examples have unbounded edge size as $n$ tends to infinity. This makes the conjecture particularly challenging.

Kahn's proof~\cite{kahn1992coloring} is based on a powerful method known as the R\"odl nibble. Roughly speaking, this method builds a large matching using an iterative probabilistic procedure. It was originally developed by R\"odl~\cite{rodl1985} to prove the Erd\H{o}s--Hanani conjecture~\cite{ErdosHanani} on combinatorial designs. Another famous result based on this method is the Pippenger--Spencer theorem~\cite{PippengerSpencer}, which implies that the chromatic index of any uniform hypergraph $\cH$ of maximum degree $D$ and codegree $o(D)$ is $D+o(D)$. (Note that this in turn implies that the EFL conjecture holds for all large $r$-uniform linear hypergraphs of bounded uniformity $r \geq 3$.) In a seminal paper, Kahn~\cite{kahn1996asymptotically} later developed the approach further to show that the same bound $D+o(D)$ even holds for the list chromatic index (an intermediate result in this direction, which also strengthens the Pippenger--Spencer theorem, was the main ingredient of his proof in~\cite{kahn1992coloring}). The best bound on the $o(D)$ error term for the list chromatic index of such hypergraphs was obtained by Molloy and Reed~\cite{MR2000}, and for the chromatic index, the best bound was proved in~\cite{KKMO2020}.  Our proof will also rely on certain properties of the R\"odl nibble.  

In addition, our proof makes use of powerful colouring results for locally sparse graphs (Theorems~\ref{local-sparsity-lemma} and~\ref{aks-local-sparsity}).
This line of research goes back to Ajtai, Koml{\'{o}}s, and Szemer\'edi~\cite{AKS1980} who (preceding R\"odl~\cite{rodl1985}) developed a very similar semi-random nibble approach to give an upper bound $O(k^2 / \log k)$ on the Ramsey number $R(3,k)$ by finding large independent sets in triangle-free graphs (the matching lower bound $R(3,k) = \Omega(k^2 / \log k)$ was later established by Kim~\cite{kim1995_ramsey}, also using a semi-random approach). The above Ramsey bound by Ajtai, Koml{\'{o}}s and Szemer\'edi was subsequently strengthened by a highly influential result of Johansson~\cite{johansson1996}, who showed that triangle-free graphs of maximum degree $\Delta$ have chromatic number $O(\Delta/\log \Delta)$ (and a related result was proved independently by Kim~\cite{kim1995}).  The result of Frieze and Mubayi~\cite{FM13} mentioned at the start of Section~\ref{intro-section} is one of several analogues and generalizations of Johansson's Theorem.
It also turns out that the condition of being triangle-free can be relaxed (in various ways) to being `locally sparse'~\cite{AIS19, AKS99, DKPS20, Vu02}. 
We will be able to apply such results to suitable parts of the line graph of our given linear hypergraph $\cH$.  

One step in our proof involves what may be considered a `vertex absorption' argument; here certain vertices not covered by a matching produced by the R\"odl nibble are `absorbed' into the matching to form a colour class. (Vertex) absorption as a systematic approach was introduced by R\"odl, Ruci\'nski, and Szemer\'edi~\cite{rrs2006} to find spanning structures in hypergraphs (with precursors including~\cite{egp1991, krivelevich1997}). Absorption ideas were first used for edge decomposition problems in~\cite{kuhn2013hamilton} to solve Kelly's conjecture on tournament decompositions. (Note that the EFL conjecture can also be viewed as a (hyper-)edge decomposition problem --
we will use this interpretation to colour `small' edges, see Section~\ref{sec:overview}.)
We will make use of an application of the main result of~\cite{kuhn2013hamilton} to the overfull subgraph conjecture (which was derived in~\cite{GKO2016}).

Kahn~\cite{kahn1995} predicted\COMMENT{In fact, he conjectured this in an informal way -- he said ``there ought to be some slack in the bound away from these extremes", as well as mentioning the strengthening of EFL by Berge-F\"uredi-Meyniel. Unfortunately, it seems to require some more ideas to prove Berge-F\"uredi-Meyniel conjecture,  even for bounded edge size case, as the neighbours of the high-degree vertices in $U$ are not longer well-distributed (i.e., cannot guarantee the regularity conditions of the absorbers), where our absorption lemmas rely much on this.} that the bound in the EFL conjecture can be improved if $\cH$ is far from being one of the extremal examples mentioned above. We confirm his prediction by proving a stability result as follows.

\COMMENT{Moreover, perhaps we can make the chromatic index $o(n)$ if $|e| \in [d_1 , g(n)]$, if $1 \ll d_1 \leq g(n) \ll \sqrt{n}$ -- it is also interesting to see how we can bound the chromatic index depending on the choice of $g(n)$.}

\begin{theorem}\label{main-thm-2}
    For every $\delta > 0$, there exist $n_0, \sigma > 0$ such that the following holds. For any $n \geq n_0$, if $\cH$ is an $n$-vertex linear hypergraph with maximum degree at most $(1-\delta)n$ such that the number of edges of size $(1 \pm \delta)\sqrt{n}$ in $\cH$ is at most $(1-3\delta)n$, then the chromatic index of $\cH$ is at most $(1-\sigma)n$.
\end{theorem}

Recall that in an $n$-vertex projective plane, every edge has size close to $\sqrt{n}$, and every pair of its vertices is covered by an edge. 
In an $n$-vertex linear hypergraph, it is easy to see that the fact that the number of edges of size $(1 \pm o(1)) \sqrt{n}$ is at least $(1-o(1))n$ is equivalent to the fact that $(1-o(1))\binom{n}{2}$ pairs of vertices are covered by the edges of size $(1 \pm o(1)) \sqrt{n}$. Thus, Theorem~\ref{main-thm-2} says that if an $n$-vertex linear hypergraph $\cH$ is not too close to a projective plane and has maximum degree at most $(1-o(1))n$, then it has chromatic index at most $(1-o(1))n$. 

Our methods also easily give a further improvement for hypergraphs which are very far from being extremal.

\begin{theorem}\label{weak-stability-thm}
    For every $\eps > 0$, there exist $n_0, \eta > 0$ such that the following holds.  For any $n \geq n_0$, if $\cH$ is an $n$-vertex linear hypergraph with maximum degree at most $\eta n$ and no edge $e\in \cH$ such that $\eta \sqrt n < |e| < \sqrt n / \eta$, then the chromatic index of $\cH$ is at most $\eps n$.
\end{theorem}

The proof of Theorem~\ref{main-thm} can be turned into a randomized polynomial-time algorithm.
The necessary modifications are discussed in detail in~\cite{KKKMO2021_focs}.

\subsection{Related results and open problems}

Formulation~(ii) of the EFL conjecture can be viewed as a statement implying that a local restriction on the local density of a graph has a strong influence on its global structure.  A famous example where this is not the case is the construction by Erd\H os of graphs of high girth and high chromatic number.  Another well known instance where this fails is a bipartite version of the EFL conjecture due to Alon, Saks, and Seymour (see Kahn~\cite{kahn1991}); they conjectured that if a graph $G$ can be decomposed into $k$ edge-disjoint bipartite graphs, then the chromatic number of $G$ is at most $k+1$. This conjecture was a generalisation of the Graham--Pollak theorem~\cite{gp1972} on edge decompositions of complete graphs into bipartite graphs, which has applications to communication complexity. However, this conjecture was disproved by Huang and Sudakov~\cite{HS2012} in a strong form, i.e., it is not even close to being true.

A natural generalization of the EFL conjecture was suggested by Berge~\cite{berge1989} and F\"{u}redi~\cite{furedi1986}; if $\cH$ is a linear hypergraph with vertex set $V$, then the chromatic index of $\cH$ is at most $\max_{v\in V}|\bigcup_{e\ni v}e|$.  This would be a direct generalization of Vizing's theorem on the chromatic index of graphs.  Another beautiful question leading on from Theorem~\ref{main-thm} is whether it can be extended to list colourings (this was first raised by Faber~\cite{F17}). Many more related results and problems, along with a detailed sketch of some aspects of the proof of the EFL conjecture are given in the recent survey by the authors~\cite{KKKMO2021survey}. 

After completion of the initial version of the paper,  the authors were able to build on the methods of Sections~\ref{fpp-extremal-section} and~\ref{large-edge-section} to solve a problem of Erd\H{o}s on colouring cliques overlapping in a bounded number of vertices~\cite{KKKMO2021_tEFL}.
Moreover, based on a nibble argument, Kelly, K\"uhn and Osthus~\cite{KKO2021} proved asymptotically tight bounds on the list chromatic number of the union of nearly disjoint graphs of bounded maximum degree
(which confirms a special case of a conjecture of Vu on colouring graphs of small codegree~\cite{Vu02} and generalizes Kahn’s bound on the list chromatic index of linear uniform hypergraphs of bounded maximum degree~\cite{kahn1996asymptotically}).

\section{Overview}
\label{sec:overview}

In this section, we provide an overview of the proof of Theorem~\ref{main-thm}.  More detailed sketches of certain aspects of the proof can also be found in the recent survey by the authors~\cite[Section 5]{KKKMO2021survey}.

\subsection{Colouring linear hypergraphs with bounded edge sizes}\label{subsection:small-edge-overview}

Here, we discuss the proof of Theorem~\ref{main-thm} in the special case when all edges of $\cH$ have bounded size. 
We will construct our colouring by successively constructing suitable matchings (which form the colour classes).
The key observation motivating our approach is the following: if $\mathrm{(i)}$ a vertex $x$ has very high degree
(i.e.~close to $n$), then almost all of these matchings need to cover $x$, which is quite challenging.
If on the other hand $\mathrm{(ii)}$ the degree of $x$ is not too high, then we have considerable flexibility as to which of the
matchings will cover $x$. This makes it feasible to colour the edges at $x$ via a probabilistic nibble approach.
But since $\cH$ is linear, the seemingly more difficult case $\mathrm{(i)}$ can only occur if almost all edges incident to $x$ have size two, i.e.~are \emph{graph} edges. This allows us to build on graph theoretical techniques to colour certain leftover edges incident to $x$. We now make this intuition more precise.

In this subsection, we fix constants satisfying the hierarchy
\begin{equation*}
    0 < 1/n_0  \ll \xi \ll 1 / r \ll \gamma \ll \eps \ll \rho \ll 1,
\end{equation*}
we let $n \geq n_0$, and we let $\cH$ be an $n$-vertex linear hypergraph such that every $e\in \cH$ satisfies $2 \leq |e| \leq r$. We first describe the ideas which already lead to the near-optimal bound $\chi'(\cH) \leq n + 1$.  A more detailed sketch of this proof is provided in the authors' survey~\cite[Section 5.1]{KKKMO2021survey}.

Let $G$ be the graph with $V(G) \coloneqq V(\cH)$ and $E(G) \coloneqq \{e \in \cH : |e| = 2\}$.  The first step of the proof is to include every edge of $G$ in a `reservoir' $R$ independently with probability $1/2$ that we will use for `absorption'.  With high probability, each $v\in V(\cH)$ satisfies $d_{R}(v) = d_G(v)/2 \pm \xi n$.
Since $\cH$ is linear, this easily implies that $\Delta(\cH\setminus R) \leq (1/2 + \xi)n$.\COMMENT{$d_{\cH \setminus R}(v) = d_{\cH \setminus G}(v)+ d_{G \setminus R}(v) \approx d_{\cH \setminus G}(v) + d_G(v)/2 = 1/2(d_G(v) + 2d_{\cH \setminus G}(v)) \leq n/2$.} So by the Pippenger--Spencer theorem~\cite{PippengerSpencer}, we obtain the nearly optimal bound $\chi'(\cH\setminus R) \leq (1/2 + \gamma)n$.  Now using $R$ as a `vertex-absorber', ideally, we would like to extend the colour classes of $\cH \setminus R$ to cover as many vertices of $U$ as possible, where  $U\coloneqq\{u\in V(\cH) : d_G(u) \geq (1 - \eps)n\}$. This would allow us to control the maximum degree in the hypergraph consisting of uncoloured edges, so that it can then be coloured with few colours. To that end, we need the following important definition.

\begin{restatable}[Perfect and nearly-perfect coverage]{definition}{perfectcoverage}\label{def:coverage}
    Let $\cH$ be a linear multi-hypergraph, let $\cN$ be a set of edge-disjoint matchings in $\cH$, and let $S \subseteq U \subseteq V(\cH)$.
    
    \begin{itemize}
        \item We say $\cN$ has \emph{perfect coverage} of $U$ if each $N \in \cN$ covers $U$.
        
        \item We say $\cN$ has \emph{nearly-perfect coverage of $U$ with defects in $S$} if
        \begin{enumerate}[(i)]
            \item each $u \in U$ is covered by at least $|\cN|-1$ matchings in $\cN$ and
    
            \item each $N \in \cN$ covers all but at most one vertex in $U$ such that $U \setminus V(N) \subseteq S$.
        \end{enumerate}
    \end{itemize}
\end{restatable}

We will construct some $\cH'\subseteq \cH$ and a proper edge-colouring $\psi : \cH' \rightarrow C$ such that 
\begin{enumerate}[(a)]
    \item \label{(a)} $\cH'\supseteq  \cH\setminus R$,
    \item $|C| = (1/2 + \gamma)n$, and
    \item \label{(c)} the set of colour classes $\{\psi^{-1}(c) : c\in C\}$ has nearly-perfect coverage of $U$ (with defects in $U$).
\end{enumerate}

Crucially, \ref{(a)} means that $\cH\setminus \cH'$ is a graph and \ref{(a)} and \ref{(c)} imply that it satisfies $\Delta(\cH \setminus \cH') \leq n - |C|$. (Indeed, every vertex $u\in U$ satisfies $d_\cH(u) \leq n - 1$ and $d_{\cH'}(u) \geq |C|-1$ as $u$ is covered by all but at most one of the colour classes of $\psi$, and every vertex $v\notin U$ satisfies $d_{\cH \setminus \cH'}(v) \leq d_{R}(v) \leq ((1 - \eps)/2 + \xi)n < n - |C|$.)  Therefore, Vizing's theorem~\cite{vizing1965} implies that $\chi'(\cH\setminus \cH') \leq \Delta(\cH\setminus \cH') + 1 \leq n - |C| + 1$, so altogether we have $\chi'(\cH) \leq \chi'(\cH') + \chi'(\cH\setminus \cH') \leq n + 1$, as claimed.

To construct $\cH'$ and $\psi$ we iteratively apply the R\"{o}dl nibble to (the leftover of) $\cH \setminus R$ to successively construct large matchings $N_i$ which are then removed from $\cH \setminus R$ and form part of the colour classes of $\psi$. (The R\"odl nibble is applied implicitly via Corollary~\ref{cor:sparse_egj}, which guarantees a large matching in a suitable hypergraph.) Crucially, each matching $N_i$ exhibits pseudorandom properties, which allow us to use some edges of $R$ to extend $N_i$ into a matching $M_i$ (which will form a colour class of $\psi$) with nearly-perfect coverage of $U$, as desired. (This is why we apply the R\"odl nibble in our proof rather than the Pippenger--Spencer theorem.) Thus, $R$ acts as a `vertex-absorber' for $U \setminus V(N_i)$ and the final edge decomposition of the unused edges of $R$ into matchings is achieved by Vizing's theorem. (Actually, this only works if $\cH \setminus R$ is nearly regular, which is not necessarily the case. Thus, we first embed $\cH \setminus R$ in a suitable nearly regular hypergraph $\cH^*$ and prove that the respective matchings in $\cH^*$ have nearly-perfect coverage of $U$, which suffices for our purposes.)

Let us now discuss how to improve the bound $\chi'(\cH) \leq n + 1$ to $\chi'(\cH) \leq n$.  Let $S \coloneqq \{u \in U : d_G(u) < n - 1\}$, and note that if $\{\psi^{-1}(c) : c \in C\}$ has either perfect coverage of $U$, or nearly-perfect coverage of $U$ with defects in $S$, then $\Delta(\cH\setminus \cH') \leq n - 1 - |C|$.  In this case, we may use the same argument as before with Vizing's theorem to obtain $\chi'(\cH) \leq n$.  However, it is not always possible to find such a colouring.  For example, if $\cH$ is a complete graph $K_n$ for odd $n$ (which is one of the extremal examples for Theorem~\ref{main-thm}), then $U = V(\cH)$ and $S = \varnothing$, so it is not possible for even a single colour class to have nearly-perfect coverage of $U$ with defects in $S$. However, we can adapt the above nibble-absorption-Vizing approach to work whenever $\cH$ is not `close' to $K_n$ in the following sense.

\begin{restatable}[$(\rho, \eps)$-full]{definition}{full}\label{def:full}
\label{def:rho-eps-full}
  Let $\cH$ be an $n$-vertex linear hypergraph, and let $G$ be the graph with $V(G) \coloneqq V(\cH)$ and $E(G) \coloneqq \{e \in \cH : |e| = 2\}$.  For $\eps, \rho \in (0, 1)$, $\cH$ is \textit{$(\rho,\eps)$-full} if
  \begin{itemize}
  \item $|\{u \in V(\cH) : d_G(u) \geq (1 - \eps)n\}| \geq (1 - 10 \eps) n$, and
  \item $|\{v \in V(\cH) : d_G(v) = n - 1\}| \geq (\rho-15\eps)n$.
  \end{itemize}
\end{restatable}

As mentioned above, when $\cH$ is not $(\rho, \eps)$-full we can adapt the nibble-absorption-Vizing approach to show that $\chi'(\cH) \leq n$ (with a reservoir of density $\rho$ rather than $1/2$). If $\cH$ is $(\rho, \eps)$-full then we will ensure that the leftover $\cH \setminus \cH' \subseteq R$ is a quasirandom almost regular graph (which involves a more careful choice of $R$ -- again it will have density close to $\rho$ rather than $1/2$ but now it consists of a `random' part and a `regularising' part). This allows us to apply a result \cite{GKO2016} on the overfull subgraph conjecture (see Corollary~\ref{cor:optcol}) which implies that $\chi'(\cH \setminus \cH') \leq \Delta(\cH \setminus \cH')$. (The result in \cite{GKO2016} is obtained as a straightforward consequence of the result in \cite{kuhn2013hamilton} that robustly expanding regular graphs have a Hamilton decomposition, and thus, a $1$-factorisation if they have even order.)

\subsection{Colouring linear hypergraphs where all edges are large}\label{subsection:large-edge-overview}

Now we discuss how to prove Theorems~\ref{main-thm} and~\ref{main-thm-2} when all edges of $\cH$ have size at least some large constant.  This case of Theorem~\ref{main-thm} is also sketched in more detail in the authors' survey~\cite[Section 5.2]{KKKMO2021survey}.  In this step it is often very useful to consider the line graph $L(\cH)$ of $\cH$ and use the fact that $\chi(L(\cH)) = \chi'(\cH)$.
In this subsection, we fix constants satisfying the hierarchy
\begin{equation*}
    0 < 1/n_0  \ll 1 / r \ll \sigma \ll \delta \ll 1,
\end{equation*}
we let $n \geq n_0$, and we let $\cH$ be an $n$-vertex linear hypergraph such that every $e\in \cH$ satisfies $|e| > r$.  Now we sketch a proof that $\chi'(\cH) \leq n$ for such $\cH$.  If $\cH$ is a finite projective plane of order $k$, where $k^2 + k + 1 = n$, then the line graph $L(\cH)$ is a clique $K_n$. Thus, $\chi'(\cH) = \chi(L(\cH)) = n$, so the bound $\chi'(\cH) \leq n$ is best possible.  Thus, we refer to the case where $\cH$ has approximately $n$ edges of size $(1 \pm \delta)\sqrt n$ as the `FPP-extremal' case.  We also sketch how to prove the improved bound $\chi'(\cH) \leq (1 - \sigma)n$ if $\cH$ is not in the FPP-extremal case.   As we discuss in the next subsection, we will need this result in the proof of Theorem~\ref{main-thm}.

Consider an ordering $\ordering$ of the edges of $\cH$ according to their size, i.e., $e \ordering e'$ if $|e| > |e'|$ for every $e,e' \in \cH$. For an edge $e \in \cH$, let $\fwddeg_\cH(e)$ denote the number of edges in $\cH$ which intersect $e$ and precede $e$ in $\ordering$. Clearly, a greedy colouring following this size-monotone ordering achieves a bound of $\chi'(\cH) \leq \max_i \fwddeg_\cH(e_i) + 1$ (this bound was also used in~\cite{chang1988, kahn1992coloring}). Moreover, since $\cH$ is linear, it is easy to see that if this greedy colouring algorithm fails to produce a colouring with at most $(1 - \sigma)n$ colours, i.e., if an edge $e$ satisfies $\fwddeg_\cH(e) \ge (1 - \sigma) n$, then almost all of the corresponding edges that intersect $e$ and precede $e$ must have size close to $|e|$.

Surprisingly, if one allows some flexibility in the ordering (in particular, if we allow it to be size-monotone only up to some edge $e^*$ such that $\fwddeg_\cH(e^*) \geq (1 - \sigma)n$ while every edge $f$ with $e^*\ordering f$ satisfies $\fwddeg_\cH(f) < (1 - \sigma)n$), then one can show much more: Either we can modify the ordering to reduce the number of edges which come before $e^*$, or there is a set $W\subseteq \cH$ (where $e^*$ is the last edge of $W$) such that \begin{enumerate}[(W1)]
    \item\label{W1overview} $|e^*| \approx |e|$ for every $e \in W$, and
    \item\label{W2overview} the edges of $W$ cover almost all pairs of vertices of $\cH$.
    \end{enumerate}
(The precise statement is given in Lemma~\ref{reordering-lemma} (Reordering lemma).)

If $|e^*| \leq (1- \delta) \sqrt{n}$, then one can show that  $L(W)$ induces a `locally sparse' graph (as $\cH$ is linear).  Moreover,~\ref{W1overview} implies that the maximum degree of $L(W)$ is not too large, and thus one can show that $\chi(L(W))$ is much smaller than $(1 - \sigma)n$ (leaving enough room to colour the edges preceding $W$ with a new set of colours). This together with \ref{W2overview} allows us to extend the colouring of $W$ to all of $\cH$ using a suitable modification of the above greedy colouring procedure for the remaining edges in $\cH$ to obtain that $\chi'(\cH) \leq (1-\sigma)n$, as desired.

If $|e^*| \geq (1- \delta) \sqrt{n}$, then we first colour the edges of size at least $(1- \delta) \sqrt{n}$ (in particular, the edges of $W$) as follows. Let $\cH' \subseteq \cH$ be the hypergraph consisting of these edges. 
Since each edge of size at least $(1-\delta)\sqrt{n}$ covers $\binom{(1-\delta)\sqrt{n}}{2}$ pairs of the vertices among $\binom{n}{2}$ pairs, by the linearity of $\cH'$,
\begin{align*}
    e(\cH') \leq \binom{n}{2} \binom{(1-\delta)\sqrt{n}}{2}^{-1} = \frac{\sqrt{n}(n-1)}{(1-\delta)((1-\delta)\sqrt{n} - 1)} \leq \frac{n}{(1-\delta)(1-1.5\delta)} \leq (1+3\delta)n.
\end{align*}
If $e(\cH') \leq n$, then, of course, we may colour the edges of $\cH'$ with different colours. Otherwise, if $t \coloneqq e(\cH') - n > 0$, the main idea is to find a matching of size $t$ in the complement of $L(\cH')$ (where $L(\cH')$ will be close to being a clique of order not much more than $n$). By assigning the same colour to the edges of $\cH'$ that are adjacent in this matching, we obtain $\chi'(\cH') = \chi(L(\cH')) \leq n$ (see Lemma~\ref{extremal-case-lemma}). Now we extend the colouring to all of $\cH$ using a suitable modification of the above greedy colouring procedure again to obtain that $\chi'(\cH) \leq n$, as desired.

\subsection{Combining colourings of the large and small edges}\label{subsection:combining-overview}

We now describe how one can prove Theorem \ref{main-thm} by building on the ideas described in Sections \ref{subsection:small-edge-overview} and \ref{subsection:large-edge-overview}. In this subsection and throughout the rest of the paper we work with constants satisfying the following hierarchy:
\begin{equation}
  0< 1 / n_0 \ll 1/r_0 \ll \xi \ll  1 / r_1 \ll \beta \ll \kappa \ll \gamma_1 \ll \eps_1 \ll \rho_1 \ll \sigma \ll \delta \ll  \gamma_2 \ll \rho_2 \ll \eps_2 \ll 1.
\end{equation}
Some of these constants are used to characterize the edges of a hypergraph by their size, as follows.

\begin{restatable}[Edge sizes]{definition}{edgesizes}\label{def:edgesize}
Let $\cH$ be an $n$-vertex linear hypergraph with $n \ge n_0$.  
\begin{itemize}

  \item Let $\cH_{\rm small} \coloneqq \{ e \in \cH \: : \: |e| \leq r_1 \}$.  An edge $e\in \cH$ is \textit{small} if $e\in\cH_{\mathrm{small}}$.
  
  \item Let $\cH_{\rm med} \coloneqq \{ e \in \cH \: : \: r_1 < |e| \leq r_0 \}$.  An edge $e\in\cH$ is \textit{medium} if $e\in\cH_{\mathrm{med}}$.
  
  \item Let $\cH_{\rm large} \coloneqq \{ e \in \cH \: : \: |e| > r_0 \}$.  An edge $e\in \cH$ is \textit{large} if $e\in\cH_{\mathrm{large}}$.
    
  \item Let $\cH_{\mathrm{ex}} \coloneqq \{ e \in \cH \: : \: |e| = (1 \pm \delta)\sqrt n\}$.  An edge $e\in\cH$ is \textit{FPP-extremal} if $e\in\cH_{\mathrm{ex}}$.
  
  \item Let $\cH_{\rm huge} \coloneqq \{ e \in \cH \: : \: |e| \geq \beta n / 4 \} $.  An edge $e\in \cH$ is \textit{huge} if $e\in\cH_{\mathrm{huge}}$.
  
  \end{itemize}
\end{restatable}
Note that $\cH_{\rm small}, \cH_{\rm med}, \cH_{\rm large}$ form a partition of the edges of $\cH$ (see Figure~\ref{fig:size}). Also note that if $\cH$ is an $n$-vertex linear hypergraph and $1/n \ll \alpha < 1$, then\COMMENT{Since $\cH$ is linear, if there are $k$ edges of size at least $\alpha n$, then $n \geq \sum_{i=1}^k\max\{\alpha n - i, 0\}$.}
\begin{equation}\label{eqn:huge-edge-bound}
    |\{e \in \cH : |e| \geq \alpha n\}| \leq 2 / \alpha.
\end{equation}
\begin{figure}
\centering
\begin{tikzpicture}
\draw(0,0)--(15,0);

\foreach \x/\xtext in {0/$1$,2/$r_1$,4/$r_0$,6/$(1-\delta)\sqrt{n}$,9/$(1+\delta)\sqrt{n}$,11/${\beta n/4}$,15/$n$}
    \draw(\x,5pt)--(\x,-5pt) node[below] {\xtext};

\draw[decorate, decoration={brace}, yshift=3ex]  (0.1,0) -- node[above=0.5ex] {small}  (1.9,0);

\draw[decorate, decoration={brace}, yshift=3ex]  (2.1,0) -- node[above=0.5ex] {medium}  (3.9,0);

\draw[decorate, decoration={brace}, yshift=3ex]  (4.1,0) -- node[above=0.5ex] {large}  (14.9,0);

\draw[decorate, decoration={brace, mirror}, yshift=-5.5ex]  (6.1,0) -- node[below=0.5ex] {FPP-extremal}  (8.9,0);

\draw[decorate, decoration={brace, mirror}, yshift=-5.5ex]  (11.1,0) -- node[below=0.5ex] {huge}  (14.9,0);

\end{tikzpicture}
\caption{Types of edges based on their size}
\label{fig:size}
\end{figure} In the proof of Theorem~\ref{main-thm}, given an $n$-vertex linear hypergraph $\cH$ with $n \ge n_0$ (where we assume $\cH$ has no singleton edges), we first find a proper edge-colouring $\psi_1 : \cH_{\mathrm{med}}\cup\cH_{\mathrm{large}} \rightarrow C_1$ as discussed in Section~\ref{subsection:large-edge-overview}, and then we extend it to a proper $n$-edge-colouring of $\cH_{\mathrm{small}}$ by adapting the argument presented in Section~\ref{subsection:small-edge-overview}.  The proof proceeds slightly differently depending on whether we are in the FPP-extremal case.  As discussed in the previous subsection, in the non-FPP-extremal case, $\chi'(\cH_{\mathrm{med}}\cup\cH_{\mathrm{large}}) \leq (1 - \sigma)n$, so we may assume $|C_1| = (1 - \sigma)n$.  In this case, we let $\gamma \coloneqq \gamma_1$, $\eps \coloneqq \eps_1$, and $\rho \coloneqq \rho_1$; in the FPP-extremal case, we let $\gamma \coloneqq \gamma_2$, $\rho \coloneqq \rho_2$, and $\eps \coloneqq \eps_2$ (see {Step}~\ref{step:step1} of  Section~\ref{proof-section} for the details).  

We define $G$ and $U$ as in Section~\ref{subsection:small-edge-overview}, and we define a suitable `defect' set $S \subseteq U$ (whose choice now depends on the structure of $\cH$). In order to extend the colouring $\psi_1$ of $\cH_{\mathrm{med}}\cup\cH_{\mathrm{large}}$ to $\cH$, we need it to satisfy a few additional properties, which are provided by Theorem~\ref{large-edge-thm}.  Roughly, we need that 
\begin{enumerate}[(1)]
    \item \label{(1)} each colour class of $\psi_1$ covers at most $\beta n$ vertices, with exceptions for colour classes containing huge or medium edges, and
    \item \label{(2)}  at most $\gamma n$ colours are assigned by $\psi_1$ to colour medium edges. 
\end{enumerate}

We choose a `reservoir' $R$ from $E(G)$; how we choose it depends on whether we are in the FPP-extremal case.  In the non-FPP-extremal case, we choose it as described in Section~\ref{subsection:small-edge-overview}, and in the FPP-extremal case, we include every edge of $G$ incident to a vertex of $U$ to be in $R$ independently with probability $\rho$  (see {Step}~\ref{step:step2} of Section~\ref{proof-section} for the details).

\begin{figure}
     \centering
     \begin{subfigure}[b]{\textwidth}
        \centering
        \begin{tikzpicture}
        \draw(0,0)--(15,0);
        
        \foreach \x in {0,1,2,3,8,11,13,15}
        	\draw(\x,5pt)--(\x,-5pt);
        	
        \path(0,0) -- node[below=0.3ex] {$C_{\rm diff}$}  (1,0);
        \path(1,0) -- node[below=0.3ex] {$C_{\rm huge}$}  (2,0);
        \path(2,0) -- node[below=0.3ex] {$C_{\rm med}$}  (3,0);
        \path(3,0) -- node[below=0.3ex] {$C_{\rm large}$}  (8,0);
        \path(8,0) -- node[below=0.3ex] {$C_{\rm main} \setminus C_{\rm large}$}  (11,0);
        \path(11,0) -- node[below=0.3ex] {$C_{\rm buff}$}  (13,0);
        \path(13,0) -- node[below=0.3ex] {$C_{\rm final}$}  (15,0);
        
        \draw[decorate, decoration={brace}, yshift=3ex]  (0.1,0) -- node[above=0.5ex] {$C_{1}$}  (7.9,0);
        
        \draw[decorate, decoration={brace,mirror}, yshift=-4.5ex]  (3.1,0) -- node[below=0.5ex] {$C_{\rm main}$}  (10.9,0);
        \draw[decorate, decoration={brace,mirror}, yshift=-4.5ex]  (0.1,0) -- node[below=0.5ex] {$C_{\rm hm}$}  (2.9,0);
        
        \draw[decorate, decoration={calligraphic brace,amplitude=10pt,mirror}, yshift=-8ex]  (3.1,0) -- node[below=2.2ex] {$C_2$}  (12.9,0);
        
        \end{tikzpicture}
        \caption{Non-FPP-extremal case}
    \end{subfigure}
	
     \begin{subfigure}[b]{\textwidth}
        \centering
        \begin{tikzpicture}
        \draw(0,0)--(15,0);
        
        \foreach \x in {0,1,2,3,11,13,15}
        	\draw(\x,5pt)--(\x,-5pt);
        	
        \path(0,0) -- node[below=0.3ex] {$C_{\rm diff}$}  (1,0);
        \path(1,0) -- node[below=0.3ex] {$C_{\rm huge}$}  (2,0);
        \path(2,0) -- node[below=0.3ex] {$C_{\rm med}$}  (3,0);
        \path(3,0) -- node[below=0.3ex] {$C_{\rm main}$}  (11,0);
        \path(11,0) -- node[below=0.3ex] {$C_{\rm buff}$}  (13,0);
        \path(13,0) -- node[below=0.3ex] {$C_{\rm final}$}  (15,0);
        
        \draw[decorate, decoration={brace}, yshift=3ex]  (0.1,0) -- node[above=0.5ex] {$C_{1}$}  (14.9,0);
        
        \draw[decorate, decoration={brace,mirror}, yshift=-4.5ex]  (0.1,0) -- node[below=0.5ex] {$C_{\rm hm}$}  (2.9,0);
        
        \draw[decorate, decoration={brace,mirror}, yshift=-4.5ex]  (3.1,0) -- node[below=0.5ex] {$C_2$}  (12.9,0);
        
        \end{tikzpicture}
        \caption{FPP-extremal case}
    \end{subfigure}

    \caption{Subsets of colours partitioning $[n]$. The entire line represents $[n]$, and segments represent colour sets. $C_2$ will be used to colour the small edges via the nibble-absorption approach described in Section~\ref{subsection:small-edge-overview}, and $C_1$ will be used to colour the large and medium edges via the methods described in Section~\ref{subsection:large-edge-overview}. $C_{\rm final}$ will be used to colour the edges of the leftover graph as described in Sections~\ref{subsection:small-edge-overview}  and~\ref{subsection:combining-overview}.}
    \label{fig:colourset}
\end{figure} 
In {Step}~\ref{step:step3}, we define various subsets of colours which partition $[n]$, where each subset of colours is used for a different purpose. 
Figure~\ref{fig:colourset} shows how these subsets are defined in the non-FPP-extremal case and the FPP-extremal case, respectively.
For clarity of presentation, we will use a less refined partition $C_{\rm hm} \cup C_2 \cup C_{\rm final} = [n]$ in this overview section.  Recall that $C_1$ is the set of colours assigned to medium or large edges by $\psi_1$.  Let  $C_{\mathrm{hm}} \subseteq C_1$ be the set of colours assigned to a huge or medium edge by $\psi_1$.

Note that $e(\cH_{\mathrm{huge}}) \leq 8/\beta$ by \eqref{eqn:huge-edge-bound}, so consequently, by \ref{(2)}, $|C_{\mathrm{hm}}| \leq 3\gamma n / 2$.
For each $c\in C_{\mathrm{hm}}$, we use Lemma~\ref{huge-edge-absorption-lemma} to extend  $\psi_1^{-1}(c)$ (in the sense of Section ~\ref{subsection:small-edge-overview}) using edges of $R$, so that $\{ \psi_1^{-1}(c) : c \in C_{\rm hm} \}$ has nearly perfect coverage of $U$ with defects in $S$ (see {Step}~\ref{step:step5} of Section~\ref{proof-section} for the details). 
However, there is possibly an exceptional colour class, which we call \textit{difficult} (see Definition~\ref{defn:difficult}).
This situation arises if $\cH$ is close to being a degenerate plane, which is defined in Subsection~\ref{subsec:intro_main}.
If $\cH$ is the degenerate plane, then there is a huge edge $e$ of size $n - 1$, and $U$ consists of a single vertex of degree $n - 1$.
Even though $\cH$ is not $(\rho, \eps)$-full, if $c$ is assigned to the edge $e$, it is clearly impossible to extend $\psi_1^{-1}(c)$ to have perfect coverage of $U$, which would be necessary in order to finish the colouring with Vizing's theorem in the final step. However, if there is a difficult colour class that we cannot absorb, then we show that we can colour $\cH$ directly (see Lemma~\ref{difficult-edge-absorption}). We deal with extending the difficult colour class/matching in {Step}~\ref{step:diff} of Section~\ref{proof-section}.

In {Steps}~\ref{step:step6} and~\ref{step:step7}, we colour $\cH_{\mathrm{small}} \setminus R$ using a set of colours $C_2$ such that
\begin{enumerate}[(A)]
    \item \label{C_2-size} $|C_2|$ is slightly larger than $(1 - \rho + \gamma)n$ and
    \item \label{C_2-hm-intersection} $C_2 \cap C_{\rm hm} = \varnothing$.
\end{enumerate}
Using the nibble-absorption approach described in Section~\ref{subsection:small-edge-overview}, we construct some $\cH'$ with $\cH_{\mathrm{small}}\setminus R \subseteq \cH' \subseteq \cH_{\mathrm{small}}$ and a proper edge-colouring $\psi_2 : \cH' \rightarrow C_2$ such that 
\begin{enumerate}[(A)]
\stepcounter{enumi}
\stepcounter{enumi}
    \item \label{psi_2-compatibility} $\psi_2$ is compatible with $\psi_1$
     and
    \item \label{psi_2-coverage} $\{\psi^{-1}_1(c) \cup \psi^{-1}_2(c) : c \in C_{\rm hm} \cup C_2\}$ has nearly-perfect coverage of $U$ with defects in $S$.
\end{enumerate}
This is accomplished via Lemmas~\ref{lem:nibble2} and \ref{lem:leftover}.
(Actually, as in Section~\ref{subsection:small-edge-overview} we obtain this coverage property only for a suitable auxiliary hypergraph $\cH^* \supseteq \cH'$, but we again ignore this here for simplicity.)  
For this, the following properties are crucial (see also the condition~\ref{C4} of Lemma~\ref{lem:nibble2}):
\begin{enumerate}[(I)]
    \item\label{(I)} $\psi_1^{-1}(c)$ covers at most $\beta n$ vertices for each $c\in C_2$ (this follows from \ref{(1)} and \ref{C_2-hm-intersection}), and
    \item\label{(II)} every vertex $v\in V(\cH)$ is contained in at most $n / (r_0-1)$ edges that are assigned a colour in $[n] \setminus C_{\rm hm} \supseteq C_2$ by $\psi_1$ (since for any $c \in [n] \setminus C_{\rm hm}$, either $\psi_1^{-1}(c)$ is empty or all the edges in $\psi_1^{-1}(c)$ are large).
\end{enumerate}
\noindent Thus, from~\ref{(II)} and ~\ref{C_2-size} we can deduce the following. 
\begin{enumerate}[resume*]
    \item\label{(III)} Each edge in $\cH_{\rm small}$ still has slightly more than $(1 - \rho)n$ colours available in $C_2$ that do not conflict with $\psi_1$ (since any edge of $\cH_{\rm small}$ intersects at most $r_1 n/(r_0-1)$ large edges and $r_1/(r_0-1) \ll \gamma$).
\end{enumerate}
  
We will use \ref{(I)}--\ref{(III)} to show that the effect of the previously coloured edges (by $\psi_1$) on the R\"odl nibble argument is negligible, i.e., we can adapt the arguments of Section \ref{subsection:small-edge-overview}, so that the colouring $\psi_2$ of $\cH'$ is compatible with $\psi_1$.

In the final step of the proof ({Step}~\ref{step:step10}), we colour the leftover graph $\cH_{\mathrm{small}}\setminus \cH' \subseteq R$ using colours from $C_{\mathrm{final}}$.  
(Prior to this, in 
{Step}~\ref{step:step8}, we prove that the leftover is indeed a graph that moreover is quasirandom and almost regular if $\cH$ is $(\rho, \eps)$-full, and in {Step}~\ref{step:step9} we prove that its maximum degree is appropriately bounded.)
In the non-FPP-extremal case, we can define $C_2$ to satisfy that $C_1 \subseteq C_2 \cup C_{\rm hm}$ (see Figure~\ref{fig:colourset}).  
Thus, we can reserve the set of colours $C_{\rm final}$ (of size close to $\rho n$) so that these colours are used neither by $\psi_1$ nor by $\psi_2$. Then we can colour the leftover graph as described in Section~\ref{subsection:small-edge-overview}.
In the FPP-extremal case, we may have $|C_1| = n$, and we need to find a proper edge-colouring of $\cH_{\mathrm{small}}\setminus \cH'$ using colours from $C_{\mathrm{final}}$ while avoiding conflicts with $\psi_1$.  But in this case most pairs of vertices are contained in an edge of $\cH_{\mathrm{ex}}$, which implies that $|U|$ is small. Moreover, every edge of the leftover graph is incident to a vertex of $U$. These two properties allow us to colour the leftover graph $\cH_{\mathrm{small}}\setminus \cH'$ with $\Delta(\cH_{\mathrm{small}}\setminus \cH')$ colours
while using \ref{(1)} and \ref{(II)} to avoid conflicts with $\psi_1$, as desired.
This leftover colouring is constructed in Lemma~\ref{hall-based-finishing-lemma}.

\subsection{Organisation of the paper}

In Section~\ref{notation-section}, we introduce some notation that we use throughout the paper, and in Section~\ref{preliminary-section} we collect some tools that we use in the proof.  In Section~\ref{fpp-extremal-section} we prove Theorem~\ref{main-thm} for hypergraphs where every edge has size at least $(1 - \delta)\sqrt n$, and in Section~\ref{large-edge-section}, we prove Theorem~\ref{large-edge-thm}, which we use to colour the large and medium edges of our hypergraph.  In Section~\ref{large-edge-section}, we also prove Theorems~\ref{main-thm-2} and~\ref{weak-stability-thm}. (In particular, Theorems~\ref{main-thm-2} and~\ref{weak-stability-thm} do not rely on the subsequent sections. Conversely, the assertions of Theorems~\ref{main-thm-2} and~\ref{weak-stability-thm} are not used directly in the proof of Theorem~\ref{main-thm}.)   In Section~\ref{absorption-section}, we prove several lemmas that we use for vertex absorption, and in Section~\ref{small-non-graph-section}, we show how to combine the results of Section~\ref{absorption-section} with hypergraph matching results (based on the R\"odl nibble) to colour the small edges of our hypergraph not in the reservoir.  In Section~\ref{graph-edge-section}, we prove Lemma~\ref{hall-based-finishing-lemma} and introduce Corollary~\ref{cor:optcol}, both of which are used in the final step of the proof to colour the uncoloured reservoir edges.  In Section~\ref{reservoir-section} we show how we select the reservoir edges, and finally in Section~\ref{proof-section}, we prove Theorem~\ref{main-thm}.
 
\section{Notation}\label{notation-section}

For $n \in \mathbb{N}$, we write $[n] := \{k \in \mathbb{N} \: : \:1 \leq k \leq n \}$. We write $c = a \pm b$ if $a-b \leq c \leq a+b$. We use the `$\ll$' notation to state our results. Whenever we write a hierarchy of constants, they have to be chosen from right to left. More precisely, if we claim that a result holds whenever $0 < a \ll b \le 1$, then this means that there exists a non-decreasing function $f : (0,1] \mapsto (0,1]$ such that the result holds for all $0 < a,b \le 1$ with $a \le f(b)$. We will not calculate these functions explicitly. Hierarchies with more constants are defined in a similar way.

A \emph{hypergraph} $\cH$ is an ordered pair $\cH = (V(\cH), E(\cH))$ where $V(\cH)$ is called the vertex set and $E(\cH) \subseteq 2^{V(\cH)}$ is called the edge set. 
If $E(\cH)$ is a multiset, we refer to $\cH$ as a multi-hypergraph. Throughout the paper we usually write $\cH$ instead of $E(\cH)$. We say that a (multi-)hypergraph $\cH$ is $r$-uniform if for every $e \in E(\cH)$ we have $|e| = r$. In particular, $2$-uniform hypergraphs are simply called \emph{graphs}.

Given any multi-hypergraph $\cH$, let $v(\cH)$ denote the number of vertices in $\cH$ and let $e(\cH)$ denote the number of edges in $\cH$. Let $\cH^{(i)} \coloneqq \{e \in \cH : |e| = i \}$. Throughout the paper, we usually denote $\cH^{(2)}$ by $G$. For any subset $S \subseteq V(\cH)$, let $\cH|_S$ be the multi-hypergraph with the vertex set $V(\cH|_S) := S$ and edge set $\cH|_S := \{e \cap S \: : \: e \in \cH\:\:{\rm and}\:\:e \cap S \ne \varnothing \}$. For any vertex $v \in V(\cH)$, let $E_{\cH}(v) \coloneqq \{ e \in \cH : v \in e\}$. We define the \emph{degree} of $v$ by $d_{\cH}(v) \coloneqq |E_{\cH}(v)|$. More generally, for any given multiset $R \subseteq \cH$, let $E_{R}(v)$ denote the multiset of edges incident to $v$ in $R$, and $d_{R}(v) := |E_{R}(v)|$.
We denote the minimum and maximum degrees of the vertices in $\cH$ by $\delta(\cH)$ and $\Delta(\cH)$, respectively. Let $V^{(d)}(\cH) \coloneqq \{ v \in \cH : d_{\cH}(v) = d \}$. Moreover, if $d\in\mathbb N$, then let $V^{(d)}_+(\cH) \coloneqq \{ v \in \cH : d_{\cH}(v) \ge d \}$ and if $x \in (0, 1)$, then let $V^{(x)}_+(\cH) \coloneqq \{ v \in \cH : d_{\cH}(v) \ge xn \}$, where $n \coloneqq v(\cH)$. For any edge $e \in \cH$, the set $N_{\cH}(e)$ denotes the multiset of edges $f \in E(\cH) \setminus \{e \}$ that intersect $e$. The subscript $\cH$ from $N_{\cH}(e)$ may be omitted if it is clear from the context.  If $\cH'\subseteq \cH$ and $e\in \cH$, then we denote $N_\cH(e)\cap \cH'$ by $N_{\cH'}(e)$, even if $e\notin\cH'$.

The \emph{line graph} $L(\cH)$ of a (multi-)hypergraph $\cH = (V(\cH), E(\cH))$ is the graph whose vertex set is $E(\cH)$, where two vertices in $L(\cH)$ are adjacent if the corresponding edges in $E(\cH)$ have a non-empty intersection. A \emph{matching} $M$ in $\cH$ is a subset of pairwise disjoint edges of $\cH$. We often regard $M$ as a hypergraph with $V(M) \coloneqq \bigcup_{e \in M}e$. For any vertex $u \in V(\cH)$, we say $u$ is \textit{covered} by a matching $M$ if $u\in e$ for some $e\in M$. For any $X\subseteq V(\cH)$, we say that a matching $M$ \textit{covers} $X$ if $M$ covers every vertex in $X$. For any integer $k \geq 0$ and a (multi-)hypergraph $\cH$, a map $\phi : \cH \to [k]$ is a \emph{proper edge-colouring} of $\cH$ if $\phi(e) \ne \phi(f)$ for any pair of distinct edges $e,f \in \cH$ such that $e \cap f \ne \varnothing$. For any integer $i$ and a proper edge-colouring $\phi : \cH \to [k]$, let $\phi^{-1}(i)$ be the set of edges $e \in \cH$ with $\phi(e) = i$. (Note that $\phi^{-1}(i)$ is a matching, for any $i$.)

A (multi-)hypergraph $\cH$ is \emph{linear} if for any distinct $e,f \in \cH$, $|e \cap f| \leq 1$. A linear hypergraph may contain singleton edges (but no edge is repeated). A linear multi-hypergraph may contain multiple singleton edges incident to the same vertex but any edge of size at least two cannot be repeated (as that would contradict linearity). Given any linear multi-hypergraph $\cH$ on $n$ vertices, and any $W \subseteq \cH$, the \emph{normalised volume} of $W$ is defined as  $\vol_\cH(W) \coloneqq \sum_{e\in W}\binom{|e|}{2} / \binom{n}{2}$. We sometimes omit the subscript and write $\vol(W)$ instead of $\vol_\cH(W)$ when it is clear from the context. Note that since $\cH$ is linear, $\vol_\cH(W) \le 1$ for any $W \subseteq \cH$.

Given any graph $G$, for any subset of vertices $V' \subseteq V (G)$, we denote the subgraph of $G$ induced by $V'$ as $G[V'] \coloneqq (V', E')$, where $E' \coloneqq \{e \in E(G) : e \subseteq V' \}$. We write $G - V'  \coloneqq G[V \setminus V']$. If $V' = \{v\}$, then we simply write $G - v$ instead of $G - \{v\}$. For any disjoint pair of subsets $S, T \subseteq V(G)$, let $E_G(S, T) := \{ st \in E(G) : s \in S, t \in T \}$, and let $e_G(S,T) \coloneqq |E_G(S, T)|$. Let $\overline{G} \coloneqq (V(G), \overline{E(G)})$ denote the complement of a graph $G$.

For the readers' convenience, we also restate the important definitions from Section~\ref{sec:overview} here.
\perfectcoverage*
\full*
\edgesizes*

\section{Preliminaries}\label{preliminary-section}
We often use the following weighted version of Chernoff's inequality.

\begin{theorem}[Weighted Chernoff's inequality~\cite{CL2002}]\label{thm:chernoff}
Let $c_1 , \dots , c_m > 0$ be real numbers, let $X_1 , \dots , X_m$ be independent random variables taking values $0$ or $1$, let $X \coloneqq \sum_{i=1}^{m} c_i X_i$ and let $C \coloneqq \max_{i \in [m]} c_i$. Then,
\begin{equation*}
	\mathbb{P}(|X - \mathbb{E}(X)| \geq t) \leq 2 e^{\frac{-t^2}{2C(\mathbb{E}[X] + t/3)}}.
\end{equation*}
\end{theorem}

\subsection{Pseudorandom hypergraph matchings}

Now we state a special case of a recent result of Ehard, Glock, and Joos~\cite{EGJ2020} that provides a matching covering  almost all vertices of every set in a given collection of sets. This result will be used in the proof of Lemma~\ref{lem:pseudorandom_matching}. A similar result, but with weaker bounds, was proved earlier by Alon and Yuster~\cite{AY2005}. The proof in~\cite{EGJ2020} is derived via an averaging argument from a result on the chromatic index of hypergraphs by Molloy and Reed~\cite{MR2000}, which in turn relies on the R\"{o}dl nibble.

\begin{theorem}[Ehard, Glock, and Joos~\cite{EGJ2020}]\label{thm:egj}
Let $r \geq 2$ be an integer, and let $\varepsilon \coloneqq 1/(1500r^2)$. There exists $\Delta_0$ such that the following holds for all $\Delta \geq \Delta_0$. Let $\cH$ be an $r$-uniform linear hypergraph with $\Delta(\cH) \leq \Delta$ and $e(\cH) \leq \exp(\Delta^{\eps^2})$. Let $\cF^*$ be a set of subsets of $V(\cH)$ such that $|\cF^*| \leq \exp(\Delta^{\varepsilon^2})$ and $\sum_{v \in S}d_\cH(v) \geq \Delta^{26/25}$ for any $S \in \cF^*$. Then, there exists a matching $M_0$ in $\cH$ such that  for any $S \in \cF^*$, we have $|S \cap V(M_0)| = (1 \pm \Delta^{-\varepsilon^2})\sum_{v \in S}d_\cH(v)/\Delta$.
\end{theorem}

We remark that Theorem~\ref{thm:egj} is a direct application of~\cite[Theorem 1.2]{EGJ2020} by setting $\delta \coloneqq 1/30$ and the weight functions $w_S (e) := |S \cap e|$ for $S \in \cF^*$, where $w_S( \cH ) \geq \max_{e \in \cH}w_S(e) \Delta^{1+\delta}$ follows by the assumption $\sum_{v \in S}d_\cH(v) \geq \Delta^{26/25}$, since $w_S(\cH) = \sum_{v \in S}d_\cH(v)$ and $w_S(e) \leq r$ for any $e \in \cH$.\COMMENT{\cite[Theorem 1.2]{EGJ2020} is stated as follows.
{\bf Theorem}. Suppose $\delta \in (0,1)$ and $r \in \mathbb{N}$ with $r \geq 2$, and let $\eps \coloneqq \delta/50r^2$. Then there exists $\Delta_0$ such that for all $\Delta \geq \Delta_0$, the following holds: Let $\cH$ be an $r$-uniform hypergraph with $\Delta(\cH) \leq \Delta$ and $\Delta^c (\cH) \leq \Delta^{1-\delta}$ as well as $e(\cH) \leq \exp(\Delta^{\eps^2})$. Suppose that $\cF$ is a set of at most $\exp(\Delta^{\eps^2})$ weight functions on $\cH$. Then, there exists a matching $M$ in $\cH$ such that
$\omega(M) = (1 \pm \Delta^{-\eps})\omega(\cH)/\Delta$ for all $\omega \in \cF$ with $\omega(\cH) \geq \max_{e \in \cH}w(e) \Delta^{1+\delta}$.\\
Indeed, then it is clear (from the assumption of Theorem~\ref{thm:egj}) that $w_S(\cH) = \sum_{v \in S}d_\cH(v) \geq \Delta^{26/25} \geq r \Delta^{31/30} \geq \max_{e \in \cH}w_S(e) \Delta^{1+\delta}$. Hence, we can apply the above theorem to obtain a matching $M_0$ in $\cH$ such that $w_S(M_0) = |S \cap V(M_0)| = (1 \pm \Delta^{-\eps})w_S(\cH)/\Delta = (1 \pm \Delta^{-\eps})\sum_{v \in S}d_\cH(v)/\Delta$, showing that Theorem~\ref{thm:egj} holds.}

Our vertex absorption arguments will actually require that the number of uncovered vertices in $S$ is small but not too small.\COMMENT{Note that Theorem~\ref{thm:egj} may imply that $|S \setminus V(M_0)| = 0$, which is bad for us} So we need the following `sparsified' version of Theorem~\ref{thm:egj}, which allows us to have better control on the number of uncovered vertices. To deduce Corollary~\ref{cor:sparse_egj} from Theorem~\ref{thm:egj}, one simply applies Theorem~\ref{thm:egj} to obtain a matching $M_0$ (in $\cH$) and then we randomly remove each edge of $M_0$ with probability $\gamma$ to obtain a matching $M$ which satisfies the assertion of Corollary~\ref{cor:sparse_egj} with positive probability. We remark that one could also derive Corollary~\ref{cor:sparse_egj} via a direct application of the R\"{o}dl nibble (see~\cite{KKMO2020} for a proof of a stronger result based on stronger assumptions).

\begin{corollary}\label{cor:sparse_egj}
Let $0 < 1 / n_0 \ll 1 / r, \kappa , \gamma < 1$. For any integer $n \geq n_0$, let $\cH$ be an $r$-uniform linear $n$-vertex hypergraph such that every vertex has degree $(1 \pm \kappa)D$, where $D \geq n^{1/100}$. Let $\mathcal{F}$ be a set of subsets of $V(\cH)$ such that $|\mathcal{F}| \le n^{2 \log n}$. Then there exists a matching $M$ of $\cH$ such that for any $S \in \cF$ with $|S| \geq D^{1/20}$, we have $|S \setminus V(M)| = (\gamma \pm 4 \kappa)|S|$.
\end{corollary}

\COMMENT{\begin{proof}
Let $\eps := 1/(1500r^2)$ and $\cF^* \subseteq \cF$ be the collection of all $S \in \cF$ with $|S| \geq D^{1/20}$. Then we have $\sum_{v \in S} d_\cH(v) \geq (1 - \kappa)D|S| \geq D^{26/25}$ as $|S| \geq D^{1/20}$. Then applying Theorem~\ref{thm:egj} with $\Delta \coloneqq (1+\kappa)D$, we obtain a matching $M_0$ in $\cH$ such that for all $S \in \mathcal{F^*}$,
\begin{equation}\label{eqn:s_uncov}
    |S \cap V(M_0)| \geq  (1 - \Delta^{-\varepsilon^2}) \frac{\sum_{v \in S} d_\cH(v)}{\Delta} \geq (1 - D^{-\varepsilon^2}) \frac{1-\kappa}{1+\kappa} |S| \geq (1-3\kappa)|S|.
\end{equation}
Now let $M$ be a matching obtained from $M_0$ by removing each $e \in M_0$ with probability $\gamma$. Let us fix $S \in \mathcal{F}^*$, and let $q \coloneqq |S \setminus V(M_0)| \overset{\eqref{eqn:s_uncov}}{\leq} 3\kappa|S|$. Let $Y_S \coloneqq |S \setminus V(M)| - q = \sum_{e \in M_0} |e \cap S| {\bf 1}_{e \notin M}$. By the linearity of expectation,
\begin{equation}\label{eqn:expected_ys}
    \mathbb{E}[Y_S] = \sum_{e \in M_0} \gamma |e \cap S| = \gamma |S \cap V(M_0)| = \gamma (|S| - q).
\end{equation}
Since $Y_S$ is the weighted sum of independent random indicators with maximum weight at most $r$, by Theorem~\ref{thm:chernoff}, we have
\begin{equation*}
    \mathbb{P} (|Y_S - \mathbb{E}[Y_S]| \geq \kappa |S|) \leq  2 \exp \left ( -\frac{ \kappa^2 |S|^2}{2 r (\gamma |S| + \kappa|S|/3)} \right ) \leq  2 \exp \left ( - r^{-1} \kappa^2 n^{1/2000} / 4 \right ),
\end{equation*}
since $|S| \geq D^{1/20} \geq n^{1/2000}$.  Then with probability at least $1 - \exp \left ( - n^{1/4000} \right )$, $Y_S = \mathbb{E}[Y_S] \pm \kappa |S|$, so by~\eqref{eqn:expected_ys} we have $|S \setminus V(M)| = q + Y_S = q + \gamma(|S|-q) \pm \kappa |S| = (\gamma \pm 4\kappa)|S|$, since $q \leq 3 \kappa |S|$. As $|\cF^*| \leq |\cF| \leq n^{2 \log n}$, by the union bound, 
$|S \setminus V(M)| = (\gamma \pm 4\kappa)|S|$ holds for all $S \in \cF^*$, with positive probability.\COMMENT{at least
\begin{equation*}
    1 - 2 |\cF^*|\exp \left (-2r^{-2} \kappa^2 n^{1/2000} \right ) \geq 1 - 2 \exp \left ( 2 \log^2 n - 2r^{-2} \kappa^2 n^{1/2000} \right ) > 0,
\end{equation*}}
Fixing such a matching $M$ completes the proof.
\end{proof}}

\subsection{Embedding lemma}
The following lemma allows us to embed any linear hypergraph $\cH$ with maximum degree $D$ into an almost regular, uniform, linear hypergraph $\cH_{\rm unif}$ with maximum degree $D$, satisfying some additional properties.

\begin{lemma}\label{lem:embed}
Let $0 < 1/N_0, 1/D_0, 1/C_0 \ll 1/r \leq 1/3$, where $r\in\mathbb N$.  Let $N \geq N_0$, let $C \geq C_0$, let $D \geq D_0$, and let $\cH$ be an $N$-vertex linear multi-hypergraph with $\Delta(\cH) \leq D$.  If every $e\in\cH$ satisfies $|e| \leq r$, then there exists an $r$-uniform linear hypergraph $\cH_{\rm unif}$ such that the following hold.
\begin{itemize}
    \item[\mylabel{E1}{(\ref{lem:embed}.1)}] $\cH \subseteq \cH_{\rm unif}|_{V(\cH)}$ and $\cH_{\rm unif}|_{V(\cH)} \setminus \cH$ only contains singleton edges.
    
    \item[\mylabel{E2}{(\ref{lem:embed}.2)}]For any $v \in V(\cH_{\rm unif})$, $D-C \leq d_{\cH_{\rm unif}}(v) \leq D$. Moreover, if $d_{\cH}(v) \geq D-C$ for $v \in V(\cH)$, then $d_{\cH_{\rm unif}}(v) = d_{\cH}(v)$.

    \item[\mylabel{E3}{(\ref{lem:embed}.3)}] $v(\cH_{\rm unif}) \leq r(r-1)^2 D^3 N$.
\end{itemize}
\end{lemma}

The proof of Lemma~\ref{lem:embed} is a straightforward modification of the proof of~\cite[Lemma 8.1]{KKMO2020}. Here we briefly sketch the proof. First, let $\cH^*$ be an $r$-uniform linear hypergraph obtained from $\cH$ by adding $r - |e|$ new vertices to each $e \in \cH$. Let $T \coloneqq (r-1)^2 D^2$. For every (sufficiently large) integer $d \le D$, by considering Steiner systems, one can easily construct a linear $T$-vertex $r$-uniform hypergraph $\cH_d$ such that every vertex of $\cH_d$ has degree between $d-c_r$ and $d$ for some constant $c_r$ depending on $r$. We define our desired multi-hypergraph $\cH_{\rm unif}$ by taking the union of $T$ vertex-disjoint copies of $\cH^*$, where the first copy is identified with $\cH^*$. Then, for each $v \in V(\cH^*)$ with $d_{\cH^*}(v) < D - C$, let $v^1 , \dots , v^T$ be the $T$ clone vertices of $v \in V(\cH^*)$ in $\cH_{\rm unif}$, and extend $\cH_{\rm unif}$ by making $\cH_{\rm unif}[\{v^1 , \dots , v^T\}]$ induce a copy of $\cH_{D - d_{\cH^*}(v)}$, which implies that $D - c_r \leq d_{\cH_{\rm unif}}(v^i) \leq D$ for $1 \le i \le T$.\COMMENT{Moreover, if $v^1 , \dots , v^T$ are the $T$ clone vertices of some  vertex $v \in V(\cH^*)$ with $d_{\cH^*}(v) \geq D - C$, then $d_{\cH_{\rm unif}}(v^i) = d_{\cH^*}(v)$ for $1 \le i \le T$. One can then check that the lemma follows by taking $C_0$ sufficiently large, since $C \ge C_0$.}

\subsection{Some colouring results}

We use Vizing's theorem~\cite{vizing1965} in the final step of our proof when $\cH$ is not $(\rho, \eps)$-full in the non-FPP-extremal case.  We also use the following stronger form in Lemma~\ref{difficult-edge-absorption}.

\begin{theorem}[Vizing~\cite{vizing1965}]\label{thm:vizing}
Every graph $G$ with $\Delta(G) \leq D$ satisfies $\chi'(G) \leq D + 1$.  Moreover, if $G$ contains at most two vertices of degree $D$, then $\chi'(G) \leq D$.
\end{theorem}

The following theorem is used as one of the ingredients to prove Theorems~\ref{main-thm-2} and~\ref{weak-stability-thm} in Section~\ref{large-edge-section}, as well as in Section~\ref{small-non-graph-section} to colour a small `leftover' part in Lemma~\ref{lem:leftover}.

\begin{theorem}[Kahn~\cite{kahn1996asymptotically}]\label{thm:kahn}
Let $0 < {1}/{D_0} \ll  \alpha, {1}/{r} < 1$, and let $D \geq D_0$. Let $\cH$ be a linear hypergraph such that $\Delta(\cH) \leq D$, and every $e\in\cH$ satisfies $|e| \leq r$. Let $C$ be a set with $|C| \geq (1+\alpha)D$, and for each $e \in \cH$, let $C(e) \subseteq C$ and $|C(e)| \leq \alpha D / 2$. Then there exists a proper edge-colouring $\phi : \cH \to C$ such that $\phi(e) \notin C(e)$ for each $e \in \cH$.
\end{theorem}

Theorem~\ref{thm:kahn} immediately follows from ~\cite[Theorem 1.1]{kahn1996asymptotically} by defining a list $S(e) \coloneqq C \setminus C(e)$ of available colours for each $e \in \cH$.\COMMENT{The original statement of Kahn~\cite[Theorem 1.1]{kahn1996asymptotically} is stated as follows: If $\cH$ is a hypergraph $\cH$ with $\Delta(\cH) \leq D$, $|e| \leq r$ for a fixed $r$, and $\Delta^c (\cH) = o(D)$, then $\chi_\ell'(\cH) = D + o(D)$, where $o(D)$ tends to 0 as $D \to \infty$. This implies the following: for any $r \in \mathbb{N}$ and $\alpha > 0$, there exists $D_0 = D_0(\alpha , r)$ such that for any $D \geq D_0$ and linear hypergraph $\cH$ with $\Delta(\cH) \leq D$ and $|e| \leq r$ for any $e \in \cH$, $\chi_\ell'(\cH) \leq (1+\alpha)D$.}

\subsection{A special case of the Lov\'asz's $(g,f)$-factor theorem}
For any non-negative integer functions $g,f : V(G) \to \mathbb{Z}$, a subset $F \subseteq E(G)$ is a \emph{$(g,f)$-factor} in $G$ if $g(w) \leq d_F(w) \leq f(w)$ for each $w \in V(G)$.

Lov\'asz~\cite[Theorem 7.3]{lovasz1970} (or see~\cite[Corollary 35.1b]{schrijver_volA}) proved a result which ensures the existence of a $(g,f)$-factor.
Here we state a special case of the result.

\begin{theorem}[Lov\'asz~\cite{lovasz1970}]\label{thm:gf_factor}
Let $G$ be a graph, and let $f,g : V(G) \to \mathbb{Z}$ be integer functions such that $0 \leq g(x) < f(x) \leq d_G(x)$ for each $x \in V(G)$. Then there exists a $(g,f)$-factor in $G$ if and only if for all disjoint subsets $S,T \subseteq V(G)$,
$\sum_{s \in S} f(s) + \sum_{t \in T} (d_G(t) - g(t)) \geq e_G(S,T)$.
\end{theorem}

\section{Colouring FPP-extremal and larger edges}\label{fpp-extremal-section}

In this section we prove Theorem~\ref{main-thm} when all edges are FPP-extremal or larger, as follows.

\begin{lemma}\label{extremal-case-lemma}
Let $0 < 1/n_0 \ll \delta \ll 1$, and let $n \geq n_0$. If $\cH$ is an $n$-vertex linear hypergraph where every $e\in \cH$ satisfies $|e| \geq (1 - \delta)\sqrt n$, then $\cH$ has a proper edge-colouring with $n$ colours, where each colour is assigned to at most two edges.
\end{lemma}

First we need the following simple observations.  For a hypergraph $\cH$, recall that $L(\cH)$ denotes the line graph of $\cH$ and $\overline{L(\cH)}$ denotes its complement.

\begin{obs}\label{lem:chromindex}
Let $\cH$ be an $n$-vertex linear hypergraph. If there is a matching $N$ in $\overline{L(\cH)}$ of size $e(\cH) - n$, then $\cH$ has a proper edge-colouring with $n$ colours, where each colour is assigned to at most two edges.\COMMENT{We may assume that $e(\cH) > n$, otherwise the proposition trivially follows by colouring the edges of $\cH$ with different colours. Let $t \coloneqq e(\cH) - n$ and $N \coloneqq \{e_1 f_1 , \dots , e_t f_t \}$ be a matching of size $t$ in $\overline{L(\cH)}$. We assign the same colour to $e_i$ and $f_i$ for each $i\in [t]$ and a distinct colour to every $e \in L(\cH) \setminus \bigcup_{i=1}^{t} \{e_i,f_i \}$. 
This gives a proper edge-colouring of $\cH$ with
\begin{equation*}
    (e(\cH)-2t) + t = e(\cH) - (e(\cH) - n) = n
\end{equation*}
colours, where each colour is assigned to at most two edges in $\cH$, as desired.}\qed
\end{obs}

A pair $\{e, f\}\subseteq \cH$ in an $n$-vertex hypergraph $\cH$ is \textit{useful} if $e\neq f$, $e\cap f\neq\varnothing$, and $|N(e) \cap N(f)| \leq n-2$.  

\begin{proposition}
\label{lem:chromindex0}
Let $\cH$ be an $n$-vertex linear hypergraph, and let $t \coloneqq e(\cH) - n$.  If $e_1 , \dots , e_{2t} \in \cH$ are distinct pairwise intersecting edges such that $\{e_{2i-1}, e_{2i}\}$ is a useful pair for each $i \in [t]$, then $\cH$ has a proper edge-colouring with $n$ colours, where each colour is assigned to at most two edges.
\end{proposition}
\begin{proof}
We may assume that $e(\cH) > n$.
We will show that there exists a matching $N$ in $\overline{L(\cH)}$ of size $t$.

For $i \in [t]$, suppose we have chosen distinct edges $z_1 , \dots , z_{i-1} \in \cH \setminus \{e_1 , \dots , e_{2t}\}$ where $z_j$ is non-adjacent to at least one of $e_{2j-1}$ or $e_{2j}$ in $L(\cH)$ for $1 \leq j \leq i-1$. We claim that one can choose $z_i \in \cH \setminus \{e_1 , \dots , e_{2t}\}$ distinct from $z_1 , \dots , z_{i-1}$ such that $z_i$ is non-adjacent to either $e_{2i-1}$ or $e_{2i}$ in $L(\cH)$. Indeed, since $|N(e_{2i-1}) \cap N(e_{2i})| \leq n-2$, letting $S \coloneqq \cH \setminus ((N(e_{2i-1}) \cap N(e_{2i})) \cup \{ e_{2i-1} , e_{2i} \})$, we have $|S| \geq e(\cH)-n$ and every $e\in S$ is non-adjacent to at least one of $e_{2i-1}$ or $e_{2i}$. Since $i-1 \leq t - 1 = e(\cH)-n-1$ we can choose $z_i \in S$ distinct from $z_1 , \dots , z_{i-1}$. Moreover, since $S \cap \{e_1 , \dots , e_{2t} \} = \varnothing$, we have $z_i \in \cH \setminus \{e_1 , \dots , e_{2t}\}$, as desired.

Let $z_1, z_2, \ldots, z_t \in \cH \setminus \{e_1 , \dots , e_{2t}\}$ be chosen using the above procedure.
Then since $z_j$ is non-adjacent to either $e_{2j-1}$ or $e_{2j}$ for each $j \in [t]$, we have a matching $N$ in $\overline{L(\cH)}$ of size $t = e(\cH)-n$. Now applying Observation~\ref{lem:chromindex}, the proof is complete.
\end{proof}

\begin{proposition}\label{lem:chromindex1}
Let $\cH$ be an $n$-vertex linear hypergraph, and let $\{A,B \}$ be a partition of $\cH$ such that $|A|+|B|-n \leq |A|/4$.
If for every distinct intersecting $e,f \in A$, the pair $\{e, f\}$ is useful, then $\cH$ has a proper edge-colouring with $n$ colours, where each colour is assigned to at most two edges.
\end{proposition}
\begin{proof}
We may assume that $|A|+|B| \geq n+1$. Let $N$ be a matching of maximum size in $\overline{L(\cH)}$. If $|N| \geq e(\cH)-n = |A|+|B|-n$, then by Observation~\ref{lem:chromindex}, we have a proper colouring of $\cH$ with the desired properties. Thus we may assume that $|N| \leq |A|+|B|-n$. 
Recalling that $V(N)$ is a subset of edges in $\cH$, we have
\begin{equation*}
    |A \setminus V(N)| - 2(|A|+|B|-n) \geq |A|-4(|A|+|B|-n) \geq 0
\end{equation*}
since we assumed $|A|+|B|-n \leq |A|/4$. By the maximality of $N$, all pairs $e, f \in A \setminus V(N)$ are adjacent. Thus we may choose $2t \coloneqq 2(|A|+|B|-n)$ distinct $a_1 , \dots , a_{2t} \in A \setminus V(N)$ such that $\{a_{2i-1}, a_{2i}\}$ is a useful pair for each $i \in [t]$. Thus we can apply Proposition~\ref{lem:chromindex0} to $\{a_1 , \dots , a_{2t} \}$ to complete the proof.
\end{proof}

Recall that a finite projective plane of order $k$ is a hypergraph on $k^2 + k + 1$ vertices in which the edges have size $k + 1$.  Thus, in the proof of Lemma~\ref{extremal-case-lemma}, it is natural to  work with the minimum $k \in \mathbb N$ such that $k^2 + k + 1 \geq n$.  Assuming $n \in \mathbb N$, we then have
\begin{equation}\label{eqn:rangen}
   k^2 - k + 2 \leq n \leq k^2 + k + 1,
\end{equation}
where the left side of the above inequality holds by the minimality of $k$.  We will often use the following crude bounds that follow from \eqref{eqn:rangen}; if \eqref{eqn:rangen} holds, then since $(k + 1)^2 \geq k^2 + k + 1 \geq n$, we have
\begin{equation}\label{eqn:rangen-crude-lower}
k \geq \sqrt n - 1,
\end{equation}
and since $(k - 1)^2 \leq k^2 - k + 2$, we have
\begin{equation}\label{eqn:rangen-crude-upper}
k \leq \sqrt n + 1.
\end{equation}

Before proving Lemma~\ref{extremal-case-lemma}, we need one more proposition.

\begin{proposition}\label{na-prop}
  Let $0 < 1/n_0 \ll \delta \ll 1$, let $n \geq n_0$, and let $k \in \mathbb N$ satisfy \eqref{eqn:rangen}.  Let $\cH$ be an $n$-vertex linear hypergraph where every $e\in\cH$ satisfies $|e|\geq (1 - \delta)\sqrt n$, and let $e, f \in \cH$ be distinct intersecting edges of size at most $k$.  Let $w \in e\cap f$, and let $m$ be the number of edges of size at most $k - 1$ containing $w$.  If either
  \begin{enumerate}[(i)]
      \item\label{na-prop:(k-1)-edge} at least one of $e$ or $f$ has size at most $k - 1$ or
      \item\label{na-prop:m-bound} $m \leq 1/(3\delta)$,
  \end{enumerate} 
  then $\{e, f\}$ is a useful pair.
\end{proposition}
\begin{proof}
    If at least one of $e$ or $f$ has size at most $k - 1$, then\COMMENT{To obtain the last inequality, note that $(k-1)(k-2) + \frac{n-1}{(1-\delta)\sqrt{n}-1} \le k^2-3k+2 + \sqrt{n} (1+2\delta) \le n - 2k + \sqrt{n} (1+2\delta) \le n - 2(\sqrt{n}-1) + \sqrt{n} (1+2\delta) \le n-2$} 
\begin{align*}
    |N(e) \cap N(f)| &\leq (|e|-1)(|f|-1) + d(w) \leq (k-1)(k-2) + \frac{n - 1}{(1-\delta)\sqrt{n} - 1}\\ 
    & \leq (k^2 - k + 2) - 2k + (1 + 2 \delta) \sqrt{n} 
    \overset{\eqref{eqn:rangen}}{\leq} n - 2k + 3\sqrt{n}/2 \overset{\eqref{eqn:rangen-crude-lower}}{\leq} n - 2,
\end{align*}
as desired.

  Now we may assume $|e| = |f| =k$ and  $m \leq {1}/{(3\delta)}$; that is, \ref{na-prop:m-bound} holds but \ref{na-prop:(k-1)-edge} does not.  Let $A$ be the set of edges in $E_{\cH}(w)$ of size at most $k - 1$, let $B \coloneqq E_{\cH}(w) \setminus (\{e, f\} \cup A)$, and let us consider the number of vertices in $V(\cH) \setminus (e \cup f)$ sharing an edge with $w$.  Since $\cH$ is linear, this quantity is
  \begin{equation*}
      \sum_{e' \in A \cup B} (|e'| - 1) \leq n - |e \cup f| = n - 2k + 1.
  \end{equation*}
  By the definition of $m$, since every edge in $\cH$ has size at least $(1 - \delta)\sqrt n$, we have
\begin{equation*}
    \sum_{e' \in A \cup B} (|e'| - 1) \geq |A| ((1-\delta) \sqrt{n} - 1) + |B|(k-1) = m ((1-\delta) \sqrt{n} - 1) + (d(w) - 2 - m)(k-1).
\end{equation*}
Combining the two inequalities above, we have
\begin{equation*}
    d(w) - 2 \leq \frac{n-2k+1}{k-1} + m \left ( 1 - \frac{(1-\delta)\sqrt{n} - 1}{k-1} \right )
    \overset{\eqref{eqn:rangen-crude-upper}}{\leq} \frac{n-2k+1}{k-1} + 2\delta m.
\end{equation*}
Since $n \geq k^2 - k + 2$, we have
\begin{equation*}\label{eqn:d_w2}
    \frac{n-2k+1}{k-1} = n - 2 - \frac{n(k - 2) + 1}{k - 1} \leq n - 2 - \frac{k^3 - 3k^2 + 4k - 3}{k - 1} = n - 2 - (k - 1)^2 - \frac{k - 2}{k - 1}. 
\end{equation*}
Combining the two inequalities above and using $m \leq {1}/{(3\delta)}$, we have
\begin{equation}\label{eqn:d_w}
    d(w) - 2 \leq n - 2 - (k - 1)^2 - \frac{k - 2}{k - 1} + 2\delta m \leq n - 2 - (k - 1)^2.
\end{equation}
Thus, since $|e| = |f| = k$,
\begin{equation*}
    |N(e) \cap N(f)| \leq (|e|-1)(|f|-1) + d(w) - 2 \overset{\eqref{eqn:d_w}}{\leq} n-2,
\end{equation*}
so $\{e, f\}$ is a useful pair, as desired.
\end{proof}

Now we prove Lemma~\ref{extremal-case-lemma}.

\begin{proof}[Proof of Lemma~\ref{extremal-case-lemma}]
First of all, we may assume that $e(\cH) > n$ and $n \in \mathbb N$.  Choose $k\in \mathbb N$ minimum such that $k^2 + k + 1 \geq n$, and note that \eqref{eqn:rangen}--\eqref{eqn:rangen-crude-upper} hold.
Let $A^- \coloneqq \{e \in \cH \: : \: |e| \leq k-1 \}$, let $A^+ \coloneqq \{e \in \cH \: : \: |e| = k \}$, let $A \coloneqq A^- \cup A^+$, and let $B \coloneqq \{e \in \cH \: : \: |e| \geq k+1 \}$.
Note that
\begin{align*}
    \vol_{\cH}(B) \geq |B| \frac{k(k+1)}{n(n-1)} \overset{\eqref{eqn:rangen}}{\geq} \frac{|B|}{n}.
\end{align*}
Since $|e| \geq (1-\delta)\sqrt{n}$ for all $e\in\cH$, we have\COMMENT{The last inequality uses $\left.\binom{(1 - \delta)\sqrt n}{2}\middle/\binom{n}{2}\right. \geq \frac{(1 - \delta)^2n - (1 - \delta)\sqrt n}{n^2}  \geq \frac{1 - 2\delta + \delta^2 - (1 - \delta)/\sqrt n}{n} \geq \frac{1 - 2\delta}{n}$.}
\begin{align}\label{eqn:volA}
    \vol_{\cH}(A) \geq |A| \binom{(1-\delta)\sqrt{n}}{2} \binom{n}{2}^{-1} \geq |A| \frac{1 - 2\delta}{n}.
\end{align}
Combining the above two inequalities with $\vol_{\cH}(A) + \vol_{\cH}(B) \le 1$, we have
\begin{align}\label{eqn:abn_0}
    |A| + |B| - n \leq 2 \delta |A|.
\end{align}

If $|A^-| \leq 300$, then by Proposition~\ref{na-prop}\ref{na-prop:m-bound}, for any distinct intersecting $e,f \in A$, we have that $\{e, f\}$ is useful. Moreover,~\eqref{eqn:abn_0} implies that $|A| + |B| - n \leq |A|/4$. Thus we can apply Proposition~\ref{lem:chromindex1} to obtain a proper edge-colouring of $\cH$ with the desired properties, proving the lemma in this case.
Hence, we may assume that
\begin{equation}\label{eqn:bounda-}
|A^-| > 300.   
\end{equation}

Note that\COMMENT{To get the last inequality, note that $\frac{k(k-1)}{n-1} \geq \frac{(k^2 + k + 1) - (2k+1)}{n} \geq \frac{n - 3\sqrt{n}}{n}$ by~\eqref{eqn:rangen}, since $k = (1+o(1))\sqrt{n}$.}
\begin{align*}
    \vol_{\cH}(A^+ \cup B) &\geq |B| \frac{k(k+1)}{n(n-1)} + |A^+| \frac{k(k-1)}{n(n-1)}
    \overset{\eqref{eqn:rangen}}{\geq} \frac{|B|}{n} + \frac{|A^+|}{n}\left(\frac{n - 1 - 2k}{n - 1}\right)
    \overset{\eqref{eqn:rangen-crude-upper}}{\geq} \frac{|B|}{n} + \frac{|A^+|}{n} \left ( 1 - \frac{3}{\sqrt{n}} \right ).
\end{align*}
Similarly as in~\eqref{eqn:volA}, we have $\vol_{\cH}(A^-) \geq |A^-| (1 - 2\delta)/n$.
Using these two inequalities, we have
\begin{equation*}
    1 \geq \vol_{\cH}(A^+ \cup B) + \vol_{\cH}(A^-) \geq \frac{|B|}{n} + \frac{|A^+|}{n} \left ( 1 - \frac{3}{\sqrt{n}} \right ) + \frac{|A^-|}{n} \left(1 - 2\delta\right),
\end{equation*}
and thus, by rearranging terms, we obtain
\begin{align}\label{eqn:abn}
    e(\cH) - n = |A^-| + |A^+ \cup B| - n \leq \frac{3|A^+|}{\sqrt{n}} + 2 \delta |A^-|.
\end{align}
Thus if $|A^+| \leq \sqrt{n} |A^-| / 15$, then we have $|A^-| + |A^+ \cup B| - n \leq |A^-|/4$. Using this inequality and Proposition~\ref{na-prop}\ref{na-prop:(k-1)-edge}, we can apply Proposition~\ref{lem:chromindex1} with $A^-$ and $A^+ \cup B$ playing the roles of $A$ and $B$, respectively,  to obtain an edge-colouring of $\cH$ with the desired properties, proving the lemma in this case.
Thus we can assume that
\begin{equation}\label{eqn:bounda+}
    |A^+| > \frac{\sqrt{n}|A^-|}{15} \overset{\eqref{eqn:bounda-}}{\geq }20 \sqrt{n}.
\end{equation}

Now let $t\coloneqq e(\cH) - n$, let $L\coloneqq L(\cH)$ be the line graph of $\cH$, and let $N$ be a maximal matching in $\overline L$.  We assume $|N| < t$, as otherwise by Observation~\ref{lem:chromindex}, we obtain the desired proper edge-colouring of $\cH$, proving the lemma.  Combining this inequality with~\eqref{eqn:abn} and~\eqref{eqn:bounda+}, we have
\begin{equation}\label{eqn:sizeN}
  |N| < t \leq \frac{3|A^+|}{\sqrt{n}} + 2 \delta |A^-| < \frac{5|A^+|}{\sqrt{n}}.
\end{equation}

Most of the remainder of the proof is devoted to the following claim.
\begin{claim}\label{claim:useful-pairs}
    There are $2t$ distinct $e_1, \dots, e_{2t} \in A^{+}$ such that
\begin{enumerate}[(i)]
\item\label{extremal-pairwise-intersecting} $e_1, \dots, e_{2t}$ are pairwise intersecting and
\item\label{extremal-useful-pairs} $\{e_{2i - 1}, e_{2i}\}$ is useful for $i\in[t]$.
\end{enumerate}
\end{claim}
\begin{claimproof}
Since $N$ is maximal, $\cH\setminus V(N)$ is a clique in $L$ (recall that $V(N)$ is a subset of edges in $\cH$).  We choose $e_1, \dots, e_{2t}$ in $A^+\setminus V(N)$, which will ensure that~\ref{extremal-pairwise-intersecting} holds.

For each $x \in V(\cH)$, let $A_x \coloneqq \{e \in A^- \: : \: x \in e \}$, and let $V_{\rm bad} \coloneqq \{ x \in V(\cH) \: : \: |A_x| \geq (4\delta)^{-1} \}$.  Thus, if $e, f\in\cH$ are distinct and intersecting such that $w\notin V_{\rm bad}$ where $w\in e\cap f$, then $\{e, f\}$ is useful by Proposition~\ref{na-prop}\ref{na-prop:m-bound}.  We choose $e_1, \dots, e_{2t}$ such that $w\in e_{2i - 1}\cap e_{2i}$ satisfies $w\notin V_{\rm bad}$, which will ensure that~\ref{extremal-useful-pairs} holds.  Let $\mathcal{P} \coloneqq \{(w,e) \: : \: \textrm{$w \in V_{\rm bad}$ and $w \in e \in A^-$}\}$, and note that $|V_{\rm bad}| \cdot (4 \delta)^{-1} \leq |\mathcal{P}| \leq (k-1) |A^-| \leq 2 \sqrt{n} |A^-|$ by \eqref{eqn:rangen-crude-upper}. Thus,
\begin{equation}\label{claim:nobad}
    |V_{\rm bad}| \leq 8 \delta |A^-| \sqrt{n}.
\end{equation}

Now let $A^* \coloneqq \{e \in A^+ \: : \: |e \cap V_{\rm bad}| \geq \sqrt{\delta n} \}$, and note that since every vertex is contained in at most $(n - 1) / (k - 1)$ edges of $A^+$, we have
\begin{equation*}
    |A^*| \sqrt{\delta n} \leq \left|\left\{(x, e) : x \in e \cap V_{\rm bad}, e \in A^* \right\}\right| \leq |V_{\rm bad}| \frac{n-1}{k - 1} \overset{\eqref{claim:nobad}}{\leq} 16 \delta n |A^-|.
\end{equation*}
Therefore,
\begin{equation}\label{eqn:bounda*}
    |A^*| \leq 16 \sqrt{\delta n} |A^-| \overset{\eqref{eqn:bounda+}}{\leq} \frac{|A^+|}{20},
\end{equation}
and thus,
\begin{equation}\label{eqn:bounda+_2}
|A^+ \setminus (A^* \cup V(N))| \geq |A^+| - |A^*| - |V(N)| \overset{\eqref{eqn:sizeN},\eqref{eqn:bounda*}}{\geq} \frac{9|A^+|}{10} \overset{\eqref{eqn:bounda+}}{\geq} 18 \sqrt{n}.
\end{equation}

Now we iterate the following procedure for $i \in [t']$, where $t' \coloneqq \lceil |A^+| / 4 \rceil$. Suppose we have chosen  distinct $e_1 , \dots , e_{2(i-1)} \in A^+ \setminus (A^* \cup V(N))$ such that $\{e_{2j-1} , e_{2j} \}$ is useful for each $ j \in [i - 1]$. Consider the set $S_i \coloneqq A^+ \setminus \left(A^* \cup V(N) \cup \{e_1 , \dots , e_{2(i-1)} \}\right)$,
which has size
\begin{equation}
\label{eq:size:Si}
    |S_i| \geq |A^+ \setminus (A^* \cup V(N))| - 2(t'-1) \overset{\eqref{eqn:bounda+_2}}{\geq} \frac{9|A^+|}{10} - \frac{|A^+|}{2} \overset{\eqref{eqn:bounda+}}{\geq} 5 \sqrt{n}.
\end{equation}
We first show that there exists a useful pair $\{e_{2i-1}, e_{2i}\} \subseteq S_i$. 
For any $e \in S_i$, we have $|e \cap V_{\rm bad}| \le \sqrt{\delta n}$ since $S_i \subseteq A^{+} \setminus A^*$. Therefore, letting $\mathcal{P}_i \coloneqq \{(w,e) : e \in S_i,\: w \in e\setminus V_{\rm bad}\}$, we have the following.

\begin{align*}
    \frac{1}{|V(\cH) \setminus V_{\rm bad}|}\sum_{w \in V(\cH) \setminus V_{\rm bad}} d_{S_i}(w) &= \frac{|\mathcal{P}_i|}{|V(\cH) \setminus V_{\rm bad}|} \geq \frac{|\mathcal{P}_i|}{n}
    = \frac{1}{n} \sum_{e \in S_i} |e \setminus V_{\rm bad}| \\&\geq \frac{|S_i| (k - \sqrt{\delta n})}{n} \overset{\eqref{eq:size:Si},\eqref{eqn:rangen}}{\geq} \frac{5 \sqrt{n} \cdot \sqrt{n}/2}{n} > 2.
\end{align*}
Note that $V(\cH) \setminus V_{\rm bad} \ne \varnothing$ since $|e \setminus V_{\rm bad}| \geq (1-\delta-\sqrt{\delta})\sqrt{n}$ for each $e \in A^+ \setminus A^*$, and $A^+ \setminus A^* \ne \varnothing$ by \eqref{eqn:bounda+_2}.
Thus, the inequality above implies that there exists a vertex $w \in V(\cH) \setminus V_{\rm bad}$ and distinct edges $e_{2i-1},e_{2i} \in S_i$ such that $w \in e_{2i-1} \cap e_{2i}$, and by Proposition~\ref{na-prop}\ref{na-prop:m-bound}, the pair $\{e_{2i-1},e_{2i} \} \subseteq S_i$ is useful. 

The above procedure constructs a useful pair $\{e_{2i-1} , e_{2i} \} \subseteq A^+ \setminus (A^* \cup V(N))$ for each $i \in [t']$.  Recall that since $N$ is maximal, the elements of $A^+ \setminus V(N)$ are pairwise intersecting.  Since $t' \geq |A^+|/4 > t = e(\cH) - n$ by~\eqref{eqn:sizeN}, $e_1, \dots, e_{2t}$ satisfy~\ref{extremal-pairwise-intersecting} and~\ref{extremal-useful-pairs}, as claimed.
\end{claimproof}

Now by combining Claim~\ref{claim:useful-pairs} and Proposition~\ref{lem:chromindex0}, there is a proper edge-colouring of $\cH$ with $n$ colours such that each colour is assigned to at most two edges, which completes the proof of Lemma~\ref{extremal-case-lemma}.
\end{proof}

\section{Colouring large and medium edges}\label{large-edge-section}

The main result of this section is the following, which we use in the proof of Theorem~\ref{main-thm} to colour large and medium edges.  

\begin{theorem}\label{large-edge-thm}
  Let $0 < 1/n_0 \ll 1/r_0 \ll  1/r_1, \beta \ll \gamma_1 \ll \sigma \ll \delta \ll \gamma_2\ll 1$, and let $n \geq n_0$.
  If $\mathcal H$ is an $n$-vertex linear hypergraph where every $e\in \cH$ satisfies $|e| > r_1$, then at least one of the following holds:
  \begin{enumerate}[(\ref{large-edge-thm}:a), topsep = 6pt]
    \item There exists a proper edge-colouring of $\mathcal H$ using at most $(1 - \sigma)n$ colours such that
    \begin{enumerate}[(i)]
    \item\label{non-extremal-huge-edge} every colour assigned to a huge edge is assigned to no other edge,
    \item\label{non-extremal-medium-bound} every medium edge is assigned a colour from a set $C_{\mathrm{med}}$ of size at most $\gamma_1 n$ such that for every $c\in C_{\mathrm{med}}$, at most $\gamma_1 n$ vertices are incident to an edge coloured $c$, and
    \item\label{non-extremal-colour-class-bound} for every colour $c\notin C_{\mathrm{med}}$ not assigned to a huge edge, at most $\beta n$ vertices are incident to an edge coloured $c$.

    \end{enumerate}\label{non-extremal-colouring}
  \item There exists a set of FPP-extremal edges of volume at least $1 - \delta$ and a proper edge-colouring of $\cH$ using at most $n$ colours such that
    \begin{enumerate}[(i)]
    \item\label{extremal-huge-edge} for every colour $c$ assigned to a huge edge, at most $\delta n$ vertices are incident to an edge coloured $c$,
    \item\label{extremal-medium-bound} every medium edge is assigned a colour from a set $C_{\mathrm{med}}$ of size at most $\gamma_2 n$ such that for every $c\in C_{\mathrm{med}}$, at most $\gamma_1 n$ vertices are incident to an edge coloured $c$, and
    \item\label{extremal-colour-class-bound} for every colour $c\notin C_{\mathrm{med}}$ not assigned to a huge edge, at most $\beta n$ vertices are incident to an edge coloured $c$.
    \end{enumerate}\label{extremal-colouring}
  \end{enumerate}
\end{theorem}

Note that every linear hypergraph $\cH$ satisfies $\vol_{\cH}(\cH) \leq 1$, so in~\ref{extremal-colouring}, the FPP-extremal edges contain almost all of the pairs of vertices.

We now deduce Theorem~\ref{main-thm-2} from Theorem~\ref{thm:kahn} and Theorem~\ref{large-edge-thm}. Note that we don't need the additional properties \ref{non-extremal-huge-edge}--\ref{extremal-colour-class-bound} for this.

\begin{proof}[Proof of Theorem~\ref{main-thm-2}]
Without loss of generality, we may assume that $\delta$ is sufficiently small. Let $0 < 1/n_0 \ll 1/r_0 \ll 1/r_1, \beta \ll \gamma \ll \sigma \ll \delta \ll 1$, and recall $\cH_{\rm small}$, $\cH_{\rm med}$, $\cH_{\rm large}$, and $\cH_{\rm ex}$  were defined in Definition~\ref{def:edgesize}.
By assumption, $e(\cH_{\rm ex}) \leq (1-3\delta)n$, so
\begin{equation*}
   \vol_{\cH}(\cH_{\rm ex}) \leq (1-3\delta)n \cdot \binom{(1+\delta)\sqrt{n}}{2} \binom{n}{2}^{-1} \leq \frac{(1-3\delta)(1+\delta)^2 n}{n-1} < 1 - \delta.
\end{equation*}
Hence, applying Theorem~\ref{large-edge-thm} with $\cH_{\rm med} \cup \cH_{\rm large}$, $\gamma$, $\sigma$, and $\delta$ playing the roles of $\cH$, $\gamma_1$, $\sigma$, and $\delta$, respectively, we obtain a proper edge-colouring $\phi' : \cH_{\rm med} \cup \cH_{\rm large} \to [(1-\sigma)n]$ and $C_{\rm med} \subseteq [(1-\sigma)n]$ such that every $e \in \cH_{\rm med}$ satisfies $\phi'(e) \in C_{\rm med}$ and $|C_{\rm med}| \leq \gamma n$.

For every $e \in \cH_{\rm small}$, let $C(e) := \{ \phi'(f) \: : \: f \in \cH_{\rm large}\:,\:e \cap f \ne \varnothing \}$. Note that for each vertex $w \in V(\cH)$, there are at most $2 r_0^{-1} n$ edges of $\cH_{\rm large}$ incident to $w$. Therefore, for every $e \in \cH_{\rm small}$, there are at most $2r_1 r_0^{-1} n$ edges $f \in \cH_{\rm large}$ such that $e \cap f \ne \varnothing$. Hence, $|C(e)| \leq \beta n$. So applying Theorem~\ref{thm:kahn} with $\cH_{\rm small}$, $3 \beta$, $r_1$, $\lfloor (1-\delta)n \rfloor$, $[(1-\sigma)n] \setminus C_{\rm med}$ playing the roles of $\cH$, $\alpha$, $r$, $D$, $C$, respectively, we obtain a proper edge-colouring $\phi'' : \cH_{\rm small} \to [(1-\sigma)n] \setminus C_{\rm med}$ such that for every $e \in \cH_{\rm small}$, $\phi''(e) \not \in C(e)$, which implies that $\phi''(e) \ne \phi'(f)$ for every $f \in \cH_{\rm large}$ with $e \cap f \ne \varnothing$. Hence $\phi := \phi' \cup \phi'' : \cH \to [(1-\sigma)n]$ is a proper edge-colouring, as desired.
\end{proof}

\subsection{Reordering}

If $\ordering$ is a linear ordering of the edges of a hypergraph $\cH$, for each $e\in\cH$, we define $\fwdnbr_\cH(e)\coloneqq \{f \in N_\cH(e) : f\ordering e \}$ and $\fwddeg_\cH(e)\coloneqq |\fwdnbr_\cH(e)|$.  We omit the subscript $\cH$ when it is clear from the context. 
For each $e\in \cH$, we also let $\cH^{\ordering e} \coloneqq \{f\in \cH : f \ordering e\}$.  The main result of this subsection is the following key lemma, which we use to find the ordering of the edges of $\cH$ mentioned in Section~\ref{subsection:large-edge-overview}.

\begin{lemma}[Reordering lemma]\label{reordering-lemma}
  Let $0 < 1/r_1 \ll \tau, 1/K$ where $\tau < 1$, $K \geq 1$, and $1 - \tau - 7\tau^{1/4}/K > 0$. 
  If $\cH$ is an $n$-vertex linear hypergraph where every $e\in \cH$ satisfies $|e| \geq r_1$, then there exists a linear ordering $\ordering$ of the edges of $\cH$ such that at least one of the following holds.
  \begin{enumerate}[(\ref{reordering-lemma}:a), topsep = 6pt]
  \item Every $e\in\cH$ satisfies $\fwddeg(e) \leq (1 - \tau)n$.\label{reordering-good}
  \item There is a set $W\subseteq \cH$ such that
    \begin{enumerate}[(W1)]
    \item\label{W-max-size} $\max_{e\in W}|e| \leq (1 + 3\tau^{1/4}K^{4})\min_{e\in W}|e|$ and
    \item\label{W-volume} $\vol_{\cH}(W) \geq \frac{(1 - \tau - 7\tau^{1/4}/K)^2}{1 + 3\tau^{1/4}K^4}$.
    \end{enumerate}
    Moreover, if $e^*$ is the last edge of $W$, then
    \begin{enumerate}[(O1)]
    \item for all $f\in \cH$ such that $e^*\ordering f$ and $f\neq e^*$, we have $\fwddeg(f) \leq (1 - \tau)n$ and\label{ordering-goodness}
    \item for all $e,f\in \cH$ such that $f\ordering e \ordering e^*$, we have $|f| \geq |e|$\label{ordering-by-size} (and in particular, $|e^*| = \min_{e \in W}|e|$).
    \end{enumerate}
    \label{reordering-volume}
  \end{enumerate}
\end{lemma}

The main application of Lemma~\ref{reordering-lemma} will be with $\tau \ll 1$ and $K=1$ to obtain a set $W_2$ consisting of edges of similar size occupying most of the volume of $\mathcal{H}$ (see the proof of Theorem~\ref{large-edge-thm}). As discussed in Section~\ref{subsection:large-edge-overview}, $W_2$ can be coloured efficiently via Corollary~\ref{sparsity-corollary}.
Prior to this, we will also apply Lemma~\ref{reordering-lemma} with $\tau \sim 1$ and $1/K \ll 1$. The structural information gained from this application will help us to extend the above colouring to all of $\mathcal{H}$.
Note that if \ref{reordering-volume} holds in Lemma~\ref{reordering-lemma}, then \ref{ordering-by-size} implies that $|e| \geq |e^*|$, since $e^* \ordering e^*$.

The set $W$ in Lemma~\ref{reordering-lemma} is obtained via a double counting argument, which shows that if there is no ordering satisfying~\ref{reordering-good}, then we can (essentially) take $W$ to be the first and second neighbourhood (in the line graph of $\mathcal{H}$) of a suitable edge $e^*$.

\begin{proposition}\label{fwdnbr-prop}
  Let $\alpha_1, \alpha_2, \tau \geq 0$.  Let $\cH$ be an $n$-vertex linear hypergraph where every $e\in\cH$ satisfies $|e|\geq 1 + \alpha_2$.  Let $e\in \cH$, let $r\coloneqq |e|$, let $F_1 \subseteq \{f \in N(e) : |f| \geq (1 + \alpha_1)r\}$, let $F_2 \subseteq \{f \in N(e) : (1 + \alpha_1)r > |f| \geq r / (1 + \alpha_2)\}$, let $m_1 \coloneqq |F_1|$, and let $m_2 \coloneqq |F_2|$.  If $r > 1 + \alpha_2$, then
  \begin{equation}\label{fwddeg-inequality}\tag{i}
    (1 + \alpha_1)m_1 + \frac{m_2}{1 + \alpha_2} \leq n + \frac{(1 + \alpha_2)n}{r - 1 - \alpha_2}.
  \end{equation}
  Moreover, if $m_1 + m_2 \geq (1 - \tau)n$ and $\alpha_1 > 0$, then
  \begin{equation}\label{fwddeg-far-bound}\tag{ii}
    m_1 \leq \left(\tau + \frac{(1 + \alpha_2)(1 + \alpha_2 r)}{r - 1 - \alpha_2}\right)\frac{n}{\alpha_1}.
  \end{equation}
\end{proposition}

If $\ordering$ is an ordering of the edges of an $n$-vertex linear hypergraph $\cH$ satisfying $e \ordering f$ if $|e| > |f|$, then Proposition~\ref{fwdnbr-prop}\eqref{fwddeg-inequality} with $\alpha_1, \alpha_2, \tau = 0$ implies that every $e\in\cH$ with $|e| \geq 2$ satisfies $\fwddeg(e) \leq (1 + 1/(|e| - 1))n$.  This well-known fact immediately implies that every linear $n$-vertex hypergraph $\cH$ satisfies $\chi'(\cH) \leq 2n + 1$\COMMENT{In fact, using the more precise bound $\fwddeg(e) \leq |e|(n - |e|) / (|e| - 1)$, it is straightforward to show that if $\cH$ is simple, then $\chi'(\cH) \leq 2(n - 2) + 1 = 2n - 3$.  If $\cH$ contains singleton edges, then we have the bound $\chi'(\cH) \leq 2n - 2$.}, and if all of its edges have size at least $r_1\geq 2$, then $\chi'(\cH) \leq (1 + 1/(r_1 - 1))n + 1$.

\begin{proof}[Proof of Proposition~\ref{fwdnbr-prop}]
  There are exactly $r (n - r)$ pairs of vertices $\{u, v\}$ of $\cH$ where $u\notin e$ and $v\in e$.  Thus, since $\cH$ is linear, $\sum_{f \in N(e)} (|f| - 1) \leq r (n - r)$.  In particular,
  \begin{equation*}
    \sum_{f \in F_1} \left((1 + \alpha_1)r - 1\right) + \sum_{f \in F_2} \left(r / (1 + \alpha_2) - 1\right) \leq r (n - r).
  \end{equation*}
  Dividing both sides of this inequality by $r$ and rearranging terms, we obtain
  \begin{equation}\label{fwddeg-inequality1}
    (1 + \alpha_1)m_1 + m_2 / (1 + \alpha_2) \leq n - r + (m_1 + m_2)/r.
  \end{equation}
  Similarly, we have
  \begin{equation}\label{fwddeg-inequality2}
    m_1 + m_2 \leq \frac{r(n - r)}{r/(1 + \alpha_2) - 1} \leq \frac{(1 + \alpha_2)rn}{r - 1 - \alpha_2}.
  \end{equation}
  Substituting this inequality in the right side of~\eqref{fwddeg-inequality1}, we obtain~\eqref{fwddeg-inequality}, as desired.

  Now suppose additionally $m_1 + m_2 \geq (1 - \tau)n$ and $\alpha_1 > 0$.  By combining the former inequality with~\eqref{fwddeg-inequality}\COMMENT{The left side of~\eqref{fwddeg-inequality} is $m_1 + m_2 + \alpha_1 m_1 + (1/(1 + \alpha_2) - 1)m_2$, so if $m_1 + m_2 \geq (1 - \tau)n$, then $(1 - \tau)n + \alpha_1 m_1 + (1/(1 + \alpha_2) - 1)m_2$ is at most the right side of~\eqref{fwddeg-inequality}.}, we obtain $\alpha_1 m_1  \leq \tau n + n(1 + \alpha_2) / (r - 1 - \alpha_2) + (1 - 1/(1 + \alpha_2))  m_2$.  Since $1 - 1/(1 + \alpha_2) \leq \alpha_2$, we have $\alpha_1 m_1  \leq \tau n + n(1 + \alpha_2) / (r - 1 - \alpha_2) + \alpha_2 m_2$, and by combining this inequality with the bound on $m_2$ from~\eqref{fwddeg-inequality2}, we obtain~\eqref{fwddeg-far-bound}, as desired.
\end{proof}

\begin{proof}[Proof of Lemma~\ref{reordering-lemma}]
  We consider an ordering $\ordering$ of the edges of $\cH$ satisfying~\ref{ordering-goodness} and~\ref{ordering-by-size} (without the ``in particular" part) for some $e^*\in\cH$ such that $e(\cH^{\ordering e^*})$ is minimum.  
  Note that such an ordering exists -- in particular, any ordering where $f \ordering e$ whenever $|f| > |e|$, that is, an ordering in which the edge sizes are monotonically decreasing, satisfies \ref{ordering-goodness} and \ref{ordering-by-size} for $e^*$, where $e^*$ is the last edge in the ordering. 

  If $e(\cH^{\ordering e^*}) = 1$, then $\fwddeg(e^*) = 0$ and \ref{ordering-goodness} implies that every edge $e \neq e^*$ in $\cH$ satisfies $\fwddeg(e) \leq (1 - \tau)n$, so \ref{reordering-good} holds, as desired.  Thus, we assume we do not have this case.  Now we have $\fwddeg(e^*) > (1 - \tau)n$, or else the predecessor of $e^*$ also satisfies~\ref{ordering-goodness} and~\ref{ordering-by-size}, contradicting the choice of $e(\cH^{\ordering e^*})$ to be minimum.

  Let $r \coloneqq |e^*|$, and let $W \coloneqq \{f \ordering e^* : |f| \leq (1 + 3\tau^{1/4}K^{4})r\}$. 
  It remains to show that $W$ satisfies~\ref{reordering-volume}.  By the choice of $\ordering$, every $e \in \fwdnbr(e^*)$ satisfies
  \begin{equation}\label{nbrs-ahead-of-e^*-bound}
    |N(e) \cap \cH^{\ordering e^*}| > (1 - \tau)n,
  \end{equation}
  or else we can make $e$ the successor of $e^*$.  Let $X \coloneqq \{e \in \fwdnbr(e^*) : |e| \leq (1 + K\sqrt\tau)r\}$.  By \eqref{nbrs-ahead-of-e^*-bound} and \ref{ordering-by-size}, we may apply Proposition~\ref{fwdnbr-prop}\eqref{fwddeg-far-bound} with $K \sqrt \tau$, $0$, and $\tau$ playing the roles of $\alpha_1$, $\alpha_2$, and $\tau$, respectively, and with $e^*$, $\fwdnbr(e^*) \setminus X$, and $X$ playing the roles of $e$, $F_1$, and $F_2$, respectively, to obtain
  \begin{equation*}
      |\fwdnbr(e^*)\setminus X| \leq \frac{\tau + 2/r}{K\sqrt\tau}n \leq \frac{2\sqrt\tau}{K}n,
  \end{equation*}
  where in the last inequality we used that $r \geq r_1$ and $1 / r_1 \ll \tau$.  Thus, again by \eqref{nbrs-ahead-of-e^*-bound}, we have
  \begin{equation}\label{nbrs-far-ahead-of-e^*-bound}
    |X| \geq (1 - \tau)n - |\fwdnbr(e^*)\setminus X| \geq \left(1 - \tau - \frac{2\sqrt\tau}{K}\right)n.
  \end{equation}

  Consider $e \in X$. We now aim to apply Proposition~\ref{fwdnbr-prop}\eqref{fwddeg-far-bound} to $e$ with $K^3\tau^{1/4}$ and $K\sqrt\tau$ playing the roles of $\alpha_1$ and $\alpha_2$, respectively, and with $\{f \in N(e) \cap \cH^{\ordering e^*} : |f| \geq (1 + K^3\tau^{1/4})|e|\}$ and $\{f \in N(e) \cap \cH^{\ordering e^*} : |f| < (1 + K^3\tau^{1/4})|e|\}$ playing the roles of $F_1$ and $F_2$, respectively.  
Since $1 / r_1 \ll \tau, 1/K$, we have $r_1 \geq 1 + K\sqrt \tau$, as required to apply Proposition~\ref{fwdnbr-prop} with this choice of $\alpha_2$.
  By \ref{ordering-by-size} and since $e\in X$, we have $|f| \geq r \geq |e| / (1 + K\sqrt \tau)$ for every $f \in N(e) \cap \cH^{\ordering e^*}$, as required to apply Proposition~\ref{fwdnbr-prop} with this choice of $\alpha_2$ and $F_2$.
  Thus, by \eqref{nbrs-ahead-of-e^*-bound} we can apply Proposition~\ref{fwdnbr-prop}\eqref{fwddeg-far-bound} to deduce that for every $e \in X$ we have
  \begin{align*}
    |\{f \in N(e) \cap \cH^{\ordering e^*} : |f| \geq (1 + K^3\tau^{1/4})|e|\}| &\leq \left(\tau + \frac{1 + K\sqrt\tau + K\sqrt\tau |e| + K^2\tau |e|}{|e|/2}\right)\frac{n}{K^3\tau^{1/4}}\\
    &\leq \frac{6K^2\sqrt\tau}{K^3\tau^{1/4}}n = \frac{6\tau^{1/4}}{K}n.
  \end{align*}
  (In the second inequality, we used that $|e| \geq r$ by~\ref{ordering-by-size}.)
  Since every $e \in X$ satisfies $(1 + K^3\tau^{1/4})|e| \leq (1 + K^3\tau^{1/4})(1 + K\sqrt \tau)r \leq (1 + 3K^4\tau^{1/4})r$, we have $\{f \in N(e) \cap \cH^{\ordering e^*} : |f| \geq (1 + K^3\tau^{1/4})|e|\}\supseteq (N(e) \cap \cH^{\ordering e^*}) \setminus W$, so the inequality above implies that every $e \in X$ satisfies
  \begin{equation}\label{nbrhood-outside-W-bound}
    |(N(e) \cap \cH^{\ordering e^*}) \setminus W| \leq \frac{6\tau^{1/4}}{K}n.
  \end{equation}

  Now we use these inequalities to lower bound the size of $W$.  First we claim that every $e\in X$ satisfies
  \begin{equation}\label{X-nbrs-in-W-bound}
    |N(e) \cap (W\setminus N(e^*))| \geq (1 - \tau - 7\tau^{1/4}/K)n - (1 + K\sqrt\tau)r^2.
  \end{equation}
  To that end, we bound $|N(e)\cap \fwdnbr(e^*)|$, as follows.  Let $N_1 \coloneqq \{f\in N(e) \cap \fwdnbr(e^*): f\cap e^* = e\cap e^*\}$ and $N_2 \coloneqq (N(e)\cap \fwdnbr(e^*))\setminus N_1$.  Since $\cH$ is linear and every edge in $\fwdnbr(e^*)$ has size at least $r$, we have $|N_1| \leq n / (r - 1)$ and $|N_2| \leq (|e| - 1)(|e^*| - 1) \leq (1 + K\sqrt \tau)r^2$, where the last inequality uses that $e \in X$.  Thus, 
  \begin{equation}\label{eqn:nbrhood-intersection}
      |N(e)\cap \fwdnbr(e^*)| \leq \frac{n}{r - 1} + (1 + K\sqrt \tau)r^2.
  \end{equation}
  
  On the other hand, since $W \subseteq \cH^{\ordering e^*}$, $|N(e) \cap (W\setminus N(e^*))| = |(N(e) \cap \cH^{\ordering e^*}) \cap (W\setminus N(e^*))|$. Moreover, we have\COMMENT{In fact, the equality holds.}
  \begin{equation*}
   |(N(e) \cap \cH^{\ordering e^*}) \cap (W\setminus N(e^*))|
  \geq |N(e) \cap \cH^{\ordering e^*}| - |(N(e) \cap \cH^{\ordering e^*}) \setminus W| - |N(e)\cap \fwdnbr(e^*)|.
  \end{equation*}
  Combining this inequality with~\eqref{nbrs-ahead-of-e^*-bound},~\eqref{nbrhood-outside-W-bound},~\eqref{eqn:nbrhood-intersection}, one can see that ~\eqref{X-nbrs-in-W-bound} follows, as claimed.

  For every $f \in W\setminus N(e^*)$, we also have
  \begin{equation}\label{W-nbrs-in-X-bound}
    |N(f) \cap X| \leq |N(f) \cap N(e^*)| \leq |f||e^*| \leq (1 + 3\tau^{1/4}K^4)r^2.
  \end{equation}
  Since $\sum_{e\in X}|N(e)\cap(W\setminus N(e^*))| = |\{(e, f) : e \in X,\ f \in W\setminus N(e^*),\ e \in N(f)\}| = 
  \sum_{f\in W\setminus N(e^*)} | N(f) \cap X|$, by combining~\eqref{X-nbrs-in-W-bound} and~\eqref{W-nbrs-in-X-bound}, we have
  \begin{equation*}
      |X|\left((1 - \tau - 7\tau^{1/4}/K)n - (1 + K\sqrt\tau)r^2\right) \leq |W \setminus N(e^*)|(1 + 3\tau^{1/4}K^4)r^2.
  \end{equation*}
  Thus, by rearranging terms,
  \begin{equation*}
    |W\setminus N(e^*)| \geq  |X|\left(\frac{(1 - \tau - 7\tau^{1/4}/K)n}{(1 + 3\tau^{1/4}K^4)r^2} - \frac{1 + K\sqrt\tau}{1 + 3\tau^{1/4}K^4}\right),
  \end{equation*}
  and since $X\subseteq N(e^*) \cap W$, this inequality implies
\begin{align*}
  |W| &\geq |X|\left(\frac{(1 - \tau - 7\tau^{1/4}/K)n}{(1 + 3\tau^{1/4}K^4)r^2} + 1 - \frac{1 + K\sqrt\tau}{1 + 3\tau^{1/4}K^4}\right)\\
  &\geq \left(1 - \tau - \frac{2\sqrt\tau}{K}\right)\left(\frac{1 - \tau - 7\tau^{1/4}/K}{1 + 3\tau^{1/4}K^4}\right)\frac{n^2}{r^2},
\end{align*}
where the second inequality follows from~\eqref{nbrs-far-ahead-of-e^*-bound} and the fact that $1 \geq \frac{1 + K\sqrt\tau}{1 + 3\tau^{1/4}K^4}$.  Thus, $\vol_{\cH}(W) \geq |W|\binom{r}{2} / \binom{n}{2} \geq |W|\frac{r^2}{n^2}(1 - 1/r) \geq \frac{(1 - \tau - 7\tau^{1/4}/K)^2}{1 + 3\tau^{1/4}K^4}$,
so $W$ satisfies \ref{W-volume}.  Moreover, $\ordering$ satisfies \ref{ordering-goodness} and~\ref{ordering-by-size}, and by the definition of $W$ and \ref{ordering-by-size}, $W$ also satisfies \ref{W-max-size}, so \ref{reordering-volume} holds, as desired.
\end{proof}

\subsection{Colouring locally sparse graphs}
To prove Theorem~\ref{large-edge-thm} we use the following theorem~\cite[Theorem~10.5]{MR02}, which has been improved in~\cite{BJ15, BPP18, HdVK20}.

\begin{theorem}[Molloy and Reed~\cite{MR02}]\label{local-sparsity-lemma}
  Let $0 < 1/\Delta \ll \zeta \leq 1$.  Let $G$ be a graph with $\Delta(G) \leq \Delta$.  If every $v\in V(G)$ satisfies $e(G[N(v)]) \leq (1 - \zeta)\binom{\Delta}{2}$, then $\chi(G) \leq (1 - \zeta/e^6)\Delta$.
\end{theorem}

\begin{corollary}\label{sparsity-corollary}
Let $0 < 1/n_0, 1/r \ll \alpha \ll \zeta < 1$, let $n \geq n_0$, and suppose $r \leq (1 - \zeta)\sqrt n$.  If $\cH$ is an $n$-vertex linear hypergraph such that every $e\in\cH$ satisfies $|e| \in [r, (1 + \alpha)r]$, then $\chi'(\cH) \leq (1 - \zeta/500)n$.
\end{corollary}
\begin{proof}
  Let $\Delta \coloneqq (1 + \alpha)r(n - r)/(r - 1)$, and let $L\coloneqq L(\cH)$.  For every edge $e\in\cH$, there are at most $(1 + \alpha)r(n - r)$ pairs of vertices $\{u, v\}$ of $\cH$ where $u\notin e$ and $v\in e$.  Thus, since $\cH$ is linear and every edge has size at least $r$, we have $\Delta(L) \leq \Delta$.  Similarly, if $e, f\in\cH$ share a vertex, then $|N_L(e)\cap N_L(f)| \leq n/(r - 1) + (1 + \alpha)^2r^2 \leq (1 - 5\zeta / 6)n$.  Thus, every $v\in V(L)$ satisfies $e(L[N(v)]) \leq \Delta (1 - 5\zeta / 6)n / 2 \leq (1 - 5\zeta / 6)\binom{\Delta}{2}$.  Therefore by Theorem~\ref{local-sparsity-lemma}, $\chi'(\cH) = \chi(L) \leq (1 - 5\zeta / (6e^6))\Delta \leq (1 - \zeta / 500)n$, as desired.
\end{proof}

In the proof of Theorem~\ref{weak-stability-thm}, we use the following theorem, which has been further improved in~\cite{AIS19, AKS99, DKPS20, Vu02}.
\begin{theorem}[Alon, Krivelevich, and Sudakov~\cite{AKS99}]\label{aks-local-sparsity}
Let $0 < \zeta, 1/K_{\ref{aks-local-sparsity}} \ll 1$.  Let $G$ be a graph with $\Delta(G) \leq \Delta$.  If every $v\in V(G)$ satisfies $e(G[N(v)]) \leq \zeta \Delta^2$, then $\chi(G) \leq K_{\ref{aks-local-sparsity}}\Delta / \log(1/\zeta)$.
\end{theorem}

We need the following corollary of Theorem~\ref{aks-local-sparsity}.  The proof is nearly identical to the proof of Corollary~\ref{sparsity-corollary}, with Theorem~\ref{local-sparsity-lemma} replaced by Theorem~\ref{aks-local-sparsity}. 

\begin{corollary}\label{aks-sparsity-corollary}
Let $0 < 1 / n_0 \ll \eta \ll \alpha, \eps < 1$, and let $n \geq n_0$.  If $\cH$ is an $n$-vertex linear hypergraph such that every $e\in \cH$ satisfies $1 / \eta \leq |e| \leq \eta\sqrt n$ and $\min_{e\in\cH}|e| \geq \alpha\max_{e\in\cH}|e|$, then $\chi'(\cH) \leq \eps n$.
\end{corollary}
\begin{proof}
Let $r\coloneqq \min_{e\in\cH}|e|$, let $\Delta \coloneqq (r / \alpha)n/(r - 1)$, and let $L\coloneqq L(\cH)$.  For every edge $e\in\cH$, there are $|e|(n - |e|) \leq (r / \alpha)n$ pairs of vertices $\{u, v\}$ of $\cH$ where $u\notin e$ and $v\in e$.  Thus, since $\cH$ is linear and every edge has size at least $r$, we have $\Delta(L) \leq \Delta$.
Let $e, e'\in\cH$ be adjacent in $L$, let $N_1\coloneqq \{f \in N(e)\cap N(e') : f\cap e = f \cap e'\}$, and let $N_2 \coloneqq (N(e) \cap N(e'))\setminus N_1$.  Since $\cH$ is linear and every $e\in\cH$ satisfies $1 / \eta \leq |e| \leq \eta\sqrt n$, we have $|N_1| \leq n / (1/\eta - 1) \leq 3\eta n / 2$ and $|N_2| \leq (|e| - 1)(|e'| - 1) \leq \eta^2 n$.  Therefore $|N(e) \cap N(e')| \leq 2\eta n$, so every $v \in V(L)$ satisfies $e(L[N(v)]) \leq \eta n \Delta \leq \eta \alpha \Delta^2$.  
Since $L$ satisfies $\Delta(L) \leq \Delta$ and $e(L[N(v)]) \leq \eta \alpha \Delta^2$, by Theorem~\ref{aks-local-sparsity} applied with $\eta \alpha$ playing the role of $\zeta$, we have $\chi'(\cH) = \chi(L) \leq K_{\ref{aks-local-sparsity}}\Delta / \log(1 / (\eta \alpha)) \leq \eps n$, as desired.
\end{proof}

We remark that the proof of Corollary~\ref{aks-sparsity-corollary} is also similar to that of~\cite[Theorem 1.1]{faber2010}, where a similar statement was proved for uniform regular linear hypergraphs (which implies that the EFL conjecture holds for all $r$-uniform regular linear $n$-vertex hypergraphs satisfying $c \leq r \leq \sqrt{n}/c$ for some constant $c>0$).

\subsection{Proof of Theorems~\ref{weak-stability-thm} and~\ref{large-edge-thm}}
Let $\phi$ be a proper edge-colouring of an $n$-vertex hypergraph $\cH$.  For $\alpha\in(0, 1)$, we say $\phi$ is \textit{$\alpha$-bounded} if every colour $c$ satisfies at least one of the following: $c$ is assigned to at most one $e\in\cH$, or $\phi^{-1}(c)$ covers at most $\alpha n$ vertices of $\cH$.

\begin{proposition}\label{few-large-colour-classes}
  Let $\alpha > 0$, and let $\cH$ be an $n$-vertex linear hypergraph where every $e\in \cH$ satisfies $|e| \geq r + 1$.  
  \begin{enumerate}[(i)]
    \item If $M_1, \dots, M_t$ are pairwise edge-disjoint matchings in $\cH$ that each cover at least $\alpha n$ vertices, then $t \leq n / (\alpha r)$.\label{few-large-colour-classes-bound}
    \item There is an $\alpha$-bounded proper edge-colouring of $\cH$ using at most $\chi'(\cH) + 2n/(\alpha^2 r)$ colours.\label{few-large-colour-classes-splitting}
  \end{enumerate}
\end{proposition}
\begin{proof}
  Since $\cH$ is linear, we have $1 \geq \vol_{\cH}\left(\bigcup_{i=1}^tM_i\right) = \sum_{i=1}^t\vol_{\cH}(M_i).$ Moreover,\COMMENT{For each $i\in [t]$, the volume of $M_i$ is at least as large as that of a matching covering exactly $\lceil \alpha n\rceil$ vertices, consisting of edges of size at most $r + 2$, with as many edges of size $r + 1$ as possible.}~\COMMENT{Jensen's Inequality implies that if $f : \mathbb R\rightarrow \mathbb R$ is a convex function, then for any $x_1, \dots, x_k \in \mathbb R$ and positive weights $a_1, \dots, a_k$, we have
  \begin{equation*}
      f\left(\frac{\sum_{i=1}^k a_i x_i}{\sum_{i=1}^k a_i}\right) \leq \frac{\sum_{i=1}^ka_i f(x_i)}{\sum_{i=1}^k a_i}.
  \end{equation*}
  Thus, since $x\mapsto \binom{x}{2}$ is convex, applying Jensen's Inequality with $|M_i|$ as $k$ and $a_1, \dots, a_k  = 1$, we have
  \begin{equation*}
      \binom{n}{2}\vol_{\cH}(M_i) \geq \binom{\sum_{e\in M_i}|e| / |M_i|}{2}|M_i| \geq \binom{r + 1}{2}\frac{\alpha n}{r + 1}.
  \end{equation*}} for each $i\in[t]$,  $\vol_{\cH}(M_i) \geq \left.\frac{\alpha n}{r+1}\binom{r + 1}{2}\middle/ \binom{n}{2}\right. \geq \alpha r / n$.  Combining these inequalities, we have $t \leq n / (\alpha r)$, as desired for~\ref{few-large-colour-classes-bound}.
  
  Now let $\phi$ be a proper edge-colouring of $\cH$ using a set $C$ of $\chi'(\cH)$ colours.  For each $c\in C$, let $M_c \coloneqq \phi^{-1}(c)$, and let $C' \coloneqq \{c \in C : |V(M_c)| > \alpha n\}$.  By~\ref{few-large-colour-classes-bound}, we have $|C'| \leq n / (\alpha r)$.  For each $c\in C'$, there is a partition of $M_c$ into a set $\cM_c$ of pairwise disjoint matchings such that every $M\in \cM_c$ covers at least $\alpha n / 2$ vertices and satisfies at least one of the following: $|M| = 1$, or $M$ covers at most $\alpha n$ vertices of $\cH$.\COMMENT{First, let $\cM'_c = \{\{e\} : e\in M_c \text{ and } |e| \geq \alpha n / 2\}$, and let $M'_c \coloneqq M_c \setminus \bigcup_{M\in\cM'_c}M$.  Now choose matchings consisting of edges in $M'_c$ greedily.}
  Note that $|\cM_c| \leq 2/\alpha$.  Now for each $c\in C'$, we choose a distinct set of $|\cM_c|$ colours $C_c$ disjoint from $C$, and we define a proper edge-colouring $\phi'$ of $\cH$ as follows.  For each $c\in C\setminus C'$ and $e\in M_c$, we let $\phi'(e) \coloneqq \phi(e)$.  For each $c\in C'$ and $e\in M_c$, we let $\phi'(e) \in C_c$ such that for every $M\in \cM_c$, every edge of $M$ is assigned the same colour.  By the choice of $\cM_c$, every colour is either assigned to at most one $e\in\cH$ by $\phi'$, or there are at most $\alpha n$ vertices of $\cH$ that are incident to an edge assigned that colour, so $\phi'$ is $\alpha$-bounded, as desired.  Moreover, by the bounds on $|\cM_c|$ and $|C'|$, the colouring $\phi'$ uses at most $|C| + 2|C'|/\alpha \leq \chi'(\cH) + 2n / (\alpha^2 r)$ colours, as desired for~\ref{few-large-colour-classes-splitting}.
\end{proof}
\begin{proposition}\label{greedy-colouring-prop}
  Let $0 < 1/n_0 \ll 1/r \ll \alpha_1, \alpha_2 < 1$, and let $n \geq n_0$. Let $\ordering$ be a linear ordering of the edges of an $n$-vertex linear hypergraph $\cH$ where every $e\in \cH$ satisfies $|e| \geq r$.  If $C$ is a list-assignment for the line graph of $\cH$ such that every $e\in \cH$ satisfies $|C(e)| \geq \fwddeg(e) + \alpha_1 n$, then there is an $\alpha_2$-bounded proper edge-colouring $\phi$ of $\cH$ such that $\phi(e) \in C(e)$ for every $e\in\cH$.
\end{proposition}
\begin{proof}
    Let $\cH_{\mathrm{big}} \coloneqq \{e \in \cH : |e| \geq \alpha_2 n / 2\}$, and note that $e(\cH_{\mathrm{big}}) \leq 4 / \alpha_2$ by~\eqref{eqn:huge-edge-bound}.  By possibly reordering $\ordering$ and replacing $\alpha_1$ with $\alpha_1 / 2$, we may assume without loss of generality that every $e\in \cH_{\mathrm{big}}$ satisfies $e\ordering f$ for $f\in \cH\setminus \cH_{\mathrm{big}}$.  
    
  Choose an edge $e^*\in\cH$ and an $\alpha_2$-bounded proper edge-colouring $\phi$ of $\cH^{\ordering e^*}$ satisfying $\phi(e) \in C(e)$ for every $e\in \cH^{\ordering e^*}$ such that $e(\cH^{\ordering e^*})$ is maximum.  Note that such a choice indeed exists, for example when $e^*$ is the first edge in $\ordering$.  We claim that $e^*$ is the last edge of $\cH$ in $\ordering$, in which case $\phi$ is the desired colouring.  Suppose to the contrary, and let $f$ be the successor of $e^*$.  We have $f\notin\cH_{\mathrm{big}}$, or else assigning $f$ a colour in $C(f)\setminus\{\phi(e) : e\ordering f\}$ would yield an $\alpha_2$-bounded colouring of $\cH^{\ordering f}$, contradicting the choice of $e^*$.

  Now let $C_1 \coloneqq \bigcup_{e\in\fwdnbr(f)}\phi(e)$, and let $C_2$ be the set of colours $c$ for which there are at least $\alpha_2 n / 2$ vertices of $\cH$ incident to an edge assigned the colour $c$.  If there is a colour $c\in C(f)\setminus (C_1\cup C_2)$, then assigning $\phi(f) \coloneqq c$ would yield an $\alpha_2$-bounded colouring of $\cH^{\ordering f}$, contradicting the choice of $e^*$.  Therefore $|C_1| + |C_2| \geq |C(f)|$, and since $|C_1| \leq \fwddeg(f)$, we have $|C_2| \geq \alpha_1 n / 2$.  However, by Proposition~\ref{few-large-colour-classes}\ref{few-large-colour-classes-bound}, $|C_2| \leq 2n / (\alpha_2(r - 1)) < \alpha_1 n / 2$, a contradiction.
\end{proof}

Before we prove Theorem~\ref{large-edge-thm}, we prove Theorem~\ref{weak-stability-thm} using Theorem~\ref{thm:kahn}, Lemma~\ref{reordering-lemma}, Corollary~\ref{aks-sparsity-corollary}, and Proposition~\ref{greedy-colouring-prop}.  The proof of Theorem~\ref{large-edge-thm} uses similar ideas, with Corollary~\ref{sparsity-corollary} instead of Corollary~\ref{aks-sparsity-corollary}.

\begin{proof}[Proof of Theorem~\ref{weak-stability-thm}]
We may assume that $\eps \ll 1$. Choose $n_0 , \eta$ to satisfy $0 < 1/n_0 , \eta \ll \eps$. First we decompose $\cH$ into three spanning subhypergraphs, as follows.  Let $\cH_1 \coloneqq \{e\in\cH : 1 / \eta < |e| < \eta \sqrt n\}$, let $\cH_2 \coloneqq \{e\in \cH : |e| \leq 1 / \eta\}$, and let $\cH_3 \coloneqq \{e\in \cH : |e| \geq \sqrt n / \eta\}$.  Since $\Delta(\cH) \leq \eta n$, by Theorem~\ref{thm:kahn} applied to $\cH_2$ with $\eta n$, $1/2$, and $1 / \eta$ playing the roles of $D$, $\alpha$, and $r$, respectively, we have $\chi'(\cH_2) \leq 3 \eta n / 2 \leq \eps n / 4$.  Since $\cH$ is linear and $1 \geq \vol_{\cH}(\cH_3) \geq e(\cH_3)\binom{\sqrt n / \eta}{2} / \binom{n}{2}$, we have $e(\cH_3) \leq 2 \eta n$, and thus, $\chi'(\cH_3) \leq 2 \eta n \leq \eps n / 4$.  Therefore it suffices to show that $\chi'(\cH_1) \leq \eps n / 2$.

Without loss of generality, let us assume $\cH_1 \ne \varnothing$. Let $\cH_0^{\mathrm{left}} \coloneqq \cH_1$, and for every positive integer $i$, define spanning subhypergraphs $\cH_i^{\mathrm{left}}, \cH_i^{\mathrm{good}}, W_i$ of $\cH_{i - 1}^{\mathrm{left}}$ as follows.  If $\cH_{i - 1}^{\mathrm{left}} = \varnothing$, then let $\cH_i^{\mathrm{left}}, \cH_i^{\mathrm{good}}, W_i\coloneqq \varnothing$. Otherwise, apply Lemma~\ref{reordering-lemma} to $\cH_{i - 1}^{\mathrm{left}}$ with $1 - \eps/6$ and $\eps^{-2}$ playing the roles of $\tau$ and $K$, respectively, to obtain an ordering $\ordering_i$.  If $\ordering_i$ satisfies~\ref{reordering-good}, then let $\cH_i^{\mathrm{good}} \coloneqq \cH_{i-1}^{\mathrm{left}}$, and let $\cH_i^{\mathrm{left}}, W_i \coloneqq \varnothing$.  Otherwise, let $W_i$ be the set $W$ obtained from~\ref{reordering-volume}, let $e^*_i$ be the edge of $W_i$ which comes last in $\ordering_i$, let $\cH_i^{\mathrm{good}} \coloneqq \cH_{i - 1}^{\mathrm{left}}\setminus (\cH_{i - 1}^{\mathrm{left}})^{\ordering e^*_i}$, let $f^*_i$ be the edge of $W_i$ which comes first in $\ordering_i$, and let $\cH_i^{\mathrm{left}}\coloneqq \cH_{i - 1}^{\mathrm{left}}\setminus \{e \in \cH_{i - 1}^{\mathrm{left}} : f^*_i \ordering_i e\}$.  In this case, by~\ref{ordering-by-size}, $|e_i^*| = \min_{e\in W}|e|$.  Also by~\ref{ordering-by-size}, we may assume without loss of generality that every $e\in \cH_{i-1}^{\rm left}$ satisfying $f^*_i \ordering_i e \ordering_i e^*_i$ is in $W_i$.  By the choices of $\tau$ and $K$, if $\ordering_i$ and $W_i$ satisfy~\ref{reordering-volume}, then\COMMENT{Here we used $\eps \ll 1$.}
\begin{enumerate}[label = {($\mathrm{W}_i\arabic*)$}]
  \item\label{Wi-max-size} $\max_{e\in W_i}|e| \leq \eps^{-10}|e_i^*|$ and
  \item\label{Wi-volume} $\vol_{\cH}(W_i) \geq \eps^{20}$.
\end{enumerate}

For any $i \geq 1$, note that the sets $W_1 , \dots , W_i$ are pairwise disjoint. Moreover, if $W_i = \varnothing$ then $\cH_i^{\rm left} = \varnothing$, and if $W_i \ne \varnothing$ then $W_1 , \dots , W_{i-1}$ are also nonempty. Thus, we have $\sum_i\vol_\cH(W_i) \leq 1$, and if $W_i \neq \varnothing$ then $i \leq \eps^{-20}$ by~\ref{Wi-volume}, so there is some integer $k \geq 1$ such that $W_k = \varnothing$, and $k \leq \eps^{-20} + 1$. Hence,
\begin{equation}\label{eq:weak-stability-partition}
      \textrm{$\cH_1$ is partitioned into $W_1, \dots, W_{k-1}$ and $\cH_1^{\mathrm{good}}, \dots, \cH_k^{\mathrm{good}}$,}
\end{equation}
where $W_1 , \dots,  W_{k-1}$ are nonempty, and $\cH_1^{\rm good} , \dots , \cH_k^{\rm good}$ could be empty.\COMMENT{We keep applying Lemma~\ref{reordering-lemma} until $W_k = \varnothing$. Then either the procedure ends with~\ref{reordering-good} or $\cH_{k-1}^{\rm left} = \varnothing$ (thus $\cH_k^{\rm left} = \cH_k^{\rm good} = W_k = \varnothing$).}

Combine $\ordering_1, \dots, \ordering_k$ to obtain an ordering $\ordering$ of $\cH_1$ where if $f \in \cH_i^{\mathrm{good}} \cup W_i$, then $e\ordering f$ for every $e \in (\cH_{i-1}^{\mathrm{left}})^{\ordering_i f}$.  Let $\cH^{\mathrm{good}} \coloneqq \bigcup_{i=1}^k \cH_i^{\mathrm{good}}$, and note that by~\ref{reordering-good} and~\ref{ordering-goodness},
\begin{equation}
    \label{reordering-goodnessi}
    \text{if $e \in \cH^{\mathrm{good}}$, then $\fwddeg_{\cH_1}(e) \leq \eps n / 6$.}
\end{equation}

By~\ref{Wi-max-size}, and since every $e\in \cH_1$ satisfies $1 / \eta < |e| < \eta \sqrt n$, for each $i \in [k-1]$ we can apply Corollary~\ref{aks-sparsity-corollary} to $W_i$ with $\eta$, $\eps^{10}$, and $\eps / (6k)$ playing the roles of $\eta$, $\alpha$, and $\eps$, respectively, to obtain a proper edge-colouring $\phi_i : W_i \rightarrow C_i$, where $|C_i| \leq \eps n /(6k)$.  
By~\eqref{reordering-goodnessi}, we can apply Proposition~\ref{greedy-colouring-prop} to $\cH^{\mathrm{good}}$ with $1/\eta$, $\eps / 6$, and $1/2$ playing the roles of $r$, $\alpha_1$, and $\alpha_2$, respectively, to obtain a proper edge-colouring $\phi' : \cH^{\mathrm{good}} \rightarrow C'$, where $|C'| \leq \eps n / 3$.  
We may assume without loss of generality that $C'$, $C_1, \dots, C_{k-1}$ are pairwise disjoint.  Therefore by~\eqref{eq:weak-stability-partition}, we can combine $\phi'$, $\phi_1, \dots, \phi_{k-1}$ to obtain a proper edge-colouring $\phi : \cH_1 \rightarrow C' \cup \bigcup_{i=1}^{k-1} C_i$.  Since $|C'| + \sum_{i=1}^{k-1} |C_i'| \leq \eps n / 2$, we have $\chi'(\cH_1) \leq \eps n / 2$, as desired.
\end{proof}

\begin{proposition}\label{prop:med-edges}
    Let $0 < 1 / n_0 \ll 1 / r_0 \ll 1 / r_1 \ll \gamma < 1$, and let $n \geq n_0$.  If $\cH$ is an $n$-vertex linear hypergraph where every $e\in\cH$ satisfies $r_1 \leq |e| \leq r_0$, then there is a $\gamma$-bounded proper edge-colouring of $\cH$ using at most $\gamma n$ colours.
\end{proposition}
\begin{proof}
  By Proposition~\ref{few-large-colour-classes}\ref{few-large-colour-classes-splitting} applied with $\gamma$ and $r_1 - 1$ playing the roles of $\alpha$ and $r$, respectively, there is a $\gamma$-bounded proper edge-colouring $\phi$ of $\cH$ using at most $\chi'(\cH) + 2n / (\gamma^2 (r_1 - 1)) \leq \chi'(\cH) + \gamma n / 2$ colours. 
  Since $\cH$ is linear and every $e\in\cH$ satisfies $|e| \geq r_1$, we have $\Delta(\cH) \leq n / (r_1 - 1)$.  Thus, by Theorem~\ref{thm:kahn}, $\chi'(\cH) \leq 2n / r_1 \leq \gamma n / 2$, so $\phi$ uses at most $\gamma n$ colours, as desired.
\end{proof}

\begin{proof}[Proof of Theorem~\ref{large-edge-thm}]
  Recall that $\cH_{\rm med} = \{ e \in \cH \: : \: r_1 < |e| \leq r_0 \}$, $\cH_{\rm large} = \{ e \in \cH \: : \: |e| > r_0 \}$, and $\cH_{\rm huge} = \{ e \in \cH \: : \: |e| \geq \beta n / 4 \} $.  Let $\cH' \coloneqq \cH\setminus \cH_{\mathrm{huge}}$.  By Proposition~\ref{prop:med-edges}, there is a $\gamma_1$-bounded proper edge-colouring $\phi_{\mathrm{med}}$ of $\cH_{\mathrm{med}}$ using a set $C_{\mathrm{med}}$ of at most $\gamma_1 n$ colours.  We use $\phi_{\mathrm{med}}$ in all cases when we prove~\ref{non-extremal-colouring}, and we define $C_{\mathrm{med}}$ differently when we prove~\ref{extremal-colouring} (see Case 2.2 below).   

    Now we apply Lemma~\ref{reordering-lemma} several times and combine the resulting orderings to obtain an ordering $\ordering$ of $\cH$.  In every application of Lemma~\ref{reordering-lemma}, $r_1$ in Lemma~\ref{reordering-lemma} is simply the present value of $r_1$.  We also define several subhypergraphs of $\cH$, which we assume are all spanning.  See Figure~\ref{fig:reordering}.  
  First, apply Lemma~\ref{reordering-lemma} to $\cH'$ with $1 - \gamma_2/3$ and $\gamma_2^{-2}$ playing the roles of $\tau$ and $K$, respectively, to obtain an ordering $\ordering_1$.  We define $e^*_1$, $W_1$, and $\cH_1^{\mathrm{good}}$, as follows.  If $\ordering_1$ satisfies~\ref{reordering-good}, then let $e^*_1$ be the first edge of $\cH'$, let $W_1 \coloneqq \varnothing$, and let $\cH_1^{\mathrm{good}} \coloneqq \cH'$.  Otherwise, let $W_1$ be the set $W$ obtained from~\ref{reordering-volume}, let $e^*_1$ be the last edge of $W_1$, and let $\cH_1^{\mathrm{good}} \coloneqq \cH'\setminus (\cH')^{\ordering_1e^*_1}$.  In this case, by \ref{ordering-by-size}, $|e^*_1| = \min_{e\in W_1}|e|$.  In both cases, let $\cH_1^{\mathrm{left}} \coloneqq \cH'\setminus \cH_1^{\mathrm{good}}$, and let $r_2\coloneqq |e^*_1|$.
  By the choices of $\tau$ and $K$, and since $\delta \ll \gamma_2 \ll 1$, if $\ordering_1$ and $W_1$ satisfy~\ref{reordering-volume}, then
  \begin{enumerate}[label = {($\mathrm{W}_1\arabic*)$}]
  \item\label{W1-max-size} $\max_{e\in W_1}|e| \leq (1 + 3 \gamma_2^{-8})|e_1^*| \leq r_2 / \gamma_2^{10}$ and
  \item\label{W1-volume} $\vol_\cH(W_1) \geq (\gamma_2 / 3 - 7\gamma_2^2)^2 / \gamma_2^{-10} \geq \gamma_2^{20} > \delta$.
  \end{enumerate}
  
  If $\cH_1^{\mathrm{left}} = \varnothing$, then let $e^*_2\coloneqq e^*_1$ and $W_2, \cH_2^{\mathrm{good}} \coloneqq \varnothing$.  If $\cH_1^{\mathrm{left}} \neq \varnothing$, we apply Lemma~\ref{reordering-lemma} to $\cH_1^{\mathrm{left}}$ with $3\sigma$ and $1$ playing the roles of $\tau$ and $K$, respectively, to obtain an ordering $\ordering_2$, and we define $e^*_2$, $W_2$, and $\cH_2^{\mathrm{good}}$, as follows.  
  If $\ordering_2$ satisfies~\ref{reordering-good}, then let $e^*_2$ be the first edge of $\cH_1^{\mathrm{left}}$ in $\ordering_2$, let $W_2 \coloneqq \varnothing$, and let $\cH_2^{\mathrm{good}} \coloneqq \cH_1^{\mathrm{left}}$.   Otherwise, \ref{reordering-volume} holds for some $W_2 \subseteq \cH_1^{\mathrm{left}}$ playing the role of $W$, so we choose such a $W_2$ that is  maximal, let $e^*_2$ be the last edge of $W_2$ in $\ordering_2$, and let $\cH_2^{\mathrm{good}} \coloneqq \cH_1^{\mathrm{left}}\setminus (\cH_1^{\mathrm{left}})^{\ordering_2 e^*_2}$. In this case, by \ref{ordering-by-size}, $|e^*_2| = \min_{e\in W_2}|e|$. In all cases, let $\cH_2^{\mathrm{left}} \coloneqq (\cH_1^{\mathrm{left}}\cup\cH_{\mathrm{huge}})\setminus\cH_2^{\mathrm{good}}$, and let $r_3 \coloneqq |e^*_2|$.  By the choices of $\tau$ and $K$, and since $\sigma \ll \delta\ll 1$, if $\ordering_2$ and $W_2$ satisfy~\ref{reordering-volume}, then
  \begin{enumerate}[label = {($\mathrm{W}_2\arabic*)$}]
  \item\label{W2-max-size} $\max_{e\in W_2}|e| \leq (1 + 3(3\sigma)^{1/4})|e_2^*| \leq (1 + 4\sigma^{1/4})r_3$ and
  \item\label{W2-volume} $\vol_\cH(W_2) \geq (1 - 3\sigma - 7(3\sigma)^{1/4})^2 / (1 + 4\sigma^{1/4}) \geq (1 - \sigma^{1/5})^3 \geq 1 - \delta^3$.
  \end{enumerate}

  Finally, if $W_2\neq\varnothing$, then let $f^*$ be the edge of $W_2$ which comes first in $\ordering_2$, and let $\cH_3 \coloneqq \cH_2^{\mathrm{left}} \setminus \{e \in \cH_1^{\mathrm{left}} : f^*\ordering_2 e\}$.  Otherwise, let $\cH_3\coloneqq\cH_{\mathrm{huge}}$. Thus in both cases, since $\cH_1^{\mathrm{left}} \subseteq \cH' = \cH \setminus \cH_{\rm huge}$ and $\cH_{\rm huge} \subseteq \cH_{2}^{\mathrm{left}}$, we have $\cH_{\rm huge} \subseteq \cH_3$. Apply Lemma~\ref{reordering-lemma} with $\cH_3$, $1 - 1/2000$, and $2000^2$ playing the roles of $\cH$, $\tau$, and $K$, respectively, to obtain an ordering $\ordering_3$.  
  By the choice of $f^*$, we have $W_2 \subseteq \{e \in \cH_1^{\mathrm{left}} : f^* \ordering_2 e\}$, so by the construction of $\cH_3$, we have $W_2\cap \cH_3 = \varnothing$.  Thus, $\vol_\cH(W_2) + \vol_\cH(\cH_3) \leq 1$, and so $\ordering_3$ satisfies~\ref{reordering-good}, because otherwise \ref{reordering-volume} would imply there is a set $W'\subseteq \cH_3$ disjoint from $W_2$ with $\vol_{\cH}(W')\geq \delta$, contradicting~\ref{W2-volume}.  
  
  By the choice of $W_2$ to be maximal and by \ref{ordering-by-size} of Lemma~\ref{reordering-lemma}, every $e\in \cH_1^{\rm left}$ satisfying $f^* \ordering_2 e \ordering_2 e^*_2$ is in $W_2$.  In particular, $\{e \in \cH_1^{\mathrm{left}} : f^*\ordering_2 e\}$ is partitioned into $W_2$ and $\cH_2^{\mathrm{good}}$, and thus, 
  \begin{equation}\label{eq:large-edge-thm-partition}
      \textrm{$\cH$ is partitioned into $\cH_2^{\mathrm{left}}$, $\cH_2^{\mathrm{good}}$, and $\cH_1^{\mathrm{good}}$, and $\cH_2^{\mathrm{left}}$ is partitioned into $\cH_3$ and $W_2$}.
  \end{equation}

  By \eqref{eq:large-edge-thm-partition}, and since $\cH_{\mathrm{huge}}\subseteq \cH_3$, we can combine $\ordering_1$, $\ordering_2$, and $\ordering_3$ to obtain an ordering $\ordering$ of $\cH$ where 
  \begin{itemize}
  \item if $f\in\cH_1^{\mathrm{good}}$, then $e\ordering f$ for every $e\in \cH_{\mathrm{huge}}\cup (\cH')^{\ordering_1 f}$, 
  \item if $f\in\cH_2^{\mathrm{good}} \cup W_2$, then $e \ordering f$ for every $e\in\cH_{\mathrm{huge}}\cup (\cH_1^{\mathrm{left}})^{\ordering_2 f}$, and
  \item if $f\in \cH_3$, then $e\ordering f$ for every $e\in \cH_3^{\ordering_3 f}$.
  \end{itemize}

  \begin{figure}
\centering
\begin{tikzpicture}
\draw(0,0)--(15,0);

\foreach \x/\xtext/\ytext in {4/{$\leq(1 + 4\sigma^{1/4})r_3$}/$f^*$/,7/$r_3$/$e^*_2$}
\draw(\x,11pt) node[above] {\ytext} --(\x,-11pt) node[below] {\xtext} ;

\foreach \x/\xtext/\ytext in {0/{}/{}, 11/$r_2$/$e^*_1$, 15/$r_1$/{}}
\draw(\x,2pt) node[above] {\ytext} --(\x,-2pt) node[below] {\xtext} ;

\draw[decorate, decoration={brace, mirror}, yshift=-1ex]  (0.2,0) -- node[below=0.5ex] {$\cH_3$}  (3.8,0);

\draw[decorate, decoration={brace, mirror}, yshift=-1ex]  (11.2,0) -- node[below=0.5ex] {$\cH_1^{\mathrm{good}}$}  (14.8,0);

\draw[decorate, decoration={brace}, yshift=6ex]  (3.9,0) -- node[above=0.5ex] {$W_2$}  (7.1,0);

\draw[decorate, decoration={brace}, yshift=1ex]  (.2,0) -- node[above=0.5ex] {$\cH_{\mathrm{huge}}$}  (.8,0);

\draw[decorate, decoration={brace, mirror}, yshift=-6ex]  (0,0) -- node[below=0.5ex] {$\cH_2^{\mathrm{left}}$}  (7.1,0);

\draw[decorate, decoration={brace, mirror}, yshift=-3.5ex]  (7.2,0) -- node[below=0.5ex] {$\cH_2^{\mathrm{good}}$}  (11.2,0);

\end{tikzpicture}
\caption{Combining three applications of the Reordering Lemma in the proof of Theorem~\ref{large-edge-thm}: the ordering $\ordering$ is increasing from left to right.  ($\cH_{\mathrm{huge}}\subseteq \cH_3$, but $\cH_{\mathrm{huge}}$ need not form an initial segment.)}
\label{fig:reordering}
\end{figure}
   
  See Figure~\ref{fig:reordering}.  Note the following.
  \begin{enumerate}[(a)]
  \item\label{huge-in-H3} $\cH_{\mathrm{huge}}\subseteq \cH_3$, and $e(\cH_{\mathrm{huge}}) \leq 8 / \beta$ by~\eqref{eqn:huge-edge-bound}.
  \item\label{reordering-goodness1} If $e\in\cH_1^{\mathrm{good}}$, then $\fwddeg_\cH(e) \leq d_{\cH'}^{\ordering_1}(e) + e(\cH_{\mathrm{huge}}) \leq \gamma_2 n / 2$, since $\ordering_1$ satisfies either \ref{reordering-good} or \ref{ordering-goodness} with $\tau = 1 - \gamma_2 / 3$, and by~\ref{huge-in-H3}.
  \item\label{reordering-goodness2} If $e\in \cH_2^{\mathrm{good}}$, then $d^{\ordering_2}_{\cH_1^{\rm left}}(e) \leq (1 - 3\sigma)n$ and $\fwddeg_\cH(e) \leq d_{\cH_1^{\mathrm{left}}}^{\ordering_2}(e) + e(\cH_{\mathrm{huge}}) \leq (1 - 2\sigma)n$, since $\ordering_2$ satisfies either~\ref{reordering-good} or~\ref{ordering-goodness} with $\tau = 3\sigma$, and by~\ref{huge-in-H3}.
  \item\label{reordering-goodness3} If $e\in \cH_3$, then $\fwddeg_\cH(e) \leq n/2000$ since $\ordering_3$ satisfies~\ref{reordering-good}.
  \item\label{d2-bound} If $e\in\cH_1^{\mathrm{left}}$, then $e \ordering_1 e^*_1$ and thus $|e| \geq r_2 = |e^*_1|$ (because if $\ordering_1$ satisfies~\ref{reordering-good}, then it's vacuously true, and otherwise it follows from \ref{ordering-by-size}).
  \item\label{d3-bound} If $e\in \cH_2^{\mathrm{left}}$, then $e \ordering_2 e^*_2$ and thus $|e| \geq r_3 = |e^*_2|$ (because if $\ordering_2$ satisfies~\ref{reordering-good}, then it's vacuously true, and otherwise it follows from \ref{ordering-by-size}).
  \end{enumerate}

  We consider two cases: $\ordering_2$ satisfies~\ref{reordering-good}, or both $\ordering_1$ and $\ordering_2$ satisfy~\ref{reordering-volume}.  Note that if $\ordering_1$ satisfies~\ref{reordering-good}, then $\ordering_2$ vacuously satisfies~\ref{reordering-good} (that is, the former case applies).

  \noindent {\bf Case 1:} $\ordering_2$ satisfies~\ref{reordering-good}.
    
    In this case, we prove~\ref{non-extremal-colouring}.  If $\ordering_2$ satisfies~\ref{reordering-good}, then $\cH_3\cup \cH_2^{\mathrm{good}}\cup \cH_1^{\mathrm{good}} = \cH$ by~\eqref{eq:large-edge-thm-partition}, so every $e\in\cH$ satisfies $\fwddeg_\cH(e) \leq (1 - 2\sigma)n$ by~\ref{reordering-goodness1},~\ref{reordering-goodness2}, and~\ref{reordering-goodness3}.  Thus, by Proposition~\ref{greedy-colouring-prop} applied to $\cH_{\mathrm{large}}$ with $r_0$, $\sigma / 2$, and $\beta/5$ playing the roles of $r$, $\alpha_1$, and $\alpha_2$, respectively, we obtain a $\beta/5$-bounded proper edge-colouring $\phi_{\mathrm{large}}$ of $\cH_{\mathrm{large}}$ using colours from a set $C_{\mathrm{large}}$ of size at most $(1 - 3\sigma / 2)n$ disjoint from $C_{\mathrm{med}}$. We combine $\phi_{\mathrm{large}}$ and $\phi_{\mathrm{med}}$ to obtain a proper edge-colouring $\phi$ of $\cH$ satisfying~\ref{non-extremal-colouring}, as follows.  For each $e\in \cH_{\mathrm{large}}$, let $\phi(e) \coloneqq \phi_{\mathrm{large}}(e)$, and for each $e\in\cH_{\mathrm{med}}$, let $\phi(e) \coloneqq \phi_{\mathrm{med}}(e)$.  Since $C_{\mathrm{large}} \cap C_{\mathrm{med}} = \varnothing$ and $|C_{\mathrm{large}} \cup C_{\mathrm{med}}| \leq (1 - \sigma)n$, the colouring $\phi$ is proper and uses at most $(1 - \sigma)n$ colours, as required.  Since $\phi_{\mathrm{large}}$ is $\beta/5$-bounded, $\phi$ satisfies~\ref{non-extremal-huge-edge} and ~\ref{non-extremal-colour-class-bound}, and since $\phi_{\mathrm{med}}$ is $\gamma_1$-bounded, $\phi$ satisfies~\ref{non-extremal-medium-bound}, as desired.
  
  \noindent {\bf Case 2:} Both $\ordering_1$ and $\ordering_2$ satisfy~\ref{reordering-volume}.
  
  By \ref{d3-bound}, $\vol_{\cH}(\cH_2^{\mathrm{left}}) \geq \left.e(\cH_2^{\mathrm{left}})\binom{r_3}{2}\middle/\binom{n}{2}\right.$, and since $\vol_{\cH}(\cH_2^{\mathrm{left}}) \leq 1$, we have $e(\cH_{2}^{\mathrm{left}}) \leq \left.\binom{n}{2}\middle/ \binom{r_3}{2}\right.$. Thus, we assume
  \begin{equation}\label{eqn:d3-upper-bound}
    r_3 \leq \sqrt{n / (1 - 4\sigma)},
  \end{equation}
   as otherwise we would have $e(\cH_2^{\mathrm{left}}) \leq (1 - 3\sigma)n$, and together with~\ref{reordering-goodness2}, this fact would imply that $\ordering_2$ satisfies~\ref{reordering-good}.
  
  We now consider two additional cases: $r_3 < (1 - \delta)\sqrt n$, and $r_3 \geq (1 - \delta)\sqrt n$.  In the former case, we prove~\ref{non-extremal-colouring}, and in the latter case we prove~\ref{extremal-colouring}.
  
  \noindent {\bf Case 2.1:} $r_3 < (1 - \delta)\sqrt n$.
    
Let $\zeta \coloneqq 1 - r_3 / \sqrt n$.  Since $r_3 < (1 - \delta)\sqrt n$, we have $\zeta > \delta$.  
  First we show how to colour $W_2\setminus\cH_{\mathrm{med}}$ in the following claim.
  \begin{claim}\label{claim:sparse-colouring}
      There is a $\beta/5$-bounded proper edge-colouring $\phi'$ of $W_2\setminus\cH_{\mathrm{med}}$ using at most $(1 - \zeta / 1000)n$ colours.
  \end{claim}
  \begin{claimproof}
    By \ref{W2-max-size} and \ref{ordering-by-size}, we can apply Corollary~\ref{sparsity-corollary} with $r_3$ and $4\sigma^{1/4}$ playing the roles of $r$ and $\alpha$, respectively, so $\chi'(W_2\setminus\cH_{\mathrm{med}}) \leq (1 - \zeta / 500)n$.  Thus, the claim follows from Proposition~\ref{few-large-colour-classes}\ref{few-large-colour-classes-splitting} with $\beta/5$ and $r_0$ playing the roles of $\alpha$ and $r$, respectively, since every edge in $W_2\setminus\cH_{\mathrm{med}}$ has size at least $r_0 + 1$, $\zeta > \delta$, and $2n / ((\beta/5)^2r_0) \leq \delta^2 n \leq \zeta n / 1000$.
  \end{claimproof}

  We will colour $\cH_3\setminus\cH_{\mathrm{med}}$ with a set of colours disjoint from those that we assign to $W_2 \setminus\cH_{\mathrm{med}}$ using the following claim.
  
  \begin{claim}\label{claim:before-W-colouring}
    There is a $\beta/5$-bounded proper edge-colouring $\phi''$ of $\cH_3\setminus\cH_{\mathrm{med}}$ using at most $(\zeta / 1000 - 2\sigma)n$ colours.
  \end{claim}
  \begin{claimproof}
   Let $k \coloneqq e(\cH_3\setminus\cH_{\mathrm{med}})$.  If $k \leq (\zeta / 1000 - 2\sigma)n$, then we can simply assign each edge of $\cH_3\setminus\cH_{\mathrm{med}}$ a distinct colour and the claim holds, so we assume $k > (\zeta / 1000 - 2\sigma)n$.  Since $\zeta > \delta$, we have $k > 2\delta^2 n$.  By \ref{d3-bound}, every edge of $\cH_3$ has size at least $r_3$, so we have $\vol_{\cH}(\cH_3\setminus \cH_{\mathrm{med}}) \geq k \left.\binom{r_3}{2} \middle/ \binom{n}{2}\right. \geq k(r_3 - 1)^2 / n^2$.  On the other hand, by~\ref{W2-volume}, and since $\cH_3 \cap W_2 = \varnothing$ by~\eqref{eq:large-edge-thm-partition}, we have $\vol_{\cH}(\cH_3) \leq \delta^3$.  Thus, $2\delta^2 n < k \leq \delta^3 n^2 / (r_3 - 1)^2$, so $r_3 < \delta^{1/4} \sqrt{n}$.
   
    Therefore, $\zeta = 1 - r_3 / \sqrt n > 1 - \delta^{1/4} > 1000/1001$ since $\delta \ll 1$.  Now by~\ref{reordering-goodness3} and Proposition~\ref{greedy-colouring-prop} applied with $\cH_3\setminus\cH_{\mathrm{med}}$, $\ordering_3$, $r_0$, $1/6000$, and $\beta/5$ playing the roles of $\cH$, $\ordering$, $r$, $\alpha_1$, and $\alpha_2$, respectively, with a list-assignment for the line graph of $\cH_3\setminus \cH_{\mathrm{med}}$ in which each list is the same, we obtain a $\beta/5$-bounded proper edge-colouring of $\cH_3\setminus\cH_{\mathrm{med}}$ using a set of at most $n/1500 \leq (\zeta / 1000 - 2\sigma)n$ colours, as desired.
    \end{claimproof}
    
    We may assume that $\phi'$ and $\phi''$ use disjoint sets of colours. By Claims~\ref{claim:sparse-colouring} and~\ref{claim:before-W-colouring} and \eqref{eq:large-edge-thm-partition}, we can combine $\phi'$ and $\phi''$ to obtain a $\beta/5$-bounded proper edge-colouring $\phi_1$ of $\cH_2^{\mathrm{left}}\setminus \cH_{\mathrm{med}}$ using a set $C_1$ of at most $(1 - 2\sigma)n$ colours.  We may assume that $C_1\cap C_{\mathrm{med}} = \varnothing$.  Let $C_2$ be a set of $\lfloor (1 - 3\sigma / 2)n \rfloor - |C_1|$ colours disjoint from $C_1$ and $C_{\mathrm{med}}$, and let $C_{\mathrm{huge}}\subseteq C_1$ where $c\in C_{\mathrm{huge}}$ if $\phi_1(f) = c$ for some $f\in \cH_{\mathrm{huge}}$.  By~\ref{huge-in-H3}, $|(C_1 \cup C_2)\setminus C_{\mathrm{huge}}| \geq (1 - 7\sigma / 4)n$.  We extend $\phi_1$ to a $\beta/5$-bounded proper edge-colouring of $\cH_{\mathrm{large}}$ using the following claim.
    \begin{claim}\label{claim:bounded-colouring-extension}
        There is a $\beta/5$-bounded proper edge-colouring $\phi_2$ of $(\cH_2^{\mathrm{good}}\cup\cH_1^{\mathrm{good}})\setminus \cH_{\mathrm{med}}$ such that every $e\in (\cH_2^{\mathrm{good}}\cup\cH_1^{\mathrm{good}})\setminus \cH_{\mathrm{med}}$ satisfies
        \begin{itemize}
            \item $\phi_2(e) \in (C_1\cup C_2)\setminus C_{\mathrm{huge}}$  and
            \item $\phi_2(e)\neq\phi_1(f)$ for every $f\in N_{\cH_2^{\mathrm{left}}\setminus\cH_{\mathrm{med}}}(e)$.
        \end{itemize}
    \end{claim}
    \begin{claimproof}
        We apply Proposition~\ref{greedy-colouring-prop} to $(\cH_2^{\mathrm{good}}\cup\cH_1^{\mathrm{good}})\setminus \cH_{\mathrm{med}}$, as follows.  We let $r_0$, $\sigma/4$ and $\beta/5$ play the roles of $r$, $\alpha_1$ and $\alpha_2$, respectively, and we define the list-assignment for $L((\cH_2^{\mathrm{good}}\cup\cH_1^{\mathrm{good}})\setminus \cH_{\mathrm{med}})$ as $C(e) \coloneqq ((C_1\cup C_2)\setminus C_{\mathrm{huge}})\setminus \bigcup_{f\in N_{\cH_2^{\mathrm{left}}\setminus\cH_{\mathrm{med}}}(e)}\phi_1(f)$.  Since every $e\in \cH_2^{\mathrm{good}}\cup\cH_1^{\mathrm{good}}$ satisfies $|N_{\cH_2^{\mathrm{left}}}(e)| + \fwddeg_{\cH^{\mathrm{good}}_2\cup\cH_1^{\mathrm{good}}}(e) = \fwddeg_{\cH}(e) \leq (1 - 2\sigma)n$ by~\ref{reordering-goodness1} and~\ref{reordering-goodness2}, we have $|C(e)| \geq (1 - 7\sigma / 4)n - |N_{\cH_2^{\mathrm{left}}}| \geq \fwddeg_{\cH^{\mathrm{good}}_2\cup\cH_1^{\mathrm{good}}}(e) + \sigma n/4$ as required.  Therefore by Proposition~\ref{greedy-colouring-prop}, there is a $\beta/5$-bounded proper edge-colouring $\phi_2$ of $(\cH_2^{\mathrm{good}}\cup\cH_1^{\mathrm{good}})\setminus \cH_{\mathrm{med}}$ such that $\phi_2(e) \in C(e)$ for every $e\in(\cH_2^{\mathrm{good}}\cup\cH_1^{\mathrm{good}})\setminus \cH_{\mathrm{med}}$, and the choice of $C(e)$ ensures that $\phi_2$ satisfies the claim.
    \end{claimproof}
    
    Now we combine $\phi_1$, $\phi_2$, and $\phi_{\mathrm{med}}$ to obtain a proper edge-colouring $\phi$, and we show that $\phi$ satisfies~\ref{non-extremal-colouring}.  Indeed, by~\eqref{eq:large-edge-thm-partition}, every edge of $\cH$ is assigned a colour by $\phi$, and since $|C_1| + |C_2| + |C_\mathrm{med}| \leq (1 - 3\sigma / 2) + \gamma_1 n \leq (1 - \sigma) n$, the colouring $\phi$ uses at most $(1 - \sigma)n$ colours, as required.  Since $\phi_2(e) \notin C_{\mathrm{huge}}$ for each $e\in (\cH_2^{\mathrm{good}}\cup\cH_1^{\mathrm{good}})\setminus \cH_{\mathrm{med}}$, since~\ref{huge-in-H3} holds, and since $\phi''$ is $\beta / 5$-bounded, $\phi$ satisfies~\ref{non-extremal-huge-edge}.  Since $\phi_1$ and $\phi_2$ are $\beta/5$-bounded, $\phi$ satisfies~\ref{non-extremal-colour-class-bound}, and since $\phi_{\mathrm{med}}$ is $\gamma_1$-bounded, $\phi$ satisfies~\ref{non-extremal-medium-bound}, as desired.

  \noindent {\bf Case 2.2:} $r_3 \geq (1 - \delta)\sqrt n$.
  
  By \ref{W2-max-size} and \eqref{eqn:d3-upper-bound}, any $e\in\cH$ such that $|e| \geq (1 + \delta)\sqrt n$ is not in $W_2$, so \ref{W2-volume} implies that $\vol_\cH(e) \leq \delta^3$, which in turn implies that $|e| < \delta n / 2$.  By~\ref{d3-bound} and Lemma~\ref{extremal-case-lemma}, there is a proper edge-colouring $\phi_1$ of $\cH_2^{\mathrm{left}}$ using a set $C$ of at most $n$ colours such that every colour is assigned to at most two edges. Hence, $\phi_1 |_{\cH_2^{\mathrm{left}}\setminus\cH_{\mathrm{huge}}}$ is $\beta/2$-bounded. 
  Moreover, since there are no edges of size at least $\delta n/2$ in $\cH$, $\phi_1$ satisfies~\ref{extremal-huge-edge} of~\ref{extremal-colouring}.  Let $C_{\mathrm{huge}}\subseteq C$ where $c\in C_{\mathrm{huge}}$ if $\phi_1(f) \in\cH$ for some $f\in\cH_{\mathrm{huge}}$, and let $C_{\mathrm{med}}\subseteq C\setminus C_{\mathrm{huge}}$ have size $\lfloor \gamma_2 n \rfloor$ (such a set exists by~\ref{huge-in-H3}).

  By~\ref{W1-volume} and~\ref{W2-volume}, $W_1\cap W_2 \neq \varnothing$, so by~\ref{d3-bound}, there is an edge $e\in W_1$ such that $|e| \geq r_3$.  Therefore by~\ref{W1-max-size}, $r_2 \geq \gamma_2^{10}r_3$.  Also, $r_3 \geq (1 - \delta)\sqrt n \geq 2r_0 / \gamma_2^{10}$, so $r_2 \geq 2r_0$.  Thus, by~\ref{d2-bound}, we have $\cH_{\mathrm{med}} \cap \cH_1^{\mathrm{left}} = \varnothing$, so $\cH_{\mathrm{med}}\subseteq \cH_1^{\mathrm{good}}$.

  We use the following two claims to colour $\cH_2^{\mathrm{good}}$ and $\cH_1^{\mathrm{good}}$.  The proofs are similar to the proof of Claim~\ref{claim:bounded-colouring-extension}.

    \begin{claim}\label{claim:bounded-colouring-extension2}
    There is a $\beta/5$-bounded proper edge-colouring $\phi_2$ of $\cH_2^{\mathrm{good}}$ such that every $e\in\cH_2^{\mathrm{good}}$ satisfies
    \begin{itemize}
    \item $\phi_2(e) \in C\setminus C_{\mathrm{huge}}$ and 
    \item $\phi_2(e) \neq \phi_1(f)$ for every $f\in N_{\cH_2^{\mathrm{left}}}(e)$.\end{itemize}
  \end{claim}
  \begin{claimproof}
      We apply Proposition~\ref{greedy-colouring-prop} to $\cH_2^{\mathrm{good}}$, as follows.  We let $r_0$, $\sigma$, and $\beta/5$ play the roles of $r$, $\alpha_1$, and $\alpha_2$, respectively, and we define the list-assignment for $L(\cH_2^{\mathrm{good}})$ as $C(e) \coloneqq (C\setminus C_{\mathrm{huge}})\setminus \bigcup_{f\in N_{\cH_2^{\mathrm{left}}}(e)}\phi_1(f)$.  By \ref{huge-in-H3}, $|C\setminus C_{\mathrm{huge}}| \geq (1 - \sigma)n$, and since every $e\in \cH_2^{\mathrm{good}}$ satisfies $|N_{\cH_2^{\mathrm{left}}}(e)| + \fwddeg_{\cH_2^{\mathrm{good}}}(e) = \fwddeg(e) \leq (1 - 2\sigma)n$ by~\ref{reordering-goodness2}, we have $|C(e)| \geq (1 - \sigma)n - |N_{\cH_2^{\mathrm{left}}}(e)| \geq \fwddeg_{\cH_2^{\mathrm{good}}}(e) + \sigma n $ as required.  Therefore by Proposition~\ref{greedy-colouring-prop}, there is a $\beta / 5$-bounded proper edge-colouring $\phi_2$ of $\cH_2^{\mathrm{good}}$ such that $\phi_2(e) \in C(e)$ for every $e\in\cH_2^{\mathrm{good}}$, and the choice of $C(e)$ ensures that $\phi_2$ satisfies the claim.
 \end{claimproof}

  \begin{claim}\label{claim:bounded-colouring-extension3}
    There is a $\gamma_1/2$-bounded proper edge-colouring $\phi_3$ of $\cH_1^{\mathrm{good}}$ such that every $e\in\cH_1^{\mathrm{good}}$ satisfies
    \begin{itemize}
    \item $\phi_3(e) \in C_{\mathrm{med}}$ and 
    \item $\phi_3(e) \neq \phi_1(f)$ for every $f\in N_{\cH_2^{\mathrm{left}}}(e)$ and $\phi_3(e) \neq \phi_2(f)$ for every $f\in N_{\cH_2^{\mathrm{good}}}(e)$.\end{itemize}
  \end{claim}
    \begin{claimproof}
      We apply Proposition~\ref{greedy-colouring-prop} to $\cH_1^{\mathrm{good}}$, as follows.  We let $r_1$, $\gamma_2/3$, and $\gamma_1/2$ play the roles of $r$, $\alpha_1$, and $\alpha_2$, respectively, and we define the list-assignment for $L(\cH_1^{\mathrm{good}})$ as $C(e) \coloneqq C_{\mathrm{med}}\setminus ((\bigcup_{f\in N_{\cH_2^{\mathrm{left}}}(e)}\phi_1(f))\cup (\bigcup_{f\in N_{\cH_2^{\mathrm{good}}}(e)}\phi_2(f))$.  
      Since every $e\in \cH_1^{\mathrm{good}}$ satisfies $|N_{\cH_2^{\mathrm{left}}}(e)| + |N_{\cH_2^{\mathrm{good}}}(e)| +  \fwddeg_{\cH_1^{\mathrm{good}}}(e) = \fwddeg(e) \leq \gamma_2 n / 2$ by~\ref{reordering-goodness1} and since $|C_{\mathrm{med}}| = \lfloor \gamma_2 n \rfloor$, we have $|C(e)| \geq \gamma_2 n - 1 - |N_{\cH_2^{\mathrm{left}}}(e)| - |N_{\cH_2^{\mathrm{good}}}(e)| \geq \fwddeg_{\cH_1^{\mathrm{good}}}(e) + \gamma_2 n / 3$ as required.  
      Therefore by Proposition~\ref{greedy-colouring-prop}, there is a $\gamma_1/2$-bounded proper edge-colouring $\phi_3$ of $\cH_1^{\mathrm{good}}$ such that $\phi_3(e) \in C(e)$ for every $e\in\cH_1^{\mathrm{good}}$, and the choice of $C(e)$ ensures that $\phi_3$ satisfies the claim.
  \end{claimproof}

    Now we combine $\phi_1$, $\phi_2$, and $\phi_3$ to obtain a proper edge-colouring $\phi$, and we show that \ref{extremal-colouring} holds.  First, since $r_3 \geq (1 - \delta)\sqrt n$, by~\ref{d3-bound},~\ref{W2-max-size}, and~\eqref{eqn:d3-upper-bound}, the edges in $W_2$ are FPP-extremal, and by~\ref{W2-volume}, $\vol_{\cH}(W_2) \geq 1 - \delta$, so there indeed exists a set of FPP-extremal edges of volume at least $1 - \delta$ (namely, $W_2$), as required.  In addition, since $|C| \leq n$, the colouring $\phi$ uses at most $n$ colours, as required, so it remains to check that $\phi$ satisfies \ref{extremal-huge-edge}--\ref{extremal-colour-class-bound} of \ref{extremal-colouring}.  Since $\phi_i(e)\notin C_{\mathrm{huge}}$ for each $e\in\cH_{4 - i}^{\mathrm{good}}$ for $i\in\{2, 3\}$, by \eqref{eq:large-edge-thm-partition} and \ref{huge-in-H3}, every colour assigned to a huge edge by $\phi$ is in $C_{\mathrm{huge}}$.  Hence, since $\phi_1$ satisfies \ref{extremal-huge-edge} of \ref{extremal-colouring} (as shown earlier), $\phi$ satisfies \ref{extremal-huge-edge}, as desired.  Since $\phi_1 |_{\cH_2^{\mathrm{left}}\setminus\cH_{\mathrm{huge}}}$ is $\beta/2$-bounded and $\phi_2$ is $\beta/5$-bounded, $\phi$ satisfies~\ref{extremal-colour-class-bound}, and since moreover, $C_{\mathrm{med}} \cap C_{\mathrm{huge}} =\varnothing$, $\cH_{\mathrm{med}}\subseteq \cH_1^{\mathrm{good}}$, and $\phi_3$ is $\gamma_1/2$-bounded, $\phi$ satisfies \ref{extremal-medium-bound}, as desired.
\end{proof}

\section{Vertex absorption to extend colour classes}\label{absorption-section}

In this section, we will define properties that that we need our absorbers to satisfy in order to carry out the vertex absorption step outlined in Section~\ref{subsection:small-edge-overview}. (These absorbers will form part of the reservoirs $R$, which will be constructed in Section~\ref{reservoir-section}.) 
We also formalise various properties that allow a matching to be extended using vertex absorption.
In particular, given a set of edge-disjoint matchings having these properties, we can extend each matching using an absorber so that the resulting set of matchings has perfect or nearly-perfect coverage of $U \coloneqq V^{(1 - \eps)}_+(G)$ where $G \coloneqq \cH^{(2)}$ (as defined in Definition~\ref{def:coverage}).

\subsection{Quasirandom properties for absorption}\label{subsection:absorbable-properties}

\begin{definition}[Typicality and upper regularity]\label{def:typicality}
 Let $\gamma, \rho, \xi \in (0,1)$, let $G$ be an $n$-vertex graph, let $\cV$ be a set of subsets of $V(G)$, and let $R\subseteq E(G)$.  We say $R$ is 
 
 \begin{itemize}
    \item \textit{$(\rho, \gamma, G)$-typical with respect to $\cV$} if for every $X\in \cV$, every vertex $v \in V(G)$ satisfies $|N_R(v) \cap X| = \rho|N_G(v)\cap X| \pm \gamma n$ and $|N_R(v) \setminus X| = \rho|N_G(v)\setminus X| \pm \gamma n$, and
  
    \item \textit{upper $(\rho, \xi, G)$-regular} if for every pair of disjoint sets $S, T\subseteq V(G)$ with $|S|, |T| \geq \xi n$, we have $|E_G(S, T)\cap R| \leq \rho e_G(S, T) + \xi |S||T|$.
 \end{itemize}
 
 We also say a graph $H$ is upper $(\rho, \xi)$-regular if $E(H)$ is upper $(\rho, \xi, K_{v(H)})$-regular. In other words, $H$ is upper $(\rho, \xi)$-regular if for every pair of disjoint sets $S, T \subseteq V(H)$ with $|S|, |T| \ge \xi v(H)$, we have $e_H(S,T) \le (\rho + \xi)|S||T|$.
\end{definition}

The following simple observation follows immediately from Definition~\ref{def:typicality}.

\begin{obs}
\label{upper_reg_subgraphs}
If $H$ is a subgraph of $G$ such that $\xi v(G) < v(H)$, and $E(H)$ is upper $(\rho, \xi, G)$-regular, then $H$ is upper $(\rho, \xi v(G)/v(H))$-regular. \qed
\end{obs}

\begin{definition}[Absorbers]\label{def:absorber}
  Let $\xi , \gamma, \rho , \eps \in (0,1)$. Let $\cH$ be a linear multi-hypergraph, let $G\coloneqq \cH^{(2)}$, let $G'$ be the spanning subgraph of $G$ consisting of those edges with at least one vertex in $V^{(1 - \eps)}_+(G)$, and let $\cV$ be a set of subsets of $V(\cH)$.
 We say $R_{\mathrm{abs}}$ is a \textit{$(\rho, \gamma, \xi, \eps)$-absorber for $\cV$} if it satisfies the following properties: 
  \begin{enumerate}[(i)]
  \item\label{cond:absorber1} $R_{\mathrm{abs}} \subseteq E(G')$,
  
  \item $R_{\mathrm{abs}}$ is $(\rho, \gamma, G')$-typical with respect to $\cV$, and
  
  \item $R_{\mathrm{abs}}$ is upper $(\rho, \xi, G')$-regular.
  \end{enumerate}
\end{definition}

\begin{obs}[Robustness of absorbers]\label{obs:del_absorber}
    Let $\xi , \gamma, \rho , \eps \in (0,1)$.
    Let $\cH$ be a linear multi-hypergraph, and let $\cV$ be a set of subsets of $V(\cH)$.
    The following hold.
    \begin{enumerate}[{\rm (i)}]
        \item A $(\rho,\gamma,\xi,\eps)$-absorber for $\cV$ is also a $(\rho,\gamma' , \xi , \eps)$-absorber for $\cV$, if $\gamma < \gamma'< 1$.
        
        \item For any $(\rho,\gamma,\xi,\eps)$-absorber $R$ for $\cV$, if $R' \subseteq R$ and $\Delta(R-R')\leq \alpha n$, then $R'$ is also a $(\rho,\gamma + \alpha , \xi , \eps)$-absorber for $\cV$. \qed
    \end{enumerate} 
\end{obs}

\begin{definition}[Pseudorandom matchings]\label{pseudorandom-matching-def}
  Let $n \in \mathbb{N}$, $\gamma , \kappa \in (0,1)$, and let $\cH$ be an $n$-vertex multi-hypergraph. For a family $\cF$ of subsets of $V(\cH)$, a matching $M$ in $\cH$ is \textit{$(\gamma, \kappa)$-pseudorandom} with respect to $\cF$ if every $S\in\cF$ satisfies $|S \setminus V(M)| = \gamma|S| \pm \kappa n$.
\end{definition}

\begin{definition}[Absorbable matchings]
\label{def:absorbable}
Let $\xi , \kappa, \gamma, \rho , \eps \in (0,1)$. Let $\cH$ be an $n$-vertex linear multi-hypergraph, let $G\coloneqq \cH^{(2)}$, let $U\coloneqq V^{(1 - \eps)}_+(G)$, and let $S\subseteq U$.
Let $R\subseteq E(G)$, and let $M$ be a matching in $\cH$.   We say $(\cH , M , R , S)$ is \emph{$(\rho, \eps, \gamma, \kappa, \xi)$-absorbable} if
\begin{enumerate}
    \item[\mylabel{AB1}{(\rm AB1)}] $R$ is a $(\rho,10\gamma,\xi,\eps)$-absorber for some $\cV$ such that $U,V(\cH) \in \cV$, 
    
    \item[\mylabel{AB2}{(\rm AB2)}] $M \subseteq \cH \setminus R$, and
    
    \item[\mylabel{AB3}{(\rm AB3)}] at least one of the following holds:
    \begin{enumerate}[(i)]
    \item\label{AB-pseudo} $M$ is  $(\gamma,\kappa)$-pseudorandom with respect to $\cF(R) \cup \{U,S\}$, where $\cF(R) \coloneqq \{N_R(u) \cap U \: : \: u \in U \} \cup \{N_R(u) \setminus U \: : \: u \in U \}$, 
    \item\label{AB-med} $v(M) \leq \gamma n$, or
    \item\label{AB-huge} $|V(M) \cap U| \leq \eps n$ and $U\cup V(M), U\setminus V(M) \in \cV$.
    \end{enumerate}
\end{enumerate}

We say $(\cH , M , R , S)$ is
\begin{itemize}
    \item $(\rho, \eps, \gamma, \kappa, \xi)$-absorbable  \textit{by pseudorandomness of $M$} if~\ref{AB-pseudo} holds,
    \item $(\rho, \eps, \gamma, \kappa, \xi)$-absorbable  \textit{by smallness of $M$} if~\ref{AB-med} holds, and
    \item $(\rho, \eps, \gamma, \kappa, \xi)$-absorbable  \textit{by typicality of $R$} if~\ref{AB-huge} holds.
\end{itemize}

We simply say $(\cH , M , R , S)$ is absorbable if it is $(\rho, \eps, \gamma, \kappa, \xi)$-absorbable and $\rho$, $\eps$, $\gamma$, $\kappa$, and $\xi$ are clear from the context.
\end{definition}
\COMMENT{It may seem unnatural to require $10\gamma$ instead of $\gamma$ as the typicality parameter for the absorber $R$ in ~\ref{AB1}. This is because we need this parameter to be strictly larger than $\gamma$ for our applications. Our choice of $10\gamma$ is just for convenience; one could, in fact, introduce a new parameter $\gamma'$ instead of $10\gamma$, which is different from the parameter $\gamma$ in ~\ref{AB3}\ref{AB-pseudo} since we only rely on the relation $\gamma , \gamma' \ll \rho,\eps$ in the proofs of the lemmas in this section.}

If $\cH$ is an $n$-vertex linear hypergraph and $\phi : \cH_{\mathrm{med}}\cup \cH_{\mathrm{large}} \rightarrow C$ is obtained from Theorem~\ref{large-edge-thm}, then we will choose an absorber $R$ such that for each $c\in C$, $(\cH, \phi^{-1}(c), R, S)$ is absorbable by smallness of $\phi^{-1}(c)$ if $c\in C_{\mathrm{med}}$ and by typicality of $R$ if $\phi$ assigns $c$ to a huge edge.  For essentially every other $c \in C$, we will find a matching $M_c \supseteq \phi^{-1}(c)$ in Section~\ref{small-non-graph-section} such that $(\cH , M_c , R , S)$ is absorbable by pseudorandomness of $M_c$.

In Sections~\ref{subsection:pseudorandom-absorption} and~\ref{subsection-huge-and-med-absorption}, we will show that if $(\cH , M , R , S)$ is absorbable, then we can extend $M$ to cover all but at most one vertex of $U$ using the edges of $R$ (see Definition~\ref{def:coverage}).  Note that if $U = \varnothing$, then there is no need to extend the matchings, and in this case, the lemmas below are vacuously true.

\subsection{Absorption for pseudorandom matchings}\label{subsection:pseudorandom-absorption}
We will use the following two lemmas (depending on whether $|U|$ is small or not) to extend matchings $N$ for which $(\cH , N , R , S)$ is absorbable by pseudorandomness of $N$.

\begin{lemma}\label{crossing-absorption-lemma}
    Let $0 < 1/n_0 \ll \xi \ll \kappa \ll \gamma \ll \rho, \eps \ll 1$, let $n \geq n_0$ and let $k \leq \kappa n$.  Let $\cH$ be an $n$-vertex linear multi-hypergraph, let $G\coloneqq \cH^{(2)}$, let $R\subseteq E(G)$, and let $U\coloneqq V^{(1 - \eps)}_+(G)$. Let $\cN := \{N_1,\dots, N_{k}\}$ be a set of edge-disjoint matchings in $\cH$ such that $(\cH , N_i, R , \varnothing)$ is absorbable by pseudorandomness of $N_i$ for each $i\in[k]$.
    
    If $|U| \leq n / 100$, then there is a set of edge-disjoint matchings $\cN' \coloneqq \{N'_1, \dots, N'_k\}$ in $\cH$ such that for all $i \in [k]$,
    \begin{itemize}
        \item $N_i' \supseteq N_i$ and $N_i' \setminus N_i \subseteq R$, and 
        \item $\cN'$ has perfect coverage of $U$.
    \end{itemize}
\end{lemma}

\begin{lemma}\label{internal-absorption-lemma}
    Let $0 < 1/n_0 \ll \xi \ll \kappa \ll \gamma \ll \rho, \eps \ll 1$, let $n \geq n_0$ and let $k \leq \kappa n$. Let $\cH$ be an $n$-vertex linear multi-hypergraph, let $G\coloneqq \cH^{(2)}$, let $R\subseteq E(G)$, let $U\coloneqq V^{(1 - \eps)}_+(G)$, and let $S \subseteq U$ satisfy $|S| \geq \gamma n$ if $|U| > (1-2\eps)n$.  Let $\cN := \{N_1,\dots, N_{k}\}$ be a set of edge-disjoint matchings in $\cH$ such that $(\cH , N_i , R , S)$ is absorbable by pseudorandomness of $N_i$ for each $i\in[k]$.
    
    If $|U| \geq n/100$, then there is a set of edge-disjoint matchings $\cN' \coloneqq \{N'_1, \dots, N'_k\}$ in $\cH$ such that for all $i \in [k]$,
    \begin{itemize}
        \item $N_i' \supseteq N_i$ and $N_i' \setminus N_i \subseteq R$, and 
        \item if $|U| \leq (1-2\eps)n$, then $\cN'$ has perfect coverage of $U$. Otherwise, $\cN'$ has nearly-perfect coverage of $U$ with defects in $S$.
    \end{itemize}
\end{lemma}

To prove the above lemmas we will need the following simple propositions, which follow easily from Hall's theorem.

\begin{proposition}
\label{crossing-matching-prop}
  Let $0 < \xi \ll \rho \leq 1$.  If $H$ is an upper $(\rho, \xi)$-regular bipartite graph with bipartition $(A, B)$ such that $v(H) \leq \rho|A|/\xi$ and every $v\in A$ satisfies $d_H(v) \geq 2\rho |A|$, then $H$ has a matching covering $A$.
\end{proposition}
\begin{proof}
  By Hall's Theorem, it suffices to show that every non-empty $S\subseteq A$ satisfies $|N(S)| \geq |S|$.  Since $|N(S)| \geq 2\rho |A| \geq \xi v(H)$, we can assume $|S| \geq \xi v(H)$.  Since every $v\in A$ satisfies $d_H(v) \geq 2\rho |A|$, every non-empty $S\subseteq A$ satisfies $e_{H}(S, N(S)) \geq 2\rho|S||A|.$
  
  However, since $H$ is upper $(\rho, \xi)$-regular and $|S|,|N(S)| \geq \xi v(H)$, we also have
 $e_{H}(S, N(S)) \leq (\rho + \xi)|S| |N(S)|.$
  Combining these two inequalities, we have $|N(S)| \geq |A| \geq |S|$, as desired.
\end{proof}

\begin{proposition}\label{internal-matching-prop}
  Let $0 < 1/m \ll \xi \ll \rho \leq 1$. If $G$ is an $m$-vertex, upper $(\rho, \xi)$-regular graph such that every $v\in V(G)$ satisfies $d_G(v) \geq 3\rho m / 4$, and $m$ is even, then $G$ has a perfect matching.
\end{proposition}
\begin{proof}
By a standard probabilistic argument, $G$ has a spanning bipartite subgraph $H$ with bipartition $(A, B)$ such that $|A| = |B| = m / 2$ and every $v\in V(G)$ satisfies $d_{H}(v) \geq \rho m / 3$.\COMMENT{Consider such a bipartition chosen uniformly at random.  The distribution of the random subset $A$ coincides with the distribution on subsets of $V(G)$ obtained by choosing each vertex to be in $A$ with probability $1/2$, conditioned on the event that $|A| =  m / 2 $.  The probability of this event is $\Theta(1/\sqrt m)$, but by Chernoff's bound, if each vertex is chosen in $A$ independently with probability $1/2$, then the probability $v$ has fewer than $\rho m / 3$ neighbors in $A$ is $\exp(-\Omega(m))$.}  By Hall's Theorem, it suffices to show that every $S\subseteq A$ satisfies $|N(S)| \geq |S|$.  If $S\neq\varnothing$, then $|N(S)| \geq \rho m / 3 \geq \xi m$, so we assume $|S| \geq \xi m$.
  Since every $v\in A$ satisfies $d_H(v) \geq \rho m / 3$, every $S\subseteq A$ satisfies $e_{H}(S, N(S)) \geq |S| \rho m / 3.$
   However, since $G$ is upper $(\rho, \xi)$-regular, we also have $e_{H}(S, N(S)) \leq (\rho + \xi)|S||N(S)|.$
  Combining these two inequalities, we have $|N(S)| \geq 3 m / 10$.  Let $T\coloneqq B\setminus N(S)$.  If $T = \varnothing$, then $S$ satisfies the Hall condition, as desired, so we assume $T \neq \varnothing$.  By the same argument applied to $T$, we have $|N(T)| \geq 3m / 10$, so $|S| \leq |A| - |N(T)| \leq m / 5$.  Now $|S| \leq m / 5 \leq 3m / 10 \leq |N(S)|$, as desired.
\end{proof}

\begin{proof}[Proof of Lemma~\ref{crossing-absorption-lemma}]
  For each $i\in [k]$, let $H_i$ be the bipartite graph consisting of edges in $R$ with the bipartition $(A_i, B_i)$, where $A_i\coloneqq U\setminus V(N_i)$ and $B_i \coloneqq (V(\cH)\setminus U)\setminus V(N_i)$.  

  We claim that there exist pairwise edge-disjoint matchings $N^{\mathrm{abs}}_i$ in $H_i$ covering $A_i$ for each $i\in[k]$. We find these matchings one-by-one using Proposition~\ref{crossing-matching-prop}, if $|A_i| \geq \xi n / \rho$. Otherwise, we find them greedily. To this end, we assume that for some $\ell \leq k$, we have found such matchings $N^{\mathrm{abs}}_i$ for $i\in[\ell - 1]$, and we show that there exists such a matching $N^{\mathrm{abs}}_{\ell}$, which proves the claim.  Let $H'_\ell := H_\ell \setminus \bigcup_{i\in[\ell-1]}N^{\mathrm{abs}}_i$.  
  
  We first show that every vertex $u\in A_{\ell}$ satisfies $d_{H'_\ell}(u) \geq 2\rho |A_\ell|$. Since $|U| \leq n / 100$ and $R$ is a $(\rho , 10\gamma , \xi , \eps)$-absorber for $\{V(\cH),U \}$ by~\ref{AB1}, every $u\in A_{\ell}$ satisfies\COMMENT{A vertex of $U$, by definition, has degree at least $(1-\eps)n$ in $G'$, so it sends at least $99n/100 - \eps n$ edges to $V(G') \setminus U$ if $|U| \leq n / 100$. Since $R$ is $(\rho , 10\gamma ,G')$-typical with respect to $U$, $|N_R(u) \setminus U| \geq (\rho(99/100 - \eps) - 10\gamma)n$.}
  \begin{equation}\label{absorption-lemma-degree-outside-tmp}
    |N_R(u) \setminus U| \geq (\rho(99/100 - \eps) - 10\gamma)n \geq 98\rho n / 100.
  \end{equation}
  Note that each $N_i$ is $(\gamma , \kappa)$-pseudorandom with respect to $\cF\coloneqq \cF(R) \cup \{U\}$ by~\ref{AB3}\ref{AB-pseudo}. Together with ~\eqref{absorption-lemma-degree-outside-tmp}, this implies that every $u\in A_{\ell}$ satisfies $d_{H_{\ell}}(u) \geq \gamma |N_R(u)\setminus U| - \kappa n \geq  97 \gamma \rho n / 100$.  Since $\ell \leq k \leq \kappa n$, we have
  \begin{equation}\label{absorption-lemma-degree-outside2-tmp}
    d_{H'_{\ell}}(u) \geq d_{H_\ell}(u) - \kappa n \geq 96 \gamma \rho n / 100.
  \end{equation}
  Since $N_i$ is $(\gamma, \kappa)$-pseudorandom with respect to $\cF \ni U$ and $|U| \leq n / 100$, we also have
  \begin{equation}\label{crossing-absorption-lemma-part-size-bound-tmp}
    |A_\ell| \leq \gamma |U| + \kappa n \leq \gamma n / 50.
  \end{equation}
  Combining~\eqref{absorption-lemma-degree-outside2-tmp} and~\eqref{crossing-absorption-lemma-part-size-bound-tmp}, we have $d_{H'_\ell}(u) \geq 2\rho |A_\ell|$, as desired.  
  
  Note that the graph $H^*_\ell := (V(\cH), E(H'_\ell))$ is upper $(\rho, \xi)$-regular\COMMENT{$E(H^*_\ell)$ is, in fact, upper $(\rho, \xi, G')$-regular where $G'$ is the spanning subgraph of $G$ consisting of those edges with at least one vertex in $U$, which implies that $H^*_\ell$ is upper $(\rho, \xi)$-regular} since $H^*_\ell \subseteq R$, and $H^*_\ell$ is bipartite with the bipartition $(A_\ell, V(\cH) \setminus A_\ell)$ where $d_{H^*_\ell}(u) \geq 2\rho |A_\ell|$, for every $u\in A_{\ell}$.
Therefore, if $|A_\ell| \geq \xi n / \rho$, then Proposition~\ref{crossing-matching-prop} implies that $H^*_\ell$ has a matching $N^{\mathrm{abs}}_{\ell}$ covering $A_\ell$, so $N^{\mathrm{abs}}_{\ell}$ is also a matching in $H_\ell'$ covering $A_\ell$, as claimed. Otherwise, for any $u \in A_{\ell}$,
  \begin{equation}\label{eqn:bounda_ell}
      |A_\ell| < \xi n / \rho \leq 96 \gamma \rho n / 100 \overset{\eqref{absorption-lemma-degree-outside2-tmp}}{\leq} d_{H_\ell'}(u),
  \end{equation}
  so we can find such a matching greedily.

  Therefore we have pairwise edge-disjoint matchings $N_i^{\mathrm{abs}}$ in $H_i$ covering $A_i$ for $i\in[k]$, as claimed, and $N_1^{\rm abs} , \dots , N_k^{\rm abs}$ are edge-disjoint from $N_1 , \dots , N_k$ by~\ref{AB2}. For each $i\in[k]$, let ${N'_i} \coloneqq N_i\cup N_i^{\mathrm{abs}}$, and let $\cN' = \{N'_1, \dots, N'_k\}$.  Hence each matching in $\cN'$ covers $U$, so $\cN'$ has perfect coverage of $U$, as desired.
\end{proof}

\begin{proof}[Proof of Lemma~\ref{internal-absorption-lemma}]
 Let $\cF := \cF(R) \cup \{U, S\}$, where $\cF(R) \coloneqq \{N_R(u) \cap U \: : \: u \in U \} \cup \{N_R(u) \setminus U \: : \: u \in U \}$. For each $i\in [k]$, let $G_i$ be the graph with $V(G_i) := V(\cH) \setminus V(N_i)$ and $E(G_i) \coloneqq \{ e \in R \: : \: e \subseteq V(\cH)\setminus V(N_i)\}$, and let $U_i \coloneqq U\setminus V(N_i)$. Since $(\cH , N_i , R , S)$ is absorbable by pseudorandomness of $N_i$, $N_i$ is $(\gamma,\kappa)$-pseudorandom with respect to $\cF \ni U$ (see \ref{AB3}). Thus, we have
  \begin{equation}\label{eqn:size_ui}
      \textrm{$|U_i| = \gamma |U| \pm \kappa n$, so $|U_i| \geq \gamma n / 200$}.
  \end{equation}

We claim the following.
\begin{claim}
\label{Ni_abs_claim}
 For each $i\in[k]$ there exists $u_i \in U_i$ and a matching $N^{\mathrm{abs}}_i$ in $G_i$ such that the following holds.  The vertices $u_1, \dots, u_k$ are distinct, the matchings $N^{\mathrm{abs}}_1, N^{\mathrm{abs}}_2, \ldots, N^{\mathrm{abs}}_k$ are pairwise edge-disjoint, and $N^{\mathrm{abs}}_i$ covers every vertex of $U_i\setminus\{u_i\}$ for each $i\in[k]$.  Moreover, if $|U| \leq (1 - 2\eps)n$, then $N^{\mathrm{abs}}_i$ covers every vertex of $U_i$ for each $i\in [k]$, and otherwise $u_i \in S$.
 \end{claim}

\begin{claimproof}
  We first choose distinct $u_i \in U_i$ for each  $i\in[k]$, as follows.

  \begin{itemize}
    \item If $|U| \leq (1 - 2\eps)n$, then since $R$ is a $(\rho,10\gamma,\xi,\eps)$-absorber for $\{V(\cH),U \}$ by~\ref{AB1}, every $u\in U_i$ satisfies $|N_R(u)\setminus U| \geq (\rho\eps - 10\gamma) n$. Indeed, if $G'$ denotes the spanning subgraph of $G$ consisting of those edges with at least one vertex in $U$, then $|U| \leq (1 - 2\eps)n$ implies that $|N_{G'}(u)\setminus U| \ge \eps n$ for every $u\in U_i$, and since $R$ is $(\rho, 10\gamma, G')$-typical with respect to $U \in \mathcal V$ (for some $\mathcal V$), we have $|N_R(u)\setminus U| \ge  \rho |N_{G'}(u)\setminus U| - 10 \gamma n \ge (\rho \eps - 10\gamma) n$.   By~\ref{AB3}\ref{AB-pseudo}, since $N_i$ is $(\gamma, \kappa)$-pseudorandom with respect to $\cF \supseteq \cF(R)$ for each $i\in[k]$, this inequality implies\COMMENT{If $N_i$ is $(\gamma, \kappa)$-pseudorandom with respect to $\cF \supseteq \cF(R)$, then $|N_{G_i}(u)\setminus U| \ge \gamma(\rho\eps - 10\gamma) n - \kappa n \ge \gamma \rho\eps n / 2$ since $\kappa \ll \gamma \ll \rho, \eps$. 
} that every $u \in U_i$ satisfies $|N_{G_i}(u)\setminus U| \geq \gamma|N_R(u)\setminus U|- \kappa n \geq \gamma(\rho\eps - 10\gamma) n - \kappa n \ge \gamma \rho\eps n / 2$ since $\kappa \ll \gamma \ll \rho, \eps$.  Since $k \le \kappa n$ and $\kappa \ll \gamma, \rho, \eps$, and \eqref{eqn:size_ui} holds, we can choose $u_i \in U_i$ one-by-one such that there is a matching $\{u_iv_i : i\in[k]\}$ where $v_i \in N_{G_i}(u_i)\setminus U$ for each $i\in[k]$.  
  
    \item Otherwise, $|S| \geq \gamma n$, and since $N_i$ is $(\gamma, \kappa)$-pseudorandom with respect to $\mathcal F \ni S$, by~\ref{AB3}\ref{AB-pseudo}, we have $|S \setminus V(N_i)| \geq \gamma |S| - \kappa n \ge \gamma^2 n /2 > \kappa n$ for each $i \in [k]$, so we can choose $u_i \in U_i \cap S = S \setminus V(N_i)$ one-by-one such that they are distinct, as required.
  \end{itemize}
  
  Now let $U'_i \coloneqq U_i\setminus\{u_i\}$ if $|U_i|$ is odd. Otherwise, let $U'_i\coloneqq U_i$.
  By the choice of the vertices $u_i$, it suffices to find pairwise edge-disjoint perfect matchings ${N'}_i^{\mathrm{abs}}$ in $G_i[U'_i]$ for each $i\in [k]$. Indeed if $|U| \leq (1-2\eps)n$ and $|U_i|$ is odd, then $N_i^{\rm abs} \coloneqq {N'}_i^{\mathrm{abs}} \cup \{u_i v_i \}$ satisfies the claim, and otherwise $N_i^{\rm abs} \coloneqq {N'}_i^{\mathrm{abs}}$ satisfies the claim.
  
  We find these matchings one-by-one using Proposition~\ref{internal-matching-prop}.  To this end, we assume that for some $\ell \leq k$, we have found such matchings ${N'}^{\mathrm{abs}}_i$ for $i\in[\ell - 1]$, and we show that there exists such a matching ${N'}^{\mathrm{abs}}_{\ell}$, which proves the claim.  Let $G'_\ell \coloneqq G_\ell[U'_\ell] \setminus \bigcup_{i\in[\ell-1]}{N'}^{\mathrm{abs}}_i$.
  Since $|U| \geq n / 100$ and $R$ is a $(\rho,10\gamma,\xi,\eps)$-absorber for some $\mathcal V$ such that $V(\cH),U \in \mathcal V$ by~\ref{AB1}, every $u\in U$ satisfies
  \begin{equation}\label{internal-absorption-degree-outside-tmp}
    |N_R(u) \cap U| \geq \rho(|U| - \eps n) - 10\gamma n \geq 99 \rho |U| / 100.
  \end{equation}
  
  Note that $N_{\ell}$ is $(\gamma , \kappa)$-pseudorandom with respect to $\cF \supseteq \cF(R) \cup \{U\}$ by~\ref{AB3}\ref{AB-pseudo}. Together with ~\eqref{internal-absorption-degree-outside-tmp}, this implies that every $u\in U'_{\ell}$ satisfies $d_{G_{\ell}[U'_\ell]}(u) \geq \gamma |N_R(u)\cap U| - \kappa n - 1 \geq  98 \gamma\rho |U| / 100$.  Since $\ell \leq k \leq \kappa n$, we have
  \begin{equation}\label{internal-absorption-degree-outside2-tmp}
    d_{G'_{\ell}}(u) \geq d_{G_\ell[U_i']}(u) - \kappa n \geq 97\gamma  \rho |U| / 100.
  \end{equation}
  
  We also have
  \begin{equation}\label{internal-absorption-lemma-part-size-bound-tmp}
    |U'_\ell| \pm 1 = |U_\ell| \overset{\eqref{eqn:size_ui}}{=} \gamma |U| \pm \kappa n, \text{ so } |U'_\ell| \leq 5\gamma |U| / 4.
  \end{equation}
  Since $R$ is upper $(\rho, \xi, G')$-regular and $|U'_\ell| \overset{\eqref{internal-absorption-lemma-part-size-bound-tmp}}{\ge} \gamma n/100 - \kappa n -1 \ge \gamma n/200$, $G'_{\ell}$ is upper $(\rho, 200\xi/\gamma)$-regular by Observation~\ref{upper_reg_subgraphs}.
  Moreover, combining~\eqref{internal-absorption-degree-outside2-tmp} and~\eqref{internal-absorption-lemma-part-size-bound-tmp}, we have $d_{G'_\ell}(u) \geq 3\rho|U'_\ell| / 4 = 3\rho|V(G'_\ell)| / 4$. So by\COMMENT{
  We used $\rho \gg 200 \xi / \gamma$ which holds since $\xi \ll \rho, \gamma$.
} Proposition~\ref{internal-matching-prop} and the fact that $|V(G'_\ell)| =|U'_\ell|$ is even (by the definition of $U'_\ell$), $G'_\ell$ has a perfect matching ${N'}_\ell^{\mathrm{abs}}$, as desired.  
  \end{claimproof}

 Therefore, by Claim~\ref{Ni_abs_claim}, we have pairwise edge-disjoint matchings $N_i^{\rm abs}$ in $G_i$, which by~\ref{AB2} are edge-disjoint from $N_1 , \dots , N_k$. For each $i \in [k]$, let $N'_i \coloneqq N_i\cup N_i^{\mathrm{abs}}$, and let $\cN' = \{N'_1, \dots, N'_k\}$. Then, since $N_i^{\mathrm{abs}}$ is a matching in $G_i$ and $V(G_i) \cap V(N_i) = \varnothing$, we have $V(N_i^{\mathrm{abs}}) \cap V(N_i) = \varnothing$, so $N_i\cup N_i^{\mathrm{abs}}$ is a matching for each $i \in [k]$. Note that $N'_i \supseteq N_i$ and $N'_i\setminus N_i\subseteq R$ for each $i\in[k]$. Moreover,  if $|U| \leq (1 - 2\eps)n$, then by Claim~\ref{Ni_abs_claim} (and the definition of $U_i$), $N'_i \in \cN'$ covers every vertex of $U$ for each $i \in [k]$, so $\cN'$ has perfect coverage of $U$; otherwise, $N'_i \in \cN'$ covers every vertex of $U \setminus \{u_i\}$  where $u_i \in S$ for each $i \in [k]$ and the vertices $u_1, \ldots, u_k$ are distinct, so $\cN'$ has nearly-perfect coverage of $U$ with defects in $S$, as desired.  \end{proof}

\subsection{Absorption for matchings having huge edges or having few vertices}\label{subsection-huge-and-med-absorption}
\begin{definition}[Difficult matching]\label{defn:difficult}
    Let $\cH$ be an $n$-vertex hypergraph, let $G\coloneqq \cH^{(2)}$, and let $U\coloneqq V^{(1 - \eps)}_+(G)$.  A matching $M$ in $\cH$ is \textit{difficult} if it covers at least $3|V(\cH)\setminus U| / 4$ of the vertices in $V(\cH)\setminus U$ and $|V(\cH)\setminus U| \geq 2$.  If the matching $M$ is difficult and consists of a single edge $e$, then we also say that $e$ is \textit{difficult}.
\end{definition}

We will use the following lemma to extend matchings $N$ for which $(\cH , N , R , S)$ is either absorbable by smallness of $N$ or by typicality of $R$ (provided $N$ is not difficult). 

\begin{lemma}\label{huge-edge-absorption-lemma}
  Let $0 < 1/n_0 \ll \xi \ll \kappa \ll \gamma \ll \rho, \eps \ll 1$, and let $n \geq n_0$. Suppose $k \leq \gamma n$. Let $\cH$ be an $n$-vertex linear hypergraph, let $G\coloneqq \cH^{(2)}$, let $R\subseteq E(G)$, let $U\coloneqq V^{(1 - \eps)}_+(G)$, and let $S \subseteq U$ satisfy $|S| > (\gamma + \eps)n$ if $|U| > (1 - 10\eps)n$.
    Let $\cN \coloneqq \{ N_1 , \dots , N_k \}$ be a set of edge-disjoint matchings in $\cH$ such that for each $i\in[k]$, either
    \begin{itemize}
        \item[\rm (a)] $(\cH, N_i, R, S)$ is $(\rho, \eps, \gamma, \kappa, \xi)$-absorbable by smallness of $N_i$, or
        \item[\rm (b)] $(\cH, N_i, R, S)$ is $(\rho, \eps, \gamma, \kappa, \xi)$-absorbable by typicality of $R$ and $N_i$ is not difficult.\COMMENT{When applying this lemma, the condition $|V(M) \cap U| \leq \eps n$ in ~\ref{AB3}\ref{AB-huge} presents no problem because in the extremal case, $|U| \leq \eps n$ and in the non-extremal case we will apply this lemma only when $M$ is the colour class consisting of exactly one huge edge $e$, in which case $|e \cap U| \leq \eps n$ by the definition of $U$ and the linearity of $\cH$.}
    \end{itemize}
  Then there is a set $\cN' := \{ N_1' , \dots , N_k' \}$ of pairwise edge-disjoint matchings such that for $i \in [k]$,
  \begin{itemize}
      \item $N_i' \supseteq N_i$ and $N_i' \setminus N_i \subseteq R$, and  
      
      \item  if $|U| \leq (1-10\eps)n$, then $\cN'$ has perfect coverage of $U$. Otherwise, $\cN'$ has nearly-perfect coverage of $U$ with defects in $S$.
  \end{itemize}
\end{lemma}

\begin{proof}
Let $G'$ be the spanning subgraph of $G$ consisting only of the edges incident to a vertex in $U$. Without loss of generality, we may assume that there is an integer $s$ such that for each $i \in [s]$, we have $|U \setminus V(N_i)| \leq n/100$, and for each $i \in \{s+1, \dots, k \}$, we have $|U \setminus V(N_i)| > n/100$.

For each $i \in [k]$, let $U_i := U \setminus V(N_i)$, let $V_i \coloneqq (V(\cH)\setminus U)\setminus V(N_i)$, let $H_i$ be the bipartite graph with the bipartition $(U_i, V_i)$ and $E(H_i) \coloneqq \{ e \in R \: : \: |e \cap U_i| = |e\cap V_i| = 1\}$, and let $G_i$ be the graph with $V(G_i) := V(\cH) \setminus V(N_i)$ and $E(G_i) \coloneqq \{ e \in R \: : \: e \subseteq V(\cH)\setminus V(N_i)\}$. 
  
 We first claim that there exist pairwise edge-disjoint matchings $N_1^{\rm abs}, \dots, N_s^{\rm abs}$ such that for each $i \in [s]$,  $N_i^{\rm abs}$ is a matching in $H_i$ covering all of the vertices in $U_i$. 
We find these matchings one-by-one using Proposition~\ref{crossing-matching-prop}. To this end, we assume that for some $\ell \leq s$, we have found such matchings $N^{\mathrm{abs}}_i$ for $i\in[\ell - 1]$, we let $H_{\ell}' := H_{\ell} \setminus \bigcup_{t=1}^{\ell-1} N_t^{\rm abs}$, and we show that there exists such a matching $N^{\mathrm{abs}}_{\ell}$ in $H_{\ell}'$, which proves the claim.
    It suffices to show that every $u\in U_\ell$ satisfies
  \begin{equation}\label{huge-absorption-crossing-degree-bound}
    d_{H_{\ell}'}(u) \geq \rho n / 7 \geq 2\rho |U_{\ell}|.
  \end{equation}
  
  Indeed, the graph $H^*_\ell := (V(\cH), E(H'_\ell))$ is upper $(\rho, \xi)$-regular since $E(H^*_\ell) \subseteq R$, and $H^*_\ell$ is bipartite with the bipartition $(U_\ell, V(\cH) \setminus U_\ell)$ where $d_{H^*_\ell}(u) = d_{H'_\ell}(u) \geq 2\rho |U_\ell|$ for every $u\in U_{\ell}$. Therefore, if $|U_{\ell}| \geq \xi n / \rho$, then we have a matching $N_{\ell}^{\mathrm{abs}}$ in $H^*_\ell$ (and so in $H_{\ell}'$) covering $U_{\ell}$ by Proposition~\ref{crossing-matching-prop}, as desired. Otherwise, by \eqref{huge-absorption-crossing-degree-bound},  $|U_{\ell}| < \xi n / \rho < \rho n / 7 \leq d_{H'_{\ell}}(u)$ for every $u \in U_{\ell}$, so we can find the desired matching $N_{\ell}^{\mathrm{abs}}$ covering $U_{\ell}$ greedily.
  
  To prove \eqref{huge-absorption-crossing-degree-bound}, first suppose (b) holds. Since~\ref{AB3}\ref{AB-huge} holds and $N_{\ell}$ is not difficult, we have\COMMENT{Note that $v(N_{\ell}) \le (3/4 + \eps) n$ also holds when $|U| \ge n-1$.} $v(N_{\ell}) \leq (3/4 + \eps)n \leq 4n/5$.  Thus,  $|V_{\ell}| \geq n - v(N_{\ell}) - |U\setminus V(N_{\ell})| \geq n/5 - n/100 \geq n/6$, and since $R$ is $(\rho, 10\gamma , G')$-typical with respect to $\cV \ni U\cup V(N_{\ell})$ by~\ref{AB3}\ref{AB-huge}, every $u \in U_{\ell}$ satisfies $d_{H_{\ell}}(u) = |N_R(u) \cap V_{\ell}| \ge \rho |N_{G'}(u) \cap V_{\ell}| - 10 \gamma n \ge \rho (n/6-\eps n) - 10 \gamma n$, so $d_{H_{\ell}'}(u) \geq (\rho(1/6 - \eps) - 10\gamma - \gamma)n \geq \rho n / 7 \geq 2\rho |U_{\ell}|$ (since $k \leq \gamma n$), as desired.
Now we assume $(\cH, N_\ell, R, S)$ is absorbable by smallness of $N_\ell$, so by~\ref{AB3}\ref{AB-med}, we have $v(N_\ell) \leq \gamma n$.  Thus, $|U| \leq |U \setminus V(N_{\ell})|+v(N_{\ell}) \leq n/100 + \gamma n \leq n/50$, and since $R$ is $(\rho, 10\gamma, G')$-typical with respect to $\cV \ni U$ by~\ref{AB1}, every $u\in U_\ell$ satisfies\COMMENT{Since $u\in U$, it sends at least $(1 - \eps)n - |U| \geq (49/50 - \eps)n$ edges of $G$ outside of $U$.  Since $R$ is $(\rho, 10\gamma, G')$-typical and $k \leq \gamma n$, at least $(\rho(49/50 - \eps) - 10\gamma - \gamma)n$ are in $R\setminus \bigcup_{i=1}^{\ell - 1}N_i^{\mathrm{abs}}$, and all but at most $\gamma n$ of these have an end in $V_\ell$ by smallness.}
  $d_{H_{\ell}}(u) = |N_R(u) \cap V_{\ell}| \ge |N_R(u) \cap (V(\cH) \setminus U)| - v(N_{\ell}) \ge  \rho |N_{G'}(u) \cap (V(\cH) \setminus U)| - 10 \gamma n - v(N_{\ell}) \ge \rho ((1 - \eps)n - |U|) - 10 \gamma n - v(N_{\ell}) \ge (\rho(49/50 - \eps) - 10\gamma - \gamma)n$, so $d_{H_{\ell}'}(u) \geq (\rho(49/50 - \eps) - 10\gamma - \gamma - \gamma)n \geq \rho n / 7 \geq 2\rho |U_{\ell}|$,
  as desired.  Therefore~\eqref{huge-absorption-crossing-degree-bound} holds in both cases, so we have the matchings $N_{1}^{\mathrm{abs}}, \dots, N_{s}^{\mathrm{abs}}$, as claimed.

  \begin{claim}\label{lem6.11claim}
    There exist matchings $N_{s+1}^{\rm abs}, \dots, N_k^{\rm abs}$ and distinct vertices $u_{s+1}, \dots, u_k$ such that for each $i \in \{s+1, \dots, k\}$, $u_i \in U_i$, $N_i^{\rm abs}$ is a matching in $G_i$ covering all the vertices of $U_i \setminus \{ u_i \}$, and the matchings $N_{1}^{\rm abs}, \dots, N_k^{\rm abs}$ are pairwise edge-disjoint. Moreover, for each $i \in \{s+1, \dots, k\}$, if $|U| \leq (1-10\eps)n$, then $N_i^{\rm abs}$ covers all vertices in $U_i$, and otherwise $u_i \in S$. 
  \end{claim}
  
\begin{claimproof} We choose distinct $u_i \in U_i$ for $s+1 \leq i \leq k$ as follows. Let $G_{i}' := G_{i} \setminus \bigcup_{t=1}^{s} N_t^{\rm abs}$ for $s+1 \leq i \leq k$.
  \begin{itemize}
      \item If $|U| \leq (1 - 10\eps)n$, then every $u\in U$ satisfies $|N_G(u) \setminus U| \geq 9\eps n$ (since $d_G(u) \ge (1-\eps)n$), and moreover if $N_i$ is not difficult, then $|V(\cH)\setminus (U\cup V(N_i))| \geq |V(\cH)\setminus U|/4 \geq 2\eps n$ (since $|V(N_i) \cap (V(\cH)\setminus U)| \leq 3|V(\cH)\setminus U|/4$), which implies that every $u\in U$ satisfies $|N_G(u) \setminus (U\cup V(N_i))| \geq |V(\cH)\setminus (U\cup V(N_i))|-\eps n \ge \eps n$ (since $d_G(u) \ge (1-\eps)n$). If (b) holds, then $R$ is $(\rho, 10\gamma , G')$-typical with respect to $\cV\ni U\cup V(N_i)$ by~\ref{AB3}\ref{AB-huge}, so we have $|N_{G_i'}(u) \setminus (U \cup V(N_i))| \geq \rho |N_G(u) \setminus (U\cup V(N_i))| - 10 \gamma n - k \geq \rho\eps  n / 2 > \gamma n \geq k$ for every $u\in U_i$.  If (a) holds, then $v(N_i) \leq \gamma n$ and $R$ is $(\rho, 10\gamma , G')$-typical with respect to $\cV \ni U$ by~\ref{AB1}, so we have $|N_{G_i'}(u) \setminus (U \cup V(N_i))| \geq \rho |N_G(u)\setminus U| - 10 \gamma n - k - v(N_i) > 8\rho\eps n > k$ for every $u\in U_i$.  Thus, $|N_{G_i'}(u) \setminus (U \cup V(N_i))| > k$ for each $i \in \{s+1, \dots, k\}$ and for all $u \in U_i$, and since $|U_i| > n/100 > k$, we can choose $u_i \in U_i$ one-by-one such that there is a matching $\{ u_i v_i \: : \: i \in \{s+1, \dots, k\}\}$ where $v_i \in N_{G_i'}(u_i) \setminus (U \cup V(N_i))$ for each $i \in \{s+1, \dots, k\}$.
      
      \item Otherwise, we have $|S| > (\gamma + \eps) n$. Using~\ref{AB3}\ref{AB-med} or~\ref{AB3}\ref{AB-huge}, and $\gamma \ll \eps$, we have $|V(N_i) \cap U| \leq \eps n$. Since $S\subseteq U$ it then follows that $|S\setminus V(N_i)| \geq |S| - |V(N_i)\cap U| > \gamma n$ for each $i \in \{s+1, \dots, k\}$. Since $k \leq \gamma n$, we can choose $u_i \in S \setminus V(N_i)$ one-by-one such that they are distinct, as required.
  \end{itemize}
Now let $U'_i \coloneqq U_i\setminus\{u_i\}$ if $|U_i|$ is odd. Otherwise, let $U'_i\coloneqq U_i$.  By the choice of the vertices $u_i$, it suffices to find pairwise edge-disjoint perfect matchings ${N'}_i^{\mathrm{abs}}$ in $G'_i[U'_i]$ for each $i\in \{s+1, \dots, k\}$. Indeed, if $|U| \leq (1-10\eps)n$ and $|U_i|$ is odd, then $N_i^{\rm abs} \coloneqq {N'}_i^{\mathrm{abs}} \cup \{u_i v_i \}$ satisfies the claim, and otherwise $N_i^{\rm abs} \coloneqq {N'}_i^{\mathrm{abs}}$ satisfies the claim.
  
  We find these matchings  (${N'}_i^{\mathrm{abs}}$ for $i\in \{s+1, \dots, k\}$) one-by-one using Proposition~\ref{internal-matching-prop}.  To this end, we assume that for some $s+1 \le \ell \leq k$, we have found such matchings ${N'}^{\mathrm{abs}}_i$ for all $s+1 \le i \le \ell - 1$, we let $G''_\ell \coloneqq G'_\ell[U'_\ell] \setminus \bigcup_{s+1 \leq i \leq \ell-1}{N'}^{\mathrm{abs}}_i$, and we show that there exists such a matching ${N'}^{\mathrm{abs}}_{\ell}$ in $G''_\ell$, which proves the claim. Note that
  \begin{equation}\label{eq:lowerboundUell}
  |U'_{\ell}| \geq |U_{\ell}|-1 \geq n / 100 - 1 \ge n/ 200.
  \end{equation}
  
 Since $R$ is upper $(\rho, \xi, G)$-regular and \eqref{eq:lowerboundUell} holds, $G''_\ell$ is upper $(\rho, 200\xi)$-regular by Observation~\ref{upper_reg_subgraphs}. So by Proposition~\ref{internal-matching-prop}, it suffices to show that     every $u\in U'_\ell$ satisfies
  \begin{equation}\label{huge-absorption-internal-degree-bound}
 d_{G_{\ell}''}(u)  \geq 3 \rho |U'_{\ell}| / 4.
  \end{equation}

  To prove \eqref{huge-absorption-internal-degree-bound}, first suppose (b) holds.  
Since $R$ is a $(\rho,10\gamma,\xi,\eps)$-absorber for $\mathcal V \ni U_{\ell}$ by~\ref{AB3}\ref{AB-huge}, every $u\in U'_{\ell}$ satisfies\COMMENT{$u$ sends $|U_{\ell}|-\eps n$ edges of $G$ to $U_{\ell}$, and since $R$ is a $(\rho,10\gamma,\xi,\eps)$-absorber, we have $d_{G_{\ell}}(u) \ge \rho(|U_{\ell}|-\eps n) - 10 \gamma n$. So $d_{G''_{\ell}}(u) \ge \rho(|U_{\ell}|-\eps n) - 10 \gamma n - k -1 \ge \rho(|U_{\ell}|-\eps n) - 10 \gamma n - \gamma n -1$}
  \begin{equation*}
   d_{G_{\ell}''}(u) \geq \rho(|U_{\ell}| - \eps n) - 10 \gamma n - \gamma n - 1 \geq \rho(|U_{\ell}'| - \eps n) - 12 \gamma n \overset{\eqref{eq:lowerboundUell}}{\geq} 3 \rho |U'_{\ell}| / 4,
  \end{equation*}
  as desired.  
  Therefore, we assume (a) holds. 
Since $R$ is $(\rho, 10\gamma , G')$-typical with respect to $\cV\ni U$ by~\ref{AB1}, every $u\in U'_\ell$ satisfies\COMMENT{$u$ sends at least $|U| - \eps n$ edges of $G$ to $U$.  Since $R$ is $(\rho, 10\gamma , G')$-typical and $k \leq \gamma n$, at least $\rho(|U| - \eps n) - 10\gamma n - \gamma n$ of these edges are in $R\setminus \bigcup_{i=1}^{\ell - 1}{N'}_i^{\mathrm{abs}}$, and all but at most $\gamma n + 1$ of these have an end in $U'_\ell$.}
  \begin{equation*}
      d_{G_{\ell}''}(u) \geq \rho(|U| - \eps n) - 10 \gamma n  - \gamma n - \gamma n - 1 \geq \rho(|U_{\ell}'| - \eps n) - 13 \gamma n \overset{\eqref{eq:lowerboundUell}}{\geq} 3 \rho |U'_{\ell}| / 4,
  \end{equation*}
  as desired.  Therefore~\eqref{huge-absorption-internal-degree-bound} holds in both cases, so we have the matchings $N_{s+1}^{\mathrm{abs}}, \dots, N_k^{\mathrm{abs}}$, which proves Claim~\ref{lem6.11claim}.
  \end{claimproof}

  Now, let $N_i' \coloneqq N_i \cup N_i^{\mathrm{abs}}$ for each $i \in [k]$, and let $\cN' := \{N_1' , \dots , N_k'\}$. Note that $N_i'\supseteq N_i$ and $N_i'\setminus N_i \subseteq R$ for each $i \in [k]$.  Moreover, by the definition of $U_i$ and Claim~\ref{lem6.11claim} (together with the discussion before it), if $|U| \leq (1 - 10\eps)n$, then $N'_i \in \cN'$ covers every vertex of $U$ for each $i \in [k]$, so $\cN'$ has perfect coverage of $U$; otherwise, $N'_i \in \cN'$ covers every vertex of $U$ for each $i \in [s]$, and $N'_i \in \cN'$ covers every vertex of $U \setminus \{u_i\}$ where $u_i \in S$ for each $i \in \{s+1, \ldots, k\}$ and the vertices $u_{s+1}, \ldots, u_k$ are distinct, so $\cN'$ has nearly-perfect coverage of $U$ with defects in $S$, as desired. 
\end{proof}

We will use the following lemma in {Step}~\ref{step:diff} of the proof in Section~\ref{proof-section} to extend a difficult matching. Unlike the other cases, this extension does not quite achieve nearly-perfect coverage, but we are able to prove something similar which is sufficiently strong for our purposes.

\begin{lemma}\label{difficult-edge-absorption}
  Let $0 < 1/n_0 \ll \beta \ll 1$, and let $n \geq n_0$.  If $\cH$ is an $n$-vertex linear hypergraph with no singleton edge, $G \coloneqq \cH^{(2)}$, and $M \coloneqq \{e\}$ is a difficult matching where $e$ is huge,
    then at least one of the following holds:
    \begin{enumerate}[(\ref{difficult-edge-absorption}:a), topsep = 6pt]
    \item There is a matching $M'$ such that $M\subseteq M'$, $M' \setminus M \subseteq E(G)$, and $M'$ covers every vertex of $V^{(n - 1)}(\cH)$ and all but at most five vertices of $V^{(n - 2)}(\cH)$, or\label{difficult-edge-covering}
    \item $\chi'(\cH) \leq n$.\label{difficult-edge-extremal}
    \end{enumerate}
\end{lemma}
\begin{proof}     
  Let $U_1 \coloneqq V^{(n - 1)}(\cH)$, let $U_2 \coloneqq V^{(n - 2)}(\cH)$, and let $X \coloneqq V(\cH)\setminus (e\cup U_1\cup U_2)$. Since $\cH$ is linear and $e$ is huge, $e \cap U_i = \varnothing$ for $i\in\{1, 2\}$.

  First, suppose $U_2 = \varnothing$.  If $|U_1|$ is even, then we can find a perfect matching $M_1$ in $G[U_1]$, and $M' \coloneqq M \cup M_1$ satisfies~\ref{difficult-edge-covering}, so we assume $|U_1|$ is odd.  If $X\neq\varnothing$, then there is an edge $uv \in E(G)$ such that $u\in U_1$ and $v \in X$, and there is a perfect matching $M_1$ in $G[U_1\setminus\{u\}]$.  Now $M'\coloneqq M\cup M_1\cup \{uv\}$ satisfies~\ref{difficult-edge-covering}, so we assume $X = \varnothing$, and we show $\chi'(\cH) \leq n$.  Note that if $X, U_2 = \varnothing$, then the only edge in $\cH\setminus E(G)$ is $e$, since otherwise, an edge (different from $e$) of size at least three would need to contain at least two vertices of $U_1$ or $e$ -- either case contradicts the linearity of $\cH$.
  Let $w \in U_1$, let $M_1$ be a perfect matching in $G[U_1\setminus\{w\}]$, and let $G' \coloneqq \cH\setminus(M_1\cup e)$.  Now $G'$ is a graph with exactly one vertex of degree $n - 1$ (namely $w$), so by Theorem~\ref{thm:vizing}, $\chi'(G') \leq n - 1$.  By combining a proper $(n - 1)$-edge-colouring of $G'$ with the colour class consisting of $M_1\cup e$, we have $\chi'(\cH) \leq n$, as desired.

  Therefore we assume $U_2\neq\varnothing$, let $u \in U_2$, and let $m\coloneqq |U_1\cup U_2|$. 
  Let $G'\coloneqq G[U_1 \cup U_2] - u$ if $m$ is odd and let $G' \coloneqq G[U_1\cup U_2]$ otherwise.  If $G'$ has a perfect matching $P$, then $M' \coloneqq M \cup P$ satisfies ~\ref{difficult-edge-covering}, so we assume otherwise.  Thus, by Tutte's theorem~\cite{tutte1947}, there is a set $S$ such that $G' - S$ has at least $|S|+1$ odd components. Since $G'$ has an even number of vertices, both $|S|$ and the number of odd components of $G'-S$ have the same parity. Therefore, $G' - S$ has at least $|S|+2$ odd components. 
  Note that if a vertex has degree at least $n-2$ in $\cH$ then it is incident to at most one edge of size at least three, so it has degree at least $n-3$ in $G$. Since $G'$ is an induced subgraph of $G$, we have $\delta(G') \geq v(G') - 3$. 
Thus, for every component $C$ of $G'-S$, there are at most two vertices in $G'-(S\cup V(C))$, so there are at most three components in $G'-S$. Since $G'-S$ has at least $|S|+2$ odd components, we have $|S| \leq 1$.
  
  If $|S| = 1$, then there are exactly three components in $G'-S$. Since there are at most two vertices outside of each component of $G'-S$, we deduce that each component of $G'-S$ has exactly one vertex, so $v(G') = 4$. Let $V(G') = \{a,b,c,d \}$ with $S=\{d\}$. Since $ab,ac,bc \notin E(G)$, we have $a,b,c \in U_2$ by the definition of $G'$. Since $U_1 \cup U_2 = V(G') \cup \{u\} = \{a,b,c,d,u\}$ and $u \in U_2$, we have $|U_2| \leq 5$ and $|U_1| \leq 1$. If $U_1 = \varnothing$, then $M'\coloneqq M$ satisfies~\ref{difficult-edge-covering}, and otherwise, $U_1 = \{ d \}$ and $M'\coloneqq M\cup\{du\}$  satisfies~\ref{difficult-edge-covering}, as desired.
  
  If $S = \varnothing$, then there are at least two odd components in $G'$, so $U_1 = \varnothing$. Moreover, since there are at most two vertices outside of each component of $G'$, each odd component of $G'$ has exactly one vertex, which implies that $\delta(G') = 0$ and $v(G') \leq \delta(G') + 3 = 3$. 
Therefore, $|U_1 \cup U_2| = |V(G') \cup \{ u \}| \leq 4$, and $M' \coloneqq M$ satisfies~\ref{difficult-edge-covering}, as desired.  
\end{proof}

Note that if $\cH$ is a degenerate plane, then our proof returns outcome ~\ref{difficult-edge-extremal}.
This is also the case for certain hypergraphs $\cH$ which are similar to a degenerate plane. 
\section{Colouring small edges that are not in the reservoir}\label{small-non-graph-section}

In this section, we prove three lemmas which will be applied to colour all of the small edges that are not in the reservoir (where the reservoir is constructed in Section~\ref{proof-section}). Since we may need to reuse the colours already used for large edges and medium edges  (given by Theorem~\ref{large-edge-thm}), we need to formulate the lemmas to colour the small edges (that are not in the reservoir) by extending the colour classes given by Theorem~\ref{large-edge-thm}.

The lemma below is used repeatedly in the proof of Lemma~\ref{lem:nibble2} to colour most of the non-reserved small edges in such a way that every colour class exhibits some pseudorandom properties.

\begin{lemma}[Nibble Lemma]\label{lem:pseudorandom_matching}
Let $0 < 1/n_0\ll 1/r , \beta \ll \kappa , \gamma \ll 1$, let $n \geq n_0$, and let $D \in [n^{1/2} , n]$.  Let $\cH$ be an $n$-vertex linear multi-hypergraph.
\begin{itemize}
    \item Let $\cH' \subseteq \cH$ be a linear multi-hypergraph such that $V(\cH) = V(\cH')$, every $e \in \cH'$ satisfies $|e| \leq r$, and for every $w \in V(\cH')$ we have $d_{\cH'}(w) = (1 \pm \beta)D$,
    
    \item let $\cF_V$ and $\cF_E$ be a family of subsets in $V(\cH')$ and $E(\cH')$, respectively, such that $|\cF_V|,|\cF_E|  \leq n^{\log n}$, and
    
    \item let\COMMENT{we can imagine $M_1 , \dots,  M_D \subseteq \cH$ as colour classes of large edges} $M_1 , \dots,  M_D \subseteq \cH \setminus \cH'$ be pairwise edge-disjoint matchings such that for every $i \in [D]$, $|V(M_i)| \leq \beta D$, and for every edge $e \in \cH'$, we have 
    $|\{ i \in [D] \: : \: e \cap V(M_i) \neq \varnothing \}| \leq \beta D$.
\end{itemize}

Then there exist pairwise edge-disjoint matchings $N_1 , \dots , N_D$ in $\cH$ such that for any $i \in [D]$,
\begin{enumerate}
    \item[\mylabel{NB1}{(\ref{lem:pseudorandom_matching}.1)}] $N_i \supseteq M_i$ and $N_i \setminus M_i \subseteq \cH'$,
    
    \item[\mylabel{NB2}{(\ref{lem:pseudorandom_matching}.2)}] $N_i$ is $(\gamma , \kappa)$-pseudorandom with respect to $\cF_V$, and
    
    \item[\mylabel{NB3}{(\ref{lem:pseudorandom_matching}.3)}] $|F \setminus \bigcup_{j=1}^{D} N_j| \leq \gamma |F| + \kappa \max(|F|,D)$ for each $F \in \cF_E$.
\end{enumerate}
\end{lemma}
Note that if we let $E(\cH') \in  \cF_E$, then (\ref{lem:pseudorandom_matching}.3) implies that $\bigcup_{j=1}^{D} N_j$ contains almost all of the edges of $\cH'$. The matchings $M_1 , \dots,  M_D$ will play the role of some of the colour classes given by Theorem~\ref{large-edge-thm}.

The overall idea of the proof of Lemma~\ref{lem:pseudorandom_matching} is as follows. First we embed $\cH'$ into an $r$-uniform linear hypergraph $\cH_{\rm unif}$ using Lemma~\ref{lem:embed}. We then embed $\cH_{\rm unif}$ into an $(r+1)$-uniform auxiliary hypergraph $\cH_{\rm aux}$, and we find a pseudorandom matching $N^*$ in $\cH_{\rm aux}$ using Corollary~\ref{cor:sparse_egj}, which yields $D$ edge-disjoint pseudorandom matchings $N'_1 , \dots , N'_D$ in $\cH'$. Then we will show that the matchings $N_i := N'_i \cup M_i$ for $i \in [D]$ satisfy the desired properties.

\begin{proof}
We apply\COMMENT{We first apply the embedding lemma to make the given multi-hypergraph $\cH'$ uniform -- each singleton will turn into a different $r$-edge, so the resulting $\cH_{\rm unif}$ will not contain multiple edges.} Lemma~\ref{lem:embed} to $\cH'$ with $(1+\beta)D$, $2\beta D$, and $n$ playing the roles of $D$, $C$, and $N$, respectively, to obtain an $r$-uniform linear hypergraph $\cH_{\rm unif}$ such that 
\begin{enumerate}[(a)]
    \item\label{cond:unif1} $V(\cH_{\rm unif}) \supseteq V(\cH')$ and $v(\cH_{\rm unif}) \leq n^5$, and
   
    \item\label{cond:unif2} every vertex $w \in \cH_{\rm unif}$ satisfies $d_{\cH_{\rm unif}}(w) = (1 \pm \beta)D$. Moreover, $d_{\cH'}(w) = d_{\cH_{\rm unif}}(w)$ for any $w \in V(\cH')$.
\end{enumerate}

Indeed, since $\cH' \subseteq \cH_{\rm unif}|_{V(\cH')}$ by ~\ref{E1}, we have $V(\cH_{\rm unif}) \supseteq V(\cH')$, and by ~\ref{E3}, we have $v(\cH_{\rm unif}) \leq r(r-1)^2((1+\beta)D)^3 n \leq n^5$, so ~\ref{cond:unif1} holds. By the first statement of ~\ref{E2}, every vertex $w \in \cH_{\rm unif}$ satisfies $d_{\cH_{\rm unif}}(w) = (1 \pm \beta)D$. Moreover, for any $w \in V(\cH')$, we have $d_{\cH'}(w) \geq (1-\beta)D = (1+\beta)D - 2\beta D$, so $d_{\cH'}(w) = d_{\cH_{\rm unif}}(w)$ by the second statement of~\ref{E2}. Thus ~\ref{cond:unif2} holds as well. 

Now by the second statement of ~\ref{cond:unif2}, $\cH_{\rm unif}|_{V(\cH')} - \cH'$ contains no singleton edges, so by ~\ref{E1}, we have
\begin{enumerate}[(c)]
    \item\label{cond:unif3} $\cH' = \cH_{\rm unif}|_{V(\cH')}$.
\end{enumerate}

Let\COMMENT{The following map $\psi$ allows us to identify the edges of $\cH'$ with some edges in $\cH_{\rm unif}$. We may have added some dummy vertices to each edge $e \in \cH'$, and $\psi^{-1}(e)$ represents the transformed edge in $\cH_{\rm unif}$.} $E_{\rm meet} := \{e \in \cH_{\rm unif} : e \cap V(\cH') \not = \varnothing \}$. By~\ref{cond:unif3}, we have a bijective map\COMMENT{$\psi$ allows us to identify the edges in  $E_{\rm meet}$ with $\cH'$}
\begin{equation*}
    \textrm{$\psi : E_{\rm meet} \to \cH'$ such that $e^* \mapsto e^* \cap V(\cH')$.}
\end{equation*}

Thus, for any $w \in V(\cH')$, we have $E_{\cH'}(w) = \{\psi(e^*) \: : \: w \in e^* \in \cH_{\rm unif} \}$.\COMMENT{For any $e^* \in \cH_{\rm unif}$, if $e^*$ satisfies $|\{ i \in [D] \: : \: e^* \cap V(M_i) \neq \varnothing \}| > \beta D$, then $e^* \in E_{\rm meet}$ and $\psi(e^*) = e^* \cap V(\cH')$; it follows that $\psi(e^*)\in \cH'$ also satisfies $|\{ i \in [D] \: : \: \psi(e^*) \cap V(M_i) \neq \varnothing \}| > \beta D$, contradicting our assumption.} Note that the assumption $| \{ i \in [D] : e \cap V(M_i) \ne \varnothing \} | \leq \beta D$ for any $e \in \cH'$, implies that
\begin{equation}\label{eqn:unif_intersect}
    \textrm{for every $e^* \in \cH_{\rm unif}$, we have  $|\{ i \in [D] \: : \: e^* \cap V(M_i) \neq \varnothing \}| \leq \beta D$.}
\end{equation}

We construct\COMMENT{Now we construct an incidence hypergraph $\cH_{\rm aux}$ of $\cH_{\rm unif}$, where $V_i$ denotes the $i$th clone set of $V(\cH_{\rm unif})$ that avoids a prescribed set $V(M_i)$.} an $(r+1)$-uniform linear hypergraph $\cH_{\rm aux}$ based on $\cH_{\rm unif}$ and the sets $V(M_1), \dots , V(M_D)$, as follows. 

\begin{itemize}
    \item For any $i \in [D]$, let $V_i^* := \{w^i \: : \:w \in V(\cH_{\rm unif}) \}$, where for any distinct $i_1, i_2 \in [D]$, we have $V_{i_1}^* \cap V_{i_2}^* = \varnothing$. Now let us define a map $\varphi : [D] \times V(\cH_{\rm unif}) \to \bigcup_{i=1}^{D} V_i^*$ such that $\varphi(i,w) := w^i$ for any $(i,w) \in [D] \times V(\cH_{\rm unif})$.\COMMENT{Thus $\varphi(i,\cdot) : V(\cH_{\rm unif}) \to V_i^*$ is a bijective map that identifies $w$ with $w^i$. For a subset $S \subseteq V(\cH_{\rm unif})$, we may use a standard notion that $\varphi(i,S) := \{ \varphi(i,s) \: : \: s\in S \}$.} 
    
    \item For any $i \in [D]$, let $V_i := V_i^* \setminus \varphi(i,V(M_i))$.
      
    \item Let $V(\cH_{\rm aux}) := \cH_{\rm unif} \cup \bigcup_{i \in [D]} V_i$, where $\cH_{\rm unif} \cap V_i = \varnothing$ for $i \in [D]$.
      
    \item Let $\cH_{\rm aux} := \{ \{f, v_1^i, \dots, v_r^i\} \: : \: f = \{v_1,\dots,v_r \} \in \cH_{\rm unif} \: , \{v_1^i,\dots,v_r^i \} \subseteq V_i \: , \: i \in [D] \}.$
\end{itemize}

Note that for every $i \in [D]$, an edge $\{f, v_1^i, \dots, v_r^i\} \in \cH_{\rm aux}$ if and only if $f = \{v_1,\dots,v_r \} \in \cH_{\rm unif}$ and $f \cap V(M_i) = \varnothing$ (since $V_i := V_i^* \setminus \varphi(i,V(M_i))$). Thus, for every $w \in V(\cH_{\rm unif})$ and $i \in [D]$ such that $w^i \in V_i$, $d_{\cH_{\rm aux}}(w^i) = d_{\cH_{\rm unif}}(w) - |\{f \in \cH_{\rm unif} \: : \: w \in f \text{ and } f \cap V(M_i) \ne \varnothing \}|$. Thus, $d_{\cH_{\rm unif}}(w) - |V(M_i)| \leq d_{\cH_{\rm aux}}(w^i) \leq d_{\cH_{\rm unif}}(w)$, since $\cH_{\rm unif}$ is linear. Since $|V(M_i)| \leq \beta D$ and (b) holds, this implies $d_{\cH_{\rm aux}}(w^i) = (1 \pm 2\beta)D$. Moreover, for every $e^* = \{u_1, \ldots, u_r\} \in \cH_{\rm unif}$, since $\{e^*, u_1^i, \dots, u_r^i\} \in \cH_{\rm aux}$ if and only if $e^* \cap V(M_i) = \varnothing$, we have $d_{\cH_{\rm aux}}(e^*) = D - |\{ i \in [D] \: : \: e^* \cap V(M_i) \ne \varnothing  \}|$, so by~\eqref{eqn:unif_intersect}, $(1-\beta)D \leq d_{\cH_{\rm aux}}(e^*) \leq D$. In summary, we have the following.
\begin{equation}\label{eqn:degree_aux}
    \textrm{For every vertex $w \in V(\cH_{\rm aux})$, we have $d_{\cH_{\rm aux}}(w) = (1 \pm 2\beta)D$.}
\end{equation}

\COMMENT{Note that $\sum_{i=1}^{D} |V_i| \geq D|V(\cH')| - \sum_{i=1}^{D}|V(M_i)| \geq (1-\beta)D |V(\cH')|$.}By the construction of $\cH_{\rm aux}$, we have
\begin{equation}\label{eqn:numvertices_aux}
    n \leq (1-\beta)D|V(\cH')| + |\cH_{\rm unif}| \leq n' := |V(\cH_{\rm aux})| \leq D|V(\cH_{\rm unif})| + |\cH_{\rm unif}| \overset{\ref{cond:unif1},\ref{cond:unif2}}{\leq} n^7.
\end{equation}

Let
\begin{equation*}
\cF^{\rm aux} := \bigcup_{i=1}^{D} \{ \varphi(i , A) \setminus \varphi(i,V(M_i)) \: : \: A \in \cF_V \} \cup \{ \psi^{-1} (F) \: : \: F \in \cF_E \},
\end{equation*}
where $|\cF^{\rm aux}| \leq D|\cF_V| + |\cF_E|  \leq n^{2 \log n} \leq (n')^{2 \log n'}$. Since ~\eqref{eqn:degree_aux} holds and $D \geq n^{1/2} \overset{\eqref{eqn:numvertices_aux}}{\geq} (n')^{1/14}$, we can apply Corollary~\ref{cor:sparse_egj}, with $n'$, $\cH_{\rm aux}$, $D$, $2 \beta$, $\gamma$, $\cF^{\rm aux}$, $r+1$, playing the roles of $n$, $\cH$, $D$, $\kappa$, $\gamma$, $\cF$, $r$, respectively, to obtain a matching $N^*$ in $\cH_{\rm aux}$, such that $|A \setminus V(N^*)| = (\gamma \pm 8 \beta)|A|$ for every $A \in \cF^{\rm aux}$ with $|A| \geq D^{1/20}$. This implies that for any $A \in \cF^{\rm aux}$,
\begin{equation}\label{eqn:a_uncovered}
    |A \setminus V(N^*)| = \gamma |A| \pm (8 \beta |A| + D^{1/20}).
\end{equation}

For $i \in [D]$, let 
\begin{equation*}
    N_i' := \{ \psi(e^*) \: : \: e^* = \{v_1,\dots,v_r \} \in E_{\rm meet}\:,\:\{e^* , v_1^i , \dots , v_r^i \} \in N^*\}\text{ and }N_i := N_i' \cup M_i.
\end{equation*}
Then $N_1' , \dots , N_D' \subseteq \cH'$ satisfy the following properties.
\begin{enumerate}[(i)]
    \item\label{Ni_vertex} For any $w \in V(\cH')$, we have $w \in V(N_i')$ if and only if $w^i = \varphi(i,w) \in V(N^*)$. (Indeed, if $w^i \in V(N^*)$, then $w^i \in \{e^* , v_1^i , \dots , v_r^i \} \in N^*$ for some $e^* = \{v_1,\dots,v_r \}$. Then $w \in e^*$. But since $w \in V(\cH')$, we have $w \in V(\cH') \cap e^* = \psi(e^*) \in N_i'$. The other direction is obvious.)
    
    \item\label{Ni_edge} For any $e \in \cH'$, we have $e \in \bigcup_{i=1}^{D} N_i'$ if and only if $\psi^{-1}(e) \in V(N^*)$. (Indeed, note that $e \in \bigcup_{i=1}^{D} N_i'$, if and only if $e = \psi(e^*)$ for some $e^* = \{v_1,\dots,v_r \}$ such that $\{e^* , v_1^i , \dots , v_r^i \} \in N^*$, i.e., $e^* = \psi^{-1}(e) \in V(N^*)$.)
\end{enumerate}
We claim that $N_1' , \dots , N_D'$ are pairwise edge-disjoint matchings in $\cH'$, and that for every $i \in [D]$, we have $V(N_i') \cap V(M_i) = \varnothing$. Indeed, since the edges of $N^*$ are disjoint, they contain distinct elements $\psi^{-1}(e)$ of $\cH_{\rm unif}$, which implies that the corresponding edges $e \in \bigcup_{i=1}^{D} N_i'$ are distinct (as $\psi$ is a bijective map), so $N_1' , \dots , N_D'$ are pairwise edge-disjoint matchings in $\cH'$. Moreover, for every $i \in [D]$ and $e^* = \{v_1,\dots,v_r \} \in E_{\rm meet}$ such that $\psi(e^*) \in N_i'$, we have\COMMENT{The first equality below follows since $\varphi(i,\cdot) : V(\cH_{\rm unif}) \to V_i^*$ is a bijection.}
\begin{equation*}
\varphi(i,e^* \cap V(M_i)) = \varphi(i,e^*) \cap \varphi(i, V(M_i)) = \{v_1^i , \dots , v_r^i \} \cap \varphi(i,V(M_i)) \subseteq V_i \cap \varphi(i,V(M_i)) = \varnothing.   
\end{equation*}
Thus $\psi(e^*) \cap V(M_i) \subseteq e^* \cap V(M_i) = \varnothing$, so $V(N_i') \cap V(M_i) = \varnothing$, as claimed.

Moreover, recall that $M_1, M_2, \dots, M_D \subseteq \cH \setminus \cH'$ are pairwise edge-disjoint. Altogether this implies that $N_1 , \dots , N_D$ are pairwise edge-disjoint matchings in $\cH$. This proves~\ref{NB1}.

For any $F \in \cF_E$, since $\psi^{-1}(F) \subseteq \cH_{\rm unif} \subseteq V(\cH^{\rm aux})$ and $\psi^{-1}(F) \in \cF^{\rm aux}$, we have $$|F \setminus \bigcup_{i=1}^{D} N_i'| \overset{\ref{Ni_edge}}{=} |\psi^{-1}(F) \setminus V(N^*)| \overset{\eqref{eqn:a_uncovered}}{=} \gamma |F| \pm (8 \beta |F| + D^{1/20} ) \leq \gamma |F| + \kappa \max(|F|,D).$$ Thus, ~\ref{NB3} holds.

Finally, we prove ~\ref{NB2}. Let us consider any $A \in \cF_V$ and $i \in [D]$. Since
\begin{equation*}
    |A \setminus V(N_i)| = |(A \setminus V(M_i)) \setminus V(N_i')| \overset{\ref{Ni_vertex}} {=} |\varphi(i , A \setminus V(M_i)) \setminus V(N^*)| = |(\varphi(i, A)\setminus \varphi(i, V(M_i)))\setminus V(N^*)|
\end{equation*}
and $\varphi(i,A) \setminus \varphi(i,V(M_i)) \in \cF^{\rm aux}$, by~\eqref{eqn:a_uncovered}, we have\COMMENT{here we use $\gamma|A \setminus V(M_i)| \pm (8 \beta |A \setminus V(M_i)| + D^{1/20}) = \gamma|A|\pm ( |V(M_i)| +8 \beta |A \setminus V(M_i)| + D^{1/20}) = \gamma|A|\pm ( \beta n +8 \beta n + D^{1/20}) = \gamma|A|\pm  \kappa n$.}  $|A \setminus V(N_i')| = \gamma|A| \pm \kappa n$.
Thus, the matching $N_i$ is $(\gamma , \kappa)$-pseudorandom with respect to $\cF_V$, proving~\ref{NB2}.
\end{proof}

In the next lemma, using the absorption lemmas from Section~\ref{absorption-section}, we extend the matchings given by the previous lemma in such a way that each matching will cover all but at most one vertex of $V_+^{(1-\eps)}(G)$, where the uncovered vertex of $V_+^{(1-\eps)}(G)$ must lie in a prescribed defect set $S$. In principle, we could apply Lemma~\ref{lem:pseudorandom_matching} directly with $D = (1 - \rho)n$ to colour almost all of the non-reserved small edges, but in order to be able to apply Lemmas~\ref{crossing-absorption-lemma} and~\ref{internal-absorption-lemma} to each matching, we actually need to partition the hypergraph into subhypergraphs of maximum degree at most $\kappa n$, and we apply Lemma~\ref{lem:pseudorandom_matching} to each part successively, alternating with applications of one of Lemma~\ref{crossing-absorption-lemma} or~\ref{internal-absorption-lemma}.

\begin{lemma}[Main colouring lemma]\label{lem:nibble2}
Let $0 < 1/n_0 \ll 1/r , \xi , \beta \ll \gamma \ll \varepsilon , \rho \ll 1$, let $n \geq n_0$, and let $D \in [n^{2/3} , n]$.  Let $\cH$ be an $n$-vertex linear multi-hypergraph, let $G := \cH^{(2)}$ and let $U := V_+^{(1-\varepsilon)}(G)$.
\begin{enumerate}
    \item[\mylabel{C1}{{\bf C1}}] Let $S \subseteq U$ satisfy $|S| \geq D + \gamma n$ if $|U| > (1-2\eps)n$,
    
    \item[\mylabel{C2}{{\bf C2}}] let $R \subseteq E(G)$ be a $(\rho,\gamma,\xi,\varepsilon)$-absorber for $\mathcal{V}$ such that $U,V(\cH) \in \cV$, 
    
    \item[\mylabel{C3}{{\bf C3}}] let $\cH' \subseteq \cH \setminus R$ be\COMMENT{we can think of $\cH'$ as the hypergraph with all of the small non-reserved edges} 
    a linear multi-hypergraph such that $V(\cH) = V(\cH')$, every edge $e \in \cH'$ satisfies $|e| \leq r$, and $d_{\cH'}(w) = (1 \pm \beta)D$ for every vertex $w \in V(\cH')$, and 
    
    \item[\mylabel{C4}{{\bf C4}}] let $\cM = \{M_1 , \dots,  M_{D}\}$ be a set of edge-disjoint matchings in $\cH \setminus (\cH' \cup R)$ such that $|V(M_i)| \leq \beta D$ for every $i \in [D]$, and $|\{ i \in [D] \: : \: e \cap V(M_i) \neq \varnothing \}| \leq \beta D$ for every edge $e \in \cH'$.
\end{enumerate}

Then there exists a set $\cN := \{N_1 , \dots , N_{D} \}$ of edge-disjoint matchings in $\cH$ satisfying the following. 
\begin{enumerate}
    \item[\mylabel{N1}{(\ref{lem:nibble2}.1)}] For every $i \in [D]$, we have $N_i \supseteq M_i$ and $N_i \setminus M_i \subseteq \cH' \cup R$.
    
    \item[\mylabel{N2}{(\ref{lem:nibble2}.2)}] For every vertex $w \in V(\cH)$, $|E_R(w) \cap \bigcup_{k=1}^{D} N_k| \leq \gamma D$ and  $|E_{\cH'}(w) \setminus \bigcup_{k=1}^{D} N_k| \leq \gamma D$.
    
    \item[\mylabel{N3}{(\ref{lem:nibble2}.3)}]  If $|U| \leq (1-2\eps)n$, then $\cN$ has perfect coverage of $U$. Otherwise, $\cN$ has nearly-perfect coverage of $U$ with defects in $S$.
\end{enumerate}
\end{lemma}
\begin{proof}
Let $K := \lceil \kappa^{-1} \rceil$, where we choose $\kappa$ so that $r^{-1} , \xi , \beta \ll \kappa \ll \gamma$. First, we find a partition of $\cH'$ into pairwise edge-disjoint hypergraphs $\cH_1' , \dots , \cH_K'$ such that $\bigcup_{j=1}^{K} \cH_j' = \cH'$, and every vertex has degree $(1 \pm 2 \beta)D/K$ in $\cH_i'$ for $i \in [K]$. To show that the desired partition exists, for each $e \in \cH'$, we put $e$ into exactly one of the hypergraphs $\cH_{1}' , \dots , \cH_{K}'$ with probability $K^{-1}$ independently at random. Then for every $w \in V(\cH)$ and $i \in [K]$, $\mathbb{E}[d_{\cH_i'}(w)] =  d_{\cH'}(w)/K = (1 \pm \beta)D/K$. Hence, by Theorem~\ref{thm:chernoff}, since $D \geq n^{2/3}$, we have $d_{\cH_{i}'}(w) = \mathbb{E}[d_{\cH_i'}(w)] \pm \beta D / K = (1 \pm 2 \beta) D/K$ with probability at least $1-\exp\left({-n^{1/4}}\right)$.

Now, for each $i \in [K]$, we may choose $n_{i}$ to be either $\lfloor D/K \rfloor$ or $\lceil D/K \rceil$ such that $\sum_{j=1}^{K} n_{j} = D$. Let us partition the set $[D]$ into $K$ disjoint parts $I_1 , \dots , I_{K}$ such that $|I_i| = n_i$. Then $n_i \leq \kappa n$, and every vertex in $\cH_{i}'$ has degree $(1 \pm 3 \beta) n_{i}$.

Let us define the following statements for $0 \le j \le K$.
\begin{enumerate}[(i)$_j$]
    \item\label{N1_appl} For any $1 \leq k \leq j$, there exists a set $\cN_k := \{ N_c \: : \: c\in I_k \}$ of $n_k$ matchings in $\cH$ such that $M_c \subseteq N_c$ and $N_c \setminus M_c \subseteq \cH_k' \cup R$ for every $c \in I_k$. Moreover, the matchings in $\bigcup_{k=1}^{j} \cN_k$ are pairwise edge-disjoint.
    
    \item\label{N2_appl} For every $w \in V(\cH)$,
    \begin{equation*}
    |E_{R}(w) \cap \bigcup_{k \in [j]}\bigcup_{N \in \cN_k} N| \leq \gamma \sum_{k \in [j]} n_k \:\:\:\textrm{and}\:\:\: |E_{\bigcup_{k=1}^{j}\cH_k' }(w) \setminus \bigcup_{k \in [j]}\bigcup_{N \in \cN_k} N| \leq \gamma \sum_{k \in [j]} n_k.
    \end{equation*}

    \item\label{N3_appl}  If $|U| \leq (1-2\eps)n$, then $\bigcup_{k=1}^{j} \cN_{k}$ has perfect coverage of $U$. Otherwise, $\bigcup_{k=1}^{j} \cN_{k}$ has nearly-perfect coverage of $U$ with defects in $S$.
\end{enumerate}

Using induction on $j$, we will show that
~\hyperref[N1_appl]{${\rm (i)}_K$}--\hyperref[N3_appl]{${\rm (iii)}_K$}
hold which clearly proves the lemma. Indeed, then $\cN = \bigcup_{k=1}^{K} \cN_{k}$ is the desired set of edge-disjoint matchings in $\cH$ satisfying ~\ref{N1} (by~\hyperref[N1_appl]{${\rm (i)}_K$}), ~\ref{N2} (by~\hyperref[N2_appl]{${\rm (ii)}_K$} and the fact that $\sum_{k \in [K]} n_k = D$ and $\bigcup_{k=1}^{K}\cH_k' = \cH'$) and ~\ref{N3} (by ~\hyperref[N3_appl]{${\rm (iii)}_K$}). 

Note that~\hyperref[N1_appl]{${\rm (i)}_0$}--\hyperref[N3_appl]{${\rm (iii)}_0$} trivially hold. Let $i \in [K]$, and suppose that
~\hyperref[N1_appl]{${\rm (i)}_{i-1}$}--\hyperref[N3_appl]{${\rm (iii)}_{i-1}$} hold. 
Our goal is to find a collection $\cN_i$ of $n_i$ pairwise edge-disjoint matchings in $\cH$ satisfying
~\hyperref[N1_appl]{${\rm (i)}_i$}--\hyperref[N3_appl]{${\rm (iii)}_i$}.

Let $R_i := R \setminus \bigcup_{k=1}^{i-1} \bigcup_{N \in \cN_k} N$, let $S_{i} := S \setminus \bigcup_{k=1}^{i-1} \bigcup_{N \in \cN_k} (U \setminus V(N))$, and 
\begin{equation}\label{def:cv_nibble}
    \textrm{$\cW := \cF(R_i) \cup \{U,S_i \}$, where $\cF(R_i) := \{ N_{R_i}(u) \cap U \: : \: u \in U \} \cup \{ N_{R_i}(u) \setminus U \: : \: u \in U \}$.}
\end{equation}

Note that $d_{\cH_{i}'}(w) = (1 \pm 3 \beta) n_{i} = (1 \pm \beta^{1/2})n_i$ for every $w \in V(\cH_{i}')$, $n^{1/2} \le D/(2K) \le n_i \le 2D/K \le n$, $|\cW| = 2|U|+2 \le n^{\log n}$, $|\{ E_{\cH_i'}(w) \: : \: w \in V(\cH) \}| = n \le n^{\log n}$, and for every $c \in I_i$ we have $|V(M_c)| \le \beta D \le \beta^{1/2} n_i$, and since $\beta \ll \kappa$, by ~\ref{C4}, for every edge $e \in \cH_{i}'$ we have $|\{ c \in [I_i] \: : \: e \cap V(M_c) \neq \varnothing \}| \leq \beta D \le \beta^{1/2} n_i$. So we can apply Lemma~\ref{lem:pseudorandom_matching} with $\cH_i' , \cW , \{ E_{\cH_i'}(w) \: : \: w \in V(\cH) \}, \{ M_c \: : \: c \in I_i \},  \beta^{1/2} , \gamma / 4 , n_i$ playing the roles of $\cH' , \cF_V , \cF_E , \{ M_1 , \dots , M_D \} , \beta , \gamma , D$\COMMENT{we use the index set $I_i$ instead of $[n_i]$ for the matchings.} to obtain a set $\cN_i' := \{ N_c' \: : \: c \in I_i \}$ of $n_{i}$ pairwise edge-disjoint matchings in $\cH$ such that the following hold.

\begin{enumerate}
    \item[\mylabel{cond:dis_1}{(a)\textsubscript{\textit{i}}}] For every $c \in I_i$, $N_c' \supseteq M_c$ and $N_c' \setminus M_c \subseteq \cH_i'$. In particular, $N_c' \cap R = \varnothing$ (by ~\ref{C3} and ~\ref{C4}).

    \item[\mylabel{cond:dis_2}{(b)\textsubscript{\textit{i}}}]  For every $c \in I_i$, $N_c'$ is $(\gamma / 4 ,\kappa)$-pseudorandom with respect to $\cW$.
    
    \item[\mylabel{cond:dis_3}{(c)\textsubscript{\textit{i}}}] For every\COMMENT{here we used that $|E_{\cH_i'}(w)| = d_{\cH_i'}(w)$} $w \in V(\cH)$, $|E_{\cH_i'}(w) \setminus \bigcup_{c \in I_i} N_c'| \leq \gamma d_{\cH_i'}(w)/4 + \kappa \max(d_{\cH_i'}(w),n_i) \le \gamma  n_i / 2$ (since $d_{\cH_i'}(w) \le (1+\beta^{1/2})n_i$ and $\beta, \kappa \ll \gamma$).
    
    \item[\mylabel{cond:dis_4}{(d)\textsubscript{\textit{i}}}] For every $w \in V(\cH)$, the number of matchings in $\cN_i'$ not covering $w$ is at most 
    \begin{equation*}
        |\cN_i'| - (d_{\cH_i'}(w) - |E_{\cH_i'}(w) \setminus \bigcup_{c \in I_i} N_c'|) \overset{\ref{cond:dis_3}}{\leq} n_i - (1 - 3 \beta)n_i + \gamma n_i / 2 \leq \gamma n_i.
    \end{equation*}
\end{enumerate}

Now we show that for any given $c \in I_i$, $(\cH , N_c' , R_i , S_i)$ is $(\rho , \eps , \gamma/4 , \kappa , \xi)$-absorbable by pseudorandomness of $N_c'$, as follows.
\begin{itemize}
    \item Using the fact that $R$ is a $(\rho,\gamma,\xi,\varepsilon)$-absorber for $\mathcal{V}$, 
    ~\hyperref[N2_appl]{${\rm (ii)}_{i-1}$}, and Observation~\ref{obs:del_absorber}, we deduce that $R_i$ is a $(\rho , 2\gamma , \xi , \eps)$-absorber for $\cV$, showing~\ref{AB1}.
    
    \item \ref{cond:dis_1} implies~\ref{AB2}.
    
    \item By~\ref{cond:dis_2}, $N_c'$ is $(\gamma / 4 ,\kappa)$-pseudorandom with respect to $\cW$, so~\ref{AB-pseudo} of~\ref{AB3} holds, as required.
\end{itemize}

Moreover, if $|U| > (1-2\eps)n$, then
by~\hyperref[N3_appl]{${\rm (iii)}_{i-1}$}, $\bigcup_{k=1}^{i-1} \cN_{k}$ has nearly-perfect coverage of $U$ with defects in $S$, so each $N \in \bigcup_{k=1}^{i-1} \cN_{k}$ satisfies $|U \setminus V(N)| \le 1$ and $U \setminus V(N) \subseteq S$. Thus, by ~\ref{C1}, $|S_{i}| = |S \setminus \bigcup_{k=1}^{i-1} \bigcup_{N \in \cN_k} (U \setminus V(N))| \geq |S|-\sum_{k=1}^{i-1} n_{k} \geq D + \gamma n - D = \gamma n$, so we can apply either Lemma~\ref{internal-absorption-lemma} or Lemma~\ref{crossing-absorption-lemma} (which does not require the set $S$) depending on the size of $U$, with $\gamma/4 , R_i , S_i , \cN_i'$ playing the roles of $\gamma , R , S , \cN$.  This yields a set $\cN_i := \{ N_c \: : \: c \in I_i \}$ of $n_{i}$ pairwise edge-disjoint matchings in $\cH$ such that for every $c \in I_i$, we have $N_c \supseteq N_c'$ and $N_c \setminus N_c' \subseteq R_i$. Moreover, if $|U| \leq (1-2\eps)n$, then $\cN_i$ has perfect coverage of $U$. Otherwise, $\cN_i$ has nearly-perfect coverage of $U$ with defects in $S_i \subseteq S$. Using these properties of $\cN_i$, we now show that~\hyperref[N1_appl]{${\rm (i)}_i$}--\hyperref[N3_appl]{${\rm (iii)}_i$} hold (as desired).

\begin{itemize}
    \item For every $c \in I_i$, we have $N_c \supseteq N_c'$ and $N_c \setminus N_c' \subseteq R_i$. Therefore, since $N_c' \supseteq M_c$ and $N_c' \setminus M_c \subseteq \cH_i'$ by~\ref{cond:dis_1}, we have $N_c \supseteq M_c$ and $N_c \setminus M_c \subseteq \cH_i' \cup R_i$ for every $c \in I_i$. Since $R_i = R \setminus \bigcup_{k=1}^{i-1} \bigcup_{N \in \cN_k} N$, and $\cH_i'$ is disjoint from $\bigcup_{k=1}^{i-1}\cH_k' \cup R$ (by ~\ref{C3}), it follows that the matchings in $\cN_i$ are edge-disjoint from the matchings in $\bigcup_{k=1}^{i-1} \cN_k$ (which are also pairwise edge-disjoint by ~\hyperref[N1_appl]{${\rm (i)}_{i-1}$}). This shows~\hyperref[N1_appl]{${\rm (i)}_i$}.
    
    \item By~\ref{cond:dis_3}, for any $w \in V(\cH)$, $|E_{\cH_i'}(w) \setminus \bigcup_{c \in I_i}N_c| \leq \gamma n_i/2$. 
    Moreover, by~\ref{cond:dis_4}, all but at most $\gamma n_i$ of the matchings in $\cN_i'$ cover $w$, and $w$ may be covered using an edge of $R_i$ only if $w$ is not covered by the matchings in $\cN_i'$ (since the matchings in $\cN_i'$ are disjoint from $R_i \subseteq R$ by ~\ref{cond:dis_1}).
    Thus, $|E_{R_i}(w) \cap \bigcup_{c \in I_i} N_c| \leq \gamma n_i$ for any $w \in V(\cH)$. Combining these facts with ~\hyperref[N2_appl]{${\rm (ii)}_{i-1}$}, 
    it follows that 
    ~\hyperref[N2_appl]{${\rm (ii)}_i$} holds.
    
    \item If $|U| \leq (1-2\eps)n$, then $\cN_i$ has perfect coverage of $U$, and $\bigcup_{k=1}^{i-1} \cN_{k}$ has perfect coverage of $U$ (by~\hyperref[N3_appl]{${\rm (iii)}_{i-1}$}), so $\bigcup_{k=1}^{i} \cN_{k}$ has perfect coverage of $U$. Otherwise, $\cN_i$ has nearly-perfect coverage of $U$ with defects in $S_i \subseteq S$, and $\bigcup_{k=1}^{i-1} \cN_{k}$ has nearly-perfect coverage of $U$ with defects in $S$ (by~\hyperref[N3_appl]{${\rm (iii)}_{i-1}$}). Since $S_{i} = S \setminus \bigcup_{k=1}^{i-1} \bigcup_{N \in \cN_k} (U \setminus V(N))$, this shows that $\bigcup_{k=1}^{i} \cN_{k}$ has nearly-perfect coverage of $U$ with defects in $S$. So~\hyperref[N3_appl]{${\rm (iii)}_i$} holds. \qedhere
\end{itemize}
\end{proof}

Lemma~\ref{lem:nibble2} colours most of the non-reserved small edges (as shown in (\ref{lem:nibble2}.2)). We will use the following lemma to colour the remaining non-reserved small edges such that every colour class covers all but at most one vertex of $V_+^{(1-\eps)}(G)$. Since the proportion of remaining non-reserved small edges is small, we can afford to be less efficient in the number of colours we use in this step in order to ensure that each colour class is small, which allows us to use Lemma~\ref{huge-edge-absorption-lemma} to extend them.

\begin{lemma}[Leftover colouring lemma]\label{lem:leftover}
Let $0 < 1/n_0 \ll 1/r,\xi \ll \gamma \ll \rho,\eps \ll 1$, let $n \geq n_0$, and let $D \in [n^{2/3} , n]$.  Let $\cH$ be an $n$-vertex linear hypergraph, let $G := \cH^{(2)}$, and let $U := V_+^{(1-\eps)}(G)$.
\begin{enumerate}
    \item[\mylabel{L1}{{\bf L1}}] Let $C$ be a set of colours with $\gamma D/2 \leq |C| \leq \gamma D$,
      
    \item[\mylabel{L2}{{\bf L2}}] let $\cM := \{M_c \: : \: c \in C \}$ be a set of pairwise edge-disjoint matchings in $\cH$, where $|V(M_c)| \leq \gamma n / 2$ for every $c \in C$, 
      
    \item[\mylabel{L3}{{\bf L3}}] let $R \subseteq E(G) \setminus \bigcup_{c \in C}M_c$ be a $(\rho , 10\gamma , \xi , \eps)$-absorber for $\cV := \{ V(\cH) , U \}$,
    
    \item[\mylabel{L4}{{\bf L4}}] let $\cH_{\rm rem} \subseteq \cH \setminus (R \cup \bigcup_{c \in C} M_c)$ such that $V(\cH) = V(\cH_{\rm rem})$, $\Delta(\cH_{\rm rem}) \leq \gamma^2 D / 20$, every edge $e \in \cH_{\rm rem}$ satisfies $|e| \leq r$ and $|\{ c \in C \: : \: e \cap V(M_c) \ne \varnothing \}| \leq \gamma^2 D / 100$,
      
    \item[\mylabel{L5}{{\bf L5}}] let $S \subseteq U$ be a subset satisfying $|S| > (\gamma + \eps)n$ if $|U| > (1-10\eps)n$.
\end{enumerate}

Then there exists a set $\cN := \{ N_c \: : \: c \in C \}$ of pairwise edge-disjoint matchings such that the following hold.
\begin{enumerate}[(\ref{lem:leftover}.1)]
    \item\label{F1} For any $c \in C$, we have $N_c \supseteq M_c$ and $\cH_{\rm rem} \subseteq \bigcup_{c \in C}(N_c \setminus M_c) \subseteq \cH_{\rm rem} \cup R$.
    
    \item\label{F2} If $|U| \leq (1-10\eps)n$, then $\cN$ has perfect coverage of $U$. Otherwise, $\cN$ has nearly-perfect coverage of $U$ with defects in $S$. 
\end{enumerate}
\end{lemma}
\begin{proof}
Let us choose $\kappa$ such that $r^{-1},\xi \ll \kappa \ll \gamma$. Let $t := \lceil 4 \gamma^{-1} \rceil$ and $C_1 , \dots , C_t$ be a partition of $C$ into $t$ sets such that $|C_i| \geq \gamma^2 D / 10$ for $1 \leq i \leq t$.

We will first show that for every $i\in [t]$, there exists a proper colouring of the edges of $\cH_{\rm rem}$ using colours from $C_i$. To that end, let $C(e) := \{c \in C \: : \: e \cap V(M_c) \ne \varnothing \}$ for every edge $e \in \cH_{\rm rem}$. Since $\Delta(\cH_{\rm rem}) \leq \gamma^2 D / 20$, $|C(e)| \leq \gamma^2 D / 100$ and $|C_i| \geq \gamma^2 D / 10$, we can apply Theorem~\ref{thm:kahn} with $\gamma^2 D / 20$ and $1/2$ playing the roles of $D$ and $\alpha$, respectively, to show that for every $i\in [t]$, there exists a proper edge-colouring  $\psi_i : \cH_{\rm rem} \to C_i$ such that $\psi_i(e) \not \in C(e)$ for every $e \in \cH_{\rm rem}$. By the definition of $C(e)$, this implies that $V(\psi_i^{-1}(c)) \cap V(M_c) = \varnothing$ for any $i \in [t]$ and $c \in C_i$.

Let us now define a proper edge-colouring $\psi : \cH_{\rm rem} \to C$ by choosing $i(e) \in [t]$ uniformly and independently at random for each $e \in \cH_{\rm rem}$ and setting $\psi(e) := \psi_{i(e)}(e)$. Fix an arbitrary colour $c \in C$. Then there is a unique $j \in [t]$ such that $c \in C_j$, and $|V(\psi^{-1}(c))| = \sum_{e \in \psi_j^{-1}(c)}|e| {\bf 1}_{i(e) = j}$, so by the linearity of expectation, $\mathbb{E}[|V(\psi^{-1}(c))|] = \sum_{e \in \psi_j^{-1}(c)} |e|\cdot\mathbb{P}(i(e) = j) \leq n/t \leq \gamma n / 4$.

Since $|V(\psi^{-1}(c))|$ is a weighted sum of independent indicator random variables with maximum weight at most $r$, by Theorem~\ref{thm:chernoff}, 	$\mathbb{P}(||V(\psi^{-1}(c))| - \mathbb{E}[|V(\psi^{-1}(c))|]| \geq \gamma n / 4) \leq e^{-\sqrt{n}}$, so by the union bound, we have $|V(\psi^{-1}(c))| < \gamma n / 2$ for all $c \in C$ with non-zero probability. Combining this with the fact that for every $c \in C$, $V(\psi^{-1}(c)) \cap V(M_c) = \varnothing$ and $|V(M_c)| \leq \gamma n / 2$, it follows that there exists a proper edge-colouring $\psi : \cH_{\rm rem} \to C$ such that $ \{ \psi^{-1}(c) \cup M_c : c \in C \}$ is a set of edge-disjoint matchings in $\cH$ where for every $c \in C$, $|V(\psi^{-1}(c)) \cup V(M_c)| \leq \gamma n$. Thus, for every $c \in C$, $(\cH , M_c \cup \psi^{-1}(c) , R , S)$ is $(\rho,\eps,\gamma,\kappa,\xi)$-absorbable\COMMENT{By Definition~\ref{def:absorbable} this means $R$ is a $(\rho , 10\gamma , \xi , \eps)$-absorber. See also the remark following the definition.} by smallness of $M_c \cup \psi^{-1}(c)$. So we can apply Lemma~\ref{huge-edge-absorption-lemma} with $\{ M_c \cup \psi^{-1}(c) : c \in C \}$ playing the role\COMMENT{we used $|C| \leq \gamma n$} of $\cN$\COMMENT{Here, we also use the index set $C$ instead of $[|C|]$.} to obtain a set $\cN = \{N_c : c \in C\}$ of pairwise edge-disjoint matchings in $\cH$ such that the following hold.
\begin{itemize}
    \item For every $c \in C$, $N_c \supseteq M_c \cup \psi^{-1}(c)$ and $N_c \setminus (M_c \cup \psi^{-1}(c)) \subseteq R$; thus~\ref{F1} holds.
    
    \item If $|U| \leq (1-10\eps)n$, then $\cN$ has perfect coverage of $U$. Otherwise, $\cN$ has nearly-perfect coverage of $U$ with defects in $S$; thus~\ref{F2} holds. \qedhere
\end{itemize} 
\end{proof}
 
\section{Optimal edge-colourings}\label{graph-edge-section}

In this section we will prove colouring results (Lemma~\ref{hall-based-finishing-lemma} and Corollary~\ref{cor:optcol}) which will be used to colour the leftover edges of the reservoir in the final step of the proof of Theorem~\ref{main-thm}. Lemma~\ref{hall-based-finishing-lemma} will be used when $\cH$ satisfies (\ref{large-edge-thm}:b) of Theorem~\ref{large-edge-thm}, i.e., when $\cH$ is close to a projective plane.
Corollary~\ref{cor:optcol} will be used when $\cH$ is close to a complete graph.

\subsection{Edge-colourings with forbidden lists}

The following proposition follows easily from Hall's theorem.

\begin{proposition}\label{obs:matching_A}
Let $G$ be a bipartite graph with bipartition $\{A,B \}$, and let $\delta_A$ and $\delta_B$ be the minimum degrees of the vertices in $A$ and $B$, respectively. If $|A| \leq |B|$ and $\delta_A + \delta_B \geq |A|$, then $G$ has a matching covering $A$.
\end{proposition}
\begin{proof}
Suppose that $G$ has no matching covering $A$.
By Hall's theorem, there exists a non-empty set $A' \subseteq A$ and $B' := N_G(A')$ such that $|B'| \leq |A'|-1$, thus $B \setminus B'$ is non-empty. Let $a \in A'$ and $b \in B \setminus B'$. Then
\begin{equation*}
    \delta_A \leq d_G(a) \leq |B'| \leq |A'|-1\:\:\:,\:\:\:\delta_B \leq d_G(b) \leq |A|-|A'|,
\end{equation*}
since $N_G(a) \subseteq B'$ and $N_G(b) \subseteq A \setminus A'$. Hence $\delta_A + \delta_B \leq |A|-1$, contradicting our assumption.
\end{proof}

\begin{lemma}\label{hall-based-finishing-lemma}
Let $\delta \in (0,1)$, let $H$ be an $n$-vertex graph, let $C$ be a set of colours satisfying $|C| \geq 7 \delta n$, and for every $w \in V(H)$, let $C_w \subseteq C$ such that the following hold.
\begin{enumerate}
    \item[\mylabel{cond:hall1}{{\rm (i)}}] For any $w \in V(H)$, $d_H(w) \leq |C|-|C_w|$.
    
    \item[\mylabel{cond:hall2}{{\rm (ii)}}] There is a set $U \subseteq V(H)$ with $|U| \leq \delta n$ such that every edge of $H$ is incident to a vertex of $U$.
    
    \item[\mylabel{cond:hall3}{{\rm (iii)}}] For every\COMMENT{we imagine $C_w$ as the set of forbidden colours} vertex $w \in V(H)$, $|C_w| \le \delta n$.
    
    \item[\mylabel{cond:hall4}{{\rm (iv)}}] For every $c \in C$, $|\{w \in V(H) \: : \: c \in C_w \}| \leq \delta n$.
\end{enumerate}
Then there exists a proper edge-colouring $\phi : E(H) \to C$ such that every edge $uv \in E(H)$ satisfies $\phi(uv) \notin C_u \cup C_v$.
\end{lemma}

\begin{proof}\COMMENT{the inductive steps of the proof are quite straightforward, perhaps we can omit the proof below when submitting our paper, and leave it in Appendix.} Let $U := \{u_1 , \dots , u_t \}$, where $t := |U| \leq \delta n$. Let $\phi_0 : \varnothing \to C$ be an empty function, and for every $1 \leq j \leq t$, let us inductively define a proper edge-colouring $\phi_j : \bigcup_{k=1}^{j} E_{H}(u_k) \to C$ such that
\begin{enumerate}
    \item[\mylabel{cond:pa}{(a)\textsubscript{\textit{j}}}] $\phi_j (uv) \notin C_u \cup C_v$ for each $uv \in \bigcup_{k=1}^{j} E_H(u_k)$, and
    \item[\mylabel{cond:pb}{(b)\textsubscript{\textit{j}}}] $\phi_j$ is a proper edge-colouring extending $\phi_{j-1}$.
\end{enumerate}

Since every edge of $H$ is incident to a vertex of $U$ by~\ref{cond:hall2}, $\phi := \phi_t$ satisfies the assertion of the lemma. Let $i \in [t]$, and suppose we have already defined $\phi_j$ satisfying both~\ref{cond:pa} and~\ref{cond:pb} for $j \in [i - 1]$; now we aim to construct $\phi_i$ satisfying
~\hyperref[cond:pa]{${\rm (a)}_i$} and~\hyperref[cond:pb]{${\rm (b)}_i$}.

For each $v \in V(H) \setminus \{u_1 , \dots , u_{i-1} \}$, let $C_v^* := \phi_{i-1} ( E_H(v) \cap \bigcup_{j=1}^{i-1} E_H(u_j) )$ be the set of colours of edges incident to $v$ in $\phi_{i-1}$. Since any vertex $v \in V(H) \setminus \{u_1 , \dots , u_{i-1} \}$ is adjacent to at most $i-1$ vertices in $\{u_1 , \dots , u_{i-1} \}$, we have
\begin{equation}\label{eqn:boundcv*}
    |C_v^*| \overset{\hyperref[cond:pb]{ {\rm (b)}_{i-1}}}{=} |E_H(v) \cap \bigcup_{j=1}^{i-1} E_H(u_j)| \leq i-1 \leq \delta n.
\end{equation}

Let $A := E_H(u_i) \setminus \bigcup_{j=1}^{i-1} E_H(u_j)$ and $B := C \setminus (C_{u_i} \cup C_{u_i}^*)$. Let $G_i$ be an auxiliary bipartite graph with the bipartition $\{A,B\}$ such that $\{e,c\} \in E(G_i)$ for $e = u_i v \in A$ and $c \in B$ if and only if $c \notin C_v \cup C_v^*$. Thus, the following hold.
\begin{itemize}
    \item $|A| \leq |B|$. Indeed, $|A| = d_{H}(u_i) - |E_H(u_i) \cap \bigcup_{j=1}^{i-1} E_H(u_j)|$ and $|B| \overset{\eqref{eqn:boundcv*}}{\geq} |C| - |C_{u_i}| - |E_H(u_i) \cap \bigcup_{j=1}^{i-1} E_H(u_j)| \overset{\ref{cond:hall1}}{\geq} |A|$.
    
    \item For each $e = u_i v \in A$, we have\COMMENT{below we used $|C| \ge 7 \delta n$.}
    \begin{equation}\label{eqn:mindeg_A}
        d_{G_i}(e) = |B|-|C_v \cup C_v^*| \geq |C| - |C_{u_i} \cup C_{u_i}^*| - |C_v \cup C_v^*| \overset{\eqref{eqn:boundcv*},~\ref{cond:hall3}}{\geq} 3 \delta n.
    \end{equation}
    
    \item For each $c \in B$, since there are at most $i-1$ edges in $\bigcup_{j=1}^{i-1} E_H(u_j)$ which could be assigned the colour $c$ by $\phi_{i-1}$, we have $| \{ v \in V(H) : c \in C_v^* \} | \leq 2(i-1) \leq 2 \delta n$. Thus 
    \begin{equation}\label{eqn:mindeg_B}
        d_{G_i}(c) \geq |A| - |\{ v \in V(H) : c \in C_v^* \}| - |\{w \in V(H) \: : \: c \in C_w \}| \overset{\ref{cond:hall4}}{\geq} |A| - 3 \delta n.
    \end{equation}
\end{itemize}

Let $\delta_A$ and $\delta_B$ be the minimum degrees of the vertices in $A$ and $B$ in $G_i$, respectively. Then by~\eqref{eqn:mindeg_A} and~\eqref{eqn:mindeg_B}, we have $\delta_A + \delta_B \geq |A|$. Moreover, $|A| \leq |B|$, so there exists a matching $M_i$ in $G_i$ covering $A$ by Proposition~\ref{obs:matching_A}.

For each $e \in A = E_H(u_i) \setminus \bigcup_{j=1}^{i-1} E_H(u_j)$, let $c_e \in B$ be the unique element such that $\{e, c_e\} \in E(M_i)$. Let us define
$\phi_{i}(e) := \phi_{i-1}(e)$ for $e \in \bigcup_{j=1}^{i-1} E_H(u_j)$, and $\phi_i(e) :=c_e$ for $e \in E_H(u_i) \setminus \bigcup_{j=1}^{i-1} E_H(u_j)$. 

Now we show that~\hyperref[cond:pa]{${\rm (a)}_i$} and~\hyperref[cond:pb]{${\rm (b)}_i$} hold. 
By the definitions of $A$, $B$, and $G_i$, for every $e = u_iv \in E_H(u_i) \setminus \bigcup_{j=1}^{i-1} E_H(u_j)$, we have $\phi_i(e) = c_e \notin C_{u_i} \cup C_{u_i}^* \cup C_v \cup C_v^*$, implying
~\hyperref[cond:pa]{${\rm (a)}_i$} holds.
Moreover, by the definitions of $C_{u_i}^*$ and $C_v^*$, the colour $c_e = \phi_i(e)$ is different  from the colours of the edges intersecting $e = u_i v$ which are already coloured in $\phi_1 , \dots , \phi_{i-1}$, and all edges in $E_H(u_i) \setminus \bigcup_{j=1}^{i-1} E_H(u_j)$ receive different colours in $\phi_i$, since $M_i$ is a matching,
so~\hyperref[cond:pb]{${\rm (b)}_i$} holds, as desired.
\end{proof}

\subsection{Edge-colouring pseudorandom graphs}\label{Delta-edge-col-section}

Here we derive an optimal colouring result for pseudorandom graphs (Corollary~\ref{cor:optcol}) from a result (Theorem~\ref{thm:optcol}) on the overfull subgraph conjecture, which in turn is a consequence of the main result in~\cite{kuhn2013hamilton} on Hamilton decompositions of robustly expanding regular graphs.

\begin{definition}[Lower regularity]
Let $\rho,\xi \in (0,1)$, and let $G$ be an $n$-vertex graph. A set $R \subseteq E(G)$ is \textit{lower $(\rho, \xi, G)$-regular} if for every pair of disjoint sets $S, T\subseteq V(G)$ with $|S|, |T| \geq \xi n$, we have $|E_{G}(S, T)\cap R| \geq \rho e_{G}(S, T) - \xi |S||T|$.
In particular, for $0 < \rho' < \rho$ and $\xi < \xi' < 1$, $R$ is also lower $(\rho',\xi',G)$-regular.

A graph $H$ is \textit{lower $(\rho,\xi)$-regular} if $E(H)$ is lower $(\rho,\xi,K_{v(H)})$-regular, i.e., for every pair of disjoint sets $S,T \subseteq V(H)$ with $|S|,|T| \geq \xi v(H)$, we have $e_H(S,T) \geq (\rho-\xi)|S||T|$.
\end{definition}

\begin{proposition}[Robustness of lower regularity]\label{obs:del_lowerreg}
Let $\alpha , \xi , \rho \in (0,1)$, and let $G$ be a graph. Then the following hold.
\begin{enumerate}
    \item[\mylabel{LR1}{(\ref{obs:del_lowerreg}.1)}] If $R \subseteq E(G)$ is lower $(\rho,\xi,G)$-regular and $R' \subseteq R$ satisfies $\Delta(R-R') \leq \alpha v(G)$, then $R'$ is lower $(\rho,\xi+\alpha^{1/2},G)$-regular.
    
    \item[\mylabel{LR2}{(\ref{obs:del_lowerreg}.2)}] If $G$ is lower $(\rho,\xi)$-regular and $G \subseteq H$ such that $v(H) \leq (1 + \alpha \xi) v(G)$, then $H$ is lower $(\rho(1-\alpha)^2 , \frac{\xi}{1-\alpha})$-regular.
\end{enumerate}
\end{proposition}
\begin{proof}
First we prove~\ref{LR1}.
Let $v(G) = n$. For any disjoint $S,T \subseteq V(G)$ with $|S|,|T| \geq (\xi + \alpha^{1/2})n$, we have
$|E_{G}(S,T) \cap R'| \geq |E_{G}(S,T) \cap R| - \alpha n|S| \geq \rho e_{G}(S,T) - \xi |S||T| - \alpha n |S| \geq \rho e_{G}(S,T) - |S| (\xi|T| + \alpha n) \geq \rho e_{G}(S,T) - |S||T| (\xi + \alpha^{1/2})$, since $|T| \geq \alpha^{1/2} n$. Thus, $R'$ is lower $(\rho , \xi+\alpha^{1/2} , G)$-regular, as desired.

Now we prove~\ref{LR2}.
For any disjoint $A,B \subseteq V(H)$ such that $|A|,|B| \geq \frac{\xi}{1-\alpha} n$, we have $|A \cap V(G)|,|B \cap V(G)| \geq \frac{\xi}{1-\alpha} n - \alpha \xi n \ge \xi n$. Moreover, $|A \cap V(G)| \geq |A| - \alpha \xi n \geq |A| - \alpha \xi (1 - \alpha)\frac{|A|}{\xi} \geq (1 - \alpha)|A|$ and, similarly, $|B \cap V(G)| \geq (1 - \alpha)|B|$. Thus, since $G$ is lower $(\rho,\xi)$-regular, $e_{H}(A,B) \geq e_G(A \cap V(G),B \cap V(G)) \geq (\rho-\xi)|A \cap V(G)| |B \cap V(G)| \geq (\rho-\xi)(1-\alpha)^2 |A||B| \geq (\rho(1-\alpha)^2 - \xi )|A||B| \geq (\rho(1-\alpha)^2 - \frac{\xi}{1 - \alpha} )|A||B|$, as desired.
\end{proof}

\begin{theorem}[Glock, K\"uhn, and Osthus \cite{GKO2016}]\label{thm:optcol}
Let $0 < 1/n_0 \ll \nu , \varepsilon \ll p < 1$, and let $n \geq n_0$.  Let $G$ be an $n$-vertex graph that is lower $(p,\varepsilon)$-regular and satisfies $\Delta(G) - \delta(G) \leq \nu n$. Let $\defi_G (v) := \Delta(G) - d_G(v)$ for any $v \in V(G)$. If $n$ is even and
\begin{equation}\label{eqn:def_ineq}
    \defi_G (w) \leq \sum_{v \in V(G) \setminus \{ w\}} \defi_G (v)
\end{equation}
for some vertex $w \in V(G)$ with $d_G(w) = \delta(G)$, then $\chi'(G) = \Delta(G)$. 
\end{theorem}

Even though the statement of~\cite[Theorem 1.6]{GKO2016} requires $G$ to have no overfull subgraph, it is shown in its proof that it suffices to assume that $G$ satisfies ~\eqref{eqn:def_ineq} (see the remark below~\cite[Theorem 1.6]{GKO2016}). We will use the following corollary of Theorem~\ref{thm:optcol}\COMMENT{which is weaker than~\cite[Conjecture 1.7]{GKO2016} but enough for our purpose}.

\begin{corollary}\label{cor:optcol}
Let $0 < 1/n_0 \ll \eta , \varepsilon \ll p < 1$, and let $n \geq n_0$.  Let $G$ be an $n$-vertex graph that is lower $(p,\varepsilon)$-regular and satisfies $\Delta(G) - \delta(G) \leq \eta n$. If there are at least $\Delta(G)$ vertices in $G$ having degree less than $\Delta(G)$, then $\chi'(G) = \Delta(G)$.
\end{corollary}

\begin{proof}\COMMENT{Since the proof of Corollary~\ref{cor:optcol} is quite straightforward, perhaps we can omit the proof from our journal version.}Note that~\eqref{eqn:def_ineq} holds for some vertex $w \in V(G)$ if and only if
\begin{align}\label{eqn:def_ineq_2}
    2(\Delta(G) - \delta(G)) \leq \sum_{v \in V(G)} (\Delta(G) - d_G(v)).
\end{align}

Let $V_1 \subseteq V(G)$ be a subset such that $\varepsilon n \leq |V_1| \leq 2 \varepsilon n$, and let $V_2 := V(G) \setminus V_1$.
Since $G$ is lower $(p,\varepsilon)$-regular,
there exists $w \in V_1$ such that 
\begin{equation*}
    d_G(w) \geq |N_G(w) \cap V_2| \geq \frac{e_G(V_1 , V_2)}{|V_1|} \geq (p - \varepsilon)|V_2| \geq (p-\varepsilon)(1 - 2 \varepsilon)n \geq (p - 3 \varepsilon)n,
\end{equation*}
and thus, $\Delta(G) \geq (p - 3 \varepsilon)n$. 

Now we prove the corollary. First suppose $n$ is even. Since by the assumption, $G$ has at least $\Delta(G)$ vertices having degree less than $\Delta(G)$, we have $2(\Delta(G) - \delta(G)) \leq 2 \eta n \leq (p-3\varepsilon)n \leq \Delta(G) \leq \sum_{v \in V(G)}(\Delta(G) - d_G(v))$. Thus ~\eqref{eqn:def_ineq_2} holds, which implies that ~\eqref{eqn:def_ineq} holds. So we can apply Theorem~\ref{thm:optcol} to show that $\chi'(G) = \Delta(G)$.

Now suppose $n$ is odd. Let $t := \lfloor \Delta(G) - 2 \eta n \rfloor$. Let $G^*$ be a graph obtained from $G$ by adding a new vertex $v^*$ adjacent to exactly $t$ vertices of $G$ having degree less than $\Delta(G)$ in $G$, which is possible by the assumption of the corollary. Then
\begin{enumerate}[(a)]
    \item\label{cond:a1} $\Delta(G^*) = \Delta(G)$,
    \item\label{cond:a2} $\delta(G^*) = d_{G^*}(v^*) = t$,
    \item\label{cond:a3} $|\{ v \in V(G) : \Delta(G^*) > d_{G^*}(v) \}| \geq \Delta(G) - d_{G^*}(v^*)$, and
    \item\label{cond:a4} $G^*$ is lower\COMMENT{set $\alpha = 1/2$, $\rho = p$ and $\xi = \eps$} $(p/4,2\eps)$-regular by~\ref{LR2}.\end{enumerate}

This implies that
\begin{align*}
    \sum_{v \in V(G^*)} (\Delta(G^*) - d_{G^*}(v)) &\overset{\ref{cond:a2}}{=} \sum_{v \in V(G)} (\Delta(G^*) - d_{G^*}(v)) + (\Delta(G^*) - \delta(G^*)) \overset{\ref{cond:a1},\ref{cond:a2},\ref{cond:a3}}{\geq} 2(\Delta(G^*) - \delta(G^*)).
\end{align*}

Thus, $G^*$ satisfies~\eqref{eqn:def_ineq_2}, so it satisfies ~\eqref{eqn:def_ineq}. Moreover (d) holds, and $\Delta(G^*) - \delta(G^*) \leq 3 \eta v(G^*)$ by (a) and (b), so applying Theorem~\ref{thm:optcol} with $G^*, p/4, 2\eps , 3\eta$ playing the roles of $G, p, \eps, \nu$, respectively, we deduce that $\Delta(G) \leq \chi'(G) \leq \chi'(G^*) = \Delta(G^*) \overset{\ref{cond:a1}}{=} \Delta(G)$, as desired.
\end{proof}

\section{Constructing reservoirs}\label{reservoir-section}

In this section, we construct a set $R_{\rm res} \subseteq E(G)$ called a \emph{reservoir} with several pseudorandom properties.

Recall that an absorber is defined in Definition~\ref{def:absorber}.
\begin{definition}[Pseudorandom / Regularising reservoirs]\label{def:pseudoreg}
Let $\varepsilon , \rho , \xi \in (0,1)$, let $\cH$ be an $n$-vertex linear hypergraph, let $G \coloneqq \cH^{(2)}$, let $G'$ be the spanning subgraph consisting of the edges of $G$ with at least one vertex in $V_+^{(1 - \eps)}(G)$, and let $\cV$ be a set of subsets of $V(\cH)$.
\begin{itemize}
    \item A subset $R_{\rm res} \subseteq E(G)$ is a \textit{$(\rho, \xi, \eps , \cV)$-pseudorandom reservoir} if
    \begin{enumerate}
        \item[\mylabel{P1}{(\rm P1)}] for each $v \in V(\cH)$, $d_{R_{\rm res}}(v) = \rho d_G(v) \pm \xi n$, and

        \item[\mylabel{P2}{(\rm P2)}] $R_{\rm res} \cap E(G')$ is a $(\rho, \xi, \xi , \eps)$-absorber for $\cV$.
    \end{enumerate}
        
    \item Suppose $R_{\rm abs}$ is a $(\rho/2,\xi,\xi,\eps)$-absorber for $\cV$. A set $R_{\mathrm{reg}} \subseteq E(G) \setminus R_{\rm abs}$ is a \textit{$(\rho, \xi, \eps , R_{\rm abs} , \cV)$-regularising reservoir} if $R_{\rm res} \coloneqq R_{\mathrm{abs}} \cup R_{\mathrm{reg}}$ satisfies the following.
    \begin{enumerate}
        \item[\mylabel{R1}{(\rm R1)}] For each $v \in V_+^{(1-\eps)}(G)$, $d_{R_{\rm res}}(v) = \rho d_G(v) \pm \xi n$.
        \item[\mylabel{R2}{(\rm R2)}] For each $w\in V(\cH)\setminus V_+^{(1-\eps)}(G)$, $\max(\rho d_G(w),(\rho - 20 \varepsilon) n) \leq d_{R_{\rm res}}(w) \leq \rho(1-\varepsilon)n + \xi n.$
    \end{enumerate}
\end{itemize}
\end{definition}

Now we define various types of reservoirs. The type of reservoir we choose to use will depend on the structure of the hypergraph $\cH$.

\begin{definition}[Types of reservoirs]\label{def:reservoir_type}
Let $\varepsilon , \rho , \xi \in (0,1)$, let $\cH$ be an $n$-vertex linear hypergraph, let $G \coloneqq \cH^{(2)}$, and let $G'$ be the spanning subgraph consisting of the edges of $G$ that are incident to a vertex in $V_+^{(1-\eps)}(G)$.

For a collection $\cV$ of subsets of $V(\cH)$, and $R_{\rm res} \subseteq E(G)$, we say $R_{\rm res}$ is

\begin{itemize}
    \item a \emph{$(\rho, \xi, \eps,\cV)$-reservoir of Type $\rm A_1$} if $R_{\rm res}$ is a $(\rho, \xi, \eps,\cV)$-pseudorandom reservoir,
    
    \item a \emph{$(\rho, \xi , \eps,\cV)$-reservoir of Type $\rm A_2$} if $R_{\rm res} = R_{\rm abs} \cup R_{\rm reg}$, where 
    \begin{itemize}
        \item[$\circ$] $R_{\rm abs}$ is a $(\rho/2, \xi , \xi,\eps)$-absorber for $\cV$ that is also lower $(\rho/2, \xi ,  G')$-regular, and
        
        \item[$\circ$] $R_{\rm reg}$ is a $(\rho,\xi,\eps,R_{\rm abs} , \cV)$-regularising reservoir, and
    \end{itemize}
    
    \item a \emph{$(\rho, \xi, \eps,\cV)$-reservoir of Type $\rm B$} if $R_{\rm res}$ is a $(\rho, \xi, \xi , \eps)$-absorber for $\cV$.
\end{itemize}

For brevity, we often omit the type if it is clear from the context.
\end{definition}

We will use reservoirs of Type~${\rm A_1}$ when $\cH$ is neither $(\rho,\eps)$-full nor FPP-extremal (which are defined in Definitions~\ref{def:rho-eps-full} and~\ref{def:edgesize}), reservoirs of Type~${\rm A_2}$ when $\cH$ is $(\rho,\eps)$-full but not FPP-extremal, and reservoirs of Type~${\rm B}$ when $\cH$ is FPP-extremal.

Now we show the existence of a suitable absorber, a pseudorandom reservoir, and a regularising reservoir.

\begin{proposition}[The existence of a pseudorandom reservoir and an absorber]\label{prop:pseudo_reservoir}
Let $0 < 1 / n_0 \ll \xi , \varepsilon , \rho < 1$, and let $n \geq n_0$. Let $\cH$ be an $n$-vertex linear hypergraph, let $G \coloneqq \cH^{(2)}$, and let $G'$ be the spanning subgraph of $G$ consisting of the edges of $G$ incident to a vertex in $V_+^{(1-\eps)}(G)$. If $\cV$ is a collection of subsets of $V(\cH)$ such that $|\cV| \leq n^{\log n}$, then there exists $R_{\rm rnd} \subseteq E(G)$ such that 
\begin{itemize}
    \item $R_{\rm rnd} \cap E(G')$ is a $(\rho , \xi , \xi , \eps)$-absorber for $\cV$, and lower $(\rho,\xi,G')$-regular, and
    
    \item $R_{\rm rnd}$ is a $(\rho, \xi , \eps , \cV)$-pseudorandom reservoir. 
\end{itemize}

In particular, $R_{\rm rnd}$ is a $(\rho, \xi, \eps,\cV)$-reservoir of Type~${\rm A_1}$ and $R_{\mathrm{rnd}} \cap E(G')$ is a $(\rho, \xi, \eps,\cV)$-reservoir of Type~${\rm B}$.
\end{proposition}

To prove Proposition~\ref{prop:pseudo_reservoir}, it suffices to consider a set $R_{\rm rnd} \subseteq E(G)$ of edges chosen independently and uniformly at random with probability $\rho$. By  Chernoff's inequality (Theorem~\ref{thm:chernoff} with all weights equal to 1), with high probability, both $R_{\rm rnd}$ and $R_{\rm rnd}' = R_{\rm rnd} \cap E(G')$ satisfy all desired conditions.
\begin{proof}
Let $R_{\rm rnd} \subseteq E(G)$ be obtained by choosing each edge $e \in E(G)$ with probability $\rho$ independently at random, and let $R_{\rm rnd}' \coloneqq R_{\rm rnd} \cap E(G')$. For any $S \in \cV$ and $w \in V(G)$, we have $\mathbb{E}[d_{R_{\rm rnd}}(w)] = \rho d_G(w)$, $\mathbb{E}[|N_{R_{\rm rnd}'}(w) \cap S|] = \rho|N_{G'}(w) \cap S|$, and $\mathbb{E}[|N_{R_{\rm rnd}'}(w) \setminus S|] = \rho|N_{G'}(w) \setminus S|$. By Theorem~\ref{thm:chernoff}, with probability at least $1 - \exp({-n^{1/2}})$, we have $d_{R_{\rm rnd}}(w) = \rho d_G(w) \pm \xi n$,
\begin{align*}
    |N_{R_{\rm rnd}'}(w) \cap S| = \rho |N_{G'}(w) \cap S| \pm \xi n, \text{ and } |N_{R_{\rm rnd}'}(w) \setminus S| = \rho|N_{G'}(w) \setminus S| \pm \xi n.
\end{align*}

Since there are at most $|\cV| \leq n^{\log n}$ choices of $S \in \cV$ and $n$ choices of $w \in V(G)$, with probability at least $1 - n |\cV| \exp(-n^{1/2}) \geq 1 - \exp(-n^{1/3})$, $d_{R_{\rm rnd}}(w) = \rho d_G(w) \pm \xi n$, thus $R_{\rm rnd}$ satisfies~\ref{P1}, 
and $R_{\rm rnd}'$ is $(\rho , \xi , G')$-typical with respect to $\cV$.
Similarly, for any two disjoint sets $A,B \subseteq V(G)$ with $|A|,|B| \geq \xi n$, $\mathbb{E}[e_{R'_{\rm rnd}}(A,B)] = \rho e_{G'}(A,B)$. Applying Theorem~\ref{thm:chernoff} again, with probability at least $1 - \exp({-n^{3/2}})$, $e_{R_{\rm rnd}'}(A,B) = \rho e_{G'}(A,B) \pm \xi |A||B|$. Since there are at most $4^n$ choices of disjoint $A,B \subseteq V(G)$, with probability at least $1 - 4^n \exp(-n^{3/2}) \geq 1 - \exp(-n^{5/4})$, $R_{\rm rnd}'$ is both lower $(\rho , \xi , G')$-regular and upper $(\rho , \xi , G')$-regular, so $R_{\rm rnd}'$ is a $(\rho , \xi , \xi , \eps)$-absorber for $\cV$ that is also lower $(\rho,\xi,G')$-regular, and thus $R_{\rm rnd}$ satisfies~\ref{P2}.
Hence, there exists a choice of $R_{\rm rnd}$ such that $R_{\rm rnd}$ is a $(\rho, \xi , \eps , \cV)$-pseudorandom reservoir, and $R_{\rm rnd}' = R_{\rm rnd} \cap E(G')$ is a $(\rho , \xi , \xi , \eps)$-absorber for $\cV$ that is also lower $(\rho,\xi,G')$-regular, as desired.
\end{proof}

\begin{lemma}[The existence of a regularising reservoir]\label{lem:reg_reservoir}
Let $0 < 1/n_0 \ll \xi \ll \varepsilon \ll \rho \ll 1$, and let $n \geq n_0$.  Let $\cH$ be an $n$-vertex linear hypergraph, and let $G \coloneqq \cH^{(2)}$. If $\cV$ is a collection of subsets of $V(\cH)$ such that $|\cV| \leq n^{\log n}$,   $V(\cH) \in \cV$, and $\cH$ is $(\rho,\varepsilon)$-full, then for any $(\rho/2 , \xi , \xi ,  \varepsilon)$-absorber $R_{\rm abs}$ for $\cV$, there exists a $(\rho , \xi , \varepsilon , R_{\rm abs} , \cV)$-regularising reservoir $R_{\rm reg} \subseteq E(G) \setminus R_{\rm abs}$.
\end{lemma}
\begin{proof}
Let $U \coloneqq V_+^{(1-\eps)}(G)$ and let $G'$ be the spanning subgraph of $G$ consisting of the edges of $G$ incident to a vertex of $U$. Let $R_{\rm abs}$ be a $(\rho/2 , \xi , \xi , \eps)$-absorber for $\cV$. Since $V(\cH) \in \cV$ and $R_{\rm abs}$ is $(\rho/2 , \xi, G')$-typical with respect to $\cV$,
\begin{equation}\label{cond:deg_abs} 
    \text{for any $v \in V(G)$, $d_{R_{\rm abs}}(v) = \rho d_{G'}(v)/2 \pm \xi n$.} 
\end{equation}
    
Let 
\begin{equation}\label{eq:defU'}
    U' \coloneqq \{w \in V(\cH) \setminus U \: : \: d_{G}(w) \geq (1 - 20 \varepsilon \rho^{-1})n \}.   
\end{equation}

\COMMENT{The main reason why we separate $U'$ from $U$ is the following. We should reserve at least $\rho d_G(w) - \xi n$ edges for each vertex $w \in V(\cH)$, to satisfy~\eqref{eqn:bound_deg_R_3} in the proof of Lemma~\ref{lem:degree-lemma}. However, the method below only utilises the edges incident to vertices of $V^{(n-1)}(G)$; we can only reserve at most $(\rho - O(\eps))n$ many edges at each non-$U$ vertex, since there could be at most $(\rho - O(\eps))n$ vertices in $V^{(n-1)}(G)$ by the definition of $(\rho,\eps)$-fullness.  However, $(\rho - O(\eps))n$ could be smaller than $\rho d_G(w) - \xi n$ for some vertices $w \in U'$, so we should reserve additional edges via $(g,f)$-factor theorem as well. For convenience, instead of first reserving the edges incident to $V^{(n-1)}(G)$ for the vertices of $U'$ and then trying to augment the degrees, we directly apply $(g,f)$-factor theorem to the vertices of $U' \cup U$.}
Since $\cH$ is $(\rho,\varepsilon)$-full, we can choose a subset $S \subseteq V^{(n-1)}(G)$ with $|S| = \lceil (\rho - 20\eps)n \rceil$. Note that every vertex of $S$ is adjacent to all the other vertices of $G$.

For each vertex $w \in V(\cH) \setminus (U \cup U')$, we choose $\lceil (\rho - 20\eps)n - d_{R_{\rm abs}}(w) \rceil > 0$ edges of $E_{G'} (S,\{w\}) \setminus R_{\rm abs}$\COMMENT{This is possible by the following reason: By \eqref{cond:deg_abs}, $d_{R_{\rm abs}}(w) \leq (\rho/2 + \xi)n < (\rho - 20\eps)n$, and the vertices in $S \subseteq V^{(n-1)}(G)$ are adjacent to all other vertices in $G$. In particular, all the vertices of $S$ are adjacent to $w$.}, and let $R' \subseteq E(G') \setminus R_{\rm abs}$ be the union of all such edges for all $w \in V(\cH) \setminus (U \cup U')$. Then for any $w \in V(\cH) \setminus (U \cup U')$, we have
\begin{equation}\label{eqn:smalldeg_res}
    \rho d_G(w) \overset{\eqref{eq:defU'}}{\leq} (\rho-20\eps)n \leq d_{R_{\rm abs}}(w) + d_{R'}(w) \leq (\rho-20\eps)n+1 \leq \rho(1-\varepsilon)n.
\end{equation}

For each vertex $w \in U \cup U'$, let us define
\begin{align}\label{eqn:def_fg}
    f(w) \coloneqq \lfloor \rho d_G(w) + \xi n - d_{R_{\rm abs}}(w) - d_{R'}(w)\rfloor \text{    and    } g(w) \coloneqq f(w) - 1.
\end{align}

\begin{claim}
There exists a $(g,f)$-factor $R''$ in $H \coloneqq G[U \cup U'] - R_{\rm abs} = G[U \cup U'] - R_{\rm abs} - R'$.
\end{claim}
\begin{claimproof}
Since $\varepsilon \ll \rho$, for any $w \in U \cup U'$, we have $d_G(w) \geq (1-\rho)n$. Moreover, for any $w \in U \cup U'$, since $\cH$ is $(\rho, \eps)$-full and $R' \subseteq E_G(V(\cH)\setminus (U \cup U') , U)$, $d_{R'}(w) \leq |V(\cH) \setminus U| \leq 10 \varepsilon n \leq \rho n/10$, and by \eqref{cond:deg_abs}, $d_{R_{\rm abs}}(w) \le \rho n/2 + \xi n$. Hence, for any $w \in U \cup U'$ we have 
\begin{align*}
    \frac{3\rho n}{10} \leq \rho(1-\rho)n + \xi n - \left(\frac{\rho n}{2} + \xi n\right) - \frac{\rho n}{10}  - 1 \leq f(w) \leq \rho n + \xi n < \frac{3\rho n}{2}.
\end{align*}

Therefore, for any $w \in U \cup U'$,
\begin{equation}
\label{eq:fg-range}
f(w),g(w) \in [\rho n / 4 , 2 \rho n]. 
\end{equation}

Moreover, for any $w \in U \cup U'$, we have
\begin{align}
    d_{G - R_{\rm abs}}(w) \geq (1-\rho)n - \left ( \frac{\rho n}{2} + \xi n \right )  \geq (1-2\rho)n. \label{eqn:bound_deg_h}
\end{align}

By Theorem~\ref{thm:gf_factor}, there exists a $(g,f)$-factor in $H$ if 
\begin{itemize}
    \item[(i)] $0 \leq g(w) < f(w) \leq d_H(w)$ for each $w \in U \cup U'$, and
    
    \item[(ii)] for any pair of disjoint sets $S,T \subseteq V(H)$, $c(S,T) \geq 0$, where 
        \begin{align*}
            c(S,T) \coloneqq \sum_{t \in T}(d_{H}(t) - g(t)) + \sum_{s \in S}f(s) - e_{H}(S,T) = \sum_{t \in T}(d_{H-S}(t) - g(t)) + \sum_{s \in S}f(s).
        \end{align*}
\end{itemize}

Since $\cH$ is $(\rho,\eps)$-full and $\eps \ll \rho$, note that
\begin{equation}\label{eqn:bound_rem}
    |V(\cH) \setminus (U \cup U')| \leq |V(\cH) \setminus U| \leq 10\eps n \leq \rho n.
\end{equation}

Hence, for any $w \in U \cup U'$, we have 
\begin{equation*}
    d_H(w) \geq d_{G - R_{\rm abs}}(w) -  |V(G) \setminus (U \cup U')| \overset{\eqref{eqn:bound_deg_h}, \eqref{eqn:bound_rem}}{\geq} (1-2\rho)n - \rho n = (1-3 \rho)n \overset{\eqref{eq:fg-range}}{\geq} f(w),
\end{equation*}
so (i) holds. Now, we verify (ii). If $|S| \leq (1 - 5\rho)n$, then for any $t \in T$, we have
\begin{equation*}
    d_{H-S}(t) \geq d_{G-R_{\rm abs}}(t) - |V(G) \setminus (U \cup U')| - |S| \overset{\eqref{eqn:bound_deg_h},\eqref{eqn:bound_rem}}{\geq} (1-2\rho)n - \rho n - |S| \geq 2 \rho n \overset{\eqref{eq:fg-range}}{\geq} g(t),
\end{equation*}
so $c(S,T) \geq 0$. Hence we may assume that $|S| > (1-5\rho)n$, which implies that $|T| < 5 \rho n$. Then
\begin{align*}
    \sum_{t \in T}g(t) &\overset{\eqref{eq:fg-range}}{\leq} |T|\cdot 2 \rho n \leq 10 \rho^2 n^2 \leq (1-5\rho)n \cdot \frac{\rho n}{4} < |S| \cdot \frac{\rho n}{4} \overset{\eqref{eq:fg-range}}{\leq} \sum_{s \in S} f(s),
\end{align*}
so $c(S,T) \geq -\sum_{t \in T}g(t) + \sum_{s \in S}f(s) \geq 0$, proving (ii) and thus the claim.
\end{claimproof}

Finally, we show that $R_{\rm reg} \coloneqq R' \cup R''$ is a $(\rho , \xi , \varepsilon , R_{\rm abs} , \cV)$-regularising reservoir. Let $R_{\rm res} \coloneqq R_{\rm abs} \cup R' \cup R''$. For each $w \in U \cup U'$, by~\eqref{eqn:def_fg},
\begin{align*}
    d_{R_{\rm res}}(w) &\geq (f(w) - 1) + d_{R_{\rm abs}}(w) + d_{R'}(w) \geq \rho d_G(w) + \xi n - 2, \\
    d_{R_{\rm res}}(w) &\leq f(w) + d_{R_{\rm abs}}(w) + d_{R'}(w) \leq \rho d_G(w) + \xi n.
\end{align*}
Thus,~\ref{R1} holds, and for $w \in U'$, $(\rho - 20\varepsilon)n \overset{\eqref{eq:defU'}}{\leq} \rho d_G(w) \leq  d_{R_{\rm res}}(w) \leq \rho(1-\varepsilon)n + \xi n$, as required by~\ref{R2}. Moreover, for any $w \in V(\cH) \setminus (U \cup U')$, $d_{R_{\rm res}}(w) = d_{R_{\rm abs}}(w) + d_{R'}(w)$ since $R'' \subseteq E(H) \subseteq E(G[U \cup U'])$. Therefore, by \eqref{eqn:smalldeg_res}, $\max(\rho d_G(w) , (\rho-20\eps)n) \leq d_{R_{\rm res}}(w) \leq \rho(1-\eps)n$, showing that \ref{R2} holds for $w \in V(\cH) \setminus (U \cup U')$. Thus,~\ref{R2} holds for all $w \in V(\cH)\setminus U$. This completes the proof.
\end{proof}

\begin{definition}[Regularised linear multi-hypergraph]\label{def:hreg}
For an $n$-vertex linear hypergraph $\cH$, let $\cH_{\rm reg}$ be the linear multi-hypergraph obtained from $\cH$ by adding\COMMENT{We cannot have $n-2-d_{\cH}(w)$ or a larger number here instead of $n-3-d_\cH(w)$ because in order to be able to apply Corollary~\ref{cor:optcol}, we have to ensure that no vertices of degree at most $n-3$ in $\cH$ will attain maximum degree in the GKO case. (Recall that a vertex of degree at most $n-3$ in $\cH$ may be marked twice.)} $\max(0 , n-3-d_{\cH}(w))$ singleton edges incident to each $w \in V(\cH)$.
\end{definition}

Recall that $\cH_{\rm small} \coloneqq \{ e \in \cH \: : \: |e| \leq r_1 \}$. In order to be able to use Lemma~\ref{lem:nibble2}, we need to embed $\cH_{\rm small} \setminus R_{\rm res}$\COMMENT{Since the reservoir $R_{\rm res}$ contains an absorber $R$ as a subset, it is clear that $\cH_{\rm small} \setminus R_{\rm res} \subseteq \cH_{\rm small} \setminus R$.} into an almost-regular linear multi-hypergraph $\cH'$ by adding singleton edges. In particular, for each vertex $w \in V(\cH)$, we add at most $\max(0 , n - 3 - d_{\cH}(w))$ singleton edges containing $w$, so that $\cH' \subseteq \cH_{\rm reg}$.

\begin{lemma}[Regularising lemma]\label{lem:degree-lemma}
Let $0 < 1/n_0 \ll \xi , 1/r_1 \ll \beta , \varepsilon , \rho \ll 1$, and let $n \geq n_0$.  Let $\cH$ be an $n$-vertex linear hypergraph, and let $\cV$ be a collection of subsets in $V(\cH)$ such that $V(\cH) \in \cV$. If either 
\begin{enumerate}[(i)]
    \item $R_{\rm res}$ is a $(\rho,\xi,\eps , \cV)$-reservoir of Type $\rm A_1$ or $\rm A_2$, or
    \item $R_{\rm res}$ is a $(\rho,\xi,\eps , \cV)$-reservoir of Type $\rm B$ and $3 \rho \leq \varepsilon$, 
\end{enumerate}
then there exists a linear multi-hypergraph $\cH' \subseteq \cH_{\rm reg}$ such that

\begin{itemize}
    \item $\cH'$ is obtained from $\cH_{\mathrm{small}} \setminus R_{\rm res}$ by adding singleton edges, and
    \item for every $w \in V(\cH)$, we have $d_{\cH'}(w) =  (1-\rho) (n-1 \pm \beta n)$.
\end{itemize}
\end{lemma}
\begin{proof}
Let $G\coloneqq \cH^{(2)}$, and let $U \coloneqq V_+^{(1 - \varepsilon)}(G)$.  Since $\cH$ is linear and every $w\in V(\cH)$ is contained in at most one singleton,\COMMENT{
since $1/2 \le 1 - \rho$.}

\begin{equation}\label{eqn:non-graph-degree}
    d_{\cH\setminus E(G)}(w) \leq \frac{n-1-d_G(w)}{2} + 1 \leq (1 - \rho)(n-1-d_G(w)) + 1.
\end{equation}

We will show that for any $w \in V(\cH)$,
\begin{equation}\label{eqn:estimate_hsmall}
d_{\cH  \setminus R_{\rm res}}(w) \leq (1-\rho) (n-1) + \xi n + 1.
\end{equation}

Let us first consider the case when $R_{\rm res}$ is a $(\rho , \xi , \eps , \cV)$-reservoir of Type~${\rm A_1}$ or~${\rm A_2}$. In this case, for any $w \in V(\cH)$, we have
\begin{equation}\label{eqn:bound_deg_R_3}
    \rho d_G(w) - \xi n \leq d_{R_{\rm res}}(w) \leq (\rho + \xi)n.
\end{equation}

Indeed, if $R_{\rm res}$ is of Type $\rm A_1$, then~\eqref{eqn:bound_deg_R_3} holds by~\ref{P1}, and if $R_{\rm res}$ is of Type $\rm A_2$, then~\eqref{eqn:bound_deg_R_3} holds by~\ref{R1} and~\ref{R2}. Now, by \eqref{eqn:non-graph-degree} and~\eqref{eqn:bound_deg_R_3}, every $w \in V(\cH)$ satisfies
\begin{align}
    d_{\cH \setminus R_{\rm res}}(w) = d_{G-R_{\rm res}}(w) + d_{\cH  \setminus E(G)}(w) \overset{\eqref{eqn:bound_deg_R_3}}&{\leq} (1-\rho)d_G(w) + \xi n + d_{\cH  \setminus E(G)}(w)  \nonumber \\
    \overset{\eqref{eqn:non-graph-degree}}&\leq (1-\rho)(n-1) + \xi n + 1,\label{eqn:hsmall_r_deg}
\end{align}
proving~\eqref{eqn:estimate_hsmall} when $R_{\rm res}$ is a $(\rho , \xi , \eps , \cV)$-reservoir of Type~${\rm A_1}$ or~${\rm A_2}$.

Now let us consider the case when $R_{\rm res}$ is a $(\rho , \xi , \eps , \cV)$-reservoir of Type~${\rm B}$. Let $G'$ be the spanning subgraph of $G$ consisting of the edges of $G$ incident to a vertex of $U$. Since $R_{\rm res}$ is a $(\rho,\xi,\xi,\eps)$-absorber for $\cV$ and $V(\cH) \in \cV$, every $w \in V(\cH)$ satisfies
\begin{equation}\label{eqn:bound_deg_R_4}
    d_{R_{\rm res}}(w) = \rho d_{G'}(w) \pm \xi n.
\end{equation}

If $w \in U$, then $d_{G'}(w) = d_G(w)$, so $d_{R_{\rm res}}(w) = \rho d_{G}(w) \pm \xi n$ and by the same reasoning as in~\eqref{eqn:hsmall_r_deg}, one can show that \eqref{eqn:estimate_hsmall} holds if $w \in U$.

It remains to show that \eqref{eqn:estimate_hsmall} holds for $w \in V(\cH) \setminus U$. Indeed, every $w\in V(\cH)\setminus U$ satisfies
\begin{align*}
    d_{\cH \setminus R_{\rm res}}(w) &= d_{G'-R_{\rm res}}(w) + d_{G-E(G')-R_{\rm res}}(w) + d_{\cH  \setminus E(G)}(w)\\ 
    \overset{\eqref{eqn:bound_deg_R_4}}&{\leq} (1-\rho)d_{G'}(w)+ \xi n + d_{G-E(G')}(w) + d_{\cH  \setminus E(G)}(w) \\
    &= (1-\rho)d_{G}(w) + \rho d_{G-E(G')}(w) + d_{\cH  \setminus E(G)}(w) + \xi n\\
    \overset{\eqref{eqn:non-graph-degree}}&\leq (1-\rho)d_{G}(w) + \rho (1-\eps)n + \frac{n-1-d_G(w)}{2} + 1 +\xi n\\
    &= (1/2-\rho)d_{G}(w) + \rho (1-\eps)n + \frac{n-1}{2} + \xi n + 1\\
    &\leq (1-\rho)(n-1) + \xi n + 1,
\end{align*}
as desired. Note that the last inequality is equivalent to $(1/2-\rho)d_G(w) + \rho(1-\eps)n \leq (1/2-\rho)(n-1)$, which holds since\COMMENT{$(1/2-\rho)d_G(w) + \rho(1-\eps)n \leq (1/2-\rho)(1-\eps)n + \rho(1-\eps)n = (1/2)(1-\eps)n$ and this is at most $(1/2-\rho)(n-1)$ provided $2\rho + \frac{1-2\rho}{n} \leq \eps$ which holds since $2\rho + \frac{1-2\rho}{n} \leq 3\rho$ and  $3\rho\leq \eps$ by our assumption.} $d_G(w) \leq (1-\eps)n$ (since $w \notin U$) and $3\rho \leq \eps$.

Now let $k \coloneqq \lfloor (1-\rho)(n-1) - \beta n / 2 \rfloor$, and let $\cH'$ be the linear multi-hypergraph obtained from $ \cH_{\mathrm{small}}  \setminus R_{\rm res}$ as follows. For every vertex $w \in V(\cH)$ satisfying $d_{ \cH_{\mathrm{small}}  \setminus R_{\rm res}}(w) < k$, we add $k - d_{ \cH_{\mathrm{small}}  \setminus R_{\rm res}}(w)$ singleton edges containing $w$.
Thus, by~\eqref{eqn:estimate_hsmall}, $k \leq \delta(\cH') \leq \Delta(\cH') \leq (1-\rho)(n-1) + \xi n + 1$. 
Since $\xi \ll \beta$, this implies that for every $w \in V(\cH)$,
\begin{equation}\label{eqn:degh*}
   d_{\cH'}(w) = (1-\rho)(n-1) \pm 2\beta n / 3 = (1-\rho) (n-1 \pm \beta n),
\end{equation}
as desired. Now we prove that $\cH' \subseteq \cH_{\rm reg}$ by showing that $\cH'$ is obtained from $\cH_{\rm small} \setminus R_{\rm res}$ by adding at most $\max(0 , n-3-d_{\cH}(w))$ singleton edges incident to each vertex $w \in V(\cH)$. Indeed, for any vertex $w \in V(\cH)$ with $d_{\cH_{\mathrm{small}}  \setminus R_{\rm res}}(w) < k$, we add at most
\begin{align*}
    k - d_{ \cH_{\mathrm{small}} \setminus R_{\rm res} }(w) &= k - d_{\cH_{\rm small}}(w) + d_{R_{\rm res}}(w) \\
    \overset{\eqref{eqn:bound_deg_R_3},\eqref{eqn:bound_deg_R_4}}&{\leq} (1-\rho)n - \frac{\beta n}{2} - \left ( d_{\cH}(w) - \frac{2n}{r_1} \right ) + (\rho + \xi)n \leq n - 3 - d_{\cH}(w)
\end{align*}
singleton edges incident to $w$, since $d_\cH (w) - d_{\cH_{\rm small}}(w) \leq 2n/r_1$ and $\xi, 1/r_1 \ll \beta$. This completes the proof of the lemma.
\end{proof}

\section{Proof of Theorem~\ref{main-thm}}\label{proof-section}

Now we are ready to prove our main theorem. As discussed in Section~\ref{sec:overview}, the proof depends on the structure of $\cH$. The relevant properties of $\cH$ are captured by the following definition. (Recall that $(\rho, \eps)$-full linear hypergraphs were introduced in Definition~\ref{def:rho-eps-full}.)

\begin{definition}[Types of hypergraphs and colourings] 
\label{Def:typeofhypandcolor}
Let $\cH$ be a linear $n$-vertex hypergraph, and let $\phi : \cH_{\rm med} \cup \cH_{\rm large} \to [n]$ be a proper edge-colouring.
We say $\phi$ is of Type~${\rm A}$ if it satisfies~\ref{non-extremal-colouring} of Theorem~\ref{large-edge-thm}, and $\phi$ is of Type~${\rm B}$ if it satisfies~\ref{extremal-colouring} of Theorem~\ref{large-edge-thm}. We say $(\cH , \phi)$ is of
\begin{itemize}
    \item Type~${\rm A_1}$ if $\phi$ is of Type A, and $\cH$ is not $(\rho,\eps)$-full,
    \item Type~${\rm A_2}$ if $\phi$ is of Type A, and $\cH$ is $(\rho,\eps)$-full,
    \item Type~${\rm B}$ if $\phi$ is of Type B.
\end{itemize}
\end{definition}
Let us note that if $\cH$ is close to a projective plane, then $(\cH, \phi)$ is of Type~${\rm B}$, and otherwise $(\cH, \phi)$ is of Type~${\rm A}$.  In this case, if $\cH$ is close to a complete graph, then $(\cH, \phi)$ is of Type~${\rm A_2}$, and otherwise $(\cH, \phi)$ is of Type~${\rm A_1}$.  If $(\cH, \phi)$ is of Type~${\rm A_1}$ and is close to a degenerate plane, then it is coloured by Lemma~\ref{difficult-edge-absorption}.

\begin{proof}[Proof of Theorem~\ref{main-thm}]
Recall the hierarchy of the parameters
$$0 < 1/ n_0 \ll 1/r_0 \ll \xi \ll 1/r_1 \ll \beta \ll \kappa \ll \gamma_1 \ll \eps_1 \ll \rho_1 \ll \sigma \ll \delta \ll \gamma_2 \ll \rho_2 \ll \eps_2 \ll 1,$$
where $r_0$ and $r_1$ are integers.  Let $n \geq n_0$, and let $\cH$ be an $n$-vertex linear hypergraph. Without loss of generality, we may assume that $\cH$ has no singleton edges. Our aim is to find $n$ pairwise edge-disjoint matchings containing all of the edges of $\cH$.\COMMENT{Indeed, if we prove $\chi'(\cH') \leq n$ holds for the linear hypergraph $\cH'$ obtained from $\cH$ by removing all singleton edges, then any vertex is incident to at most $n-1$ edges in $\cH'$, so for any singleton edge incident to a vertex $w$ in $\cH$, there is an unused colour available that can be assigned to this singleton edge.}

The first step of the proof is to find a colouring $\phi$ of the medium and large edges of $\cH$. 

\begin{pf-step}\label{step:step1}
    Colour large and medium edges, and define the corresponding parameters.
\end{pf-step}
Let $\phi : \cH_{\rm med} \cup \cH_{\rm large} \to [n]$ be the proper edge-colouring given by Theorem~\ref{large-edge-thm}. 

Now we define some parameters depending on the type of $(\cH , \phi)$ as follows.
\begin{itemize}
    \item $\rho \coloneqq \rho_1$ , $\rho_{\rm abs} \coloneqq \rho_1$ , $\eps \coloneqq \eps_1$ , and $\gamma \coloneqq \gamma_1$ if $(\cH , \phi)$ is of Type~${\rm A_1}$,
    
    \item $\rho \coloneqq \rho_1$ , $\rho_{\rm abs} \coloneqq \rho_1 / 2$ , $\eps \coloneqq \eps_1$ , and $\gamma \coloneqq \gamma_1$ if $(\cH , \phi)$ is of Type~${\rm A_2}$, and
    
    \item $\rho \coloneqq \rho_2$ , $\rho_{\rm abs} \coloneqq \rho_2$ , $\eps \coloneqq \eps_2$ , and $\gamma \coloneqq \gamma_2$ if $(\cH , \phi)$ is of Type~${\rm B}$.
\end{itemize}

Thus, we have 
\begin{equation}\label{eqn:hierarchy}
    \kappa \ll \gamma \ll \rho, \rho_{\rm abs}, \eps \ll 1.
\end{equation}

Let $G \coloneqq \cH^{(2)}$, let $U \coloneqq V_+^{(1-\eps)}(G)$, and let $G'$ be the spanning subgraph of $G$ consisting of the set of edges of $G$ incident to a vertex of $U$.

Recall that difficult matchings were defined in Definition~\ref{defn:difficult}, and as discussed in Section~\ref{subsection:combining-overview}, difficult matchings arise when $\cH$ is close to the degenerate plane.
The following claim is used in the later steps. 
\begin{claim}\label{claim:easyobs}
The following hold.
\begin{enumerate}
    \item[\mylabel{EP1}{(\ref{claim:easyobs}.1)}] For any edge $e \in \cH$ such that $e \cap U \ne \varnothing$, we have $|e| \leq \eps n$.
    
    \item[\mylabel{EP2}{(\ref{claim:easyobs}.2)}] For any edge $e \in \cH_{\rm small}$, we have $|\{f \in \cH_{\rm large} : e \cap f \ne \varnothing \}| \leq 2r_1 n/r_0$.
    
    \item[\mylabel{EP3}{(\ref{claim:easyobs}.3)}] If $\phi$ is of Type B, then $|U| \leq 2 \delta n$, and there is no difficult colour class in $\phi$.
    
    \item[\mylabel{EP4}{(\ref{claim:easyobs}.4)}] If $\phi$ is of Type~${\rm A}$, then there is at most one colour $c \in [n]$ such that $\phi^{-1}(c)$ contains a huge edge and $\phi^{-1}(c)$ is difficult. Moreover, if such a colour $c$ exists, then $\phi^{-1}(c) = \{e\}$ for some huge edge $e$.
\end{enumerate}
\end{claim}
\begin{claimproof}
Let us first prove~\ref{EP1}. For any edge $e \in \cH \setminus E(G)$ containing a vertex $w \in U$, since $\cH$ is linear, the vertex $w$ is not adjacent (in $G$) to any vertex of $e$.  Thus, $(1-\eps)n \leq d_G(w)  \leq n - |e|$, implying that $|e| \leq \eps n$, as desired.

Now we prove~\ref{EP2}. Since $\cH$ is linear, every vertex $w \in V(\cH)$ is incident to at most $n/(r_0-1)$ edges of $\cH_{\rm large}$.  Thus, 
$|\{f \in \cH_{\rm large} : e \cap f \ne \varnothing \}| \leq {|e| n} / (r_0-1) \leq 2r_1n/r_0$.

Now we show~\ref{EP3}. Since there is a set of FPP-extremal edges in $\cH$ with volume at least $1-\delta$, the set $E(G)$ has volume at most $\delta$. 
Thus, $|U| (1-\eps)n/2 \leq \sum_{w \in U} d_G(w) / 2 \leq |E(G)| \leq \delta \binom{n}{2}$, implying $|U| \leq 2 \delta n$. 
If $\phi^{-1}(c)$ is difficult, then $|V(\phi^{-1}(c))| \geq 3|V(\cH) \setminus U|/4 \geq n/2$, which is impossible since each colour class covers at most $\delta n$ vertices by~\ref{extremal-colouring}\ref{extremal-huge-edge},~\ref{extremal-colouring}\ref{extremal-medium-bound}, and~\ref{extremal-colouring}\ref{extremal-colour-class-bound} of Theorem~\ref{large-edge-thm}.

Finally, we prove~\ref{EP4}.
Suppose for some $c_1 \ne c_2$, that $\phi^{-1}(c_1)$ and $\phi^{-1}(c_2)$ are difficult matchings containing huge edges $e_1$ and $e_2$, respectively.
By~\ref{non-extremal-colouring}\ref{non-extremal-huge-edge} of Theorem~\ref{large-edge-thm}, every colour class of $\phi$ containing a huge edge consists of a unique edge, so $\phi^{-1}(c_1) = \{e_1 \}$ and $\phi^{-1}(c_2) = \{ e_2 \}$. 
Moreover, by the definition of a difficult matching, $|V(\cH) \setminus U| \geq 2$ and for $i \in \{1,2\}$, $e_i \setminus U$ has at least $\lceil 3|V(\cH) \setminus U|/4 \rceil$ vertices in $V(\cH) \setminus U$.
If $|V(\cH) \setminus U| = 2$, then both $e_1$ and $e_2$ contain $V(\cH) \setminus U$ since $\lceil 3|V(\cH) \setminus U|/4 \rceil = 2$, contradicting the linearity of $\cH$. Otherwise, if $|V(\cH) \setminus U| \geq 3$, then for $i \in \{1,2 \}$, we have $|e_i \setminus U| \geq 3|V(\cH) \setminus U|/4$, so $|(e_1 \cap e_2) \setminus U| \geq |V(\cH) \setminus U|/2 > 1$, also contradicting the linearity of $\cH$.
\end{claimproof}

\begin{pf-step}
\label{step:step2}
    Choose a reservoir $R_{\rm res}$ and a defect-set $S$.
\end{pf-step}
Let us define 
\begin{equation}\label{def:cv}
    \cV \coloneqq \{U , V(\cH) \} \cup \bigcup_{i=1}^{n} \{ U \cup V(\phi^{-1}(i)) , U \setminus V(\phi^{-1}(i)) \}.
\end{equation}

If $(\cH , \phi)$ is of Type ${\rm A_1}$ or Type ${\rm B}$, then since $|\cV| \leq 2n+2$, by Proposition~\ref{prop:pseudo_reservoir}, there exists a $(\rho , \xi , \eps , \cV)$-reservoir $R_{\rm res} \subseteq E(G)$ of Type ${\rm A_1}$ or Type ${\rm B}$, respectively, which also contains a $(\rho , \xi , \xi , \eps)$-absorber $R_{\rm abs} \subseteq E(G')$ for $\cV$ as a subset as described in Definitions~\ref{def:pseudoreg} and~\ref{def:reservoir_type}.
Otherwise, if $(\cH , \phi)$ is of Type ${\rm A_2}$, then since $|\cV| \leq 2n+2$, by Proposition~\ref{prop:pseudo_reservoir}, there exists a $(\rho/2 , \xi , \xi , \eps)$-absorber $R_{\rm abs} \subseteq E(G')$ for $\cV$ which is lower $(\rho/2 , \xi , G')$-regular. Since $\cH$ is $(\rho,\eps)$-full, applying Lemma~\ref{lem:reg_reservoir}, we obtain a $(\rho , \xi , \eps , R_{\rm abs} , \cV)$-regularising reservoir $R_{\rm reg} \subseteq E(G) \setminus R_{\rm abs}$, thus obtaining a reservoir $R_{\rm res} \coloneqq R_{\rm abs} \cup R_{\rm reg}$ of Type ${\rm A_2}$. To summarise,
\begin{enumerate}
    \item[\mylabel{RES1}{\bf RES1}] $R_{\rm res}$ is a $(\rho , \xi , \eps , \cV)$-reservoir of Type $i$ if $(\cH , \phi)$ is of Type $i$, for $i \in \{ {\rm A_1},{\rm A_2},{\rm B} \}$, and
    \item[\mylabel{RES2}{\bf RES2}] $R_{\rm res}$ contains a $(\rho_{\rm abs} , \xi , \xi , \eps)$-absorber $R_{\rm abs}$ for $\cV$. Moreover, if $(\cH,\phi)$ is of Type~${\rm A_2}$, then $R_{\rm abs}$ is lower $(\rho/2 , \xi , G')$-regular, and if $(\cH,\phi)$ is of Type~${\rm B}$ then $R_{\rm res} = R_{\rm abs}$.
\end{enumerate}

Now let us define the `defect' set $S$ by
\begin{align}\label{def:s}
S \coloneqq
  \begin{cases}
    U \setminus V^{(n-1)}(\cH) &\text{if }(\cH,\phi) \text{ is of Type~${\rm A_1}$.}\\
    U &\text{if }(\cH,\phi) \text{ is of Type~${\rm A_2}$ or~${\rm B}$.}
  \end{cases}
\end{align}

Suppose $|U| > (1 - 10\eps)n$. 
If $\phi$ is of Type ${\rm A_1}$, then $\cH$ is not $(\rho,\eps)$-full, so $|V^{(n-1)}(\cH)| \leq (\rho - 15\eps)n$ and $|S| = |U| - |V^{(n-1)}(\cH)| \geq (1-\rho)n + 5 \eps n$. If $\phi$ is of Type ${\rm A_2}$, then $S=U$, so $|S| = |U| \geq (1 - 10\eps)n \geq (1-\rho)n + 5 \eps n$ since $\eps = \eps_1 \ll \rho_1 = \rho$.
Otherwise if $\phi$ is of Type ${\rm B}$, then by~\ref{EP3} of Claim~\ref{claim:easyobs}, $|U| \leq 2\delta n < (1-10\eps)n$. Thus, we can deduce the following.
\begin{equation}\label{eqn:rs}
    \textrm{If $|U| > (1-10\eps)n$, then $|S| \geq (1-\rho)n + 5 \eps n$.}
\end{equation}

\begin{pf-step}
\label{step:step3}
    Define various subsets of colours.
\end{pf-step}

In this step, we will define sets of colours $C_{\rm med}$, $C_{\rm diff}$, $C_{\rm huge}$, $C_{\rm main}$, $C_{\rm buff}$, $C_{\rm final}$ which partition $[n]$ (i.e., $[n] = C_{\rm med} \cup C_{\rm diff}  \cup C_{\rm huge} \cup C_{\rm main} \cup C_{\rm buff} \cup C_{\rm final}$) and a set $C_{\rm large} \subseteq [n]$ (see Figure~\ref{fig:colourset} in Section~\ref{subsection:combining-overview}).

In the following steps, roughly, our goal is to extend the colour classes $\phi^{-1}(c)$ for $c \in [n] \setminus C_{\rm final}$ in such a way that the maximum degree in the hypergraph of remaining uncoloured edges is at most $|C_{\rm final}|$.
To that end, first, for each $c \in C_{\rm diff}$, we will extend the colour class $\phi^{-1}(c)$ to cover every vertex of $V^{(n-1)}(\cH)$ and all but at most five vertices of $V^{(n-2)}(\cH)$ using some edges of $G$. Then for each $c \in C_{\rm huge} \cup C_{\rm med}$ we will extend the colour class $\phi^{-1}(c)$ to cover all but at most one vertex of $U$ using some edges of $R_{\rm abs}$.
Finally, using the nibble and absorption strategy described at the end of Section~\ref{subsection:combining-overview}, we will extend the colour classes $\phi^{-1}(c)$ for $c \in C_{\rm main} \cup C_{\rm buff}$ to contain all of the remaining edges in $\cH_{\rm small} \setminus R_{\rm res}$ and further extend the resulting colour classes using some edges of $R_{\rm abs}$.

Now we define the following parameters and sets of colours.
\begin{enumerate}[label=${\bf S{\arabic*}}$, topsep = 6pt]
    \item\label{def:dd'} Let $D \coloneqq \lfloor (1-\rho)(n-1) \rfloor$ and $D' \coloneqq \lfloor 10 \gamma^{1/2} D \rfloor$.
    
    \item\label{def:cmed} Let $C_{\rm med} \subseteq [n]$ be the set of at most $\gamma n$ colours such that $\phi(\cH_{\rm med}) \subseteq C_{\rm med}$ which is guaranteed by~\ref{non-extremal-colouring}\ref{non-extremal-medium-bound} and~\ref{extremal-colouring}\ref{extremal-medium-bound} of Theorem~\ref{large-edge-thm}.
    
    \item\label{def:cdiff} By~\ref{EP3} and~\ref{EP4} of Claim~\ref{claim:easyobs}, there is at most one colour $c_{\rm diff} \in \phi(\cH_{\rm huge}) \setminus C_{\rm med}$ such that $\phi^{-1}(c_{\rm diff})$ is difficult. If such a colour $c_{\rm diff}$ exists, then let $C_{\rm diff} \coloneqq \{c_{\rm diff} \}$. Otherwise, let $C_{\rm diff} \coloneqq \varnothing$.
    
    \item\label{def:chuge} Let $C_{\rm huge} \coloneqq \phi(\cH_{\rm huge}) \setminus (C_{\rm med} \cup C_{\rm diff})$. Note that, since $\cH$ is linear and for every $e \in C_{\rm huge}$, $|e| \geq \beta n/4$, we have
    \begin{equation}\label{eqn:num_huge}
        |C_{\rm huge} \cup C_{\rm diff}| \leq e(\cH_{\rm huge}) \overset{\eqref{eqn:huge-edge-bound}}{\leq} 8 \beta^{-1}.
    \end{equation}
    
    \item\label{def:clarge} Let $C_{\rm large} \coloneqq \phi(\cH_{\rm large}) \setminus (C_{\rm med} \cup C_{\rm diff} \cup C_{\rm huge})$.
    
    \item\label{def:cmain} Let $C_{\rm main} \subseteq [n] \setminus (C_{\rm med} \cup C_{\rm diff} \cup C_{\rm huge})$ be a subset of size $D$ that maximises $|C_{\rm large} \cap C_{\rm main}|$. Note that such a subset exists since $n - |C_{\rm med} \cup C_{\rm diff} \cup C_{\rm huge}| \geq n - \gamma n - 8\beta^{-1} \geq D$ by~\eqref{eqn:hierarchy} and~\eqref{eqn:num_huge}. In particular, if $\phi$ is of Type~${\rm A}$, then $C_{\rm main} \supseteq C_{\rm large}$, since $|C_{\rm large}| \leq (1-\sigma)n < D$ by~\ref{non-extremal-colouring} of Theorem~\ref{large-edge-thm} and~\ref{def:dd'}.
    
    \item\label{def:cbuff} Let $C_{\rm buff} \subseteq [n] \setminus (C_{\rm med} \cup C_{\rm diff} \cup C_{\rm huge} \cup C_{\rm main})$ be a subset of size $D'$. Note that such a subset exists since $n - |C_{\rm med} \cup C_{\rm diff} \cup C_{\rm huge} \cup C_{\rm main}| \geq n - \gamma n - 8\beta^{-1} - (1-\rho)n \geq \rho n / 2 \geq D'$ by~\eqref{eqn:hierarchy},~\ref{def:dd'},  and~\eqref{eqn:num_huge}.
    
    \item\label{def:cfinal} Let $C_{\rm final} \coloneqq [n] \setminus (C_{\rm med} \cup C_{\rm diff} \cup C_{\rm huge} \cup C_{\rm main} \cup C_{\rm buff})$.
\end{enumerate}

We will use the following observations later.

\begin{enumerate}
    \item[\mylabel{T1}{\bf T1}] If $\phi$ is of Type~${\rm A}$, then for any $c \in C_{\rm buff} \cup C_{\rm final}$, we have $\phi^{-1}(c) = \varnothing$. 
\end{enumerate}

To see this, note that $\phi(\cH_{\rm med} \cup \cH_{\rm large}) \subseteq C_{\rm med} \cup C_{\rm diff} \cup C_{\rm huge} \cup C_{\rm large}$ by~\ref{def:cmed},~\ref{def:chuge}, and~\ref{def:clarge}.
If $\phi$ is of Type~${\rm A}$ then $C_{\rm large} \subseteq C_{\rm main}$ by~\ref{def:cmain}, so $(C_{\rm buff} \cup C_{\rm final}) \cap \phi(\cH_{\rm med} \cup \cH_{\rm large}) = \varnothing$ by~\ref{def:cbuff} and~\ref{def:cfinal}. Since the domain of $\phi$ is $\cH_{\rm med} \cup \cH_{\rm large}$, this shows~\ref{T1}. 

\begin{enumerate}
    \item[\mylabel{T2}{\bf T2}] If $\phi$ is of Type~${\rm B}$, then for any $c \in C_{\rm buff} \cup C_{\rm final}$, we have $\phi^{-1}(c) \subseteq \cH_{\rm large} \setminus \cH_{\rm huge}$. 
Thus,~\ref{extremal-colouring}\ref{extremal-colour-class-bound} of Theorem~\ref{large-edge-thm} implies that $|V(\phi^{-1}(c))| \leq \beta n$ for any $c \in C_{\rm buff} \cup C_{\rm final}$.
\end{enumerate}

To see this, note that $\phi(\cH_{\rm med} \cup \cH_{\rm huge}) \subseteq C_{\rm med} \cup C_{\rm diff} \cup C_{\rm huge}$ by~\ref{def:cmed} and~\ref{def:chuge}. Thus by~\ref{def:cbuff} and~\ref{def:cfinal}, $(C_{\rm buff} \cup C_{\rm final}) \cap \phi(\cH_{\rm med} \cup \cH_{\rm huge}) = \varnothing$. Since the domain of $\phi$ is $\cH_{\rm med} \cup \cH_{\rm large}$, this shows~\ref{T2}.

Also note that
\begin{align}\label{eqn:bound_prev_col}
    (1-\rho)n \leq |C_{\rm med} \cup C_{\rm diff} \cup C_{\rm huge} \cup C_{\rm main} \cup C_{\rm buff}| \leq (1-\rho + 15 \gamma^{1/2} )n,
\end{align}
since~\ref{def:dd'},~\ref{def:cmed},~\ref{def:cmain}, and~\ref{def:cbuff} together imply $|C_{\rm med}| \leq \gamma n$, $|C_{\rm main} \cup C_{\rm buff}| = D + D' \leq (1-\rho + 10\gamma^{1/2})n$, and~\eqref{eqn:num_huge} holds.

\begin{pf-step}\label{step:diff}
    Extend the colour classes in $\{\phi^{-1}(c) : c \in C_{\rm diff}\}$ using Lemma~\ref{difficult-edge-absorption}.
\end{pf-step}

In this step, for each $c \in C_{\rm diff}$, we extend the colour class $\phi^{-1}(c)$ to cover every vertex of $V^{(n-1)}(\cH)$ and all but at most five vertices of $V^{(n-2)}(\cH)$, by  using only edges of $G$. 

If $C_{\rm diff} = \varnothing$, we let $\cM_{\rm diff} \coloneqq \varnothing$. Otherwise if $C_{\rm diff} \ne \varnothing$, then by~\ref{def:cdiff}, 
$C_{\rm diff} = \{c_{\rm diff} \}$, $\phi$ is of Type~${\rm A}$, and $\phi^{-1}(c_{\rm diff}) = \{e\}$ for some huge edge $e$. Applying Lemma~\ref{difficult-edge-absorption} with $\phi^{-1}(c_{\rm diff})$ playing the role of $M$, either $\chi'(\cH) \leq n$ or we have a set $\cM_{\rm diff} \coloneqq \{ M_{c_{\rm diff}} \}$ such that the following holds.
\begin{enumerate}
    \item[\mylabel{D1}{\bf D1}] 
    If $C_{\rm diff} \ne \varnothing$, then $C_{\rm diff} = \{ c_{\rm diff} \}$. Moreover, $M_{c_{\rm diff}} \supseteq \phi^{-1}({c_{\rm diff}})$, $M_{c_{\rm diff}} \setminus \phi^{-1}({c_{\rm diff}}) \subseteq E(G)$, $M_{c_{\rm diff}}$ covers every vertex of $V^{(n-1)}(\cH)$, and $|V^{(n-2)}(\cH) \setminus V(M_c)| \leq 5$.
\end{enumerate}

Let us define
\begin{equation}\label{def:r1s1}
    \textrm{$R_1 \coloneqq R_{\rm abs} \setminus \bigcup_{c \in C_{\rm diff}}M_c$ and $S_1 \coloneqq S \setminus \bigcup_{c \in C_{\rm diff}}(V^{(n-2)}(\cH) \setminus V(M_c))$.}    
\end{equation}

Combining ~\ref{RES2} and Observation~\ref{obs:del_absorber} with the fact that $\Delta(R_{\rm abs} \setminus R_1) \leq |C_{\rm diff}| \leq 1$, we have the following.
\begin{align}
&\text{$R_1$ is a $(\rho_{\rm abs} , 10\xi , \xi , \eps)$-absorber for $\cV$, so it is also a $(\rho_{\rm abs} , 3\gamma/2 , \xi , \eps)$-absorber for $\cV$.}\label{eq:R1abs}\\
&\text{If $|U| > (1-10\eps)n$, then $|S_1| \overset{\ref{D1}}{\geq} |S| - 5 \overset{\eqref{eqn:rs}}{\geq} 2 \eps n > (\eps + 3 \gamma/2)n$.}\label{eq:boundS1}
\end{align}

\begin{pf-step}
\label{step:step5}
    Extend the colour classes in $\{ \phi^{-1}(c) : c \in C_{\rm huge} \cup C_{\rm med} \}$ using Lemma~\ref{huge-edge-absorption-lemma}.
\end{pf-step}

In this step, for each $c \in C_{\rm huge} \cup C_{\rm med}$, we extend the colour class $\phi^{-1}(c)$ to cover all but at most one vertex of $U$, by  using only edges in $R_1 \subseteq R_{\rm abs}$. 
Combining \eqref{eqn:num_huge} and the fact that $|C_{\mathrm{med}}| \leq \gamma n$, we have
\begin{equation}\label{eqn:size_hm}
    |C_{\rm huge} \cup C_{\rm med}| \leq 3\gamma n / 2.
\end{equation}

First suppose $c \in C_{\rm huge}$. The colour $c$ is assigned to a huge edge by~\ref{def:chuge}. We claim that $|V(\phi^{-1}(c)) \cap U| \leq \eps n$.  Indeed, if $\phi$ is of Type~${\rm A}$, then $\phi^{-1}(c)$ contains exactly one edge by  \ref{non-extremal-colouring}\ref{non-extremal-huge-edge} of Theorem~\ref{large-edge-thm}, so $|V(\phi^{-1}(c)) \cap U| \leq \eps n$ by~\ref{EP1} of Claim~\ref{claim:easyobs}. Otherwise, if $\phi$ is of Type~${\rm B}$, by~\ref{extremal-colouring}\ref{extremal-huge-edge} of Theorem~\ref{large-edge-thm}, we have $|V(\phi^{-1}(c))| \leq \delta n < \eps n$, so again $|V(\phi^{-1}(c)) \cap U| \leq \eps n$. Moreover, as noted in~\eqref{eq:R1abs}, $R_1$ is a $(\rho_{\rm abs} , 3\gamma/2 , \xi , \eps)$-absorber for $\cV$, $\phi^{-1}(c) \subseteq \cH \setminus \cH_{\rm small} \subseteq \cH \setminus R_1$, 
and $U \cup V(\phi^{-1}(c)) , U \setminus V(\phi^{-1}(c)), U, V(\cH) \in \cV$ by \eqref{def:cv}, so $(\cH , \phi^{-1}(c) , R_1 , S_1)$ is $(\rho_{\rm abs} , \eps , 3\gamma/2 , \kappa , \xi)$-absorbable by typicality of $R_1$ if $c \in C_{\rm huge}$. Moreover, if $c \in C_{\rm huge}$ then $\phi^{-1}(c)$ is not difficult by~\ref{def:cdiff} and~\ref{def:chuge}.

Now suppose $c \in C_{\rm med}$. Then $|V(\phi^{-1}(c))| \leq \gamma n$ by~\ref{def:cmed},~\ref{non-extremal-colouring}\ref{non-extremal-medium-bound}, and~\ref{extremal-colouring}\ref{extremal-medium-bound} of Theorem~\ref{large-edge-thm}. So, again, since $R_1$ is a $(\rho_{\rm abs} , 3\gamma/2 , \xi , \eps)$-absorber for $\cV$, since $U, V(\cH) \in \cV$ by~\eqref{def:cv}, and since $\phi^{-1}(c) \subseteq \cH \setminus \cH_{\rm small} \subseteq \cH \setminus R_1$, 
it follows that $(\cH , \phi^{-1}(c) , R_1 , S_1)$ is $(\rho_{\rm abs} , \eps , 3\gamma/2 , \kappa , \xi)$-absorbable by smallness of $\phi^{-1}(c)$ if $c \in C_{\rm med}$. Moreover, since~\eqref{eq:boundS1} and~\eqref{eqn:size_hm} hold, we can\COMMENT{we apply Lemma~\ref{huge-edge-absorption-lemma} with $3\gamma/2$ playing the role of $\gamma$. So we need that the number of matchings is at most $3 \gamma n/2$, which holds by \eqref{eqn:size_hm}. Moreover, since $R_1$ is a $(\rho_{\rm abs} , 10\xi , \xi , \eps)$-absorber, it is also a $(\rho_{\rm abs} , 10(3\gamma/2) , \xi , \eps)$-absorber by $\xi \ll \gamma$.} 
apply Lemma~\ref{huge-edge-absorption-lemma} with $\{ \phi^{-1}(c) : c \in C_{\rm huge} \cup C_{\rm med} \}$, $R_1$, $S_1$, $\rho_{\rm abs}$, and $3\gamma/2$ playing the roles of $\cN$, $R$, $S$, $\rho$, and $\gamma$ respectively, to obtain the set $\cM_{\rm hm} \coloneqq \{ M_c : c \in C_{\rm huge} \cup C_{\rm med} \}$ of pairwise edge-disjoint matchings in $\cH$ such that the following hold.
\begin{enumerate}
    \item[\mylabel{HM1}{\bf HM1}] For each $c \in C_{\rm huge} \cup C_{\rm med}$, we have $M_c \supseteq \phi^{-1}(c)$ and $M_c \setminus \phi^{-1}(c) \subseteq R_1 \subseteq R_{\rm abs}$. Thus, $M_c$ is edge-disjoint from the matchings in $\cM_{\rm diff}$ by~\ref{D1} and~\eqref{def:r1s1}.
    
    \item[\mylabel{HM2}{\bf HM2}] If $|U| \leq (1-10\eps)n$, then $\cM_{\rm hm}$ has perfect coverage of $U$. Otherwise, $\cM_{\rm hm}$ has nearly-perfect coverage of $U$ with defects in
    $S_1$.\COMMENT{Note that $S_1$ may contain vertices of degree at most $n-3$, as we only removed some vertices with degree $n-2$ from $S$ to define $S_1$. This implies that the vertices with degree at most $n-3$ are marked/uncovered when extending the difficult matching, and they are allowed to be marked again when dealing with non-difficult huge edges.}
\end{enumerate}

Now let us define
\begin{equation}\label{def:r2s2}
    R_2 \coloneqq R_1 \setminus \bigcup_{c \in C_{\rm huge} \cup C_{\rm med}} M_c \overset{\eqref{def:r1s1}}{=} R_{\rm abs} \setminus \bigcup_{M \in \cM_{\rm diff} \cup \cM_{\rm hm}}M \text{  and  }S_2 \coloneqq S_1 \setminus \bigcup_{c \in C_{\rm huge} \cup C_{\rm med}}(U \setminus V(M_c)).
\end{equation}

By~\eqref{eq:R1abs},~\eqref{eqn:size_hm},~Observation~\ref{obs:del_absorber}(ii), and~\ref{def:dd'}, the following hold.
\begin{align}
 &\text{$R_2$ is a $(\rho_{\rm abs} , 2\gamma , \xi , \eps)$-absorber for $\cV$.}\label{eq:R2abs}\\
 &\text{If $|U| >  (1-10\eps)n$, then } |S_2| \overset{\eqref{eqn:size_hm},\ref{HM2}}{\geq} |S_1| - 3 \gamma n/2 \overset{\eqref{eq:boundS1}}{\geq} |S| - 5 - 3 \gamma n/2 \overset{\eqref{eqn:hierarchy},\eqref{eqn:rs}}{\geq} D + 4\eps n. \label{eq:boundS2}
\end{align}

\begin{pf-step}
\label{step:step6}
    Colour most of the edges of $\cH_{\rm small} \setminus (R_{\rm res} \cup \bigcup_{M \in \cM_{\rm diff}}M )$ by extending the  colour classes in $\{\phi^{-1}(c) : c \in C_{\rm main} \}$ using Lemma~\ref{lem:nibble2}.
\end{pf-step}

Note that $\cH_{\rm small} \setminus (R_{\rm res} \cup \bigcup_{M \in \cM_{\rm diff}}M )$ does not contain any edge of the matchings in $\cM_{\rm diff} \cup \cM_{\rm hm}$ by~\ref{HM1} and since $R_{\rm abs} \subseteq R_{\rm res}$ by~\ref{RES2}. 
In this step, we will first colour most of the edges in $\cH_{\rm small} \setminus (R_{\rm res} \cup \bigcup_{M \in \cM_{\rm diff}}M )$ with colours from $C_{\rm main}$ by extending the colour classes in $\{ \phi^{-1}(c) : c \in C_{\rm main} \}$, and the resulting colour classes are further extended by  using only edges of $R_2 \subseteq R_{\rm abs}$. To do this, we will use Lemma~\ref{lem:nibble2}. (Note that after this step there are only a few remaining uncoloured edges in $\cH_{\rm small} \setminus (R_{\rm res} \cup \bigcup_{M \in \cM_{\rm diff}}M )$ incident to each vertex, which will be coloured in the next step.)

To be able to apply Lemma~\ref{lem:nibble2}, we need to first embed $\cH_{\rm small} \setminus (R_{\rm res} \cup \bigcup_{M \in \cM_{\rm diff}}M )$ into an almost-regular linear multi-hypergraph $\cH^*$. In order to define $\cH^*$, let $\cH'$ be the linear multi-hypergraph obtained by applying Lemma~\ref{lem:degree-lemma} (Regularising lemma) with $\cH$, $R_{\rm res}$, and $\beta/2$ playing the roles of $\cH$, $R_{\rm res}$, and $\beta$, respectively. In particular, $\cH'$ is a linear multi-hypergraph obtained from $\cH_{\rm small} \setminus R_{\rm res}$
by adding singleton edges such that $\cH' \subseteq \cH_{\rm reg}$ (which is defined in Definition~\ref{def:hreg}). Now let $\cH^* \coloneqq \cH' \setminus \bigcup_{M \in \cM_{\rm diff}} M$. Then it is clear that $\cH^*$ can be obtained from $\cH_{\rm small} \setminus (R_{\rm res} \cup \bigcup_{M \in \cM_{\rm diff}}M )$ by adding singleton edges, so we have\COMMENT{we wrote $2 \beta$ below simply because we will make $2 \beta$ to play the role of $\beta$ later on. We can change it if you'd like.} the following.\COMMENT{Applying Lemma~\ref{lem:degree-lemma} we get $d_{\cH'}(w) = (1-\rho)(n-1 \pm \beta n) = (1 \pm 3\beta/2)D$, since $D = \lfloor (1-\rho)(n-1) \rfloor$. Since $\cH^*$ is a linear multi-hypergraph obtained from $\cH'$ by removing at most one edge incident to each vertex, we have $d_{\cH^*}(w) = (1 \pm 2\beta)D$.}
(Recall that $\cH$ has no singleton edges, assumed in the very first paragraph of the proof.)
\begin{equation}\label{eqn:h*_inside}
    \textrm{$\cH^* = \cH' \setminus \bigcup_{M \in \cM_{\rm diff}} M \subseteq \cH_{\rm reg} \setminus R_{\rm res} \subseteq \cH_{\rm reg} \setminus R_2$, and $d_{\cH^*}(w) = (1 \pm 2\beta)D$ for any $w \in V(\cH)$.}
\end{equation}

Now we want to apply Lemma~\ref{lem:nibble2} with $\cH_{\rm reg}$, $\cH^*$, $\{ \phi^{-1}(c) : c\in C_{\rm main} \}$, $R_2$, $S_2$, $r_1$, $\rho_{\rm abs}$, $2\beta$, and $2 \gamma$ playing the roles of $\cH$ , $\cH'$, $\cM$, $R$, $S$, $r$, $\rho$, $\beta$, and $\gamma$, respectively. To that end, we need to check that the  assumptions~\ref{C1}--\ref{C4} of Lemma~\ref{lem:nibble2} are satisfied.

First, note that $|C_{\rm main}| = D$ by~\ref{def:cmain} and~\eqref{eq:boundS2} implies that if $|U| >  (1-10\eps)n$, then $|S_2| \geq D + 2 \gamma n$, so \ref{C1} holds. By \eqref{eq:R2abs} and \eqref{def:cv}, \ref{C2} holds. 
Since every edge $e \in \cH^*$
satisfies $|e| \leq r_1$,~\ref{C3} follows from~\eqref{eqn:h*_inside}. Lastly, we show that~\ref{C4} holds. By~\ref{def:cmain}, 
for any $c \in C_{\rm main}$, $\phi^{-1}(c)$ is either empty or is contained in $\cH_{\rm large}$, where $\cH_{\rm large} \subseteq \cH_{\rm reg}$ and $\cH_{\rm large} \cap (\cH^* \cup R_2) = \varnothing$ from the definition of $\cH^*$, so $\phi^{-1}(c) \subseteq \cH_{\rm reg} \setminus (\cH^* \cup R_2)$\COMMENT{since $\cH^* \cap \cH_{\rm large} = \varnothing$}.  Moreover, by (\ref{large-edge-thm}:a)(iii) and (\ref{large-edge-thm}:b)(iii) of Theorem~\ref{large-edge-thm}, $|V(\phi^{-1}(c))| \leq \beta n \leq 2 \beta D$. Furthermore, by~\ref{EP2} of Claim~\ref{claim:easyobs}, for any $e \in \cH^*$, $| \{c \in C_{\rm main} : e \cap V(\phi^{-1}(c)) \ne \varnothing \} | \leq 2 r_1 n/r_0 \leq 2\beta D$, as desired.

Thus, we can apply Lemma~\ref{lem:nibble2} to obtain a set $\cM_{\rm main}^* \coloneqq \{ M_c^* : c \in C_{\rm main} \}$ of $D$ edge-disjoint matchings in $\cH_{\rm reg}$ such that the following hold.

\begin{enumerate}
    \item[\mylabel{MA1_old}{\bf MA1*}] For any $c \in C_{\rm main}$, we have $M_c^* \supseteq \phi^{-1}(c)$ and $M_c^* \setminus \phi^{-1}(c) \subseteq \cH^* \cup R_2 \subseteq \cH^* \cup R_{\rm abs}$. Thus, $M_c^*$ is edge-disjoint from all the matchings in $\cM_{\rm diff} \cup \cM_{\rm hm}$ by~\ref{HM1} and~\eqref{def:r2s2}.
    \COMMENT{This follows since $\cH^*$ is disjoint from the matching in $\cM_{\rm diff}$, medium/large edges, and $R_{\rm res}$, by the definition of $\cH^*$. Thus~\ref{HM1} and the definition of $R_2$ together  imply that $M_c^*$ is edge-disjoint from all the matchings in $\cM_{\rm hm}$ as well.}
    \item[\mylabel{MA2_old}{\bf MA2*}] For any $w \in V(\cH)$, 
    \begin{itemize}
        \item[(i)] $|E_{R_{\rm abs}}(w) \cap \bigcup_{c \in C_{\rm main}} M_c^*| =\COMMENT{By the definitions of $R_1$ and $R_2$, we have $R_2 = R_{\rm abs} \setminus \bigcup_{c \in C_{\rm diff} \cup C_{\rm huge} \cup C_{\rm med}} M_c$. Thus by~\ref{MA1_old}, the equality holds.}
        |E_{R_{2}}(w) \cap \bigcup_{c \in C_{\rm main}} M_c^*| \leq 2 \gamma D$, and
        \item[(ii)] $|E_{\cH^*}(w) \setminus \bigcup_{c \in C_{\rm main}} M_c^*| \leq 2 \gamma D$.
    \end{itemize}

    \item[\mylabel{MA3}{\bf MA3*}] If $|U| \leq (1-2\eps)n$, then $\cM_{\rm main}^*$ has perfect coverage of $U$. Otherwise, $\cM_{\rm main}^*$ has nearly-perfect coverage of $U$ with defects in $S_2$.
\end{enumerate}

For every $c \in C_{\rm main}$, let $M_c$ be obtained from $M_c^*$ by removing all singleton edges, and let $\cM_{\rm main} \coloneqq \{ M_c : c \in C_{\rm main} \}$. Then, since $\cM_{\rm main}^*$ is a set of matchings in $\cH_{\rm reg} \supseteq \cH$, and $\cH$ is obtained from $\cH_{\rm reg}$ by removing all singleton edges, $\cM_{\rm main}$ is a set of matchings in $\cH$. Since we obtain $\cH_{\rm small} \setminus (R_{\rm res} \cup \bigcup_{M \in \cM_{\rm diff}}M )$ after removing all singleton edges from $\cH^*$,~\ref{MA1_old} and~\ref{MA2_old} immediately imply~\ref{MA1} and~\ref{MA2} stated below, since all hypergraphs in~\ref{MA1} and~\ref{MA2} are subhypergraphs of $\cH$, and $\cH$ has no singleton edge.
Let 
\begin{align*}
    \cH_{\rm rem} \coloneqq \cH_{\rm small} \setminus (R_{\rm res} \cup \bigcup_{M \in \cM_{\rm diff} \cup \cM_{\rm main}}M) = \cH_{\rm small} \setminus (R_{\rm res} \cup \bigcup_{c \in C_{\rm diff} \cup C_{\rm main}} M_c).
\end{align*}

\begin{enumerate}
    \item[\mylabel{MA1}{\bf MA1}] For any $c \in C_{\rm main}$, $M_c$ is edge-disjoint from all the matchings in $\cM_{\rm diff} \cup \cM_{\rm hm}$, $M_c \supseteq \phi^{-1}(c)$, and $M_c \setminus \phi^{-1}(c) \subseteq \left ( \cH_{\rm small} \setminus (R_{\rm res} \cup \bigcup_{M \in \cM_{\rm diff}}M ) \right ) \cup R_2$.

    \item[\mylabel{MA2}{\bf MA2}]
    \begin{itemize}
        \item[(i)] For any $w \in V(\cH)$, $|E_{R_{\rm abs}}(w) \cap \bigcup_{c \in C_{\rm main}} M_c| = |E_{R_{2}}(w) \cap \bigcup_{c \in C_{\rm main}} M_c| \leq 2 \gamma D$, and
        \item[(ii)] $\Delta(\cH_{\rm rem}) \leq 2 \gamma D$.
    \end{itemize}
\end{enumerate}

Now let us define
\begin{align}\label{def:r3}
    R_3 \coloneqq R_2 \setminus \bigcup_{c \in C_{\rm main}} M_c^* = R_2 \setminus \bigcup_{c \in C_{\rm main}} M_c \overset{\eqref{def:r2s2}}{=} R_{\rm abs} \setminus \bigcup_{M \in \cM_{\rm diff} \cup \cM_{\rm hm} \cup \cM_{\rm main}} M \overset{\ref{RES2}}{\subseteq} R_{\rm res},
\end{align}
where the first equality holds since $M_c$ is obtained from $M_c^*$ by removing all singleton edges and $R_2 \subseteq \cH$ does not contain any singleton edge. Let us also define
\begin{align}\label{def:s3}
    S_3 \coloneqq S_2 \setminus \bigcup_{c \in C_{\rm main}} (U \setminus V(M_c^*)).
\end{align}

Since $|C_{\rm main}| = D$, 
\begin{equation}\label{eq:boundS3}
    \textrm{if $|U| > (1-10\eps)n$, then $|S_3| \overset{\ref{MA3}}{\geq} |S_2| - D \overset{\eqref{eq:boundS2}}{\geq} 4 \eps n$.}    
\end{equation}
Moreover, by~\ref{MA2_old}(i), $\Delta(R_2 \setminus R_3) \leq 2 \gamma D$. Thus by Observation~\ref{obs:del_absorber}(ii) and~\eqref{eq:R2abs},
\begin{equation}
\label{eq:R3abs}
\text{$R_3$ is a $(\rho_{\rm abs} , 4\gamma , \xi , \eps)$-absorber for $\cV$.}
\end{equation}

\begin{pf-step}
\label{step:step7}
    Finish colouring the remaining uncoloured edges of $\cH_{\rm small} \setminus R_{\rm res}$ by extending the colour classes in $\{\phi^{-1}(c) : c \in C_{\rm buff} \}$ using Lemma~\ref{lem:leftover}.
\end{pf-step}

Recall that $\cH_{\rm rem} = \cH_{\rm small} \setminus (R_{\rm res} \cup \bigcup_{c \in C_{\rm diff} \cup C_{\rm main}} M_c)$ and $R_{\rm abs} \subseteq R_{\rm res}$ by~\ref{RES2}. Thus, by~\ref{HM1}, 
\begin{align}\label{eqn:hrem_prop}
    \text{$\cH_{\rm rem}$ does not contain any edge of the matchings in $\cM_{\rm diff} \cup \cM_{\rm hm} \cup \cM_{\rm main}$.}
\end{align}

In this step, we will first colour all the edges of $\cH_{\rm rem}$ with colours from the `buffer' set $C_{\rm buff}$ by extending the colour classes in $\{\phi^{-1}(c) : c \in C_{\rm buff} \}$, and the resulting colour classes are further extended to cover all but at most one vertex of $U$, by using only edges of $R_3 \subseteq R_{\rm abs}$. To do this, we want to apply Lemma~\ref{lem:leftover} with $C_{\rm buff}$, $\{ \phi^{-1}(c) : c\in C_{\rm buff} \}$, $r_1$, $\rho_{\rm abs}$, $R_3$ and $10 \gamma^{1/2}$ playing the roles of $C$, $\cM$, $r$, $\rho$, $R$ and $\gamma$, respectively. So now we check that the assumptions~\ref{L1}--\ref{L5} of Lemma~\ref{lem:leftover} are satisfied.

First, $|C_{\rm buff}| = D' = \lfloor 10 \gamma^{1/2} D \rfloor$, so~\ref{L1} holds. 
For each $c \in C_{\rm buff}$, by~\ref{T1} and~\ref{T2}, the set  $\phi^{-1}(c)$ is either empty or is contained in $\cH_{\rm large} \setminus \cH_{\rm huge}$, 
and \COMMENT{Here we used $\beta \ll \gamma \ll 1$} $|V(\phi^{-1}(c))| \leq \beta n \leq 5\gamma^{1/2} n$. 
Thus~\ref{L2} holds. Note that $\{U, V(\cH) \} \subseteq \cV$ by~\eqref{def:cv}, and $\bigcup_{c \in C_{\rm buff}}\phi^{-1}(c) \subseteq \cH \setminus \cH_{\rm small} \subseteq \cH \setminus R_3$. 
It follows that~\ref{L3} holds by~\eqref{eq:R3abs}\COMMENT{so $R_3$ is a $(\rho_{\rm abs} , 100 \gamma^{1/2} , \xi , \eps)$-absorber as well} and Observation~\ref{obs:del_absorber}(i). By~\eqref{eq:boundS3}, if $|U| >  (1-10\eps)n$, then $|S_3| \geq 4 \eps n \geq (10 \gamma^{1/2} + \eps) n$\COMMENT{Here we used $\gamma \ll \eps \ll 1$}, so~\ref{L5} holds. Lastly, we show that~\ref{L4} holds. Note that~\ref{MA2}(ii) implies that $\Delta(\cH_{\rm rem}) \leq 2 \gamma D \leq (10 \gamma^{1/2})^2 D / 20$. Moreover, since for each $c \in C_{\rm buff}$, $\phi^{-1}(c)$ is either empty or is contained in $\cH_{\rm large}$, and $R_3 \subseteq R_{\rm abs} \subseteq R_{\rm res}$ by~\ref{RES2}, we have $\cH_{\rm rem} \subseteq \cH \setminus (R_3 \cup \bigcup_{c \in C_{\rm buff}} \phi^{-1}(c))$, and for any $e \in \cH_{\rm rem} \subseteq \cH_{\rm small}$, we have $| \{c \in C_{\rm buff} : e \cap V(\phi^{-1}(c)) \ne \varnothing \} | \leq 2 r_1 n/r_0 \leq (10\gamma^{1/2})^2 D/100$,\COMMENT{Here we used $r_0^{-1} \ll \gamma \ll 1$} by~\ref{EP2} of Claim~\ref{claim:easyobs}, as desired.

Thus, by Lemma~\ref{lem:leftover}, we obtain a set of pairwise edge-disjoint matchings $\cM_{\rm buff} \coloneqq \{M_c : c \in C_{\rm buff} \}$ such that the following hold.
\begin{enumerate}
    \item[\mylabel{B1}{\bf B1}] For any $c \in C_{\rm buff}$, we have $M_c \supseteq \phi^{-1}(c)$ and $\cH_{\rm rem} \subseteq \bigcup_{c \in C_{\rm buff}}(M_c \setminus \phi^{-1}(c)) \subseteq \cH_{\rm rem} \cup R_3$. Thus, by~\eqref{def:r3} and~\eqref{eqn:hrem_prop},
    $M_c$ is edge-disjoint from all the matchings in $\cM_{\rm diff} \cup \cM_{\rm hm} \cup \cM_{\rm main}$. Moreover, since $\cH_{\rm rem} \subseteq \cH_{\rm small} \setminus R_{\rm res}$ and by~\eqref{def:r3}, 
$\cH_{\rm small} \setminus R_{\rm res} \subseteq \bigcup_{c \in C_{\rm diff} \cup C_{\rm main} \cup C_{\rm buff}}(M_c \setminus \phi^{-1}(c)) \subseteq \cH_{\rm small}$.
    
    \item[\mylabel{B2}{\bf B2}] If $|U| \leq (1-10\eps)n$, then $\cM_{\rm buff}$ has perfect coverage of $U$. Otherwise, $\cM_{\rm buff}$ has nearly-perfect coverage of $U$ with defects in $S_3$.
\end{enumerate}

Now we combine all the matchings constructed previously. Let us define
\begin{equation*}
\cM_{\rm prev}^* \coloneqq \cM_{\rm diff} \cup \cM_{\rm hm} \cup \cM_{\rm main}^* \cup \cM_{\rm buff} \text{ and } \cM_{\rm prev} \coloneqq \cM_{\rm diff} \cup \cM_{\rm hm} \cup \cM_{\rm main} \cup \cM_{\rm buff},
\end{equation*}
where both $\cM_{\rm prev}^*$ and $\cM_{\rm prev}$ consist of pairwise edge-disjoint matchings by~\ref{HM1},~\ref{MA1_old}, \ref{MA1}, and~\ref{B1}. 
Let us also define
\begin{equation}\label{def:finalsets}
    R_{\rm final} \coloneqq R_{\rm res} \setminus \bigcup_{M \in \cM_{\rm prev}} M,\:\:\:\:G_{\rm final} \coloneqq (V(\cH) , R_{\rm final}),\:\:\:\text{and }\:\:\:\cH_{\rm final} \coloneqq R_{\rm final} \cup \bigcup_{c \in C_{\rm final}} \phi^{-1}(c).
\end{equation}

In the next step, we will show that the edges of $G_{\rm final}$ are the only remaining edges of $\cH$ which are not contained in either the matchings of $\cM_{\rm prev}$ or the colour classes of $\phi$ (see~\ref{eqn:h_final}). 
Thus, in the final step, we will colour the edges of $G_{\rm final}$ with colours in $C_{\rm final}$ in such a way that the colouring is compatible with $\phi$.  Prior to this, in {Step}~\ref{step:step9}, we show that the degrees in $\cH_{\rm final}$ are appropriately bounded.

\begin{pf-step}
\label{step:step8}
    Analyse properties of $G_{\rm final}$ and $\cH_{\rm final}$.

\end{pf-step}

In this step we will prove the following properties of $G_{\rm final}$ and $\cH_{\rm final}$. 
\begin{enumerate}[label = {{\bf{F}\arabic*}}]
    \item\label{eqn:bound_prev_colours}
    $(1-\rho)n \leq |\cM_{\rm prev}^*| = |\cM_{\rm prev}| = n - |C_{\rm final}| \leq (1-\rho + 15 \gamma^{1/2} )n.$
    \item\label{stat:r_rem}
    If $(\cH , \phi)$ is of Type~${\rm A_2}$, then $\Delta(G_{\rm final}) - \delta(G_{\rm final}) \leq 22\eps n$, and $G_{\rm final}$ is lower $(\rho/4, \eps^{1/3})$-regular.
    \item\label{eqn:h_final}
    $\cH_{\rm final} = \cH \setminus \bigcup_{M \in \cM_{\rm prev}} M = \cH \setminus \bigcup_{M \in \cM_{\rm prev}^*} M$. Moreover, if $\phi$ is of Type~${\rm A}$, $\cH_{\rm final} = R_{\rm final}$.
\end{enumerate}

First, since $\cM_{\rm prev} = \{M_c : c \in C_{\rm med} \cup C_{\rm diff} \cup C_{\rm huge} \cup C_{\rm main} \cup C_{\rm buff} \}$ and $C_{\rm final} = [n] \setminus (C_{\rm med} \cup C_{\rm diff} \cup C_{\rm huge} \cup C_{\rm main} \cup C_{\rm buff})$,~\ref{eqn:bound_prev_colours} follows from~\eqref{eqn:bound_prev_col}.

Now we prove~\ref{stat:r_rem}. First we show that for any $w \in V(\cH)$,
\begin{equation}\label{eqn:deficit}
    d_{R_{\rm res} \setminus R_{\rm final}}(w) \leq 15 \gamma^{1/2} n.
\end{equation}

Indeed, by~\eqref{def:finalsets},
we have
\begin{align*}
    d_{R_{\rm res} \setminus R_{\rm final}}(w) \leq |\cM_{\rm diff} \cup \cM_{\rm hm} \cup \cM_{\rm buff}| + \left |E_{R_{\rm res}}(w) \cap \bigcup_{c \in C_{\rm main}} M_c \right | \overset{\ref{MA1},~\ref{MA2}(i)}{\leq} |\cM_{\rm diff} \cup \cM_{\rm hm} \cup \cM_{\rm buff}| + 2 \gamma n,
\end{align*}
so it suffices to bound $|\cM_{\rm diff} \cup \cM_{\rm hm} \cup \cM_{\rm buff}|$. By~\eqref{eqn:size_hm}, $|\cM_{\rm hm}| \leq 3\gamma n / 2$. Together with the facts that $|\cM_{\rm diff}| \leq 1$ and $|\cM_{\rm buff}| = D' \leq 10\gamma^{1/2}n$ (by ~\ref{def:cbuff}), we obtain 
~\eqref{eqn:deficit} as desired.

Now we show that  $\Delta(G_{\rm final}) - \delta(G_{\rm final}) \leq 22 \eps n$. Indeed, since $\xi \ll \eps$, $R_{\rm final} \subseteq R_{\rm res}$ by~\eqref{def:finalsets}, and $R_{\rm res}$ is a $(\rho , \xi , \eps , \cV)$-reservoir of Type $A_2$, \ref{R1} and~\ref{R2} of Definition~\ref{def:pseudoreg} imply that for each $w \in V(\cH)$,
\begin{equation*}
    (\rho - 20\eps)n \leq d_{R_{\rm res}}(w) = d_{R_{\rm final}}(w) + d_{R_{\rm res} \setminus R_{\rm final}}(w) \leq (\rho + \xi)n.
\end{equation*}
Thus, $\Delta(G_{\rm final}) \leq (\rho + \xi)n$, and by~\eqref{eqn:deficit}, we have $d_{R_{\rm final}}(w) \geq d_{R_{\rm res}}(w) - 15 \gamma^{1/2} n \geq (\rho - 20 \eps)n - 15 \gamma^{1/2} n \geq (\rho - 21 \eps)n$ for any $w \in V(\cH)$, so $\delta(G_{\rm final}) \geq (\rho - 21\eps)n$. Thus, $\Delta(G_{\rm final}) - \delta(G_{\rm final}) \leq (21 \eps + \xi)n \le 22 \eps n$, as desired. 

Now we show that $G_{\rm final}$ is lower $(\rho/4 , \eps^{1/3})$-regular. 
Since $(\cH , \phi)$ is of Type~${\rm A_2}$, $\cH$ is $(\rho,\eps)$-full, so $|V(\cH) \setminus U| \leq 10\eps n$. Thus, by~\ref{LR2} (with $G_{\rm final}[U]$, $G_{\rm final}$, $\rho/3$, $\eps^{1/2}$, and $10^{-4}$ playing the roles of $G$,$H$,$\rho$,$\xi$, and $\alpha$, respectively), it suffices to show that $G_{\rm final}[U]$ is lower $(\rho/3 , \eps^{1/2})$-regular.\COMMENT{We may apply~\ref{LR2} with a small $\alpha$, for example, $\alpha = 1/10000$, and $\rho/3, \eps^{1/2}$ playing the roles of $\rho, \xi$.  Indeed, $v(G_{\rm final})/v(G_{\rm final}[U]) = n/|U| \leq n/(n - 10\eps n) \leq 1 + 100 \eps \leq 1 + \alpha \eps^{1/2}$.}

To that end, note that by~\ref{RES2}, $R_{\rm abs}$ is a lower $(\rho/2 , \xi , G')$-regular subset of $R_{\rm res}$, so~\eqref{eqn:deficit} and~\ref{LR1}\COMMENT{Note that if $d_{R_{\rm res} \setminus R_{\rm final}}(w)\leq 15 \gamma^{1/2} n$ then, $d_{R_{\rm abs} \setminus R_{\rm final}}(w)\leq 15 \gamma^{1/2} n$ since $R_{\rm abs} \subseteq R_{\rm res}$.} imply that $R_{\rm abs} \cap R_{\rm final}$ is lower $(\rho/2 , \gamma^{1/5} , G')$-regular.
Moreover, for any two disjoint sets $A,B \subseteq U$ with $|A|,|B| \geq \eps^{1/2}|U| \geq \eps^{1/2}n/2$,\COMMENT{here we used $|U| \geq (1-10\eps)n$ as $\cH$ is $(\rho,\eps)$-full} we have $e_{G'}(A,B) \geq (|A| - \eps n)|B| \geq (1 - 2 \eps^{1/2})|A||B|$, since each vertex in $B$ is adjacent to all but at most $\eps n$ vertices in $A$ by the definition of $U$.
Therefore, for any two disjoint sets $A,B \subseteq U$ with $|A|,|B| \geq \eps^{1/2}|U| \geq \gamma^{1/5} n$, we have
\begin{equation*}
    e_{G_{\rm final}[U]} (A,B) \geq e_{R_{\rm abs} \cap R_{\rm final}}(A,B) \geq 
    \frac{\rho}{2} e_{G'}(A,B) - \gamma^{1/5} |A||B| \geq \frac{\rho }{2} (1 - 2 \eps^{1/2})|A||B| - \gamma^{1/5} |A||B| \geq \frac{\rho}{3}|A||B|,
\end{equation*}
so $G_{\rm final}[U]$ is lower $(\rho/3 , \eps^{1/2})$-regular, completing the proof of~\ref{stat:r_rem}.

Finally, we prove~\ref{eqn:h_final}. Since $\cH = \cH_{\rm small} \cup (\cH_{\rm med} \cup \cH_{\rm large})$ and $\cH_{\rm final} = R_{\rm final} \cup \bigcup_{c \in C_{\rm final}} \phi^{-1}(c)$, in order to show that $\cH_{\rm final} = \cH \setminus \bigcup_{M \in \cM_{\rm prev}}M$, it suffices to prove the following two statements.
\begin{equation}\label{eqn:h_final_eq}
    \cH_{\rm small} \setminus \bigcup_{M \in \cM_{\rm prev}} M = R_{\rm final}\:\:\:\:\:\text{ and }\:\:\:\:\:(\cH_{\rm med} \cup \cH_{\rm large}) \setminus \bigcup_{M \in \cM_{\rm prev}} M = \bigcup_{c \in C_{\rm final}}\phi^{-1}(c).
\end{equation}

The first statement of~\eqref{eqn:h_final_eq} directly follows from the definition of $R_{\rm final}$ and the fact that $\cH_{\rm small} \setminus R_{\rm res} \subseteq \bigcup_{M \in \cM_{\rm prev}} M$ which is guaranteed by \ref{B1}. 
Now we show the second statement of~\eqref{eqn:h_final_eq}. 
To that end, first note that~\ref{D1},~\ref{HM1},~\ref{MA1},~\ref{B1}, and~\ref{def:cfinal} together imply that for each $c \in [n] \setminus C_{\rm final}$, we have $\phi^{-1}(c) \subseteq M_c \in \cM_{\rm prev}$ and $M_c \setminus \phi^{-1}(c) \subseteq \cH_{\rm small}$. 
Thus, $\bigcup_{c \in [n]\setminus C_{\rm final}} \phi^{-1}(c) \subseteq \bigcup_{M \in \cM_{\rm prev}} M \overset{\ref{eqn:bound_prev_colours}}{=} \bigcup_{c \in [n]\setminus C_{\rm final}}M_c \subseteq \bigcup_{c \in [n] \setminus C_{\rm final}} \phi^{-1}(c) \cup \cH_{\rm small}$, where the right-hand side does not contain any edges of $\bigcup_{c \in C_{\rm final}} \phi^{-1}(c)$.
Since $\cH_{\rm med} \cup \cH_{\rm large} = \bigcup_{c \in [n]} \phi^{-1}(c) = \bigcup_{c \in [n] \setminus C_{\rm final}} \phi^{-1}(c) \cup \bigcup_{c \in C_{\rm final}} \phi^{-1}(c)$, we have
\begin{align*}
    (\cH_{\rm med} \cup \cH_{\rm large}) \setminus \bigcup_{M \in \cM_{\rm prev}} M = \Big( \bigcup_{c \in [n] \setminus C_{\rm final}} \phi^{-1}(c) \setminus \bigcup_{M \in \cM_{\rm prev}} M \Big) \cup \bigcup_{c \in C_{\rm final}} \phi^{-1}(c) = \bigcup_{c \in C_{\rm final}}\phi^{-1}(c),
\end{align*}
proving the second statement of~\eqref{eqn:h_final_eq}. 

Note that~\eqref{eqn:h_final_eq} shows that
$\cH_{\rm final} = \cH \setminus \bigcup_{M \in \cM_{\rm prev}}M$. Since $\cH$ has no singleton edges, it immediately follows that $\cH \setminus \bigcup_{M \in \cM_{\rm prev}}M = \cH \setminus \bigcup_{M \in \cM_{\rm prev}^*} M$. Moreover, if $\phi$ is of Type~${\rm A}$, then by~\ref{T1}, $\bigcup_{c \in C_{\rm final}}\phi^{-1}(c) = \varnothing$, so $\cH_{\rm final} = R_{\rm final}$. This completes the proof of~\ref{eqn:h_final}.

\begin{pf-step}
\label{step:step9}
    Bound the degrees of vertices in $\cH_{\rm final}$.
\end{pf-step}

In this step, we prove the following statements which bound the degrees of vertices in $\cH_{\rm final}$ in terms of the number of colours available for the final step.
(In particular, when $\phi$ is of Type B, to apply Lemma~\ref{hall-based-finishing-lemma} in the final step, one needs to bound the degrees of vertices in $\cH_{\rm final}$ instead of $G_{\rm final}$ and take into account how many large edges in $\bigcup_{c \in C_{\rm final}} \phi^{-1}(c)$ are incident to each vertex.)

\begin{enumerate}
    \item[\mylabel{UC1}{\bf UC1}] For any $w \in V(\cH) \setminus U$, $d_{\cH_{\rm final}}(w) \leq |C_{\rm final}| - \rho \eps n/4$.
    
    \item[\mylabel{UC2}{\bf UC2}] For any $w \in U \setminus V^{(n-1)}(\cH)$, $d_{\cH_{\rm final}}(w) \leq |C_{\rm final}| - 1$.
    
    \item[\mylabel{UC3}{\bf UC3}] For any $w \in V^{(n-1)}(\cH)$, $d_{\cH_{\rm final}}(w) \leq |C_{\rm final}|$. Moreover, there are at least $|C_{\rm final}|$ vertices of degree less than $|C_{\rm final}|$ in $\cH_{\rm final}$.
    
    \item[\mylabel{UC4}{\bf UC4}] If $(\cH , \phi)$ is either Type~${\rm A_1}$ or~${\rm B}$, then for any $w \in V^{(n-1)}(\cH)$, $d_{\cH_{\rm final}}(w) \leq |C_{\rm final}| - 1$.
\end{enumerate}

Now we prove~\ref{UC1}. First we show that for any $w \in V(\cH) \setminus U$, $d_{R_{\rm res}} (w) \leq \rho(1-\eps)n + \xi n$.  Indeed, if $R_{\rm res}$ is a $(\rho, \xi , \eps , \cV)$-reservoir of Type~${\rm A_1}$ or~${\rm A_2}$, then it follows from~\ref{P1} or \ref{R2}. Otherwise, if $R_{\rm res}$ is a $(\rho, \xi , \eps , \cV)$-reservoir of Type~${\rm B}$, then it also follows since $R_{\rm res}$ is $(\rho, \xi , G')$-typical with respect to $\mathcal V \ni V(\cH)$ (see Definitions~\ref{def:typicality},~\ref{def:absorber},~\ref{def:reservoir_type}, and~\eqref{def:cv}).
Thus, for any $w \in V(\cH) \setminus U$, since $\xi,\gamma \ll \eps,\rho$,\COMMENT{For the first inequality where we used \ref{eqn:bound_prev_colours} below, the calculation is as follows. Since $\xi, \gamma \ll \rho, \eps$, we have $\rho(1-\eps)n + \xi n \leq (\rho - 15 \gamma^{1/2})n + (15 \gamma^{1/2} + \xi - \rho \eps)n \leq n - |\cM_{\rm prev}| - \rho \eps n / 2$.} 
\begin{align}
    d_{G_{\rm final}} (w) &\leq d_{R_{\rm res}} (w) \leq \rho(1-\eps)n + \xi n
    \leq (\rho - 15\gamma^{1/2} - \rho \eps / 2)n \nonumber \\
    \overset{~\ref{eqn:bound_prev_colours}}&{\leq}
    n - |\cM_{\rm prev}| - \rho \eps n / 2 \overset{~\ref{eqn:bound_prev_colours}}{=} |C_{\rm final}| - \rho \eps n / 2. \label{boundingGfinal}
\end{align}    

Since $w$ is incident to at most $2n/r_0 \leq \rho \eps n / 4$ edges of $\cH_{\rm large}$, and $\cH_{\rm final} \setminus R_{\rm final} = \bigcup_{c \in C_{\rm final}} \phi^{-1}(c) \subseteq \cH_{\rm large}$ by~\ref{T1} and~\ref{T2}, we deduce that $d_{\cH_{\rm final} \setminus R_{\rm final}}(w) \leq \rho \eps n / 4$. This together with \eqref{boundingGfinal} implies that for any $w \in V(\cH) \setminus U$, $d_{\cH_{\rm final}}(w) = d_{G_{\rm final}}(w) + d_{\cH_{\rm final} \setminus R_{\rm final}}(w) \leq |C_{\rm final}| - \rho \eps n / 4$, proving \ref{UC1}.

Before proving~\ref{UC2},~\ref{UC3}, and~\ref{UC4}, we need to collect some facts. For any $w \in V(\cH)$, let $m(w)$ be the number of the matchings in $\cM_{\rm prev}^*$ not covering $w$. Since $\cM_{\rm prev}^*$ is a set of edge-disjoint matchings in $\cH_{\rm reg}$, and $\cH_{\rm reg}$ is a linear multi-hypergraph obtained from $\cH$ by adding $\max(0 , n - 3 - d_\cH(w))$ singleton edges incident to each vertex $w \in V(\cH)$ (where $\cH$ has no singleton edges), we deduce that all but at most $m(w) + \max(0 , n-3-d_\cH(w))$ matchings in $\cM_{\rm prev}$ cover $w$. Moreover, by \ref{eqn:h_final}, $\cH_{\rm final} = \cH \setminus \bigcup_{M \in \cM_{\rm prev}}M$, so for any $w \in V(\cH)$ we have\COMMENT{The equality below follows by using $|\cM_{\rm prev}| = n - |C_{\rm final}|$ and $d_{\cH}(w) - n + \max(0,n-3-d_\cH(w)) = \max(d_\cH(w) - n , -3) = -\min(n - d_\cH(w) , 3)$.}
\begin{align}
    d_{\cH_{\rm final}}(w) &\leq d_{\cH}(w) - (|\cM_{\rm prev}| - m(w) - \max(0,n-3-d_{\cH}(w))) \nonumber \\
    \overset{\ref{eqn:bound_prev_colours}}&{=} d_\cH(w) - n + |C_{\rm final}| + m(w) + \max(0 , n-3-d_\cH(w)) \nonumber \\
    &= |C_{\rm final}| + m(w) - \min(n - d_{\cH}(w) , 3).\label{eqn:deg_h_final}
\end{align}

Recall that $S_1 = S \setminus \bigcup_{c \in C_{\rm diff}}(V^{(n-2)}(\cH) \setminus V(M_c))$ by~\eqref{def:r1s1}. Note that~\ref{D1} implies that every matching in $\cM_{\rm diff}$ covers all of the vertices in $V^{(n-1)}(\cH)$, 
$\cM_{\rm hm}$ has nearly-perfect coverage of $U$ with defects in $S_1$ by~\ref{HM2}, $\cM_{\rm main}^*$ has nearly-perfect coverage of $U$ with defects in $S_2$ by~\ref{MA3}, where $S_2$ is the set of vertices in $S_1$ contained in all matchings in $\cM_{\rm hm}$ by~\eqref{def:r2s2}, and by~\ref{B2}, $\cM_{\rm buff}$ has nearly-perfect coverage of $U$ with defects in $S_3$, where $S_3$ is the set of vertices in $S_2$ contained in all matchings in $\cM_{\rm main}^*$ by~\eqref{def:s3}. Thus,
\begin{equation}\label{eqn:nearly_perfect}
    \text{$\cM_{\rm hm} \cup \cM_{\rm main}^* \cup \cM_{\rm buff} = \cM_{\rm prev}^* \setminus \cM_{\rm diff} $ has nearly-perfect coverage of $U$ with defects in $S_1$.}
\end{equation}

In particular, every matching in $\cM_{\rm prev}^* \setminus \cM_{\rm diff}$ covers all the vertices in $U \setminus S_1$, every matching in $\cM_{\rm diff}$ covers all the vertices in $V^{(n-1)}(\cH) \setminus S_1$, and $|\cM_{\rm diff}| \leq 1$. Thus, we have
\begin{equation}\label{eqn:mbound_1}
    \text{$m(u) \leq 1$ for any $u \in U \setminus S_1$, and $m(v) = 0$ for any $v \in V^{(n-1)}(\cH) \setminus S_1$.}
\end{equation}

By~\eqref{eqn:nearly_perfect}, note that every vertex in $S_1$ is covered by all but at most one matching in $\cM_{\rm prev}^* \setminus \cM_{\rm diff}$, every matching in $\cM_{\rm diff}$ covers all the vertices in $V_+^{(n-2)}(\cH) \cap S_1$ by~\eqref{def:r1s1}, and $|\cM_{\rm diff}| \leq 1$. So we have $m(u) \leq 2$ for any $u \in S_1$ and $m(v) \leq 1$ for any $v \in V_+^{(n-2)}(\cH) \cap S_1$. Combining this with~\eqref{eqn:mbound_1}, we deduce that
\begin{equation}\label{eqn:mbound_2}
    \text{$m(u) \leq 2$ for any $u \in U \setminus V_+^{(n-2)}(\cH)$, and $m(v) \leq 1$ for any $v \in V_+^{(n-2)}(\cH)$.}
\end{equation}

Finally, if $|U| \leq (1-10\eps)n$ then $\cM_{\rm prev}^* \setminus \cM_{\rm diff}$ has perfect coverage of $U$ by~\ref{HM2},~\ref{MA3}, and~\ref{B2}. This combined with the fact that every matching in $\cM_{\rm diff}$ covers all the vertices in $V^{(n-1)}(\cH)$ and $|\cM_{\rm diff}| \leq 1$, implies that 
\begin{equation}\label{eqn:mbound_3}
    \text{if $|U| \leq (1-10\eps)n$, then $m(u) \leq 1$ for $u \in U \setminus V^{(n-1)}(\cH)$, and $m(v) = 0$ for $v \in V^{(n-1)}(\cH)$.}
\end{equation}

Now we are ready to prove~\ref{UC2},~\ref{UC3}, and~\ref{UC4}. Note that \eqref{eqn:mbound_2} and~\eqref{eqn:deg_h_final} together imply that for any $v \in U \setminus V^{(n-1)}(\cH)$, $d_{\cH_{\rm final}}(v) \leq |C_{\rm final}| - 1$ , thus~\ref{UC2} holds.

Now we prove~\ref{UC3}. Note that \eqref{eqn:mbound_2} and~\eqref{eqn:deg_h_final} together imply that for any $v \in V^{(n-1)}(\cH)$, $d_{\cH_{\rm final}}(v) \leq |C_{\rm final}|$. To prove the second statement of \ref{UC3}, we will bound the number of vertices $v \in V^{(n-1)}(\cH)$ satisfying $m(v) = 0$. Since $\cM_{\rm prev}^* \setminus \cM_{\rm diff}$ has nearly-perfect coverage of $U$ with defects in $S_1$, every matching in $\cM_{\rm prev}^* \setminus \cM_{\rm diff}$ covers all but at most one vertex in $U$. Moreover, every matching in $\cM_{\rm diff}$ covers all of the vertices in $V^{(n-1)}(\cH)$. Thus, by~\ref{eqn:bound_prev_colours}, there are at most $|\cM_{\rm prev}| = n - |C_{\rm final}|$ vertices $v \in V^{(n-1)}(\cH)$ satisfying $m(v) \geq 1$, so every other vertex of $V^{(n-1)}(\cH)$ has degree less than $|C_{\rm final}|$ by~\eqref{eqn:deg_h_final}.  
This fact combined with~\ref{UC1} and~\ref{UC2} implies that there are at least $|C_{\rm final}|$ vertices with degree less than $|C_{\rm final}|$ in $\cH_{\rm final}$, proving \ref{UC3}.

Finally, if $(\cH , \phi)$ is of Type~${\rm A_1}$ then $S_1 \subseteq S = U \setminus V^{(n-1)}(\cH)$ by~\eqref{def:s}, so by \eqref{eqn:mbound_1} every vertex $v \in V^{(n-1)}(\cH)$ satisfies $m(v) = 0$. On the other hand, if $(\cH , \phi)$ is of Type~${\rm B}$, then  by~\ref{EP3} of Claim~\ref{claim:easyobs}, $|U| \leq 2 \delta n$, so again every vertex $v \in V^{(n-1)}(\cH)$ satisfies $m(v) = 0$ by \eqref{eqn:mbound_3}. Thus, in either case we have $d_{\cH_{\rm final}}(v) \leq |C_{\rm final}| - 1$ by~\eqref{eqn:deg_h_final}, proving~\ref{UC4}.

\begin{pf-step}
\label{step:step10}
    Colour the edges of $G_{\rm final}$ with colours in $C_{\rm final}$.
\end{pf-step}

Recall from~\ref{eqn:h_final} in Step~\ref{step:step8} that the only edges which are not in the matchings $\cM_{\rm prev}$ or are not coloured by $\phi$ are those of $G_{\rm final}$.
In this step, we colour the edges of $G_{\rm final}$ using colours in $C_{\rm final}$ in such a way that the colouring is compatible with $\phi$. Since $\cH_{\rm final} = R_{\rm final} \cup \bigcup_{c \in C_{\rm final}} \phi^{-1}(c)$, it suffices to show that there is a set $\cM_{\mathrm{final}} = \{M_c : c \in C_{\mathrm{final}}\}$ of edge-disjoint matchings in $\cH_{\mathrm{final}}$ such that for every $c\in C_{\mathrm{final}}$, $\phi^{-1}(c) \subseteq M_c$ and $\bigcup_{c\in C_{\mathrm{final}}} M_c = \cH_{\mathrm{final}}$. Indeed, then $\cH = \bigcup_{M \in \cM_{\rm prev} \cup \cM_{\rm final}} M$ by~\ref{eqn:h_final}, so $\cM_{\rm prev} \cup \cM_{\rm final}$ would be the desired set of $n$ pairwise edge-disjoint matchings in $\cH$, proving Theorem~\ref{main-thm}. We divide the proof into three cases depending on the type of $(\cH , \phi)$.\\

\noindent \textit{Case 1}: $(\cH , \phi)$ is of Type~${\rm A_1}$. 

Note that in this case $\cH_{\rm final} = R_{\rm final}$ by~\ref{eqn:h_final}, so $R_{\rm final}$ is precisely the set of edges not coloured so far. Also recall that $G_{\rm final} = (V(\cH), R_{\rm final})$. Thus, by~\ref{UC1},~\ref{UC2}, and~\ref{UC4}, $\Delta(G_{\rm final}) \leq |C_{\rm final}| - 1$. Applying Vizing's theorem (Theorem~\ref{thm:vizing}) to $G_{\rm final}$, we obtain a set $\cM_{\rm final} = \{ M_c : c \in C_{\rm final} \}$ of edge-disjoint matchings in $G_{\rm final}$ such that $\bigcup_{c \in C_{\rm final}} M_c = R_{\rm final} = \cH_{\mathrm{final}}$. Moreover, by~\ref{T1}, for any $c \in C_{\rm final}$, $\phi^{-1}(c) = \varnothing$, so  $\phi^{-1}(c) \subseteq M_c$ trivially holds, as desired. This completes the proof of Theorem~\ref{main-thm} in the case when $(\cH , \phi)$ is of Type~${\rm A_1}$.\\

\noindent \textit{Case 2}: $(\cH , \phi)$ is of Type~${\rm A_2}$. 

Note that in this case again $\cH_{\rm final} = R_{\rm final}$ by~\ref{eqn:h_final}. Thus, by \ref{UC1},~\ref{UC2}, and~\ref{UC3},  $\Delta(G_{\rm final}) \leq |C_{\rm final}|$, and 
\begin{equation}\label{eqn:count_maxdeg}
    \textrm{there are at least $|C_{\rm final}|$ vertices having degree less than $|C_{\rm final}|$ in $G_{\rm final}$.}
\end{equation}

If $\Delta(G_{\rm final}) \leq |C_{\rm final}| - 1$, then we may apply Vizing's theorem (Theorem~\ref{thm:vizing}) to $G_{\rm final}$ to obtain the desired set $\cM_{\rm final} = \{ M_c : c \in C_{\rm final} \}$ of edge-disjoint matchings where for any $c\in C_{\mathrm{final}}$, $\phi^{-1}(c) \subseteq M_c$, and $\bigcup_{c \in C_{\rm final}} M_c = R_{\rm final} = \cH_{\mathrm{final}}$ as in the previous case. 

Otherwise, if $\Delta(G_{\rm final}) = |C_{\rm final}|$, then by~\eqref{eqn:count_maxdeg},~\ref{stat:r_rem}, and since $\eps = \eps_1 \ll \rho = \rho_1$, we can apply\COMMENT{we used $\rho \gg \eps$ in this case} Corollary~\ref{cor:optcol} with $\rho/4$, $\eps^{1/3}$, $22\eps$, and $G_{\rm final}$ playing the roles of $p$, $\eps$, $\eta$, and $G$, respectively, to obtain a set $\cM_{\rm final} = \{ M_c : c \in C_{\rm final} \}$ of edge-disjoint matchings such that $\bigcup_{c \in C_{\rm final}} M_c = R_{\rm final} = \cH_{\mathrm{final}}$. Moreover, by~\ref{T1},  for any $c \in C_{\rm final}$, $\phi^{-1}(c) = \varnothing$, so  $\phi^{-1}(c) \subseteq M_c$ trivially holds, as desired. This completes the proof of Theorem~\ref{main-thm} in the case when $(\cH , \phi)$ is of Type~${\rm A_2}$.\\

\noindent \textit{Case 3}: $(\cH , \phi)$ is of Type~${\rm B}$. 

First, for each $w \in V(\cH)$, let us define $C_w \coloneqq \{c \in C_{\rm final} : w \in V(\phi^{-1}(c)) \}$. 
By~\eqref{def:finalsets} and since the domain of $\phi$ is $\cH_{\rm med} \cup \cH_{\rm large}$ while $R_{\rm final} \subseteq R_{\rm res} \subseteq E(G)$, we have $\cH_{\rm final} \setminus R_{\rm final} = \bigcup_{c \in C_{\rm final}} \phi^{-1}(c)$. Moreover, by~\ref{T2}, for any $c \in C_{\rm final}$, $\phi^{-1}(c) \subseteq \cH_{\rm large}$. Thus, for any vertex $w \in V(\cH)$, we have
\begin{equation}\label{eqn:boundc_w}
    |C_w| = d_{\cH_{\rm final} \setminus R_{\rm final}}(w) \leq 2 \delta n,
\end{equation}
since $\phi$ is a proper edge-colouring, and $w$ is incident to at most $2n / r_0 \leq 2 \delta n$ edges of $\cH_{\rm large}$. Note that~\ref{UC1},~\ref{UC2}, and~\ref{UC4} imply $\Delta(\cH_{\rm final}) \leq |C_{\rm final}| - 1$, so for any vertex $w \in V(\cH)$, we have
\begin{equation}\label{eqn:bound_deg_all}
    d_{R_{\rm final}}(w) = d_{\cH_{\rm final}}(w) - d_{\cH_{\rm final}\setminus R_{\rm final}}(w) \overset{\eqref{eqn:boundc_w}}{\leq} |C_{\rm final}| - 1 - |C_w|.
\end{equation}

Now we apply Lemma~\ref{hall-based-finishing-lemma} with $G_{\rm final}$, $C_{\rm final}$, and $2 \delta$ playing the roles of $H$, $C$, and $\delta$, respectively. To that end, we need to check that the assumptions~\ref{cond:hall1}--\ref{cond:hall4} of Lemma~\ref{hall-based-finishing-lemma} are satisfied. First, by~\ref{eqn:bound_prev_colours}, $|C_{\rm final}| \geq (\rho - 15 \gamma^{1/2})n \geq 14\delta n$. 
By~\eqref{eqn:bound_deg_all}, \ref{cond:hall1} of Lemma~\ref{hall-based-finishing-lemma} holds. Now, by~\ref{EP3} of Claim~\ref{claim:easyobs}, $|U| \leq 2 \delta n$, and by Definition~\ref{def:absorber}\ref{cond:absorber1} and~\ref{RES2}, $R_{\rm final} \subseteq R_{\rm res} = R_{\rm abs} \subseteq E(G')$. Thus, every edge of $G_{\rm final}$ is incident to a vertex of $U$ by the definition of $G'$ given in {Step~\ref{step:step1}}, so \ref{cond:hall2} of Lemma~\ref{hall-based-finishing-lemma} holds. By~\eqref{eqn:boundc_w}, \ref{cond:hall3} of Lemma~\ref{hall-based-finishing-lemma} holds. Finally, by \ref{T2}, for any $c \in C_{\rm final}$, $|\{w \in V(\cH) : c \in C_w \}| = |V(\phi^{-1}(c))| \leq \beta n \leq 2 \delta n$, so \ref{cond:hall4} of Lemma~\ref{hall-based-finishing-lemma} holds.

Hence, by applying Lemma~\ref{hall-based-finishing-lemma}, we obtain a proper edge-colouring $\psi : R_{\rm final} \to C_{\rm final}$ such that every $uv \in R_{\rm final}$ satisfies $\psi(uv) \notin C_u \cup C_v$, implying that $V(\psi^{-1}(c)) \cap V(\phi^{-1}(c)) = \varnothing$ for every $c \in C_{\rm final}$. Now let $M_c \coloneqq \psi^{-1}(c) \cup \phi^{-1}(c)$ for each $c \in C_{\rm final}$. Then the set $\cM_{\rm final} \coloneqq \{ M_c : c \in C_{\rm final} \}$ consists of pairwise edge-disjoint matchings\COMMENT{This easily follows since $\psi^{-1}(c)$ and $\phi^{-1}(c)$ are colour classes.} such that $\phi^{-1}(c) \subseteq M_c$ for each $c \in C_{\rm final}$, and $\bigcup_{c \in C_{\rm final}} M_c = R_{\rm final} \cup \bigcup_{c \in C_{\rm final}} \phi^{-1}(c) = \cH_{\rm final}$, as desired. This completes the proof of Theorem~\ref{main-thm}.
\end{proof}

\section*{Acknowledgements}
We are grateful to the referees for a detailed reading of the manuscript and thoughtful comments which were very helpful in improving the presentation.

\bibliographystyle{amsabbrv}
\bibliography{ref}

\vspace{1cm}

{\footnotesize \obeylines \parindent=0pt
Dong Yeap Kang, Tom Kelly, Daniela K\"{u}hn, Abhishek Methuku, Deryk Osthus
\vspace{0.3cm}
School of Mathematics, University of Birmingham, Edgbaston, Birmingham, B15 2TT, UK}
\vspace{0.3cm}

{\footnotesize \parindent=0pt
\begin{flushleft}
{\it{E-mail addresses}:}
\tt{[d.y.kang.1,d.kuhn,d.osthus]@bham.ac.uk, tom.kelly@gatech.edu, abhishekmethuku@gmail.com}
\end{flushleft}
}

\end{document}